\documentclass[11pt]{article}
\usepackage{amsfonts,mathrsfs}
\usepackage{graphicx}
\usepackage{amsmath,mathabx}
\usepackage{amsthm}
\usepackage{mathtools}
\usepackage{amssymb}
\usepackage{xparse}
\usepackage{color}

\usepackage{amsfonts}
\usepackage{dsfont}
\DeclareMathOperator{\eqlaw}{\stackrel{\Lc}{=}}

\usepackage{tabulary}
\usepackage{booktabs}

\usepackage{todonotes}

\usepackage[colorlinks,citecolor=blue]{hyperref}

\newtheorem{theorem}{Theorem}
\newtheorem{conjecture}[theorem]{Conjecture}
\newtheorem{lemma}[theorem]{Lemma}
\newtheorem{proposition}[theorem]{Proposition}
\newtheorem{corollary}[theorem]{Corollary}

\theoremstyle{definition} 
\newtheorem{definition}[theorem]{Definition}
\newtheorem{remark}[theorem]{Remark}
\newtheorem{rmk}[theorem]{Remark}

\numberwithin{theorem}{section}
\numberwithin{equation}{section}

\newcommand{\Z}{{\mathbb Z}}
\newcommand{\T}{{\mathbb T}}
\newcommand{\R}{{\mathbb R}}
\newcommand{\Q}{{\mathbb Q}}
\newcommand{\C}{{\mathbb C}}
\newcommand{\N}{{\mathbb N}}

\newcommand{\hohc}[1]{[ #1 )}

\newcommand{\Exp}{\mathbb{E}}
\newcommand{\E}{\mathbb{E}}
\newcommand{\prob}{\mathbb{P}}
\newcommand{\filt}{\mathscr{F}}
\newcommand{\Gfilt}{\mathscr{H}}

\renewcommand{\Pr}{\prob}
\DeclareDocumentCommand \one { o }
{%
  \IfNoValueTF {#1}
  {\mathbf{1}  }
  {\mathbf{1}\left\{ {#1} \right\} }%
}

\newcommand{\Var}{\operatorname{Var}}

\newcommand{\Beta}{\operatorname{Beta}}
\newcommand{\Exponential}{\operatorname{Exp}}

\newcommand{\Gam}{\operatorname{Gamma}}

\newcommand{\Unif}{\operatorname{Unif}}
\newcommand{\As}{\operatorname{a.s.}}
\newcommand{\lawequals}{\overset{\mathscr{L}}{=}}
\DeclareDocumentCommand{\Prto} {o} {
  \IfNoValueTF {#1}
  {\overset{\Pr}{\longrightarrow}}
  { \xrightarrow[ #1 \to \infty]{\Pr }}
}
\DeclareDocumentCommand{\Asto} {o} {
  \IfNoValueTF {#1}
  {\overset{\operatorname{a.s.}}{\longrightarrow}}
  {
    \xrightarrow[ #1 \to \infty]{\operatorname{a.s.} }
  }
}
\DeclareDocumentCommand{\Mgfto} {o} {
  \IfNoValueTF {#1}
  {\overset{\operatorname{mgf}}{\longrightarrow}}
  { \xrightarrow[ #1 \to \infty]{\operatorname{mgf} }}
}

\DeclareDocumentCommand{\Wkto} {o} {
  \IfNoValueTF {#1}
  {\overset{(d)}{\longrightarrow}}
  { \xrightarrow[ #1 \to \infty]{(d) }}
}

\newcommand{\Ex}{\operatorname{Ex}}
\newcommand{\Ext}{\operatorname{Ext}}
\newcommand{\Extr}{\operatorname{Extr}}
\newcommand{\Extre}{\operatorname{Extre}}

\newcommand{\corO}[1]{\textcolor{blue}{ {#1} }}
\author{Elliot Paquette\thanks{Department of Mathematics, McGill University, Montreal, Canada.  Partially supported by a
NSERC discovery grant. {\tt email:elliot.paquette@mcgill.ca}.}\; and Ofer Zeitouni\thanks{Department of Mathematics, Weizmann Institute of Sciences, Rehovot, Israel.  Partially supported by the European Research Council (ERC) under the European Union's Horizon 2020 research and innovation programme (grant agreement No. 692452), and by the Israel Science Foundation grant number 421/20. {\tt email:ofer.zeitouni@weizmann.ac.il}.}}
\title{The extremal landscape for the C$\beta$E ensemble}

\begin{document}
\maketitle
\begin{abstract}
  We consider the extremes of the  logarithm of the characteristic polynomial of matrices from the C$\beta$E ensemble. We prove convergence in distribution of the centered maxima (of the real and imaginary parts) towards the sum of a Gumbel variable and another independent variable,  which we characterize as the total mass of a ``derivative martingale''. We also provide a description of the landscape near extrema points.
\end{abstract}
\section{Introduction}
The Circular-$\beta$ ensemble (C$\beta$E) is a distribution on $n$ points $(e^{i\omega_1},e^{i\omega_2}, \dots, e^{i\omega_n})$ on the unit circle with a joint density given by
\begin{equation}
\label{eq-density}
\frac{1}{Z_{n,\beta}} \prod_{1 \leq j < k \leq n} | e^{i \omega_j} - e^{i \omega_k} |^{\beta} d \omega_1 \cdots d \omega_n.
\end{equation}
In the special case of $\beta=2$ this is the joint distribution of eigenvalues of a Haar-distributed unitary random matrix.  The characteristic polynomial $X_n(z) \coloneqq \prod_{j=1}^n (1-e^{i\omega_j} z)$ of the C$\beta$E has attracted a considerable interest, for its connections to the theories of logarithmically--correlated fields and (when $\beta=2$) analytic number theory.

A particular quantity of interest is $M_n\coloneqq\max_{|z|=1}
\log |X_n(z)|$. Let
\begin{equation}
\label{eq-mn}
 m_n = \log n - \tfrac{3}{4}\log\log n.
\end{equation}
The random matrix
part of the Fyodorov--Hiary--Keating conjecture \cite{FHK12}
states that in the special case that $\beta=2$, $M_n-m_n$ converges in distribution
towards a limiting random variables $R_2$, with
\begin{equation}
  \label{eq-R2}
  P(R_2\in dx)=4e^{2x} K_0(2e^{x})dx.
\end{equation}
It was later observed in \cite{SZ} that the probability density in \eqref{eq-R2} is the law of the sum of two independent Gumbel random variables.

For general $\beta>0$, an important step forward was obtained by
\cite{CMN}, who proved that $M_n-\sqrt{2/\beta}m_n$ is tight. One of our  main results
strengthens this to convergence in distribution, and gives a description
of the limiting law.
\begin{theorem}
  \label{theo-main}
  The sequence of random variables $M_n-\sqrt{2/\beta}m_n$ converges in distribution
  to a random variables $R_\beta$. Further,
  \begin{equation}
    \label{eq-convolution}
    R_\beta=
  C_\beta+G_\beta+\frac{1}{\sqrt{2\beta}}\log (\mathscr{B}_\infty(\beta)) ,
\end{equation}
where $C_\beta$ is an (implicit) constant, $G_\beta$ is Gumbel distributed with
parameter $1/\sqrt{2\beta}$, and
$\mathscr{B}_\infty(\beta)$ is a random variable that is  independent
of $G_\beta$.
\end{theorem}
\begin{remark}
We identify below, see Theorem \ref{thm:Bj}, $\mathscr{B}_\infty(\beta)$
as the total mass of a certain
\textit{derivative martingale}.  For a specific log-correlated field on the circle, \cite{Remy} computes the law of the total mass of the associated GMC and confirms the Fyodorov-Bouchaud prediction
\cite{FBo}
for it. It is possible (and even anticipated, 
especially in light of \cite{CN}) but not proved, that the distribution of $\mathscr{B}_\infty$ is also Gumbel. If true (even if only for $\beta=2$),
Theorem \ref{theo-main} would then yield a proof of the
random matrix side of the Fyodorov-Hiary-Keating conjecture \cite{FHK12}.
\end{remark}

Theorem \ref{theo-main} is a consequence of a more general result,
which gives the convergence of the distance between 
 certain marked point processes
built from a sequence of orthogonal polynomials, and a sequence of
($n$-independent) decorated
Poisson point process. This general result
also applies to the imaginary part of $\log X_n(z)$ (and thus, allows for  control on maximal fluctuation of eigenvalue count on intervals).
We postpone a discussion of
Theorem \ref{theo-main} and a historical context of our results
to after the introduction of the necessary preliminaries and
the statement of our more general results.

\subsection{OPUC preliminaries and formulation of main results}
A major advance in the study of $M_n$ was achieved in \cite{CMN}, who used
the Orthogonal Polynomials on the Unit Circle (OPUC) representation of the  C$\beta$E measure due to \cite{KillipNenciu}; we refer to \cite{SimonBook} for an encyclopedic account of the OPUC theory.
Let $\left\{ \gamma_k \right\}$ be independent, complex, rotationally invariant random variables for which $|\gamma_k|^2 = \Beta(1, \beta(k+1)/2),$ that is with density on $[0,1]$ proportional to $(1-x)^{\beta(k+1)/2-1}$.  
The \emph{Szeg\H{o}} recurrence is, for all $k \geq 0,$
\begin{equation}
  \begin{pmatrix}
    \Phi_{k+1}(z) \\
    \Phi^*_{k+1}(z)
  \end{pmatrix}
  \coloneqq
  \begin{pmatrix}
    z & -\overline{\gamma_k} \\
    -\gamma_k z & 1
  \end{pmatrix}
  \begin{pmatrix}
    \Phi_{k}(z) \\
    \Phi^*_{k}(z)
  \end{pmatrix},
  \quad
  \biggl\{
    \begin{aligned}
      &\Phi_0(z) \equiv 1, \\
      &\Phi_k^*(z) = z^k \overline{ \Phi_k(1/\overline{z})}.
  \end{aligned}
    \biggr\}
  \label{eq:szego}
\end{equation}
where $\Phi_{k}^{*}$ and $\Phi_{k}$ are polynomials of degree at most $k$.
Define in terms of these coefficients the \emph{Pr\"ufer phases}
\begin{equation}\label{eq:Psi}
  \Psi_{k+1}(\theta)=\Psi_{k}(\theta) + \theta - 2\Im\left( \log(1-\gamma_ke^{i\Psi_k(\theta)})\right)
  ,\quad \Psi_0(\theta) = \theta,
\end{equation}
where here and below 
we take the principal branch of the logarithm with discontinuity along
the negative real line. Then,
$\Psi_k(\cdot)$
may be identified as a continuous version of the logarithm of $\theta \mapsto 
\tfrac{1}{i}\log( e^{i\theta} \tfrac{\Phi_k(e^{i\theta})}{\Phi_k^*(e^{i\theta})})$ (see \cite[Lemma 2.3]{CMN}).
These polynomials $\{\Phi_{k}^{*}\}$ can be used to give an effective representation for the characteristic polynomial $X_n$ by setting $\alpha$ to be a uniformly distributed element of the unit circle, independent of $\left\{ \gamma_k : k \geq 0 \right\},$ and setting for any $\theta \in \R,$
\begin{equation}\label{eq:charpoly}
  X_n(e^{i\theta})
  \coloneqq \Phi_{n-1}^*(e^{i\theta}) - \alpha e^{i\theta} \Phi_{n-1}(e^{i\theta})
  =\Phi_{n-1}^*(e^{i\theta})\bigl( 1-\alpha e^{i \Psi_{n-1}(\theta)}\bigr).
\end{equation}
See \cite[(2.2)]{CMN} for details of the demonstration, based on \cite[Proposition B.2]{KillipNenciu}, that this indeed has the law of the characteristic polynomial of C$\beta$E.

  The polynomials $\left\{ \Phi_k^* \right\}$ satisfy the recurrence
  \begin{equation}\label{eq:Phi}
    \log \Phi^*_{k+1}(e^{i\theta})=\log\Phi^*_{k}(e^{i\theta}) +\log(1-\gamma_ke^{i\Psi_k(\theta)})
    ,\quad \Phi_0^*(e^{i\theta}) = 1.
  \end{equation}
  We also recall the \emph{relative Pr\"ufer phase} \cite[Lemma 2.4]{CMN} given by the recurrence
  \begin{equation}\label{eq:prufer}
    \psi_{k+1}(\theta)=\psi_{k}(\theta) + \theta - 2\Im\left( \log(1-\gamma_ke^{i\psi_k(\theta)}) - \log(1-\gamma_k) \right)
    ,\quad \psi_0(\theta) = \theta.
  \end{equation}
In law $\{\psi_{k}(\theta) : k \in \N, \theta \in [0,2\pi]  \}$ is equal to $\{\Psi_{k}(\theta)-\Psi_{k}(0) : k \in \N,\theta \in [0,2\pi] \}.$

We will be interested in the extreme values of the fluctuations of real and imaginary parts of $\log\Phi^*,$ for which reason we will formulate our results in terms of
the recurrence
\begin{equation}\label{eq:upsilon}
  \varphi_{k+1}(\theta)=\varphi_{k}(\theta) + 2\Re \big\{ \sigma \left( \log(1-\gamma_ke^{i\Psi_k(\theta)}) \right) \big\}
  ,\quad \varphi_0(\theta) = 0,
\end{equation}
where $\sigma$ is one of $\left\{ 1, \pm i\right\}.$  Then, for $\sigma=1$,
$\varphi_k(\cdot)= 2 \Re \log \Phi_k^*(\cdot)$ while
for $\sigma=i$,  $\varphi_k(\cdot)=-2\Im\log \Phi_k^*(\cdot)$.
The entire analysis of $\varphi$ occurs for a fixed $\sigma,$ and so we shall not display the dependence on $\sigma.$

\subsection*{Formulation of main results}
We already stated our main result concerning the maximum of the real part of the logarithm of the characteristic polynomial. In this section, we describe the rest of our results.

We recall the following, with $m_n = \log n - \tfrac{3}{4}\log\log n$ as in \eqref{eq-mn}.
\begin{theorem}[\cite{CMN}]
  For any $\sigma \in \left\{ 1, \pm i\right\},$ the centered maximum
  \(
  \max_{\theta \in [0,2\pi]} \varphi_n(\theta) - \sqrt{{8}/{\beta}}\, m_n
  \)
  is tight.  The same holds for the real and imaginary parts of the logarithm of the  characteristic polynomial.
  \label{thm:tightness}
\end{theorem}

We shall expand upon this result and show that indeed this maximum converges in distribution.  Moreover, we shall show that the process of almost maxima converges.
The following result, which complements Theorem \ref{theo-main}, yields the convergence in law of the centered maxima of $\varphi_n$.
\begin{theorem}\label{thm:max}
  For any $\sigma \in \left\{ 1, \pm i\right\},$ the centered maximum
  \(
  \max_{\theta \in [0,2\pi]} \varphi_n(\theta) - \sqrt{{8}/{\beta}}\, m_n
  \)
  converges in law to a randomly shifted Gumbel of parameter $\sqrt{2/\beta}$. In the notation of Theorem \ref{theo-main},
the limit is $C_\beta^\sigma+2G_\beta^\sigma+\sqrt{2/\beta} \log \mathscr{B}_\infty$, where $G_\beta^\sigma$ has the same law as $G_\beta$ and $C_\beta^\sigma$ is an (implicit)
constant.
\end{theorem}

\begin{remark}
  By the distributional identity:
  \[
    (\log \Phi_n^*(\theta) : \theta \geq 0)
    \lawequals
    ( \overline{\log \Phi_n^*(-\theta)} : \theta \geq 0),
  \]
  which follows from the conjugation invariance of the law of $\gamma_k$ and the symmetry $-2\Im \log (1-z) = 2 \Im \log(1-\overline{z}),$
  the case of $\sigma=-i$ in Theorem \ref{thm:max} follows similarly to the case of $\sigma=i$.
\end{remark}
\noindent To describe the random shift we need to introduce some machinery.

\subsection{The derivative martingale}

We will need the so-called derivative martingale.
Define the random measure and its total mass
\begin{equation}\label{eq:Bk}
  \mathscr{D}_k(\theta) d\theta
  \coloneqq
  \tfrac{1}{2\pi}
  e^{ \sqrt{ \tfrac{\beta}{2}} \varphi_k(\theta) - \log k}
  \bigl(\sqrt{2}\log k - \sqrt{\tfrac{\beta}{4}}\varphi_k(\theta) \bigr)_+
  d\theta
  \quad
  \text{and}
  \quad
  \mathscr{B}_k \coloneqq
  \int_0^{2\pi}
  \mathscr{D}_k(\theta)
  d\theta.
\end{equation}
We equip the space of finite measures with the weak-* topology, and then we show that this measure converges almost surely.
We recall, see \eqref{eq:upsilon}, that $\varphi_k(\cdot)$ coincides with
either twice the real or imaginary parts of $\log \Phi_k^*$ (depending on $\sigma$).

\begin{theorem}\label{thm:Bj}
  For any $\sigma \in \left\{1, \pm i \right\}$ and any $\beta > 0,$
  there is an almost surely finite random variable $\mathscr{B}_\infty$ and an almost surely finite, nonatomic random measure $\mathscr{D}_\infty$ so that
  \[
  \mathscr{D}_{2^j} d\theta
    \Asto[j] \mathscr{D}_\infty
    \quad
    \text{and}
    \quad
    \mathscr{B}_{2^j}
    \Asto[j] \mathscr{B}_\infty.
  \]
  Furthermore for any $\epsilon > 0$ there is a compact $K \subset (0,\infty)$ so that with
  \[
    \chi(\theta) = \one[{ (\sqrt{2}\log k - \sqrt{\tfrac{\beta}{4}}\varphi_k(\theta))/\sqrt{\log k} \not\in K },]
  \]
  it holds that for any $k \in \N$,
  \begin{equation}\label{eq:Btight}
    \Pr
    \biggl(\int_0^{2\pi}
    e^{ \sqrt{ \tfrac{\beta}{2}} \varphi_k(\theta) - \log k}\bigl|\sqrt{2}\log k - \sqrt{\tfrac{\beta}{4}}\varphi_k(\theta) \bigr| \chi(\theta) d\theta > \epsilon
    \biggr)
    <\epsilon.
  \end{equation}
\end{theorem}
\noindent This is shown in Section \ref{sec:mgle}.  We remark that $(\mathscr{B}_{2^j} : j \in \N)$ is not in fact a martingale, but it is easily compared to a process $(\widehat{\mathscr{B}_{2^j}} : j \in \N)$ which is a martingale (see Section \ref{sec:mgle} for details).
\begin{remark} 
  We do not claim that $\mathscr{B}_\infty$ is positive almost surely in Theorem \ref{thm:Bj}.  However, by combining Theorem \ref{thm:Bj} and \ref{thm:max} with tightness of the recentered maximum (Theorem \ref{thm:tightness}), we conclude that $\mathscr{B}_\infty$ must in fact be positive almost surely.
\end{remark}

\subsection{Sequential Poisson process approximation and extremal landscape}
\label{subsec-sequential}
We introduce parameters $\{k_p : p \in \N\}$ which will be chosen large but independent of $n.$  These parameters will be taken large after $n$ is sent to infinity.  Moreover, they will be ordered in a decreasing fashion, so that $k_j \gg k_{j+1}.$  We shall not attempt to find any quantitative dependence of how these parameters are sent to infinity.
All parameters will be assumed to be larger than $1.$

We formulate several  sequential extremal processes approximation for the processes of near  maxima, these extremal processes will be indexed by $k_1$.
Divide the unit circle into consecutive arcs $\{ \widehat{I_{j,n}}\}$ by the formula that for any $j,n \in \N,$
\begin{equation}\label{eq:hatij}
  \widehat{I_{j,n}} \coloneqq 2\pi\hohc{\tfrac{(j-1)k_1}{n}, \tfrac{j k_1}{n}}.
\end{equation}
To avoid cumbersome notation we  suppress the $n$ dependence in $\widehat{I_{j,n}}$, writing $\widehat{I_j}$ instead, and we continue to do so in the forthcoming
$D_{j,n},\theta_{j,n},\widehat{W}_{j,n}$.
Let ${\mathcal{D}}_{n/k_1}$ denote the collection of indices $j=1,2,\dots,\lceil \tfrac{n}{k_1}\rceil.$  
We let $\theta_j=\theta_{j,n}$ be the supremum of $\widehat{I_{j,n}}.$
Over each of these intervals, we define the process
\begin{equation}
  \begin{split}
  D_j &= D_{j,n} : [-2\pi k_1,0] \to \C,\\
  D_j(\theta) &\coloneqq
  \begin{cases}
    \bigl(\Phi_n^{*}\bigr)^2( \exp(i(\theta_j + \tfrac{\theta}{n})))\cdot\exp(-i(n+1)\theta_j {-\sqrt{\tfrac{8}{\beta}}m_n}), & \text{ if }\sigma = 1,\\
    \exp\bigl(\varphi_n( \theta_j + \tfrac{\theta}{n})-\sqrt{\tfrac{8}{\beta}}m_n\bigr), & \text{ if }\sigma = i.
  \end{cases}
\end{split}
  \label{eq:decorations}
\end{equation}
This will serve as the \emph{decoration process}, although we will not need to (and will not)  prove their convergence as $k_1\to\infty$.  
\begin{remark}
  The choice of $D_j$ in the case of $\sigma=1$ is motivated by the application to Theorem \ref{theo-main}, which concerns the characteristic polynomial.  While the characteristic polynomial $X_n$ is coarsely approximated by the OPUC $\Phi_{n-1}^*$, the coupling between them means the phase of $\Phi_{n-1}^*$ influences the modulus of $X_n$.  So, to prove Theorem \ref{theo-main}, it is insufficient to record only the modulus of $\Phi_{n-1}^*$, whereas to prove Theorem \ref{thm:max}, one could take the definition used in the case $\sigma=i$ for both.
\end{remark}

We next define for all $j \in {\mathcal{D}}_{n/k_1}$ random variables
\begin{equation}
  \begin{aligned}
    \widehat{W_j}
    =\widehat{W}_{j,n}
    \coloneqq \max_{\theta \in \widehat{I_j}} \{\varphi_n(\theta)\}
    - \sqrt{\tfrac{8}{\beta}}m_{n}
  \end{aligned}
  \label{eq:localmaxhat}
\end{equation}
which is a local maximum, appropriately centered.
We now define a random measure which we shall show is well--approximated by a Poisson process with random intensity.  Define a measure on
Borel subsets of 
\begin{equation}
  \label{eq-090924}
  \Gamma=[0,2\pi] \times \R \times \mathcal{C}( [-2\pi k_1,0], \C),
\end{equation}
by
\begin{equation}
\label{eq-Exn}
  \Ex_n
  =\Ex_n^{k_1}
  \coloneqq
  \sum_{j \in \mathcal{D}_{n/k_1}}
  \delta_{(\theta_j, \widehat{W_j}, D_j)}.
\end{equation}

A central technical challenge will be to show that $\varphi_k$ and $\Psi_k$ are essentially constant on the interval $\widehat{I_j}$ for $k \approx n/{k_1}$, and that hence it suffices to track both $\varphi_k$ and $\Psi_k$ only at the point $\theta_j \in \widehat{I_j}.$ In this direction, it will be helpful to further decompose the local maximum $\widehat{W_j}.$
We define two new parameters $k_1^+$ and $\widehat{k}_1,$ as functions of $k_1$ in such a way that $k_1^+/\widehat{k}_1, \widehat{k}_1/k_1 \to \infty,$
specifically:
\begin{equation}
  {k_1^+} = {k_1}\exp({(\log k_1)}^{(29/30)})
  \quad\text{and}\quad
  {\widehat{k}_1} = {k_1}\exp({(\log k_1)}^{(19/20)}).
  \label{eq:k2k3}
\end{equation}
We define
$n_1 = \lfloor n/k_1\rfloor,$
and we define $\widehat{n}_1$ and $n_1^+$ analogously.
Define
\begin{equation}
  \begin{aligned}
    V_j
    \coloneqq
    \sqrt{2}m_{n_1^+}
    -
    \sqrt{\tfrac{\beta}{4}}
    \varphi_{n_1^+}(\theta_j)
    .
  \end{aligned}
  \label{eq:leafheight}
\end{equation}
Define for Borel subsets of $[0,2\pi] \times (-\infty,0] \times \mathcal{C}( [-2\pi k_1,0], \C),$
\begin{equation}\label{eq:Ext}
  \Ext_n
  =\Ext_n^{k_1}
  \coloneqq
  \sum_{j \in \mathcal{D}_{n/k_1}}
  \delta_{(\theta_j, V_j, D_je^{\sqrt{4/\beta}V_j})}.
\end{equation}
Note that the processes $\Ex_n$ and $\Ext_n$ are closely related, see the proof of Theorem \ref{thm:process}.
Our goal will be to approximate the processes $\Ext_n$ (and $\Ex_n$) by Poisson processes with random intensity.

Toward this end, recall that an important strategy used throughout the analysis of extrema of branching processes is effectively conditioning on the initial portion of the process, wherein the extrema gain a nontrivial correlation.  We will do the same and condition on the first Verblunsky coefficients.  We use the parameter \label{k2} $k_2,$ which we assume is a power of $2$ (to apply Theorem \ref{thm:Bj}), to refer to how many Verblunsky coefficients on which we condition.  We also use $(\filt_{k}: k \in \N_0)$ to refer to the natural $\sigma$-algebra generated by the sequence of Verblunsky coefficients $( \gamma_k : k \in \N_0).$

We will compare $\Ext_n$  to the $\filt_{k_2}$--conditional Poisson random measures $\Pi^{k_1,k_2}$, $\Pi^{k_1,k_2,'}$, $\Pi^{k_1,k_2,''}$
with respective intensity measures
\begin{equation}\label{eq:trueintensity}
  \mathscr{D}_{k_2}(\theta)d\theta \times I(v)dv \times \mathfrak{p}_{k_1}(v,df), \quad   \mathscr{D}_{k_2}(\theta)d\theta \times I'(v)dv 
  \times \mathfrak{p}_{k_1}(v,df),
  \quad  \mathscr{D}_{k_2}(\theta)d\theta \times I''(v)dv \times \mathfrak{p}_{k_1}(v,df)
\end{equation}
where
\begin{align*}
  I(v)&= \sqrt{\tfrac{2}{\pi}}ve^{\sqrt{2}v}
  \one[{( \log k_1^+)^{1/10} \leq v \leq (\log k_1^+)^{9/10}}], \\
 I'(v)&= \sqrt{\tfrac{2}{\pi}}ve^{\sqrt{2}v}
  \one[{( 0.5 \log k_1^+)^{1/10} \leq v \leq (2\log k_1^+)^{9/10}}]\\
 I''(v)&= \sqrt{\tfrac{2}{\pi}}ve^{\sqrt{2}v}
  \one[{( 2 \log k_1^+)^{1/10} \leq v \leq (0.5 \log k_1^+)^{9/10}}].
\end{align*}
Here and in many places below, we slightly abuse notation by using
$\times$ to denote also products of transition kernels, i.e. semidirect products.
Thus $I'$ is defined over a slightly longer interval than $I$ and $I''$ is defined over a slightly shorter interval than $I$.
The law $\mathfrak{p}_{k_1}(v,\cdot)$ is that of a random function on $\theta \in (-2\pi k_1,0)$ which is related to  the exponential of the solution of a family of coupled diffusions $\mathfrak{U}^o_{T_+}(\theta)$
in an auxiliary time parameter (see \eqref{mtl:intensity} and \eqref{eq:dSDE}).

\begin{remark}
  This process $\mathfrak{U}^o_{T_+}(\theta)$ consists of
  the terminal values of a coupled family of diffusions $(t \mapsto \mathfrak{U}^o_t(\theta))$.  This coupled family of diffusions can be related to the complex stochastic sine equation \cite{ValkoVirag2} (see also the closely related stochastic sine equation \cite{ValkoVirag,KillipStoiciu}).  The extreme values of this diffusion in $\theta$ are then needed to describe $\mathfrak{p}_{k_1}$, at least for how it appears here.
\end{remark}

Similarly, the measure $\Ex_n$  will be approximated by a Poisson random measure with a random intensity on the same space.
This intensity on $[0,2\pi] \times \R \times \mathcal{C}( [-2\pi k_1,0], \C)$ will take the form of a product measure
\(
\mathscr{D}_\infty \times \widehat{\mathfrak{p}_{k_1}},
\)
where $(\widehat{\mathfrak{p}_{k_1}} :k_1 \in \N)$ is a deterministic
Radon measure on $\R \times \mathcal{C}( [-2\pi k_1,0], \C)$, which is constructed as follows. Let
\begin{equation}
\label{eq-iota}
\iota( v, f) \coloneqq \bigl(\max_{x \in [-2\pi k_1,0]} ( -\sqrt{4/\beta}v+ \log|f(x)|), f e^{-\sqrt{4/\beta}v} \bigr)
\end{equation}
 be a map of $\R\times \mathcal{C}( [-2\pi k_1,0], \C) $ to itself, and let
\begin{equation}
\label{eq-defpk1}
\widehat{\mathfrak{p}_{k_1}}(dv,df)\;  \mbox{\rm denote the push-forward of $I(v)dv\times \mathfrak{p}_{k_1}(v,df)$ by $\iota$}.
\end{equation}
We let $\Pi^{k_1}$ be a Poisson random measure on $\Gamma \coloneqq [0,2\pi] \times \R \times \mathcal{C}( [-2\pi k_1,0], \C)$ with intensity
\(
\mathscr{D}_\infty \times \widehat{\mathfrak{p}_{k_1}}.
\)
We may define similarly $\Pi^{k_1,'}$ and $\Pi^{k_1,''}$.

To compare point processes on $\Gamma$,
we endow the latter with the distance
\[
  \partial_0( (\theta_1, z_1, f_1), (\theta_2, z_2, f_2) ) \coloneqq  \bigl(d_{\T}(\theta_1,\theta_2) + |z_1-z_2| + \sup_{t \in [-2\pi k_1,0]} |f_1(t) - f_2(t)|\bigr) \wedge 1.
\]
In terms of this we define (compare with $d_1'$ from \cite{ChenXia}) a Wasserstein distance on point configurations $\xi_1 = \sum_{i=1}^m \delta_{y_i}$ and $\xi_2 = \sum_{i=1}^n \delta_{z_i}$
\[
  \partial_1(\xi_1, \xi_2) \coloneqq
  \begin{cases}
    0, & \text{if } m = n = 0, \\
    \min_{\pi} \max_{{i=1,\ldots,n}} \partial_0(y_i, z_{\pi(i)}), & \text{if } m = n > 0, \\
    1, & \text{if } m \neq n. \\
  \end{cases}
\]
with the minimum being the distance over all permutations $\pi$ of $\{1,2,\dots,n\}.$
Finally, for two point processes $Q_1$ and $Q_2$ we define the pseudometric
\[
  \partial_2( Q_1, Q_2) \coloneqq \inf_{(\xi_1,\xi_2)} \Exp( \partial_1( \xi_1, \xi_2) ),
\]
with the infimum over couplings $(\xi_1, \xi_2)$ in which $\xi_1 \sim Q_1$ and $\xi_2 \sim Q_2.$  Note that this pseudometric only depends on the laws of the point processes.
These distances are somewhat unorthodox;  we develop Poisson approximations using these distances in Appendix \ref{sec:pointprocesses}, as well as some comparisons of these distances to other more standard metrics used in Poisson approximation.

To get a comparison between $\Ext_n$ and the point processes $\Pi^{k_1,k_2}$, it is necessary to restrict to the case that the maximum of the modulus of the  decoration $|D_je^{\sqrt{4/\beta}V_j}|$ is sufficiently large.  So, we set
  \begin{eqnarray}\nonumber
  \Gamma_{k_7}& \coloneqq &\bigl( (\theta,v, f) \in [0,2\pi] \times \R \times \mathcal{C}( [-2\pi k_1,0], \C) : \max_{x \in [-2\pi k_1,0]}
  |f(x)e^{-\sqrt{4/\beta}v}| \in [e^{-k_7},e^{k_7}] \bigr),\\
\label{eq:gk7}
\Gamma_{k_7}^+& \coloneqq &\bigl( (\theta,v, f) \in [0,2\pi] \times \R \times \mathcal{C}( [-2\pi k_1,0], \C) : \max_{x \in [-2\pi k_1,0]}
|f(x)e^{-\sqrt{4/\beta}v}| \geq e^{-k_7} \bigr).
\end{eqnarray}
In what follows, for any measure (or point process) $\mathbf{Q}$ on $\Gamma$ we write $\iota\# \mathbf{Q}$ for the push forward under the transformation that preserves the first coordinate
and applies $\iota$ of \eqref{eq-iota} to the last two. Note that under $\iota\#\mathbf{Q}$, the second coordinate of the point process can always be recovered from the third.
Our main result, which will imply Theorems \ref{theo-main} and \ref{thm:max}, is the following convergence of marked processes.
Here and in the sequel, for a point process $\Pi$ and a 
set $B\subset \Gamma$, we write $\Pi\cap B$ for the restriction of $\Pi$  to $B$, i.e. $\Pi\cap B(\cdot)=\Pi(\cdot\cap B)$. 
\begin{theorem}\label{prop:bush}
  The restrictions of $\Ext_n^{k_1}$ and $\Pi^{k_1,k_2}$ to $\Gamma_{k_7}$ satisfy
  \begin{equation}
\label{eq-Pik1k2comp}
    \limsup_{k_2,k_1,n \to \infty}
\partial_2(\iota\# (\Pi^{k_1,k_2} \cap \Gamma_{k_7}), \iota\#(\Ext_n^{k_1} \cap \Gamma_{k_7}))= 0.
  \end{equation}
The same holds with $\Pi^{k_1,k_2,'}$ or $\Pi^{k_1,k_2,''}$ replacing $\Pi^{k_1,k_2}$ in \eqref{eq-Pik1k2comp}.
\end{theorem}
\noindent The meaning of the $\limsup$ is that the parameters are taken to infinity in order, with $n$ followed by $k_1$ followed by $k_2.$

Similarly,
to make a comparison between $\Pi^{k_1}$ and $\Ex_n^{k_1},$ we will only make a comparison in which their second coordinate is in a compact set. Hence we shall further restrict the space ${\Gamma}$ from \eqref{eq-090924}
to
\begin{equation}
\label{eq-defhatGammak7}
\widehat{\Gamma}_{k_7} \coloneqq [0,2\pi] \times [-k_7,k_7] \times \mathcal{C}( [-2\pi k_1,0], \C).
\end{equation}
Theorems \ref{prop:bush} and \ref{thm:Bj} lead directly to 
\begin{theorem}\label{thm:process}
  For any $k_7 > 0,$ the restrictions of the point processes to $\widehat{\Gamma}_{k_7}$ satisfy
  \begin{equation}
\label{eq-Pik1}
    \limsup_{k_1,n \to \infty}
    \partial_2\bigl(
    \Pi^{k_1} \cap \widehat{\Gamma}_{k_7},
    \Ex_n^{k_1} \cap \widehat{\Gamma}_{k_7}
    \bigr)
    =0.
  \end{equation}
The same holds with $\Pi^{k_1,'}$ or $\Pi^{k_1,''}$ replacing $\Pi^{k_1}$ in \eqref{eq-Pik1}
\end{theorem}

Equipped with Theorem \ref{thm:process}, we can now complete the proof of Theorems \ref{theo-main} and \ref{thm:max}, assuming a technical estimate contained in Corollary \ref{cor:pregularity}.
\begin{proof}[Proof of Theorems \ref{theo-main} and \ref{thm:max}]
  We give first the details for Theorem \ref{theo-main}, since its 
  proof is more involved. Using \eqref{eq:charpoly},
we have the representation for all $z \in \C$
\[
 X_{n+1}(z) = \Phi_{n}^*(z)- \alpha z \Phi_{n}(z),
\]
where we recall $\alpha$ is uniformly distributed on the unit circle and independent of $\Phi_{n}.$
We begin by proving convergence.
From Theorem \ref{thm:tightness}, as $m_{n+1}-m_{n} \to 0$ as $n \to \infty,$ the random variables
  \[
    R_{n+1} \coloneqq \max_{|z|=1} \log|X_{n+1}(z)| - \sqrt{\tfrac{2}{\beta}} m_{n}
  \]
  are tight.  Let $R_*$ be any subsequential weak limit.

  We recall that $\Phi_{n}^*(z) = z^{n}\overline{\Phi_{n}(1/\overline{z})},$ and hence on the unit circle we have
  \[
    \Phi_{n}(z)
    = z^{n}\overline{\Phi^*_{n}(z)}
    \quad \implies
    \quad
    X_{n+1}(z) = \Phi_{n}^*(z)- \alpha z^{n+1} \overline{\Phi_{n}^*(z)},
    \quad
    \text{for}
  \quad |z|=1.
  \]
  Thus, for $|z|=1$,
  we have the representation of the log-modulus of the characteristic polynomial
  \begin{equation}\label{eq:det2}
    \log |X_{n+1}(z)| =
    \log\bigl|\Phi_{n}^*(z)-\alpha z^{n+1}{\overline{\Phi_{n}^*(z)}}\bigr|
    \leq
    \log\bigl|\Phi_{n}^*(z)\bigr|
  +\log 2.
  \end{equation}
  Then using \eqref{eq:decorations} we can represent $R_{n+1}$ by
  \begin{equation}
\label{eq-Rn+1}
    R_{n+1} =
    \max_{j \in \mathcal{D}_{n/k_1}}
      \max_{\theta \in [-2\pi k_1,0]}
    \biggl\{
      \tfrac{1}{2}
      \biggl(
      \log 2 +
      \log\biggl( |D_j(\theta)|- \Re\bigl( \overline{\alpha}e^{-i\theta(1+{1}/{n})}{{D_j(\theta)}}\bigr)\biggr)
      \biggr)
    \biggr\}.
  \end{equation}

  By \eqref{eq:det2}, 
  the log-modulus of the characteristic polynomial can only increase by a $\log 2$ over the reversed OPUC $\Phi^*_{n},$ and so for any $k_7$ sufficiently large
  \[
    \widehat{W}_j \leq -k_7
    \quad \implies
    \quad
    \max_{\theta \in [-2\pi k_1,0]}
    \biggl\{
      \log\biggl( |D_j(\theta)|- \Re\bigl( \overline{\alpha}e^{-i\theta(1+{1}/{n})}{{D_j(\theta)}}\bigr)\biggr)
    \biggr\}
    < -k_7 + \log 2.
  \]
  Let $\phi_{k_7}(x)=  \min\{ \max\{ x, -k_7/2 \}, k_7/2\}$ for all $x \in \R.$
  Then, on the event $\max_j \widehat{W}_j>-k_7$,
  \[
    \phi_{k_7}(R_{n+1})
    =
    \max_{
      \substack{j \in \mathcal{D}_{n/k_1} \\
      \widehat{W}_j > -k_7}
    }
    \max_{\theta \in [-2\pi k_1,0]}
    \phi_{k_7}\biggl(
    \tfrac12
    \biggl(
    \log 2 + 
    \log\biggl( |D_j(\theta)|- \Re\bigl( \overline{\alpha}e^{-i\theta(1+{1}/{n})}{{D_j(\theta)}}\bigr)
    \biggr)
    \biggr)
    \biggr)
    .
  \]
  We would like to apply Theorem \ref{thm:process} to establish the convergence of this statistic.  For this we need a Lipschitz bound on the mapping
  \begin{equation}\label{eq:Exnl}
    \Ex_{n}^{k_1} \cap \widehat{\Gamma}_{k_7} \mapsto \phi_{k_7}(R_{n+1})
  \end{equation}
  with respect to the $\partial_1$ metric on point configurations. 
  We note that there is $n$-dependence in the mapping \eqref{eq:Exnl} beyond the $n$ dependence in the process $\Ex_n$, which we would like to remove.  It follows from the definition of the $\partial_1$ metric that \eqref{eq:Exnl} is $C(k_7)$--Lipschitz.
  Moreover, the difference 
  \[
    \begin{aligned}
    \max_{\theta \in [-2\pi k_1,0]}
    &|\phi_{k_7}\bigl(
  \log\bigl( |D_j(\theta)|- \Re\bigl( \overline{\alpha}e^{-i\theta}{{D_j(\theta)}}\bigr)\bigr)
    \bigr)
    -
    \phi_{k_7}\bigl(
  \log\bigl( |D_j(\theta)|- \Re\bigl( \overline{\alpha}e^{-i\theta(1+{1}/{n})}{{D_j(\theta)}}\bigr)\bigr)
    \bigr)| \\
    &< \tfrac{1}{n}C(k_7)(C(k_7) + \max_{\theta \in [-2\pi k_1,0]} |D_j(\theta)|).
  \end{aligned}
  \]

Now the mapping $ D_j(\theta) \mapsto \phi_{k_7}\bigl(
  \log\bigl( |D_j(\theta)|- \Re\bigl( \overline{\alpha}e^{-i\theta}{{D_j(\theta)}}\bigr)\bigr)
    \bigr)$
  has a Lipschitz constant with respect to the sup-norm that is bounded solely in terms of $k_7.$ 
  Thus if we define for each $k_1$ a random variable 
  by
  \[
    2\mathcal{R}^{k_1} - \log 2
    \coloneqq
    \max\biggl\{ 
    \max_{\theta \in [-2\pi k_1,0]}
    \biggl(
      \log\bigl( |f(\theta)|- \Re\bigl( \overline{\alpha}e^{-i\theta}{{f(\theta)}}\bigr)\bigr)\biggr)
      : (\theta,v,f) \in \Pi^{k_1} \cap \widehat{\Gamma}_{k_7}
    \biggr\},
  \]
  then from Theorem \ref{thm:process} and the monotonicity of $\phi_{k_7}$
  \[
    \lim_{k_1 \to \infty}
    \limsup_{n \to \infty}
    \sup_{\vartheta}
    \Exp | \vartheta(\phi_{k_7}(R_{n+1})) - \vartheta(\phi_{k_7}(\mathcal{R}^{k_1}))|
    = 0,
  \]
  with the supremum over all $1$--Lipschitz real--valued functions $\vartheta$.
  It follows that
  \[
    \lim_{k_1 \to \infty}
    \sup_{\vartheta}
    \Exp | \vartheta(\phi_{k_7}(R_{*})) - \vartheta(\phi_{k_7}(\mathcal{R}^{k_1}))|
    = 0.
  \]
  Recall that $R_*$ is any subsequential limit point of $\{R_n\},$ and hence for any other subsequential limit $R_*',$
  \[
    \sup_{\vartheta}
    \Exp | \vartheta(\phi_{k_7}(R_{*})) - \vartheta(\phi_{k_7}(R'_{*}))|
    =0.
  \]
  As $k_7$ is arbitrary, it follows that $\{R_n\}$ has a unique weak-$*$ limit point, i.e.\ it converges in law.
  We also observe that since in Theorem \ref{thm:process} also applies to $\Pi^{k_1,'}$ or $\Pi^{k_1,''}$ in place of $\Pi^{k_1}$ the same argument above holds if in the definition of $\mathcal{R}^{k_1}$ we replace $\Pi^{k_1}$ by either of these.  Hence, if we define $\mathcal{R}^{k_1,'}$ and $\mathcal{R}^{k_1,''}$ by making that replacement, then for all $k_7$,
  \[
    \lim_{k_1 \to \infty}
    \sup_{\vartheta}
    \bigl(
    \Exp | \vartheta(\phi_{k_7}(\mathcal{R}^{k_1,'}) - \vartheta(\phi_{k_7}(\mathcal{R}^{k_1}))|
    +
    \Exp | \vartheta(\phi_{k_7}(\mathcal{R}^{k_1,''}) - \vartheta(\phi_{k_7}(\mathcal{R}^{k_1}))|
    \bigr)
    =0.
  \]

We next characterize the limit. For that, it is enough to evaluate for fixed $x\in \R$ the limit of the probability
$\Pr(\mathcal{R}^{k_1}\leq x)$ as $k_1\to\infty$. Recall \eqref{eq-Rn+1}.
For any $\alpha=e^{i\hat \psi}$ with $\hat \psi\in [0,2\pi]$, any $x \in \R$
and any $(\theta,v) \in [0,2\pi]\times (-\infty,0]$
introduce the Borel subset of $\mathcal{C}( [-2\pi k_1,0], \C)$
\[ 
  \mathcal{A}_{\hat \psi,\theta,v}^{k_1}(x)
  = \biggl\{ f : \max_{\eta \in [-2\pi k_1,0]} 
\log\Big( e^{-\sqrt{4/\beta} v}\bigl(|f(\eta)|-\Re(e^{-i(\hat \psi+\eta)} f(\eta)\bigr) \Big)\geq 2x-\log 2\biggr\}.
\]
  From the definition of $\Pi^{k_1}$
\[
  \Pr(\mathcal{R}^{k_1}\leq x)=
  \Exp_{\hat \psi} 
  \Pr[ \mbox{\rm there are no points $(\theta,v,f)$ in $\Pi^{k_1}$ 
  such that $e^{\sqrt{4/\beta}v}f \in \mathcal{A}_{\hat\psi,\theta,v}$}
  ],
\]
and by approximation of the step function, we have for all $x$
\[
  \lim_{k_1 \to \infty}
  |
  \Pr(\mathcal{R}^{k_1}\leq x)-
  \Pr(\mathcal{R}^{k_1,'}\leq x)
  |
  +
  |
  \Pr(\mathcal{R}^{k_1}\leq x)-
  \Pr(\mathcal{R}^{k_1,''}\leq x)
  |=0.
\]
Using the fact that $\Pi^{k_1}$ is Poisson of random intensity, the last probability can be written as
$\Exp_{\hat\psi,\mathscr{B}_{\infty}} e^{-\mathcal{I}_{\hat\psi,x}\mathscr{B}_{\infty}}$ where
\[\mathcal{I}_{\hat\psi,x}= \int I(v)\mathfrak{p}_{k_1} (v,\mathcal{A}_{\hat\psi,\theta,v})dv.\]
Let
\[
  F_{\beta,\hat\psi}(v,y)=  \mathfrak{p}_{k_1}\Big(v,\max_{\eta\in [-2\pi k_1,0]} \log \Big(|f(\eta)|-\Re(e^{-i(\hat\psi+\eta)} f(\eta))\Big)\geq y\Big).
\]
From Corollary \ref{cor:pregularity}, we have that uniformly on compact sets of $\alpha$ and $y$ and uniformly in $|v| \leq (\log k_1)^{17/18}$
\[
  \begin{aligned}
  F_{\beta,\hat\psi}(v+\alpha, y-\sqrt{\tfrac{4}{\beta}}\alpha)
  &= \mathfrak{p}_{k_1}\Big(v+\alpha,\max_{\eta\in [-2\pi k_1,0]} \log \Big(|f(\eta)|-\Re(e^{-i(\hat\psi+\eta)} f(\eta))\Big)\geq y-\sqrt{\tfrac{4}{\beta}}\alpha\Big) \\
  &= e^{\sqrt{2}\alpha + o_{k_1}}\mathfrak{p}_{k_1}\Big(v,\max_{\eta\in [-2\pi k_1,0]} \log \Big(|f(\eta)|-\Re(e^{-i(\hat\psi+\eta)} f(\eta))\Big)\geq y\Big) \\  
  &= e^{\sqrt{2}\alpha + o_{k_1}} F_{\beta,\hat\psi}(v, y).    
  \end{aligned}
\]

Note that due to the random phase in the definition of $\mathfrak{p}_{k_1}$, the function $F_{\beta,\hat\psi}(v,y)$ is actually independent of $\hat\psi$, and we can write
$F_{\beta}(v,y)=F_{\beta,\hat\psi}(v,y)$.
Then, using the change of variables $v=w-\sqrt{\tfrac{\beta}{4}}x$ and setting $J=[(\log  k_1^+)^{1/10}, (\log  k_1^+)^{9/10}]$,
\begin{eqnarray}
\label{eq-punch}
\mathcal{I}_{\hat\psi,x}
&=&\int_J  ve^{\sqrt 2v} F_{\beta}(v,2x-\log 2+\sqrt{\tfrac{4}{\beta}}v) dv\nonumber\\
&=&
\int_{\sqrt{\beta/4}x+ J} (w-\sqrt{\tfrac{\beta}{4}}x) e^{\sqrt{2}w-\sqrt{\tfrac{\beta}{2}}x}F_{\beta}(w-\sqrt{\tfrac{\beta}{4}}x,x-\log 2+\sqrt{\tfrac{4}{\beta}}v) dw. \nonumber\\
&=&
\int_{\sqrt{\beta/4}x+ J} (w-\sqrt{\tfrac{\beta}{4}}x) e^{\sqrt{2}w-\sqrt{{2\beta}}x + o_{k_1}}F_{\beta}(w,-\log 2+\sqrt{\tfrac{4}{\beta}}v) dw. \nonumber
\end{eqnarray}
In particular, $\mathcal{I}_{\hat\psi,x}$ does not depend on  $\hat\psi$, and we can omit it from the notation. Since $x$ is fixed, we obtain for $w$ in the stated interval that $(w-\sqrt{\tfrac{\beta}{4}}x)/w=1+o_{k_1}(1)$.  Hence if we let $\mathcal{I}_{x}'$ and $\mathcal{I}_{x}''$ to denote $\mathcal{I}_{x}$ with $I',I''$ replacing $I$ respectively, we obtain
\begin{equation}
\label{eq-punch1}
\mathcal{I}_{0}''(1+o_{k_1}(1))\leq e^{x \sqrt{2\beta} } \mathcal{I}_{x}\leq \mathcal{I}_{0}'(1+o_{k_1}(1)).
\end{equation}
Using that all of the distributions functions of $\mathcal{R}^{k_1}, \mathcal{R}^{k_1,'}, \mathcal{R}^{k_1,''}$ converge together in the limit, we conclude
\[
  \lim_{k_1\to\infty} 
  \Pr(\mathcal{R}^{k_1}\leq x)
  =\Exp \frac1{2\pi} \int_0^{2\pi} e^{-\mathscr{B}_\infty e^{-x\sqrt{2\beta}} A_\beta} d\hat\psi
  =\Exp e^{-\mathscr{B}_\infty e^{-x\sqrt{2\beta}} A_\beta},
\]
where
\begin{equation}
\label{eq-Abeta}
A_\beta=\int_{J} w e^{\sqrt{2}w }F_{\beta}(w,-\log 2+\sqrt{\tfrac{4}{\beta}}v) dw.
\end{equation}
is a constant that does depend on $\hat\psi$, and which takes values in $(0,\infty)$ due to tightness.

The proof of Theorem \ref{thm:max} is identical, except that 
instead of working with $R_n$ as in \eqref{eq-Rn+1}, we can work directly with $\varphi_n$, and in the right hand side of \eqref{eq-Rn+1}  one replaces 
the expression by $\log D_j(\theta)$, resulting in a simplification of the proof. 
\end{proof}

Theorem \ref{thm:process} is a direct corollary of Theorem \ref{prop:bush} and some estimates from Section \ref{sec-arc} below.
 \begin{proof}[Proof of Theorem \ref{thm:process}]
 Note first that from the data
  \(
    (\theta_j, V_j, D_je^{\sqrt{4/\beta}V_j}),
  \)
  we can express the triple
  \(
    (\theta_j, \widehat{W}_j, D_j)
  \)
  by a continuous transformation of the second two coordinates using $\iota$, viz.
   $ (\widehat{W}_j, D_j)
    =\iota( V_j, D_je^{\sqrt{4/\beta}V_j}).$
  Moreover this maps $\Gamma_{k_7}$ to $\widehat{\Gamma}_{k_7}.$  Furthermore, the transformation $\iota$ is Lipschitz with some constant $L({k_7})$ when restricted to this set.
  Let $P = \Pi^{k_1,k_2} \cap \Gamma_{k_7}$ and $Q = \Ext_n^{k_1} \cap \Gamma_{k_7}.$
  Hence for any $k_2$
  \begin{equation}\label{eqo:twoerrors}
    \partial_2\bigl(
    \Pi^{k_1} \cap \widehat{\Gamma}_{k_7},
    \Ex_n^{k_1} \cap \widehat{\Gamma}_{k_7}
    \bigr)
    \leq
    L(k_7)
    \partial_2( Q, P)
    +
    \partial_2( \widehat{\iota}(P),\Pi^{k_1} \cap \widehat{\Gamma}_{k_7}),
  \end{equation}
  where $\widehat{\iota}(P)$ is a Poisson point process on $\widehat{\Gamma}_{k_7}$ with intensity
  \(
  \mathscr{D}_{k_2}(\theta)d\theta \times
  \widehat{\mathfrak{p}_{k_1}}
  \)
  and
  where $\widehat{\mathfrak{p}_{k_1}}$ is the pushforward of $I(t)dt 
  \times \mathfrak{p}_{k_1}(t,df)$ under $\iota$, which is a Radon measure.  
  The first term in \eqref{eqo:twoerrors} goes to $0$ from Theorem \ref{prop:bush}.

  We now turn to the second term. Define for any two finite Borel measures on $\Gamma,$
  \begin{equation*}
    d_{\operatorname{BL}}(\pi,\lambda) \coloneqq \sup_{ g } \biggl|\int_{\Gamma} g d(\pi -\lambda)\biggr|,
  \end{equation*}
  with the supremum over all $g$ with both $|g(x)-g(y)| \leq \partial_0(x,y)$ for all $x,y \in \Gamma$ and $|g(x)| \leq 1$ for all $x \in \Gamma.$
  We need the following bound:
    \begin{equation}\label{eq:ibound}
      \sup_{k_1}
      \widehat{\mathfrak{p}_{k_1}}\bigl(
      [-k_7,k_7] \times \mathcal{C}( [-2\pi k_1,0], \C)
      \bigr)
      \eqqcolon \max\widehat{\mathfrak{p}}
    <\infty \quad \As
    \end{equation}
    which is shown in Lemma \ref{zrh:2moment} (in the notation of that lemma, it is $\mathcal{H}$).
    If $f : \widehat{\Gamma}_{k_7} \to \R$ is a $1$-bounded, $1$-Lipschitz function with respect to $\partial_0$, we therefore have
    \[
      \biggl|\int_{\widehat{\Gamma}_{k_7}} f(x,y) \bigl(\mathscr{D}_{k_2}(dx) -\mathscr{D}_{\infty}(dx)\bigr)\widehat{\mathfrak{p}_{k_1}}(dy)
      \biggr|
      \leq \sup_g \biggl|\int_{[0,2\pi]} g(x)\bigl(\mathscr{D}_{k_2}(dx) -\mathscr{D}_{\infty}(dx)\bigr)\biggr| \max\widehat{\mathfrak{p}},
    \]
    where $g$ ranges over all the fibers $f( \cdot, y)$ over all $y\in[-k_7,k_7]\times \mathcal{C}([-2\pi k_1,0],\mathbb{C})$.  These are all $1$-bounded and $1$-Lipschitz with respect to the metric $d(x,y) = (|x-y| \wedge 1).$  It follows from Arzel\`a-Ascoli and Theorem \ref{thm:Bj} (which gives almost sure weak-* convergence of $\mathscr{D}_{k_2}$ to $\mathscr{D}_{\infty}$ and the almost sure finiteness of $\mathscr{D}_\infty$) that
    \[
      \limsup_{k_2,k_1 \to \infty}
      \biggl|\int_{\widehat{\Gamma}_{k_7}} f(x,y) \bigl(\mathscr{D}_{k_2}(dx) -\mathscr{D}_{\infty}(dx)\bigr)\widehat{\mathfrak{p}_{k_1}}(dy)
      \biggr|
      =0\quad \As
    \]
    Hence again by Arzel\`a-Ascoli, taking supremum over all such $f$ we conclude
  \[
    \limsup_{k_2,k_1 \to \infty}
    {d_{\operatorname{BL}}( \mathscr{D}_{k_2} \times \widehat{\mathfrak{p}_{k_1}}, \mathscr{D}_{\infty} \times \widehat{\mathfrak{p}_{k_1}} )
    }
    =0 \quad \As
  \]
  Hence from Theorem \ref{thm:PPcom}, $\partial_2( \widehat{\iota}(P),\Pi^{k_1} \cap \widehat{\Gamma}_{k_7} ) \to 0$ as $k_1\to \infty$ followed by $k_2 \to \infty,$ which completes the proof.

The proof for $\Pi^{k_1,'}$ and $\Pi^{k_1,''}$ is identical.
\end{proof}

\subsection{Imaginary part of the log-determinant}
\label{subsec-imaginarylog}
The following corollary, which handles the imaginary part of the logarithm of the characteristic polynomial, follows from Theorems \ref{prop:bush} and \ref{thm:process} in
the same way that Theorem \ref{theo-main} followed from them; some adjustments are necessary because $\Im \log X_n(\theta)$ takes values on a (shifted) 
lattice, see
\eqref{eq:ImX} below.
\begin{corollary}\label{cor:icpoly}
There are deterministic constants $a_n\in [0,2\pi]$ and 
almost surely finite random variables $I_{\pm}$ so that
  \[
    \begin{aligned}
      &\max_{\theta \in [0,2\pi]} \{ 2\Im \log X_n(\theta) - n\theta \} - \sqrt{\tfrac{8}{\beta}} m_n-a_n
    \Wkto[n] I_+, \quad\text{and}\quad \\
    &\min_{\theta \in [0,2\pi]} \{ 2\Im \log X_n(\theta) - n\theta \} + \sqrt{\tfrac{8}{\beta}} m_n+a_n
  \Wkto[n] I_-.
    \end{aligned}
  \]
\end{corollary}
The adaptation of the proof needed to handle here the discreteness
is somewhat simpler than that needed in the 
case of lattice valued branching random walks, as described in 
\cite[Section 5]{BDZ1}; this is because the effects of discreteness here
are only felt in the ``decoration'' process.

The imaginary part of the logarithm of the characteristic polynomial is related to the eigenvalue counting process of the C$\beta$E.
Specifically we take an increasing version of the map
\begin{equation}\label{eq:ImX}
  \theta \mapsto n\theta - 2\Im( \log X_n(e^{i\theta}) - \log X_n(1)),
\end{equation}
(c.f.\ Lemma \ref{lem:psikmonotone} below)
which in fact counts $2\pi| \{ j : 1 \leq j \leq n, -\omega_j \in [0, \theta]\}|,$ at all those $\theta$ except $\left\{ -\omega_1,-\omega_2,\dots,-\omega_n \right\}.$  This has the same law (as a process in $\theta$) as
\[
  N_n(\theta) \coloneqq 2\pi| \{ j : 1 \leq j \leq n, \omega_j \in [0, \theta]\}|.
\]
As $\{\log X_n(1)/\sqrt{\log n}\}$ becomes a centered Gaussian (this follows from Proposition~\ref{prop:mgapproximation} with some minor adjustments) and 
\[
\max_{\theta \in [0,2\pi]} \{ 2\Im \log X_n(\theta) - n\theta \} - \sqrt{\tfrac{8}{\beta}} m_n
\]
is tight, it follows that
\[
  \biggl\{\max_{\theta \in [0,2\pi]} \{ N_n(\theta) - n\theta \} - \sqrt{\tfrac{8}{\beta}} m_n : n \in \N \biggr\}
\]
is not tight, and in fact when scaled down by $\sqrt{\log n}$ converges in law to a centered nondegenerate Gaussian.

On the other hand, the maximum over all arcs of the centered counting function admits the representation
\[
  \max_{\theta_j \in [0,2\pi]}
  \biggl\{
    N_n(\theta_2) - N_n(\theta_1) - n(\theta_2-\theta_1)
  \biggr\}
  \lawequals
  \max_{\theta \in [0,2\pi]} \{ 2\Im \log X_n(\theta) - n\theta \}
  -\min_{\theta \in [0,2\pi]} \{ 2\Im \log X_n(\theta) - n\theta \}.
\]
This motivates understanding the joint convergence of the maximum and the minimum of the imaginary part of the logarithm of the characteristic polynomial, which we do not pursue here.
\begin{conjecture}
\label{conj-joint}
  The convergence in Corollary \ref{cor:icpoly} holds jointly.
\end{conjecture}
Note that in the model of Branching Brownian Motion, the analogue of Conjecture \ref{conj-joint} was recently proved in \cite{SBM}, see also \cite{BKLMZ}.

\subsection{Related literature}
The analysis in this paper falls within the topic of \textit{logarithmically correlated fields}.  Indeed, in the case of $\beta=2$, it was shown in \cite{HKO} that
the process $ Y_n(z)=\{\log |X_n(z)|\}_{|z|=1}$ is logarithmically correlated in the sense that for any smooth test function $\phi$ on $S^1$ with $\int \phi(\theta) d\theta=0$, the random variable
$Y_\phi \coloneqq (2\pi)^{-1} \int_0^{2\pi} Y_n(e^{i\theta}) \phi(\theta) d\theta$ converges in distribution to a centered Gaussian random variable with
variance  $ \sum_{k=-\infty}^\infty \frac{|\hat \phi_k|^2}{8|k|}$, where $\hat \phi_k$ is the $k$-th Fourier coefficient of $\phi$.  See also \cite{BF, KS, Wieand} for related results.
Note that the above variance expression corresponds to a generalized Gaussian field with correlation having singularity of the form $-\log  |\theta-\theta'|$.

For Gaussian logarithmically correlated fields, the study of the maximum has a long history, going back to the seminal work \cite{Bramson78,bramson83} concerning Branching Brownian motion (BBM). Bramson
introduced the truncated second moment and barrier methods that are the core tools in all subsequential analysis (see \cite{Roberts} for a modern perspective).
Extremal processes for the BBM were constructed in \cite{ABBS} and \cite{ABK}. Following some earlier work on tightness and rough structure of the extrema,
the extension of the convergence of the maximum to the discrete Gaussian free field in the critical dimension $2$
was obtained in \cite{BDZ} and (for the extremal process) in \cite{BL}. See \cite{biskup} for an updated account, and \cite{zeitouni} for an introduction.
A convergence result for general log-correlated Gaussian
field is presented in \cite{DRZ17}, see also \cite{Madaule}.

In the physics literature, the notion of \textit{freezing} in the context of extremal processes of logarithmically correlated fields arose in the seminal work
\cite{FBo}. The link between the freezing phenomenon and extremal processes of the decorated, randomly shifted Poisson type was
rigorously elucidated in \cite{SZ}.
The highly influential work \cite{FHK12} (see also \cite{FH}) applied the freezing paradigm to making predictions  for the maximum  of the logarithm of the modulus of the characteristic polynomial of CUE matrices, see \eqref{eq-R2}, and using a conjectured dictionary going back to \cite{KS}, made predictions concerning the maximum of the Riemann $\zeta$ function over
short intervals of the critical axis. This has stimulated much work, both on the random matrix side (which we will review shortly) and on the Riemann side, for which we refer
to \cite{ABBRS} and \cite{ABRI} for the latest progress on verifying the FHK conjectures for the Riemann $\zeta$ function.

On the circular ensembles side, the first verification of the leading order in the FHK prediction was obtained in \cite{ABB}, followed in short order by the verification of the second order term
\cite{PaquetteZeitouni02}; both works used explicit computations facilitated by dealing with $\beta=2$, and a decomposition of the log-determinant according to Fourier modes for the former or to
spatial approximations for the latter. It was \cite{CMN} who introduced the use of Verblunsky coefficients and the \emph{Szeg\H{o}} recursions in \eqref{eq:szego} to not only
handle arbitrary $\beta>0$, but also to obtain the tightness of $M_n-m_n$. Along the way, \cite{CMN} derived various Gaussian approximations and barrier estimates that are fundamental for the current paper.

The analysis for C$\beta$E has natural analogues for Hermitian ensembles of the G$\beta$E type. We mention in particular \cite{FLD} for an early application of the freezing scenario in that context. More recent work include \cite{LambertPaquette01, ABZ,  CFLW, BMP, BLZ}. At this time, the results in the Hermitian setup are much less sharp than for circular ensembles.

An important role in our analysis and results  is played by the derivative martingale measure $\mathscr{D}_\infty(d\theta)$ and its total mass $\mathscr{B}_\infty$, both related to the theory of Gaussian Multiplicative Chaos (GMC), see \cite{RV} for a review. This is not
surprising - the appearance of such martingales in the expression for the law of the maximum of the BBM was discovered already in  \cite{LS}. Subsequentially,
it appeared in the analogous studies for Branching Random Walks \cite{Aidekon}, in the context of the Gaussian free field \cite{DRSV,BL}, and for more general log-correlated fields
\cite{Madaule}. As noted above, for a specific log-correlated field on the circle, \cite{Remy} computes the law of the total mass of the associated GMC and confirms the Fyodorov-Bouchaud prediction
\cite{FBo}
for it. 

We also mention that for a related model, \cite{CN} provide a direct link with the relevant GMC.
Their work implies convergence of the random measure $|\Phi^*_n|^{\gamma}(e^{i\theta})d\theta$ when $\gamma=-2$ and $\beta \geq 2$.  After appropriate rescaling, this measure converges to the GMC of \cite{FBo,Remy,CN} \emph{rescaled} to be a probability measure.  Note that the exponent $\gamma=-2$ never quite appears in the problem we consider, and for any $\beta$ the relevant $\gamma$ is the positive critical point.  When $\beta=2$, the full Fyodorov-Hiary-Keating conjecture would follow from convergence of $\mathscr{D}_n(\theta)d\theta$ (which up to rescaling) to the same GMC; this is expected to have the same limit as $\sqrt{\log n}|\Phi^*_n|^{2}(e^{i\theta})d\theta$ (up to constants).

Very recently, \cite{LN24} gave a proof of convergence of the 
random measure $|\Phi^*_n|^{\gamma}(e^{i\theta})d\theta$ to the (subcritical)
GMC whose total mass was evaluated in \cite{Remy}, 
for all $\gamma<\sqrt{2\beta}$. Their techniques might be relevant to the
evaluation of the law of $\mathscr{D}_\infty$, which corresponds to the case of critical GMC.

\subsection{A high level description of the proof}
We now provide a high level description of the proof, that glosses over many important details. A  detailed description of the proof that includes 
precise statements is provided in Section  \ref{sec-arc}.

As in \cite{CMN}, the key observation is that the recursion \eqref{eq:upsilon} contains all information needed in order to evaluate the determinant, due to \eqref{eq:charpoly}.
For fixed $\theta$, the variable $\varphi_k(\theta), k=1,\ldots,n$ can be well approximated by a random walk, and further, for all $k$ large, the increments of the random walk
are essentially Gaussian.

We  explain the strategy toward the proof of Theorem \ref{thm:max}.
We fix large constants $k_1,k_2$ (with $k_1\gg k_2$). Instead of directly studying the field $\varphi_n(\theta), \theta\in[0,2\pi]$, we consider a sublattice $\mathcal{D}_{n/k_1}$
of angles $\theta_j$
of cardinality $\lceil n/k_1\rceil$, and associate to each an interval
  $\widehat{I_{j,n}}$ as in \eqref{eq:hatij}. We write
\[
  \varphi_n(\theta)
  =\varphi_{k_2}(\theta)+ (\varphi_{n/k_1}(\theta)-\varphi_{k_2}(\theta))
  +(\varphi_n(\theta)-\varphi_{n/k_1}(\theta))
  \eqqcolon 
  \varphi_{k_2}(\theta)+ \Delta_{k_2,n/k_1}(\theta)+\Delta_{n/k_1,n}(\theta),
\]
and 
\[\max_{\theta\in [0,2\pi]}  \varphi_n(\theta)= \max_{j} \max_{\theta\in \widehat{I_{j,n}}}\Big(\varphi_{k_2}(\theta)+ \Delta_{k_2,n/k_1}(\theta)+\Delta_{n/k_1,n}(\theta)\Big).\]
We claim that the last expression can be approximated  as
\begin{equation}
\label{eq-approx}
\max_j\Big( \varphi_{k_2}(\theta_j)+ \Delta_{k_2,n/k_1}(\theta_j)+\max_{\theta\in  \widehat{I_{j,n}}}\Delta_{n/k_1,n}(\theta)\Big).\end{equation}
(This is not quite right, and in reality we will need to consider an intermediate point $n/k_1^+$ with $k_1^+>k_1$, but we gloss over this detail at this high-level description.)

To analyze the maximum in \eqref{eq-approx}, we introduce  the field 
$f_{n,j}(\eta) \coloneqq \Delta_{n/k_1,n}(\theta_j+\eta /n)$,
with $\eta\in [-2\pi k_1,0]$ and $\Delta$ defined above \eqref{eq-approx},
and write  \eqref{eq-approx} as
\begin{equation}
\label{eq-approx1}
\max_j\Big( \varphi_{k_2}(\theta_j)+ \Delta_{k_2,n/k_1}(\theta_j)+\max_{\eta\in [-2\pi k_1,0]} f_{n,j}(\eta)\Big)=: \max_j\Big(
\varphi_{k_2}(\theta_j)+ \Delta_{k_2,n/k_1}(\theta_j)+\Delta'_{n/k_1,n}(j)\Big).\end{equation}

The main contribution to the maximum comes from $j$s with
$\Delta_{k_2,n/k_1}(\theta_j)$
large, of the order of $\sqrt{8/\beta}(m_n-\log(k_1k_2))$.
However, the $\Delta_{k_2,n/k_1}(j)$ are far from independent for different
$j$. In order to begin controlling this, we introduce two ``good events'':
a global good event $\mathscr{G}_n$, which allows us to replace the recursion
by one driven by Gaussian variables (called $\mathfrak{z}_t(\theta)$, and taken for convenience in continuous time) and also impose an \emph{a priori}
upper limit on the recursion, and a barrier event
$\widehat{\mathscr{R}}$, which ensures that the Gaussian-driven recursion
$\mathfrak{z}_t(\theta)$ stays within a certain entropic envelope. We will
also insist that  $\mathfrak{z}_{n/k_1}(\theta_j)$ stays within an appropriate
window.
These steps are similar to what is done in \cite{CMN}
and prepare the ground for the application of the second moment method.

We next claim that the fields $f_{n,j}(\eta)$
converge in
distribution to the solution of a system of coupled stochastic differential equations as in \eqref{eq:dSDE} (again, this is not literally the case, and requires some pre-processing in the form
of restriction to appropriate events and using $k_1^+$ as before). In particular, the law of those fields are determined by the Markov kernel $\mathfrak{p}_{k_1}$.
Further and crucially, the fields $f_{n,j}$ can be constructed so that for well separated $j$s, they are independent. This analysis is contained in Sections \ref{sec:diffusion} and
\ref{sec-init}.

As in many applications of the second moment method, to allow for some decoupling it is necessary to
condition on $\filt_{k_2}$. We need to find high points of the right
side of \eqref{eq-approx1}.
The basic estimate, for  a given $j$, is that with $w_j=\sqrt{8/\beta} \log k_2-\varphi_{k_2}(\theta_j)$,
\begin{equation}
  \label{eq-approx2}
  \Pr\Big(\Delta_{k_2,n/k_1}(\theta_j)\sim \sqrt{8/\beta}(m_n-\log (k_1k_2)-v)\mid \filt_{k_2}\Big) \sim C \frac{ve^{2v} w_j e^{-2w_j}}{n}.\end{equation}
This estimate (after some pre-processing) is taken from \cite{CMN}, see Appendix
\ref{section:lower_bound}.

If the variables $\{\Delta_{k_2,n/k_1}(\theta_j)+\Delta'_{n/k_1,n}(j) : j\}$ were an independent family, we would be at this point done,
for then we would have that
\begin{eqnarray}
  \label{eq-approx3}
 && \Pr \Big(\varphi_{k_2}(\theta_j)+\Delta_{k_2,n/k_1}(\theta_j)+\Delta'_{n/k_1,n}(j) > \sqrt{8/\beta}(m_n +x) ~\mid~ \filt_{k_2}\Big)
 \nonumber\\ &&
 \sim C \frac{w_j e^{-2w_j}}{n/k_1} \Exp_{\mathfrak{p}_{k_1}}\Big( (V_j-x) e^{2(V_j-x)}\Big)
 \sim C\frac{\mathscr{D}_{k_{2}}\left(\theta_{j}\right)}{n/k_1} e^{-2x}.
\end{eqnarray}
Hence, we have using independence over different $j$ that
\[
  \begin{aligned}
    \Pr( \max_{\theta\in [0,2\pi]}  \varphi_n(\theta) \le \sqrt{8/\beta}(m_n +x)~\mid~ \filt_{k_2})
    &\sim \prod_j  \Big( 1- C\frac{\mathscr{D}_{k_{2}}\left(\theta_{j}\right)}{n/k_1} e^{-2x}\Big) \\
    &\sim \exp\bigl( -C\mathscr{B}_{k_2} e^{-2x}\bigr),
  \end{aligned}
\]
which would then yield Theorem \ref{thm:max}.

Unfortunately, different $j$s are not independent. We handle that through several Poisson approximations. First, we condition on $\filt_{n/k_1}$ and use the ``two moments suffice'' method of \cite{ArratiaGoldsteinGordon} to show
that the process of near maxima (together with the shape $(\varphi_n(\theta)-\varphi_{n/k_1}(\theta), \theta\in \widehat{I_{j,n}})$)
can be well approximated, as $k_1\to\infty$,
by a Poisson point process of intensity $\mathfrak{m}$ which depends still on $k_1$, see \eqref{eq:2ndintensity} and Proposition \ref{prop:ppp1}.
In the proof, the independence for well separated $j$s and the second moment
computations play a crucial role. (As mentioned above, we need to
replace $k_1$ by $k_1^+\gg k_1$ to make the argument work; for this reason, we
introduce a second Poisson approximation, see Section \ref{sec:2ndapprox}.)

This 
is close to the statement of Theorem  \ref{prop:bush}, except
that the Poisson process we obtained so far
has a random intensity, measurable on
$\filt_{n/k_1^+}$. Our final step is another standard use of the second moment method to show that this random intensity concentrates, see Section
\ref{sec-thirdppp}.
As proved earlier in Section \ref{subsec-sequential}, all 
our theorems, and in particular Theorem \ref{thm:max}, follow from Theorem \ref{prop:bush}.

\subsection{Glossary}
For convenience and easy reference, we record in Tables \ref{table1}-\ref{table3} a glossary of notation.
\begin{table}[htbp]   \begin{tabular}{ r p{1.5cm} p{11cm}}
    \toprule
    Symbol& Ref. & Meaning \\
    \midrule
    $k_1$ & \eqref{eq:hatij} & The width (and height) of the ``bushes'', i.e.\ the size of decorating processes, which are deep in the tree.  \\
    $k_1^+,\widehat{k}_1$ & \eqref{eq:k2k3} & Height parameters, which are larger than $k_1$ by a multiplicative factor which is subpolynomial in $k_1$.  $k_1^+ \gg \widehat{k}_1.$ \\
    $n_1,n_1^+,\widehat{n}_1$ & \eqref{eq:leafheight} & The floors of $n/k_1$, $n/k_1^+$ and $n/\widehat{k}_1.$  \\
    $T_-,T_{\dagger},T_+$
    & \eqref{eq:LU}, \eqref{eq:dSDE} & $\log k_1 - \log k_1^+$, $\log k_1 - \log \widehat{k}_1$, $\log k_1$.  These correspond to $n_1^+$, $\widehat{n}_1$ and $n$, in the timescale of the decoration. \\
    $k_2$ & p.\,\pageref{k2} & The height from the root of the tree, i.e.\ the count of initial Verblunsky coefficients, on which the process is conditioned.  Always taken as an integer power of $2$.  The derivative martingale is adapted to $(\filt_{k_2} : k_2).$\\
    $k_3$ & \eqref{eq:G3} & The error tolerance of the approximating Gaussian process ${Z}_{2^k}^{k_2}$. \\
    $k_4$ & \eqref{eq:Rhat} & The number of steps (in logarithmic time) that are not covered by the barrier event. \\
    $k_5$ & p.\,\pageref{k5} & The mesh spacing parameter. \\
    $k_6$ & \eqref{eq:g3} & Generic good--event parameter, controlling the global maximum of $\varphi_n(\theta)$, the quality of the Gaussian approximation, the definition of near--leader, and some downstream errors. \\
    $k_7$ & Thm.\,\ref{thm:process} & Window in which the extremal process approximation is done. \\
    $\beta_j$ & & Constant appearing in the C$\beta$E Verblunksy recurrence, $\beta_j = \sqrt{\frac{\beta}{2}(j+1)}$ \\
    \bottomrule
  \end{tabular}
  \vspace{0.2cm}
  \caption{Table of large constants.  For $p < q$, $k_p \gg k_q$}
\label{table1}
\end{table}

\begin{table}[htbp]   \begin{tabular}{ r p{2.2cm} p{11.5cm}}
%
    \toprule
    Symbol& Ref. & Meaning \\
    \midrule
    $\sigma$ & \eqref{eq:upsilon} & Either $1, i$ according to which log--correlated field is considered.\\
    $\varphi$ & \eqref{eq:upsilon} & The log--correlated field whose extrema is under consideration.\\
    $\Phi^*$ & \eqref{eq:Phi} & The reversed OPUC. \\
    $\Psi$ & \eqref{eq:Psi} & The (absolute) Pr\"ufer phase. \\
    $\widehat{I}_j$ & \eqref{eq:hatij} & An arc in the continuum which represents the domain of a decoration. \\
$H_k,\Gamma_k^a,\beta_k$& \eqref{eq-Hk},\eqref{eq:betagamma} &Harmonic number, Gamma variable and its (square-root) mean\\
    ${I}_j$ & \eqref{eq:Ij} & A discretized arc representing the domain of a decoration. \\
    $\mathfrak{W}$ & \eqref{eq:W} & A complex Brownian motion with normalization $\Exp |\mathfrak{W}_t|^2 = 2t.$ \\
  $\mathfrak{Z}_t(\theta)^\C$,  $\mathfrak{Z}_t(\theta)$ &\eqref{eq:Z}, \eqref{eq:GkZt} & Standard real (and complex) Brownian motions so that $|\varphi_k(\theta) - \sqrt{\tfrac{4}{\beta}}\mathfrak{Z}_{H_k}| \leq k_6$ for the $k$-th harmonic number $H_k$ and for all $k_2 \leq k \leq n$ on the good event $\mathscr{G}_n^2$ from \eqref{eq:g3}. \\
    $G_k(\theta)$ & \eqref{eq:GkZt} & Shorthand for $\sqrt{\tfrac{4}{\beta}}\mathfrak{Z}_{H_k}(\theta)$. \\
        $\widehat{W}_j$ & \eqref{eq:localmaxhat}& The (shifted) maximum of the field $\varphi_n(\theta)$ over the arc $\widehat{I}_j$ \eqref{eq:hatij}, with $j$ ranging over $\mathcal{D}_{n/k_1},$ the natural numbers up to $\lceil \tfrac{n}{k_1} \rceil$. \\
    $V_j,V_j'$ &\eqref{eq:leafheight}, \eqref{eq:localmax1}& value of $\sqrt{2}m_{n_1^+}-\varphi_{n_1^+}(\theta_j)$ and its $\mathfrak{Z}$ approximation\\
${W}_j$ & \eqref{eq:localmax}& The (shifted) local maximum of the field $\varphi_n(\theta)$ over the discrete mesh ${I}_j$ -- see above \eqref{eq:localmax}. \\
  $W_j',W_j^o$ & \eqref{eq:localmax1}, \eqref{eq:decoration}& The local maximum of the field  $\mathfrak{U}^j_{T_+}(\theta)$, respectively $\mathfrak{U}^{o,j}_{T_+}$, over a discrete mesh corresponding to ${I}_j$ \\
    $D_j,D_j',D_j^o$ & \eqref{eq:decorations},\,\eqref{eq:decorationD_j}, \eqref{eq:decoration} & The decoration processes. The latter two are piecewise continuous. \\
    $A_t^{p,\pm}$ & \eqref{eq:barrier} & Upper and lower barrier functions that control that support the paths of near--leaders with width parameter $p$.\\
 $\mathcal{A}_t^{\pm}$&\eqref{eq:decorationbarrierA}& A barrier that is used in the ray events $\mathscr{P}_j'$, $\mathscr{P}_j'(\theta)$.\\
    $\mathfrak{L}^j_t, \mathfrak{U}^j_t, \mathfrak{W}^j_t$ & \eqref{eq:LU} & The diffusion approximation to $2\log \Phi^*$ on the arc $\widehat{I}_j$, the approximation to $\varphi$, and the driving complex Brownian motion on time interval $[T_-,T_+]$ (see above \eqref{eq:Wtheta}).  The superscript $j$ denoting the interval is suppressed when it will not cause confusion. \\
$\Ex_n$,  $
  \Ext_n^{k_1}
  $ & \eqref{eq-Exn}, \eqref{eq:Ext} & Extremal point processes which are the main object of study. \\
  $\Extr_n,\Extre_n$
   & \eqref{eq:Extr}, \eqref{eq:Extre} & Simplifications of $\Ext_n$. \\
$\Pi(\Lambda)$&\eqref{eq:2ndintensity}& Poisson process of intensity $\Lambda$\\
$\Pi^{k_1,k_2}$, $\Pi^{k_1}$&\eqref{eq:trueintensity}, \eqref{eq-defpk1}& Limiting Poisson processes.\\
  $\widehat{\Gamma}_{k_7}, \Gamma_{k_7}, \Gamma_{k_7}^+$ & \eqref{eq-defhatGammak7},
\eqref{eq:gk7} & The domains on which the points processes are considered.  \\
$\mathscr{D}_{k_2}(\theta)$, $\mathscr{B}_{k_2}$
& \eqref{eq:Bk}
& Derivative martingales, with  limits $\mathscr{D}_\infty(d\theta)$, $\mathscr{B}_\infty$ from  Theorem \ref{thm:Bj}. \\
$\mathfrak{p}, \mathfrak{s}$
& 
\eqref{mtl:intensity},
\eqref{eq:decorationlaw}
&
Decoration laws, depending on $k_1$, $k_4$, $k_5.$  
The $\mathfrak{s}(h,\cdot)$ is the law of $D_j^o(h)$ (for any or all $j$). In the case of $\sigma=i$, $\mathfrak{p}$ and $\mathfrak{s}$ are the same.  In the case of $\sigma=1$, $\mathfrak{p}(h,\cdot)$ is the law of $D_j^o(h)$ multiplied by a uniform phase.\\
$ \mathfrak{m}_j, \mathfrak{m},  \overline{\mathfrak{m}}, \mathfrak{n}$&\eqref{eq:2ndintensity},\,\eqref{eq:2ndintensityprime}, \eqref{mtl:intensity}&Intensity of $\Extre_n$ and its approximation and limit.\\
    \bottomrule
\end{tabular}
\vspace{0.2cm}
\caption{Table of processes, other symbols}
\end{table}

\begin{table}[htbp] 
\begin{tabular}{ r p{2.2cm} p{11.5cm}}
      \toprule
    Symbol& Ref. & Meaning \\
    \midrule
        $\mathscr{T}_M$&\eqref{eq:good}& The event where truncation
    of variables is possible.\\
    $\mathscr{B}_{n,k_2,k_6}$ & \eqref{eq:allupperbarriers}& Event that the recursion $\varphi$ is below the entropic barrier at dyadic time points.\\
    $\mathscr{G}_n$ & Lemma \ref{lem:goodstuff1} & Generic global good event, on which the Gaussian approximation $|\varphi_k(\theta) - \sqrt{\tfrac{4}{\beta}}\mathfrak{Z}_{H_k}| \leq 1$
 holds, the global maximum of $\varphi_n(\theta)$ is in control (moreover the whole process $k\mapsto \sup_{\theta} \varphi_{2^k}(\theta)$ stays below a concave barrier), and the driving random variables are truncated. \\
    $\mathscr{G}_n^1$ & \eqref{eq:g1} & Good event that relates $\{\varphi_k\}$ to the diffusion $\mathfrak{U}$ after time $n_1^+$. \\
    $\mathscr{L}(\theta)$ & \eqref{eq:Ltheta} & The event that $\varphi_n(\theta)$ is a near--leader, that is $\varphi_n(\theta)$ is at least $\sqrt{\tfrac{8}{\beta}}m_n - k_6.$ \\
    $\widehat{\mathscr{R}}(\theta)$ & \eqref{eq:Rhat} & The event that $\mathfrak{Z}_t(\theta)$ stays between two convex barrier functions (the entropic envelope of a near-leader), up to time $H_n - k_4$.  In addition, the process stays below a concave barrier afterwards.  \\
    $\mathscr{N}_j$ & Def.\,\ref{def:N} & The event that the interval $\widehat{I}_j$ has nearly maximal values of $\varphi_n(\theta)$ and in addition the profile of those near--maximal values are well--described by the mesh.  \\
    $\mathscr{O}_j^{\Psi},\mathscr{O}_j^+,\mathscr{O}_j^{\Psi}$
    &\eqref{eq:Oj}& The event that the profile of $\varphi_{\widehat{n}_1}$ (resp.\,$\varphi_{n_1^+}$, $\Psi_{n_1^+}$) is flat (has small oscillation) in the interval $\widehat{I}_j.$ \\
    $\widehat{\mathscr{O}}$ & \eqref{eq:oscO} & The event that $\Psi_k(\theta)$ has small oscillation in $\theta$ for all $k$ in $k_2 \leq k \leq n_1^+$. \\ 
    $\mathscr{U}(\theta)$ & \eqref{eq:UU} & The event that $\mathfrak{Z}_{H_{k_2}}(\theta)$ is in a small entropic window. \\
    $\mathscr{R}^p_{j}(m)$ & \eqref{eq:rayevent} & The event that $\mathfrak{Z}_t(\theta_j)$ (with $\theta_j$ the largest element of $\widehat{I}_j$) stays within the entropic envelope up to time $\log m$, with width parameter $p$. \\
    $\widehat{\mathscr{R}}_{j}$ & \eqref{eq:Rhatj} & The event that all $\theta \in I_j$ which are near-leaders (in that $\mathscr{L}(\theta)$ holds) stay within the entropic envelope. \\
    $\mathscr{P}_j',\mathscr{P}_j'(\theta)$  &\eqref{eq:decorationbarrier} & Ray events for the decoration $\mathfrak{U}^j_t$.  The latter event is that a single process $t \mapsto \mathfrak{U}^j_t(\theta)$ stays within the entropic envelope.  The former is that all $\theta$ in the decoration window which are near leaders stay in the envelope. \\
$   \mathscr{P}_j(\theta,h)$& \eqref{eq:decorationbarrier2}&Ray events similar to $\mathscr{P}_j'(\theta)$, except with offset $h$ that will correspond to height at time $T_\dagger$.\\
    $\mathscr{A}_{j},\mathscr{A}'_{j}$ & \eqref{eq:qj} & Shorthand: $\mathscr{A}_{j} = \mathscr{R}_{j}^2({n}_1^+)
    \cap \mathscr{O}_j
    \cap \mathscr{O}_j^{\Psi}$ and $\mathscr{A}'_{j} =
    \mathscr{N}_j
    \cap \mathscr{P}_j'
    \cap
    \mathscr{A}_{j}.$  The first is $(\filt_{n_1^+})$--adapted. \\
 $ \mathscr{Z}_{k_2}$&\eqref{eq:G3}& The event that ${Z}_{2^k}^{k_2}(\theta)$ is close to 
    $2(\log \Phi^*_{2^k}(e^{i\theta})-\log \Phi^*_{k_2}(e^{i\theta}))$ for $k$ in the bulk.\\
    \bottomrule
  \end{tabular}
  \vspace{0.2cm}
  \caption{Table of events}
\label{table3}
\end{table}

\section{The arc of the proof}
\label{sec-arc}
In this (long) section, we will give a proof of the main theorem, in which we defer many of the major technical arguments to later sections.  Each subsection, save for the first which contains preliminaries, carries one of the major steps in the proof of Theorem \ref{prop:bush}.
\subsection{Soft properties of the Pr\"ufer phases}
\label{subsec-softprufer}

We make here some elementary and important observations about the Pr\"ufer phases.  The Pr\"ufer phases have that $\theta \mapsto \Psi_k(\theta)$ and $\theta \mapsto \psi_k(\theta)$ are increasing functions of $\theta \in \R$ (see the discussion in \cite[Section 2]{KillipStoiciu}).  Furthermore, for any $\theta \in \R$ and $k \in \N,$ $\Psi_k(\theta + 2\pi) =  \Psi_k(\theta) + (k+1)2\pi.$

It follows that the relative Pr\"ufer phase $\psi_k(\theta) \geq 0$ for all $\theta \geq 0$ and all $k \in \N.$
However more is true.
For any real number $x,$ let $\lfloor x \rfloor_{2\pi}$ be the largest element of $2\pi \Z$ which is less than or equal to $x.$ Also, let $\{ x \}_{2\pi} = x -\lfloor x \rfloor_{2\pi}.$ Then
\begin{lemma} \label{lem:psikmonotone}
  For $\theta \geq 0,$  $\lfloor \psi_k(\theta) \rfloor_{2\pi}$ is a nondecreasing sequence in $k.$ Also, for any $k \geq 0,$ $\{ \psi_k(\theta) \}_{2\pi} \geq \theta.$
\end{lemma}
\begin{proof}
  We claim that the function
  \[
    (\psi,z) \mapsto \psi - 2\Im\left( \log(1-ze^{i\psi}) - \log(1-z) \right)
  \]
  defined on $\hohc{0,2\pi} \times \mathbb{D}$ is non-negative.  
  Observe that the function is harmonic in $z$ for fixed $\psi$ and bounded below for all $|z| < 1$.  Hence by the Lindel\"of maximum principle \cite[Lemma 1.1]{GarnettMarshall} the infimum of this function over $|z| < 1$ is bounded below by its infimum over $\{|z|=1\} \setminus F$ for any finite set $F$.  On the other hand $-2\Im\log(1-e^{i\alpha})=\pi-\alpha$ for $\alpha \in (0,2\pi),$ and hence for $|z|=1,$ the function takes values in $\{0,2\pi\},$ except for when $z \in \{ 1, e^{-i\psi}\}.$

  Next we decompose, for any $k \geq 0$ and any $\theta \geq 0,$
  \begin{equation*}
    \begin{aligned}
      \psi_{k+1}(\theta)
      &=\psi_{k}(\theta) + \theta - 2\Im\left( \log(1-\gamma_ke^{i\psi_k(\theta)}) - \log(1-\gamma_k) \right) \\
      &= \lfloor \psi_k(\theta) \rfloor_{2\pi} + \theta + \{\psi_k(\theta)\}_{2\pi} - 2\Im\left( \log(1-\gamma_ke^{i\{\psi_k(\theta)\}_{2\pi}}) - \log(1-\gamma_k) \right) \\
      &\geq \lfloor \psi_k(\theta) \rfloor_{2\pi} + \theta
      \geq \lfloor \psi_k(\theta) \rfloor_{2\pi}.
    \end{aligned}
  \end{equation*}
  Hence the claim follows.
\end{proof}

\subsection{Marginally Gaussian approximations}

We follow \cite{CMN} in introducing a family of processes with Gaussian marginals that well approximate the OPUC recurrence.
We recall that the Pr\"ufer phases can be expressed in terms of $\Gam$ distributed random variables as
\begin{equation}
  \gamma_j = \sqrt{
    \frac{E_j}{E_j+\Gamma^a_j}
  } e^{i\Theta_j},
  \quad
  \text{where}
  \quad
  \begin{aligned}
    &E_j \sim \Exponential(1), \\
    &\Gamma^a_j \sim \Gam(\tfrac{\beta}{2}(j+1)), \\
    &\Theta_j \sim \Unif([0,2\pi])
  \end{aligned}
  \label{eq:betagamma}
\end{equation}
with all variables independent.  We define the variables
\begin{equation}
  X_j = \sqrt{E_j} \Re e^{i\Theta_j},
  \quad
  Y_j = \sqrt{E_j} \Im e^{i\Theta_j},
  \quad
  \text{and}
  \quad
  Z_j = X_j+iY_j.
  \label{eq:gaussians}
\end{equation}
Then $\left\{ \{X_j,Y_j\}\!:\! j \geq 0 \right\}$ are jointly iid $N(0,\tfrac 12)$ and independent of $\left\{ \Gamma^a_j\! :\! j \geq 0\right\}.$   Set $\beta_j \coloneqq  \sqrt{\frac{\beta}{2}(j+1)},$ noting that $\Exp \Gamma^a_j = \beta_j^2.$

We will also introduce a continuous time marginally Gaussian process $(\mathfrak{Z}_t^\C : t \geq 0),$ which we shall use for some comparisons.  As $\left\{ Z_k \right\}$ are independent standard complex Gaussians with $\Exp |Z_k|^2 = 1,$ we may enlarge the probability space to include a complex Brownian motion $(\mathfrak{W}_t : t \geq 0)$ such that
\begin{equation}\label{eq:W}
  [ \mathfrak{W}_t, \mathfrak{W}_t] = 2t
  \quad
  \text{for all }
  t \geq 0,
  \quad\text{and}\quad
  \sqrt{\tfrac{2}{k+1}}Z_k = \mathfrak{W}_{H_{k+1}}-\mathfrak{W}_{H_{k}}
  \quad
  \text{for all }
  k \geq 0,
\end{equation}
with 
\begin{equation}
\label{eq-Hk}
H_k = \sum_{j=1}^k \frac{1}{j}
\end{equation}
 the $k$-th harmonic number (and $H_0=0$).
In terms of $\mathfrak{W}$ we define $\mathfrak{Z}^\C(\theta)$ by
\begin{equation}\label{eq:Z}
  \mathfrak{Z}_t^\C(\theta) = -\int_{H_{k_2}}^t
  {e^{i\Psi_{k(s)}(\theta)}}
  d\mathfrak{W}_s
  +
  \sqrt{\beta}\log \Phi^*_{H_{k_2}}(e^{i\theta})
\end{equation}
where for each $k \geq 0,$ $k(s) = k$ on $s \in \hohc{H_k,H_{k+1}}.$
Then by construction we have the relationship
\begin{equation}\label{eq:ZkZt}
  -\sqrt{\frac{2}{k+1}} Z_ke^{i\Psi_k(\theta)} =
  \mathfrak{Z}_{H_{k+1}}^\C(\theta)
  -\mathfrak{Z}_{H_{k}}^\C(\theta),
  \quad
  \text{for all } k \geq k_2.
\end{equation}
We also define $\mathfrak{Z}_t = \Re \bigl(\sigma \mathfrak{Z}_t^{\C}\bigr)$ for all $t \geq H_{k_2},$ which therefore satisfies that $(\mathfrak{Z}_{t+H_{k_2}}-\mathfrak{Z}_{H_{k_2}} : t \geq 0)$ is a standard real Brownian motion.

In \cite[Proposition 3.1]{CMN} it is shown that we can approximate $\log \Phi^*_k$ by $\mathfrak{Z}_{H_k}^\C$.
\begin{proposition} \label{prop:mgapproximation}
  \[
    \begin{aligned}
      &\lim_{k_2 \to \infty} \sup_{n \geq k_2} \sup_{\theta \in [0,2\pi]} |\sqrt{\tfrac{4}{\beta}}\mathfrak{Z}_{H_n}^\C(\theta)
      - 2\log \Phi^*_n(e^{i\theta})
      | = 0 \quad \As \\
    \end{aligned}
  \]
\end{proposition}
\noindent This is proven in \cite{CMN} though not explicitly stated as such.  A careful reading of \cite[Proposition 3.1]{CMN} (in which $\left\{ \mathfrak{Z}_{H_k}^\C(\theta) \right\}$ is denoted $\left\{ Z_k(\theta) \right\}$) shows that
their proof gives the statement in Proposition \ref{prop:mgapproximation}, which is stronger than what is claimed in \cite[Proposition 3.1]{CMN}.
We shall notate, where convenient, the coupled Gaussian random walks
\begin{equation}\label{eq:GkZt}
  G_k(\theta) = \sqrt{\tfrac{4}{\beta}} \mathfrak{Z}_{H_k}(\theta)
  \quad
  \text{for all } k \geq k_2.
\end{equation}

\subsection*{Truncation}
We will frequently need to truncate these variables.
We let $M \in \N$ be a fixed large parameter, and define an event $\mathscr{T}_M$ on which
\begin{equation}
  \max_{t \in [H_j,H_{j+1}]} | \mathfrak{W}_t - \mathfrak{W}_{H_j}|^2 \leq \frac{4\log (j)}{j}
  \quad\text{and}
  \quad
  |\sqrt{\Gamma^a_j} - \beta_j| \leq 4\sqrt{\log (j)}
  \quad \text{for all } M < j \leq \infty.
  \label{eq:good}
\end{equation}
By the construction of $\mathfrak{W},$ this event controls the magnitude of $E_j = |X_j+i Y_j|^2.$
By adjusting the cutoffs for $\Gamma^a_j$, we may assume that on $\mathscr{T}_M,$ $\Exp[(\Gamma^a_j - \beta_j^2)\one[\mathscr{T}_M]] = 0.$ We may further do this truncation in such a way that $\Pr(\mathscr{T}_M) \to 1$ as $M \to \infty.$

\subsection{First moment simplifications}\label{sec:milieu}
In this section, we will show that near-maxima typically arise with many simplifying features.  All the estimates in this section will use first-moment type estimates, which is to say that we will control the expected number of $j \in \mathcal{D}_{n/k_1}$ such that $W_j$ is large without some other simplifying properties taking place.

\subsubsection{The upsloping barriers}
\label{sec:upslope}

We will use the same barrier functions employed in \cite{CMN}, so as to reuse as much of the machinery developed there.
We begin by introducing a high barrier $A_k^{\ll}$, which can be used to give \emph{a priori} bounds on the growth of the processes $n \mapsto \varphi_n(\theta)$:
\begin{equation}
  A_k^{\ll}
  =A_k^{\ll,(n)}
  \coloneqq
  H_k +
  \begin{cases}
    (H_k)^{1/100} & \text{ if } H_k \leq \frac{1}{2}\log(n) \\
    (H_n - H_k)^{1/100} - \frac{3}{4}\log \log n& \text{ if }  \frac{1}{2}\log(n) < H_k \leq H_n,
  \end{cases}
  \label{eq:upperbarrier}
\end{equation}
where we recall 
that $H_k = \sum_{j=1}^k \frac{1}{j}$ is the $k$-th harmonic number,
see \eqref{eq-Hk}.

We let $\mathscr{B}_{n,k_2,k_6}$ be the event that the process is below this barrier at \emph{all} $\theta \in [0,2\pi]$ at dyadic time points and at the final time, i.e.\ for some $k_6$ and with $\log_2$ denoting the logarithm on 
basis $2$,
\begin{equation}
  \begin{aligned}
    \mathscr{B}_{n,k_2,k_6}
    \coloneqq
    &\bigl\{
      \forall~ \log_2 k_2 \leq k \leq \log_2 n, \theta \in [0,2\pi]
      :
      \sup_{ \theta \in [0,2\pi]}
      \varphi_{2^k}{(\theta)}
      \leq \sqrt{\tfrac{8}{\beta}}A_{2^k}^{\ll} + k_6
    \bigr\} \\
    &\cap
    \{ \sup_{\theta \in [0,2\pi]}
    \varphi_{n}{(\theta)}
    \leq \sqrt{\tfrac{8}{\beta}} A_n^{\ll} + k_6 \}.
  \end{aligned}
  \label{eq:allupperbarriers}
\end{equation}
In the definition, we use only integer $k,$ and we recall that $k_2 \in 2^{\N}.$
Then the event that the barrier is exceeded satisfies:
\begin{lemma}
  \[
    \lim_{k_6 \to \infty}
    \liminf_{k_2 \to \infty}
    \inf_{n \geq 1} \Pr\bigl( \mathscr{B}_{n,k_2,k_6} \bigr) = 1.
  \]
  \label{lem:barrier}
\end{lemma}
\noindent Note that this barrier curves \emph{above} the straight barrier (and further the function $H_k \mapsto A_{k}^{\ll}$ is piecewise concave).  We will also use barriers that curve \emph{below}, but the previous statement is not true for such barriers.
\begin{proof}
  By Proposition \ref{prop:mgapproximation} it is enough to prove the statement with $\varphi_{j}{(\theta)},$ replaced by $\{G_j(\theta)\}$.
  For the $\{G_j(\theta)\}$ process, this is \cite[(4.4)]{CMN}.
\end{proof}

Besides using \emph{a priori} upper bound on the growth of the process $\varphi_n,$ we will also only work on the event that the processes $G_k(\theta)$ and $\varphi_k(\theta)$ are close, and so we define
\begin{equation}\label{eq:g3}
  \begin{aligned}
    \mathscr{G}^2_{n,k_2}
    &=\left\{
      \sup_{\theta \in [0,2\pi]}
      \sup_{ {k_2} \leq k  \leq n}
      | G_{k}(\theta)
      -\varphi_{k}(\theta)
      | \leq 1
    \right\}.
  \end{aligned}
\end{equation}
We shall want to work on the event that for some $k_2,k_6 \in \N,$
\[
  \mathscr{G}_{n}
  =
  \mathscr{G}_{n,k_2,k_6}
  =
  \mathscr{G}^2_{n,k_2}
  \cap
  \mathscr{B}_{n,k_2,k_6}
  \cap
  \mathscr{T}_{k_2}
  ,
\]
where we recall that the truncation event $\mathscr{T}_{k_2}$ is defined in \eqref{eq:good}
and emphasize that this is indeed a typical event due to Proposition \ref{prop:mgapproximation}:
\begin{lemma}
  If we take $n$ large followed by $k_2$ and $k_6,$
  \[
    \liminf_{n \to \infty} \Pr ( \mathscr{G}_{n} ~\vert~\filt_{k_2}) \Prto[k_6,k_2] 1.
  \]
  \label{lem:goodstuff1}
\end{lemma}

\subsubsection{Downsloping barriers and the banana}
We shall use that extremal statistics are well approximated by restricting the $\{\varphi_n(\theta)\}$ to a fine mesh in $\theta$. \label{k5} We let $k_5 \in \N$ be a parameter we use to control the mesh size, which will be $\tfrac{2\pi}{4 k_5 n}\Z$.
We introduce the near--leader event $\mathscr{L}(\theta),$ which is simply that $\varphi_n(\theta)$ is large
\begin{equation}\label{eq:Ltheta}
  \begin{aligned}
    \mathscr{L}(\theta)
    &=
    \mathscr{L}_{k_6}(\theta)
    =
    \{
      \varphi_n{(\theta)}
      \in \sqrt{\tfrac{8}{\beta}}m_n
      +
      [-k_6, \infty)
    \}. \\
  \end{aligned}
\end{equation}
We also introduce further barrier functions $t \mapsto A_t^{p,\pm}$ for any $p \in \N$,
\begin{equation}
  A_t^{p,\pm}
  =
  A_t^{p,\pm, (n)}
  \coloneqq
  t
  +
  \begin{cases}
    - t ^{1/2 \mp {p}/({2p+1})} & \text{ if } t \leq \frac{1}{2}\log(n), \\
    -(\log n - t)^{1/2 \mp {p}/({2p+1})} - \frac{3}{4}\log\log n& \text{ if }  \frac{1}{2}\log(n)  < t \leq \log(n).
  \end{cases}
  \label{eq:barrier}
\end{equation}
A random walk conditioned to lie below the barrier $A_{(\cdot)}^{\ll}$ and conditioned to end near the barrier will tend to stay in the banana--like envelope $A_{(\cdot)}^{1,\pm}$ (and hence also $A_{(\cdot)}^{p,\pm}$ for any $p \geq 1$).

We introduce the barriers as it will be convenient over the course of the argument to change between barrier functions (all of which functionally play the same role).  To aid in this changing of barriers, it is convenient if we further restrict the process at the entrance time $H_{k_2}$ to be in an even more restrained window.  Recall that $s_{k_2}^{\pm} = \sqrt{\tfrac{8}{\beta}}( \log k_2 - (\log k_2)^{0.5\mp 0.01}),$ and define
\begin{equation}\label{eq:UU}
  \mathscr{U}(\theta)
  =
  \{
    \varphi_{k_2}(\theta) \in [s_{k_2}^{-}, s_{k_2}^+]
  \}
  =
  \{
    \sqrt{\tfrac{4}{\beta}}\mathfrak{Z}_{H_{k_2}}(\theta) \in [s_{k_2}^{-}, s_{k_2}^+]
  \}.
\end{equation}

For fixed $\theta,$ $(\mathfrak{Z}_t(\theta) : t)$ has the law of a standard Brownian motion.
This motivates the introduction of the following event.
\begin{align}\label{eq:Rhat}
    \widehat{\mathscr{R}}(\theta)
    &=
    \mathscr{U}(\theta)
    \cap
    \left\{
      \forall~t \in [H_{k_2},H_n-k_4]
      :
      \sqrt{\tfrac{8}{\beta}}A_{t}^{1,-}
      \leq
      \sqrt{\tfrac{4}{\beta}}\mathfrak{Z}_t(\theta)
      \leq \sqrt{\tfrac{8}{\beta}}A_{t}^{1,+}
    \right\} \\
    &\quad \bigcap
     \left\{
      \forall t \in [H_n-k_4,H_n]
   \!   :
     \! \sqrt{\tfrac{4}{\beta}}\mathfrak{Z}_t(\theta)
      \leq \sqrt{\tfrac{8}{\beta}}\bigl(t-\tfrac{3}{4}\log\log n +
      \tfrac{1}{\sqrt{2}}\bigl( \tfrac{t (H_n - t + (\log k_5)^{50})}{H_n}\bigr)^{1/50}\bigr)
    \right\}.\nonumber
  \end{align}
We end the barrier $k_4$ time steps early, which is relatively early in the sense that $k_4 \gg k_5.$  In some instances, we need barrier information which continues all the way to the end, for which reason we include the second part. We note that this is essentially provided to us by the good event $\mathscr{B}_{n,k_2,k_6}$ in \eqref{eq:allupperbarriers}, and it is a small argument to simply include the continuous part.

We now show by a first moment argument that we may restrict attention to those angles for which this event occurs.
\begin{proposition}  For all $k_6$
  \[
    \limsup_{n \to \infty}
    \sum_{j = 1}^{4k_5 n}
    \Exp[
      \one\{\mathscr{L}(\tfrac{\pi j}{2k_5 n})\}
      \one\{
        \widehat{\mathscr{R}}(\tfrac{\pi j}{2k_5 n})^c
      \}
      \one\{\mathscr{G}_{n}\}
      ~\vert~
      \filt_{k_2}
    ]
    \Prto[k_5,k_4,k_2]
    0.
  \]
  \label{prop:nearmaximarays}
\end{proposition}
\begin{proof}[Proof of Proposition \ref{prop:nearmaximarays}]
  On the event $\mathscr{G}_{n}$ (see \eqref{eq:allupperbarriers} and  \eqref{eq:g3}) we have that for any $\theta \in [0,2\pi]$
  \begin{equation}\label{eq:nma1}
    \sqrt{\tfrac{4}{\beta}}\mathfrak{Z}_{H_{2^k}}(\theta)
    =
    G_{2^k}(\theta)
    \leq 2k_6 + \sqrt{\tfrac{8}{\beta}}A_{2^k}^{\ll},
    \quad
    \text{for all}
    \quad
    \log_2 k_2 \leq k \leq \log_2 n,
  \end{equation}
  for integer $k$ and for $k= \log_2 n.$
  On the other hand, on the event $\mathscr{L}(\theta) \cap \mathscr{G}_{n},$ we have
  \begin{equation}
    \sqrt{\tfrac{4}{\beta}}\mathfrak{Z}_{H_{n}}(\theta)
    \geq \varphi_n(\theta)-k_6
    \geq \sqrt{\tfrac{8}{\beta}}m_n - 2k_6.
    \label{eq:nma2}
  \end{equation}
   We recall that
  $s_{k_2}^{\pm} = \sqrt{\tfrac{8}{\beta}}( \log k_2 - (\log k_2)^{0.5\mp 0.01}).$
It is a classical estimate on Gaussian random walk, that when \eqref{eq:nma1}
  and 
  \eqref{eq:nma2} and $\mathscr{U}(\theta)$ occur, the entropic envelope condition $\widehat{\mathscr{R}}(\theta)$ is typical. 
  Indeed,  for $\varphi_{k_2}(\theta) \in [s_{k_2}^{-},s_{k_2}^{+}],$
  i.e. when $\mathscr{U}(\theta)$ occurs,
  we show in Lemma \ref{lem:allhailbanana} that there are constants $C,c$ so that for all $n \gg k_2 \gg k_4$ sufficiently large,
  \begin{equation}\label{eq:nma4}
    \Pr(\widehat{\mathscr{R}}(\theta)^c \cap \eqref{eq:nma2}
    \cap \eqref{eq:nma1} \vert \mathfrak{Z}_{H_n}(\theta), \filt_{k_2})
    \leq C
    \frac{
      (\sqrt{2}\log k_2 - \mathfrak{Z}_{H_{k_2}}(\theta))
      e^{-(\log k_5)^{2}}
    }{\log n}.
  \end{equation}
  Now as we still have that the increment
  \[
    \mathfrak{Z}_{H_n}(\theta)-\mathfrak{Z}_{H_{k_2}}(\theta)
    \in \sqrt{2}m_n-\mathfrak{Z}_{H_{k_2}}(\theta) + [-2k_6, 2k_6],
  \]
  it follows from integrating the Gaussian density that
  \begin{equation}\label{eq:nma4a}
    \begin{aligned}
      &\Pr(\widehat{\mathscr{R}}(\theta)^c \cap \eqref{eq:nma2}
      \cap \eqref{eq:nma1}\cap \mathscr{G}_n
       \vert \filt_{k_2}) \\
      &\leq C
      \frac{
        (\sqrt{2}\log k_2 - \mathfrak{Z}_{H_{k_2}}(\theta))
e^{-(\log k_5)^{2}}
      k_6}{\log n (H_n - H_{k_2})^{1/2}}\exp\left(
      -\frac{(\sqrt{2}m_n-\mathfrak{Z}_{H_{k_2}}(\theta) - 2k_6)^2}{2(H_n - H_{k_2})}
      \right). \\
      &\leq C
      \frac{
        (\sqrt{2}\log k_2 - \mathfrak{Z}_{H_{k_2}}(\theta))
e^{-(\log k_5)^{2}}
      k_6}{ n}\exp\left(
      2\sqrt{2}k_6
      +
      \sqrt{2}\mathfrak{Z}_{H_{k_2}}(\theta)- \log k_2
      \right). \\
    \end{aligned}
  \end{equation}
  We conclude that on the event 
  $\varphi_{k_2}(\theta) \in [s_{k_2}^{-},s_{k_2}^{+}],$
  \begin{equation}\label{eq:nma4b}
    \begin{aligned}
      \Pr(
      \mathscr{L}(\theta) \cap \widehat{\mathscr{R}}(\theta)^c
      \cap \mathscr{G}_n
      \vert \filt_{k_2})
      &\leq C_\beta
      \frac{
	(\sqrt{2}\log k_2 -\mathfrak{Z}_{H_{k_2}}(\theta))
	 k_6
e^{-(\log k_5)^{2}}
      }{ n}
      e^{6k_6
        +
        \sqrt{\tfrac{\beta}{2}}\varphi_{k_2}(\theta)- \log k_2
      }.
    \end{aligned}
  \end{equation}

  We must also give a relatively sharp bound when $\varphi_{k_2}(\theta)$ is outside this good range.  These estimates are already given in \cite{CMN}, but are essentially standard Gaussian random walk estimates going back to \cite{bramson83}.
  As we work on the event $\mathscr{G}_n,$ \eqref{eq:nma1} is given to us.  The process $k \mapsto G_{2^k}(\theta) - G_{k_2}(\theta)$ is a Gaussian random walk whose increments have (nearly) variance ${\frac{4}{\beta}}\log 2$ started from $0.$  From \eqref{eq:nma1}, this process stays below the barrier $k\mapsto \sqrt{\frac{8 \log 2}{\beta}}(k +g(k))$ with $g$ controlled by $({k \wedge (\log_2 n-k)})^{1/100}.$ The probability of \eqref{eq:nma2} happening is thus uniformly bounded from the appropriate ballot theorem (see e.g. \cite[Lemma 2.1]{BRZ})
  \begin{align*}
    & \Pr(\eqref{eq:nma2} \cap \mathscr{G}_n
\vert \eqref{eq:nma1}, \filt_{k_2})\\
 &   \leq
    C_\beta
    \exp\left(-\frac{(\sqrt{\tfrac{8}{\beta}}m_n - \varphi_{k_2}(\theta) - 
    3k_6)^2}{\frac{8}{\beta}\log(n/k_2)}\right)
    \frac{ (\sqrt{2}\log k_2 - \sqrt{\tfrac{\beta}{4}} \varphi_{k_2}(\theta)+
    2k_6) k_6}{(\log(n/k_2))^{3/2}},
  \end{align*}
  for some constant $C_\beta.$ On sending $n \to \infty,$ we therefore have 
  (after increasing $C_\beta$) that the left side of the last display is bounded above by the bound 
  \begin{align}\label{eq:nmr0}
 &   \Pr(\eqref{eq:nma2}\cap \mathscr{G}_n
 \vert \eqref{eq:nma1}, \filt_{k_2})\\
   & \nonumber \leq
    C_\beta
    k_6^2
    \exp\left(
    -\log( n k_2) + \sqrt{\tfrac{\beta}{2}}(  \varphi_{k_2}(\theta) + 
    c_\beta 
    k_6)(1+o_n(1))
    \right)
    \bigl(\sqrt{2}\log k_2 - \sqrt{\tfrac{\beta}{4}} \varphi_{k_2}(\theta)+
    2k_6\bigr)
    .
  \end{align}

  We can now return to the summation we wish to bound.
  \begin{equation}\label{eq:nmr1}
    \begin{aligned}
    &  \limsup_{n \to \infty}
      \sum_{j = 1}^{4k_5 n}
      \Exp[
        \one\{\mathscr{L}(\tfrac{\pi j}{2k_5 n})\}
        \one\{
          \widehat{\mathscr{R}}(\tfrac{\pi j}{2k_5 n})^c
        \}
        \one\{\mathscr{G}_{n}\}
        ~\vert~
        \filt_{k_2}
      ] \\
      &\leq
      \limsup_{n \to \infty}
      \sum_{j = 1}^{4k_5 n}
      \one[{ \varphi_{k_2}(\tfrac{\pi j}{2k_5 n}) \in [s_{k_2}^{-},s_{k_2}^+] }]
      \Pr(
      \widehat{\mathscr{R}}(\tfrac{\pi j}{2k_5 n})^c\cap \mathscr{G}_{n}
      \vert \eqref{eq:nma2}\eqref{eq:nma1}, \filt_{k_2})
      \Pr(\eqref{eq:nma2}\mathscr{G}_{n} \vert \eqref{eq:nma1}, \filt_{k_2}) \\
      &\qquad +
      \limsup_{n \to \infty}
      \sum_{j = 1}^{4k_5 n}
      \one[{ \varphi_{k_2}(\tfrac{\pi j}{2k_5 n}) \not\in [s_{k_2}^{-},s_{k_2}^+] }]
      \Pr(\eqref{eq:nma2} \cap \mathscr{G}_{n}\vert \eqref{eq:nma1}, \filt_{k_2}).
      \\
    \end{aligned}
  \end{equation}
  We emphasize that in the events $\eqref{eq:nma2}$ and $\eqref{eq:nma1}$ we have taken $\theta \to \tfrac{\pi j}{2k_5 n}.$

  From the almost sure continuity of $\varphi_{k_2}(\theta)$, we get convergence of these sums to integrals, and using \eqref{eq:nmr0} and \eqref{eq:nma4b}
  \begin{equation}\label{eq:nmr3}
    \begin{aligned}
      \limsup_{n \to \infty}
      &\sum_{j = 1}^{4k_5 n}
      \Exp[
        \one\{\mathscr{L}(\tfrac{\pi j}{2k_5 n})\}
        \one\{
          \widehat{\mathscr{R}}(\tfrac{\pi j}{2k_5 n})^c
        \}
        \one\{\mathscr{G}_{n}\}
        ~\vert~
        \filt_{k_2}
      ] \\
      &\leq
      C_{\beta}
      k_5
e^{-(\log k_5)^{2}}
      \int_0^{2\pi}
      e^{\sqrt{\tfrac{\beta}{2}}( 6k_6 + \varphi_{k_2}(\theta)) -\log(k_2)}
      \bigl(C_\beta k_6+\sqrt{2}\log k_2 - \sqrt{\tfrac{\beta}{4}} \varphi_{k_2}(\theta)\bigr)
      d\theta \\
      &+
      k_5\int_0^{2\pi}
      e^{\sqrt{\tfrac{\beta}{2}}(  \varphi_{k_2}(\theta) + k_6) -\log(k_2)}
      \one[{ \varphi_{k_2}(\theta) \not\in [s_{k_2}^{-},s_{k_2}^+] }]
      \bigl|\sqrt{2}\log k_2 - \sqrt{\tfrac{\beta}{4}} \varphi_{k_2}(\theta)
      +c_\beta k_6\bigr|
      d\theta.
    \end{aligned}
  \end{equation}
  Hence by Theorem \ref{thm:Bj} the convergence follows.
\end{proof}

\subsubsection{Meshing}

We define for each $j \in {\mathcal{D}}_{n/k_1}$ the set $I_j$ of $\tfrac{2\pi}{4 k_5n}\mathbb{Z}$ which is at least distance $\tfrac{2\pi \log k_1}{n}$ from the complement of $\widehat{I_j}$, i.e.
\begin{equation}\label{eq:Ij}
  {I_j} = \tfrac{2\pi}{4 k_5n}\mathbb{Z} \cap 2\pi[ \tfrac{(j-1)k_1}{n} + \tfrac{\log k_1}{n},\tfrac{jk_1}{n} - \tfrac{\log k_1}{n}  ].
\end{equation}
For any $j\in\Z,$ define
\begin{equation}
  \begin{aligned}
    &{W_j} \coloneqq \max_{\theta \in {I_j}} \{\varphi_n(\theta)\} - \sqrt{\tfrac{8}{\beta}}m_{n}.
  \end{aligned}
  \label{eq:localmax}
\end{equation}
Deterministically, it is possible to bound the error of $\widehat{W}_j$ (recall \eqref{eq:localmaxhat}) and the maximum of $\varphi_n$ over $\tfrac{2\pi}{4 k_5n}\mathbb{Z}$ (for example as in \cite[Proposition 4.5]{CMN}).  We shall expand upon this by showing that in fact we can guarantee $\widehat{W}_j$ is large implies $W_j$ is large, except with a negligible probability, as a consequence of interpolation.

\begin{proposition}  Any interval in which $\widehat{W}_j$ is large must also have $W_j$ is large in the sense that
for any $k_7$
  \[
    \limsup_{k_1,n \to \infty}
    \sum_{j \in \mathcal{D}_{n/k_1}}
    \Exp[
      \one\{ \widehat{W}_j \in [-k_7,k_7]\}
      \one\{ W_j \not{\in} [-k_6,k_6]\}
      \one\{\mathscr{G}_{n}\}
      ~\vert~
      \filt_{k_2}
    ]
    \Prto[k_6,k_5,k_2]
    0.
  \]
  \label{prop:interpolation0}
\end{proposition}
\noindent We delay the proof to the end of Section \ref{sec:meshing} until after introducing the relevant interpolation tools.

\subsubsection{A canonical trunk with small oscillations}
Tightness of 
$\max_j \widehat{W}_j$ and Proposition \ref{prop:interpolation0} guarantee that we can find choices of $\theta$ in the sufficiently fine mesh that are close to maximum.  Combining this with Proposition \ref{prop:nearmaximarays}, we conclude that all leading $\theta$ behave like leading particles in a branching random walk in the sense that they are confined to the banana-like region.  We would like to be able to say that for each interval $I_j,$ for which $W_j \in [-k_6,k_6],$ all the random walks $(\varphi_{k}(\theta) : k)$ for all $\theta \in I_j$ behave like 
a single random walk, at least for $k \leq {n}_1^+$ (which is long before the effective branching time).

Recall that $\theta_j$ is the largest point of $\widehat{I}_j.$  We proceed to define the events for $j \in \mathcal{D}_{n/k_1}$
\begin{equation}
  \begin{aligned}
   &\mathscr{O}_j =
  \bigl\{
    \sup_{\theta \in \widehat{I}_j}
    \max_{k_2 \leq k \leq \widehat{n}_1}
    |\varphi_{k}{(\theta)}-\varphi_{k}{(\theta_j)}|
    \leq \sqrt{\tfrac{{k_1}}{\widehat{k}_1 }} 
  \bigr\}, \\
  &\mathscr{O}_j^+ =
  \bigl\{
    \sup_{\theta \in \widehat{I}_j}
    \max_{k_2 \leq k \leq {n}^+_1}
    |\varphi_{k}{(\theta)}-\varphi_{k}{(\theta_j)}|
    \leq \sqrt{\tfrac{{k_1}}{\widehat{k}_1} } 
  \bigr\},
  \quad
  \text{and}  \\
  \quad
  &\mathscr{O}_j^{\Psi} =
  \bigl\{
    \sup_{\theta \in \widehat{I}_j} |\Psi_{n_1^+}{(\theta)}-\Psi_{n_1^+}{(\theta_j)}| \leq
    \tfrac{k_1(\log k_1)^{50}}{k_1^+ }
  \bigr\}
  .
  \label{eq:Oj}
  \end{aligned}
\end{equation}
In the imaginary case $\sigma = i$, we only use the event 
$\mathscr{O}_j^{\Psi}$; in the real case we use the event $\mathscr{O}_j$, 
which of course implies 
the event $\mathscr{O}_j^+$.

These events control the oscillation of the respective functions on an interval of $\theta$ which is much smaller than the natural smoothness scale of $\varphi_{n_1^+}(\theta)$ is small.  We are interested in controlling this oscillation when the random walk $(\varphi_{k}{(\theta_j)} : k)$ is a near-leader.
Indeed, to that end, for $j \in \mathcal{D}_{n/k_1},$ define the \emph{good ray events}
\begin{equation}
  \begin{aligned}
    &\mathscr{R}^p_{j}(m)
    =
    \mathscr{U}(\theta_j)
    \cap
    \left\{
      \forall~ H_{k_2}  \leq t \leq H_m
      :
      \sqrt{\tfrac{8}{\beta}}A_{t}^{p,-}
      \leq
      \sqrt{\tfrac{4}{\beta}} \mathfrak{Z}_t(\theta_j)
      \leq \sqrt{\tfrac{8}{\beta}}A_{t}^{p,+}
    \right\},
  \end{aligned}
  \label{eq:rayevent}
\end{equation}
This differs from the previous \eqref{eq:Rhat} in that we specialize to the angle $\theta_j$ and that it only bounds the behavior of the walk up to time $\log m.$
We have also slightly increased the barrier sizes from \eqref{eq:Rhat}.

Our main strategy for this estimate in the $\sigma=1$ case is to use the \emph{a priori} bounds which ultimately follow from $\Phi^*_k$ being a polynomial of degree $k.$  This is contained in the following:
\begin{proposition}  (Real case $\sigma=1$) For any $k_5,k_6$ and 
  all $n \gg k_1$ sufficiently large,
  on the event $\mathscr{G}_n$,  for any $j$ for which there is a
  $\theta' \in I_j$ satisfying 
  \[
    \sqrt{\tfrac{\beta}{8}}
    \varphi_{k}(\theta') \geq
    A_{k}^{4,-} \quad \text{for all}\quad k_2 \leq k \leq \widehat{n}_1,
  \]
  also $\mathscr{O}_j$ holds. 
  \label{prop:trunkray1}
\end{proposition}

\begin{proof}
  Recall that in the real case $(\sigma = 1)$,
  $\varphi_k(\theta) = 2\log | \Phi_k^*(e^{i\theta})|.$
    Here we will show that almost surely
    on the events in question,
    there is a deterministic error $\varepsilon_{k_1,n}$
    so that for all $k_1$ sufficiently large,
  \[
    \max_{\theta \in I_j}
    \max_{k_2 \leq k \leq \widehat{n}_1} |\varphi_k(\theta) - \varphi_k(\theta')| \leq \varepsilon_{k_1,n}
    \quad
    \text{where}
    \quad
    \limsup_{n\to \infty} \varepsilon_{k_1,n} \leq
    \sqrt{\tfrac{{k_1}}{\widehat{k}_1 }}.
  \]
  This will show that $\mathscr{O}_j$ holds for all $n$ and $k_1$ large.

  On $\mathscr{G}_{n},$ we have that for any $k_2\leq  k \leq \widehat{n}_1$ and $|z| = 1$
  \[
    |\Phi_k^*(z)|^2 \leq
    \exp(
    \sqrt{\tfrac{8}{\beta}}A_{k}^{\ll}
    +k_6
    )
    \leq
    \exp(
    \sqrt{\tfrac{8}{\beta}}
    (
    A_{H_k}^{n,-}
    + 2\log(\tfrac n k)^{17/18}
    )
    +k_6
    )
  \]
  (A-priori, we only have this bound on dyadic integers; however, the
  event $\mathscr{T}_{k_2}$, which forms part of $\mathscr{G}_{k_2}$, yields
  together with the latter that one can interpolate the bound to all integers in the indicated range.)
  From Bernstein's inequality, see Theorem \ref{thm:Bernstein}, for any $k \leq \widehat{n}_1$ and $|z| = 1,$
  \[
    |\tfrac{d}{dz}\Phi_k^*(z)| \leq k
    \exp(
    \sqrt{\tfrac{2}{\beta}}A_{k}^{\ll}
    + \tfrac{k_6}{2}
    ).
  \]
  By construction $|\theta - \theta'| \leq \frac{2\pi k_1 }{n}$ for all $\theta \in I_j.$
  Hence for any $\theta \in I_j,$
  using the lower bound assumption on $\varphi_{k}(\theta')$,
  \[
    \begin{aligned}
      \frac{ |\Phi_k^*(e^{i\theta}) - \Phi_k^*(e^{i\theta'})|}
      {|\Phi_k^*(e^{i\theta'})|}
      &\leq
      \frac{7k k_1}{n}
      \exp(
      \tfrac{k_6}{2}+
      \sqrt{\tfrac{32}{\beta}}
      \log(\tfrac n k)^{17/18}
      )
      \\ &=
      7 
      \exp(
      -\log \tfrac{n}{k}
      +\log(k_1)
      +\tfrac{k_6}{2}+
      \sqrt{\tfrac{32}{\beta}}
      \log(\tfrac n k)^{17/18}
      ).
    \end{aligned}
  \]
  The mapping $x \mapsto -\log x  + C(\log x)^{17/18}$ is decreasing for all $x$ bigger than some $x_0$ depending only on $C,$ and therefore
  as $k \leq \widehat{n}_1$, we have that for all $\widehat{k}_1$ sufficiently large
  \[
    |\varphi_k(\theta) - \varphi_k(\theta')|
    =
    2 \biggl| \log\biggl(
    1+
    \frac{ |\Phi_k^*(e^{i\theta})| - |\Phi_k^*(e^{i\theta'})|}
    {|\Phi_k^*(e^{i\theta'})|}
    \biggr)
    \biggr|
    \leq
    14 
    \exp(
    \log(\tfrac{k_1}{\widehat{k}_1})
    +\tfrac{k_6}{2}+
    \sqrt{\tfrac{32}{\beta}}
    \log(\widehat{k}_1)^{17/18}
    ),
  \]
  which
  tends to $0$ for any $\{k_p : p \geq 2\}.$
  Moreover, recalling that from how $\widehat{k}_1$ is chosen
  $\log(\tfrac{k_1}{\widehat{k}_1}) = \log(\widehat{k}_1)^{19/20}(1+o_{k_1})$.
  Hence we conclude that
  for all $k_1$ sufficiently large
  \[
    |\varphi_k(\theta) - \varphi_k(\theta')|
    =
    2  \biggl| \log\biggl(
    1+
    \frac{ |\Phi_k^*(e^{i\theta})| - |\Phi_k^*(e^{i\theta'})|}
    {|\Phi_k^*(e^{i\theta'})|}
    \biggr)
    \biggr|
    \leq
    2\exp(
    \tfrac12\log(\tfrac{k_1}{\widehat{k}_1})
    ).
  \]
  This implies that $\mathscr{O}_j$ occurs.
 \end{proof}

\begin{corollary}  (Real case $\sigma=1$) For any $k_5,k_6$
  \[
    \limsup_{k_1,n \to \infty}
    \sum_{j \in \mathcal{D}_{n/k_1}}
    \Exp[
      \one\{W_j \in [-k_6,k_6]\}
      (\one\{{\bigl(\mathscr{R}_{j}^2(\widehat{n}_1)\bigr)^c}\} + \one\{\mathscr{O}_j^c\})
      \one\{\mathscr{G}_{n}\}
      ~\vert~
      \filt_{k_2}
    ]
    \Prto[k_4,k_2]
    0.
  \]
  \label{cor:trunkray1}
\end{corollary}
\begin{proof}
  Recall the ray event $\widehat{\mathscr{R}}(\theta)$, c.f.\ \eqref{eq:Rhat}. 
  We further define the event $\widehat{\mathscr{R}}_j$ as the event that all almost--leaders in $I_j$ satisfy the ray event $\widehat{\mathscr{R}}(\theta)$:
  \begin{equation}\label{eq:Rhatj}
    \widehat{\mathscr{R}}_j =
    \bigcap_{ \theta \in I_j}
    (\mathscr{L}(\theta)^c
    \cup
    \widehat{\mathscr{R}}(\theta)).
  \end{equation}
  We note that using Proposition \ref{prop:nearmaximarays}, with high probability at any mesh point $\theta$ at which $\mathscr{L}(\theta)$ holds, the ray event $\widehat{\mathscr{R}}(\theta)$ holds as well.
  Hence it suffices to show
  \[
    \limsup_{k_1,n \to \infty}
    \sum_{j \in \mathcal{D}_{n/k_1}}
    \Exp[
      \one\{W_{j} \in [-k_6,k_6]\}
      \one\{\widehat{\mathscr{R}}_j\}
      (\one\{\bigl(\mathscr{R}_{j}^2(\widehat{n}_1)\bigr)^c\} + \one\{\mathscr{O}_j^c\})
      \one\{\mathscr{G}_{n}\}
      ~\vert~
      \filt_{k_2}
    ]
    \Prto[k_4,k_2]
    0.
  \]
  Note that on the event $\{W_{j} \in [-k_6,k_6]\} \cap \widehat{\mathscr{R}}_j$ (from the definition \eqref{eq:Rhatj}) we have that there is a $\theta' \in I_j$ for which $\mathscr{L}(\theta') \cap \widehat{ \mathscr{R}}(\theta')$ holds.  On $\mathscr{G}_n$, we have the uniform bound
   for all $\theta \in I_j$
    \[
    |\varphi_k(\theta)
    -\sqrt{\tfrac{4}{\beta}}
    (
    \mathfrak{Z}_{H_k}(\theta)
    )
    |
    \leq k_6,
  \]
  and hence at $\theta'$, for all $k_2 \leq k \leq \widehat{n}_1$
  \[
    \varphi_{k}(\theta')
    \geq
    \sqrt{\tfrac{4}{\beta}}
    \mathfrak{Z}_{H_{k}}(\theta)
    -k_6
    \geq
    \sqrt{\tfrac{8}{\beta}}
    A^{1,-}_{k}
    -k_6.
  \]
  Thus from Proposition \ref{prop:trunkray1}, $\mathscr{O}_j$ holds for all $k_1 \gg k_2$ sufficiently large.

  We must still show that $\mathscr{R}_j^2(\widehat{n}_1)$ occurs.
  As the event $\widehat{\mathscr{R}}(\theta')$ holds and on $\mathscr{O}_j$ we can bound $|\varphi_{\widehat{n}_1}(\theta')-\varphi_{\widehat{n}_1}(\theta_j)|=o_{k_2}(1)$,
  then when $t=H_k$ for integer $k$ with $\log_2 k_2 \leq k \leq \log_2 \widehat{n}_1$
  \[
    \sqrt{\tfrac{8}{\beta}}A_{t}^{1,-}-2k_6
    \leq
    \sqrt{\tfrac{4}{\beta}} \mathfrak{Z}_t(\theta_j)
    \leq \sqrt{\tfrac{8}{\beta}}A_{t}^{1,+}+2k_6.
  \]
  Now the oscillations of $\mathfrak{Z}_t(\theta')$ for $t \in [H_{k},H_{k+1}]$ are controlled, on $\mathscr{G}_n$ (see \eqref{eq:good}), by something which is $o_{k_2}(1).$  Thus the conclusion holds with $p = 2$, which absorbs the extra $k_6$ term.
\end{proof}

This does not complete the analysis of the real case.  We will also need that the Pr\"ufer phases do not oscillate much.  This we reduce to a conditional first moment estimate on $\Psi$.
\begin{proposition}  (Real case $\sigma=1$) For any $k_2,k_4,k_5,k_6$
  and for any $p \leq 4$,
  \[
    \limsup_{k_1,n \to \infty}
    \frac{
    (\log k_1)^{25}
    k_1}{\widehat{k}_1}
    \times
    \!\!
    \sum_{j \in \mathcal{D}_{n/k_1}}\!\!
    \Exp[
      e^{\sqrt{\tfrac{\beta}{2}}\varphi_{\widehat{n}_1}(\theta_j) - 2m_{\widehat{n}_1}}
      \one\{{\mathscr{R}}^p_{j}(\widehat{n}_1)\}
      \one\{\mathscr{O}_j\}
      (1-\one\{\mathscr{O}_j^\Psi\})
      \one\{\mathscr{G}_{n}\}
      ~\vert~
      \filt_{k_2}
    ]
    =0, \As
  \]
  \label{prop:trunkray2}
\end{proposition}
We note that the factor $(\log k_1)^{25}$ should be considered as a rate in this statement: the remainder of the constants would correctly balance the sum if the $\mathscr{O}^\Psi_j$ term were removed.  Indeed our first application of this will be to show:
\begin{corollary}  (Real case $\sigma=1$) For any $k_2,k_4,k_5,k_6$
  and for any $p \leq 4$,
  \[
    \limsup_{k_1,n \to \infty}
    \sum_{j \in \mathcal{D}_{n/k_1}}
    \Exp[
      \one\{W_j \in [-k_6,k_6]\}
      \one\{{\mathscr{R}}^p_{j}(\widehat{n}_1)\}
      \one\{\mathscr{O}_j\}
      (1-\one\{\mathscr{O}_j^\Psi\})
      \one\{\mathscr{G}_{n}\}
      ~\vert~
      \filt_{k_2}
    ]
    =0 \quad \As
  \]
  \label{cor:trunkray2}
\end{corollary}
\begin{proof}[Proof of Corollary \ref{cor:trunkray2}]
  On the event $\mathscr{O}_j \cap \mathscr{G}_n,$ for all $n$ sufficiently large, the event $W_j \in [-k_6,k_6]$ is contained in the event that for one of the $C k_5 k_1$ elements $\theta \in I_j$ and a constant $C(\beta,k_6)$ sufficiently large
  \begin{equation}\label{eqo:Ginc}
    \mathfrak{Z}_{H_n}(\theta)
    -\mathfrak{Z}_{H_{\widehat{n}_1}}(\theta)
    \geq
    \sqrt{2}
    \log(n/\widehat{n}_1)
    +
    (
    \sqrt{2}
    (\log(\widehat{n}_1)-\tfrac34 \log\log n)
    -
    \sqrt{\tfrac{\beta}{4}}
    \varphi_{\widehat{n}_1}(\theta_j)
    )
    -C(\beta,k_6).
  \end{equation}
  Thus conditioning on $\filt_{\widehat{n}_1}$
  and
  using the standard Gaussian tail bound,
  \begin{equation}\label{eq:tr22}
    \begin{aligned}
      \Exp[
        &\one\{W_j \in [-k_6,k_6]\}
	      \one\{{\mathscr{R}}^p_{j}(\widehat{n}_1)\}
        \one\{\mathscr{O}_j\}
        (1-\one\{\mathscr{O}_j^\Psi\})
        \one\{\mathscr{G}_{n}\}
        ~\vert~
        \filt_{k_2}
      ] \\
      &\leq
      C_\beta
      \frac{k_5k_1 \widehat{n}_1e^{C(\beta,k_6)}}{n}
      \Exp\biggl[
        e^{ \sqrt{\tfrac{\beta}{2}}
	\varphi_{\widehat{n}_1}(\theta_j)
	-2 m_{\widehat{n}_1}}
        \one\{{\mathscr{R}}_{j}^p(\widehat{n}_1)\}(1-\one\{\widehat{\mathscr{O}_j}\})
        \one\{\mathscr{G}_{n}\}
        ~\vert~
        \filt_{k_2}
      \biggr].
    \end{aligned}
  \end{equation}
  By how $\widehat{n}_1$ is defined (see \eqref{eq:k2k3}) and how the barrier is defined \eqref{eq:barrier}, this tends to $0$ with $k_1 \ll n$ (deterministically).
  Thus the result follows from Proposition \ref{prop:trunkray2}.
\end{proof}
\begin{proof}[Proof of Proposition \ref{prop:trunkray2}]
  Using the bound on the difference of $\sqrt{\tfrac{\beta}{2}}\varphi_{\widehat{n}_1}(\theta_j)$ and $\sqrt{2}\mathfrak{Z}_{H_{\widehat{n}_1}}$,
  it suffices to show
  \begin{equation}
    \label{eq-241223}
    \begin{aligned}
    &\limsup_{k_1,n \to \infty}
    \frac{
    (\log k_1)^{25}
    k_1}{\widehat{k}_1} \\
    &\times\!\!
    \sum_{j \in \mathcal{D}_{n/k_1}}
    \Exp[
      \exp\bigl(\sqrt{2}\mathfrak{Z}_{H_{\widehat{n}_1}}(\theta_j) - 2m_{\widehat{n}_1}\bigr)
      \one\{{\mathscr{R}}_{j}^p(\widehat{n}_1)\}
      \one\{\mathscr{O}_j\}
      (1-\one\{\mathscr{O}_j^\Psi\})
      \one\{\mathscr{G}_{n}\}
      ~\vert~
      \filt_{k_2}
    ]
    =0 \quad \As
  \end{aligned}
\end{equation}

    Let $\theta_{j}^{\mp}$ be the largest and smallest element of $\widehat{I}_j$ respectively.
  Introduce the event
  \begin{equation}\label{eq:tr21}
    \begin{aligned}
      &\widehat{\mathscr{O}}_j
      =
      \biggl\{
        \max_{k_2 \leq k \leq n_1^+}
        |\Psi_{k}(\theta^{-}_j) - \Psi_{k}(\theta^{+}_j)| \leq
        \tfrac{k_1(\log k_1)^{50}}{2k_1^+ }
      \biggr\}.
    \end{aligned}
  \end{equation}
  We emphasize the ray event $\mathscr{R}^p_j(\widehat{n}_1)$ runs to a much later time (when $k_1$ is large) than the oscillation event $\widehat{\mathscr{O}}_j.$

  Our main task will be to show that conditionally on
  $\mathfrak{Z}_{H_{\widehat{n}_1}}(\theta_j),$
  $\filt_{k_2},$
  and $\mathscr{R}^p_j(\widehat{n}_1),$ the probability of the event $\bigl(\widehat{\mathscr{O}_j}\bigr)^c$ is
  controlled.
  We show in Corollary \ref{cor:osc} (recall \eqref{eq:GkZt} for the change in notation)
  that
  there is a constant $C_\beta > 0$ sufficiently large that
  for all $n,k_1$ large, uniformly in $\theta_j$
  \begin{align}\label{eq:oscillation}
 &  \Pr[
      \widehat{ \mathscr{O}_j}^c
      \cap \mathscr{R}_j^p(\widehat{n}_1)
      \cap \mathscr{G}_n
    ~\vert~ \filt_{k_2}, \mathfrak{Z}_{H_{\widehat{n}_1}}(\theta_j)]\nonumber\\
 & \qquad  \leq
    C_\beta
    \frac{
      \left( \sqrt{2}\log k_2 - \sqrt{\tfrac{\beta}{4}} \varphi_{k_2}(\theta_j)+k_6 \right)_+
      \left( \sqrt{2}m_n - \mathfrak{Z}_{H_{\widehat{n}_1}}(\theta_j) \right)_+
    }{(\log k_1)^{50}\log(\widehat{n}_1/k_2)}.
  \end{align}

  We can now return to 
  \eqref{eq-241223}, and substitute in the left hand side the last estimate.
  We can treat the exponential of
  $\sqrt{2}(\mathfrak{Z}_{H_{\widehat{n}_1}} - \mathfrak{Z}_{H_{k_2}}) -  \log(\widehat{n}_1/k_2)$
  as a change of measure.  Under this tilted measure $\Q$ the increment
  $\mathfrak{Z}_{H_{\widehat{n}_1}} - \mathfrak{Z}_{H_{k_2}}$
  has mean $\sqrt{2}\log(\widehat{n}_1/k_2).$
Then recalling the definition of $\widehat{n}_1$ (see \eqref{eq:k2k3})
  for all $n$ sufficiently large
  \[
    \Q\left(
    \one[{\mathfrak{Z}_{H_{\widehat{n}_1}}  \in
      [A_{\widehat{n}_1}^{p,-},A_{\widehat{n}_1}^{p,+}] }]
    \left( \sqrt{2}m_n - \mathfrak{Z}_{H_{\widehat{n}_1}}\right)_+
    ~\vert~ \filt_{k_2}
    \right)
    \leq C_\beta\frac{(\log k_1)^{\tfrac{19}{20}\tfrac{4p+1}{4p+2}}}{\sqrt{\log(\widehat{n}_1/k_2)} }.
  \]
  Thus we
  conclude
  for all $n$ sufficiently large 
  using the Brownian Bridge ballot theorem (\cite[(3.40)]{KaratzasShreve}),
  \begin{equation}\label{eq:tr23}
    \begin{aligned}
      \frac{k_1}{\widehat{k}_1}
      \Exp[
      &\exp\bigl(\sqrt{2}\mathfrak{Z}_{H_{\widehat{n}_1}}(\theta_j) - 2m_{\widehat{n}_1}\bigr)
      \one\{{\mathscr{R}}_{j}^p(\widehat{n}_1)\}
      \one\{\mathscr{O}_j\}
      (1-\one\{\mathscr{O}_j^\Psi\})
      \one\{\mathscr{G}_{n}\}
      ~\vert~
      \filt_{k_2}
      ] \\
      &\leq
      C(\beta,k_6)
      \frac{k_1 (\log n)^{3/2}}
      {n(\log \widehat{n}_1/k_2)^{3/2}(\log k_1)^{49}}
      e^{ \sqrt{\tfrac{\beta}{2}}\varphi_{k_2}(\theta_j)- \log(k_2)}
      \left( \sqrt{2}\log k_2 - \sqrt{\tfrac{\beta}{4}} \varphi_{k_2}(\theta_j)+k_6 \right)_+.
    \end{aligned}
  \end{equation}
  Hence summing over $j \in \mathcal{D}_{n/k_1}$ and
  sending $n\to \infty,$
  \begin{equation}\label{eq:tr25}
    \begin{aligned}
      \limsup_{n \to \infty}
      &\sum_{j \in \mathcal{D}_{n/k_1}}
      \frac{k_1}{\widehat{k}_1}
    \Exp[
      \exp\bigl(\sqrt{2}\mathfrak{Z}_{H_{\widehat{n}_1}}(\theta_j) - 2m_{\widehat{n}_1}\bigr)
      \one\{{\mathscr{R}}_{j}^p(\widehat{n}_1)\}
      \one\{\mathscr{O}_j\}
      (1-\one\{\mathscr{O}_j^\Psi\})
      \one\{\mathscr{G}_{n}\}
      ~\vert~
      \filt_{k_2}
    ] \\
      &\leq
      \frac{C_\beta
      }
      {2\pi(\log k_1)^{49}}
      \int_0^{2\pi}
      e^{\sqrt{\tfrac{\beta}{2}}\varphi_{k_2}(\theta) -\log(k_2)}
      \bigl(\sqrt{2}\log k_2 - \sqrt{\tfrac{\beta}{4}} \varphi_{k_2}(\theta)+k_6\bigr)_{+}
      d\theta
      \quad \As
    \end{aligned}
  \end{equation}
  Hence this tends to $0$ on multiplying by $(\log k_1)^{25}$ and sending $k_1 \to \infty.$
\end{proof}

\begin{proposition}  (Imaginary case $\sigma=i$) For any $k_5,k_6$ and $p = 2$
  \[
    \limsup_{k_1,n \to \infty}
    \frac{
      (\log k_1)^{25}
    k_1}{\widehat{k}_1}
    \times
    \!\!
    \sum_{j \in \mathcal{D}_{n/k_1}}
    \Exp[
      e^{\sqrt{\tfrac{\beta}{2}}\varphi_{\widehat{n}_1}(\theta_j) - 2m_{\widehat{n}_1}}
      \one\{{\mathscr{R}}^p_{j}(\widehat{n}_1)\}
      (1-\one\{\mathscr{O}_j^\Psi\})
      \one\{\mathscr{G}_{n}\}
      ~\vert~
      \filt_{k_2}
    ]
    =0 \quad \As
  \]
  Furthermore, for all $k_2$ sufficiently large and $p = 2$
  \[
    \limsup_{k_1,n \to \infty}
    \sum_{j \in \mathcal{D}_{n/k_1}}
    \Exp[
      \one\{W_j \in [-k_6,k_6]\}
      \bigl(
      (1-\one\{{\mathscr{R}}^p_{j}(\widehat{n}_1)\})
      +
      (1-\one\{\mathscr{O}_j^\Psi\})
      \bigr)
      \one\{\mathscr{G}_{n}\}
      ~\vert~
      \filt_{k_2}
    ]
    =0 \quad \As
  \]
  \label{prop:trunkray3}
\end{proposition}
\begin{proof}
   When $\sigma = i,$  $\varphi_k(\theta) = \Psi_k(\theta) - (k+1)\theta$ for all $k,\theta.$  The proof of the first claim is identical to the real case Proposition \ref{prop:trunkray2}.

    We turn to the second claim.  From Proposition \ref{prop:nearmaximarays}, we may assume that whenever $\mathscr{L}(\tfrac{\pi j}{2k_5 n})$ occurs for some $j$ then so does $\widehat{\mathscr{R}}(\tfrac{\pi j}{2k_5 n})$.  We will further show now that whenever $\mathscr{L}(\tfrac{\pi j}{2k_5 n}) \cap \widehat{\mathscr{R}}(\tfrac{\pi j}{2k_5 n})$ occurs, so does the event 
   \[
      \widehat{\mathscr{O}}(\theta)
      \coloneqq
      \biggl\{
        \max_{k_2 \leq k \leq n_1^+}
        |\Psi_{k}(\theta  - \tfrac{2\pi k_1}{n}) - \Psi_{k}(\theta  + \tfrac{2\pi k_1}{n})| \leq
        \tfrac{k_1(\log k_1)^{50}}{2k_1^+ }
      \biggr\}.
  \]
  In particular, in a manner similar to Proposition \ref{prop:trunkray2}, we can derive 
  \begin{equation}\label{eq:2024_1}
    \limsup_{k_1,n \to \infty}
      \sum_{j = 1}^{4k_5 n}
      \Exp[
        \one\{\mathscr{L}(\tfrac{\pi j}{2k_5 n})\}
        \one\{
          \widehat{\mathscr{R}}(\tfrac{\pi j}{2k_5 n})
        \}
        \one\{\widehat{\mathscr{O}}(\tfrac{\pi j}{2k_5 n})^c\}
        \one\{\mathscr{G}_{n}\}
        ~\vert~
        \filt_{k_2}
      ]
      \Prto[k_5,k_4,k_2]
      0.      
  \end{equation}

We comment briefly on the proof of \eqref{eq:2024_1} at the end, and we turn to deriving the second claim of this proposition.  The event $\widehat{\mathscr{O}}(\theta)$, for any $\theta \in I_j$, directly implies $\mathscr{O}_j^\Psi$ (recalling \eqref{eq:Oj}), from the monotonicity of the Pr\"ufer phases.  On the other hand, on the event $W_j \in [-k_6,k_6]$, we have a $\theta \in I_j$ so that $\widehat{\mathscr{R}}(\theta) \cap \widehat{\mathscr{O}}(\theta)$ holds.  Using monotonicity of the Pr\"ufer phases and $\widehat{\mathscr{O}}(\theta)$,  for any $k$ with $k_2 \leq k \leq n_1^+$
  \[
    \Psi_k(\theta_j) 
    \leq \Psi_k(\theta + \tfrac{2\pi k_1}{n}) 
    \leq \Psi_k(\theta - \tfrac{2\pi k_1}{n}) + \tfrac{k_1(\log k_1)^{50}}{2k_1^+}
    \leq \Psi_k(\theta) + \tfrac{k_1(\log k_1)^{50}}{2k_1^+}.
  \]
  Furthermore, trivially by monotonicity, $\Psi_k(\theta_j) \geq \Psi_k(\theta)$, and so $\max_{k_2\leq k \leq n_1^+} |\Psi_k(\theta_j) - \Psi_k(\theta)| = o_{k_1}(1)$.  Subtracting the mean behavior from $\Psi_k$ for $\theta$ in this range also is vanishing, i.e:  $\max_{k_2\leq k \leq n_1^+} |\varphi_k(\theta_j) - \varphi_k(\theta)| = o_{k_1}(1)$.  As the oscillations of $\mathfrak{Z}_t$ for $t \in [H_k,H_{k+1}]$ are $o_{k_2}(1)$ on $\mathscr{G}_n$ and $|\sqrt{\tfrac{4}{\beta}}\mathfrak{Z}_{H_k}-\varphi_k|\leq k_6$  we conclude for all $H_{k_2} \leq t \leq H_{n_1^+}$
  \[
  \begin{aligned}
  &\sqrt{\tfrac{4}{\beta}}\mathfrak{Z}_{t}(\theta_j)
  \leq 
  \sqrt{\tfrac{4}{\beta}}\mathfrak{Z}_{t}(\theta) + o_{k_2}(1)  + 2k_6
  \leq
  \sqrt{\tfrac{8}{\beta}}A_t^{1,+} + o_{k_2}(1)+2k_6, \\
  &\sqrt{\tfrac{4}{\beta}}\mathfrak{Z}_{t}(\theta_j)
  \geq 
  \sqrt{\tfrac{4}{\beta}}\mathfrak{Z}_{t}(\theta) - o_{k_2}(1) - 2k_6
  \geq
  \sqrt{\tfrac{8}{\beta}}A_t^{1,-} - o_{k_2}(1) - 2k_6,
\end{aligned}
  \]
  which hence implies the event $\mathscr{R}^{1.1}_{j}({n}_1^+)$ for all $k_2$ sufficiently large.  
  Thus we have reduced the second claim to showing
  \begin{equation}\label{eq:2024_2}
    \limsup_{k_1,n \to \infty}
      \sum_{j \in \mathcal{D}_{n/k_1}}
      \Exp\biggl[
        \one\{{\mathscr{R}}^{1.1}_{j}({n}^+_1)\}
        \one\{{\mathscr{R}}^2_{j}(\widehat{n}_1)^c\}
        \one\{\mathscr{O}_j^\Psi \cap \mathscr{G}_{n}\}
        \sum_{\theta \in I_j}
        \one\{\mathscr{L}(\theta)\}
        \one\{
          \widehat{\mathscr{R}}(\theta)
        \}
        ~\vert~
        \filt_{k_2}
      \biggr]
      \Prto[k_5,k_4,k_2]
      0.      
  \end{equation}
  
  We separately consider the events that $\mathscr{R}^2_{j}(\widehat{n}_1)$ fails due to 
  \(
      \sqrt{\tfrac{4}{\beta}} \mathfrak{Z}_t(\theta_j)
      \geq \sqrt{\tfrac{8}{\beta}}A_{t}^{p,+}
  \)
  or due to 
  \(
        \sqrt{\tfrac{8}{\beta}}A_{t}^{p,-}
        \geq
        \sqrt{\tfrac{4}{\beta}} \mathfrak{Z}_t(\theta_j)
  \)
  for some $t$ in $[H_{n_1^+}, H_{\widehat{n}_1}],$ which we call $\operatorname{REsc}_j^+$ and $\operatorname{REsc}_j^-$ respectively. 
  We start with the failure event $\operatorname{REsc}_j^+$.
  For any $\theta \in I_j$, on the event $\mathscr{L}(\theta)$, we must have
  \[
    \mathfrak{Z}_{H_n}(\theta)
    -\mathfrak{Z}_{H_{{n}_1}}(\theta)
    \geq 
    \sqrt{2}
    \log(n/{n}_1)
    +
    (
    \sqrt{2}
    (\log({n}_1)-\tfrac34 \log\log n)
    -
    \mathfrak{Z}_{H_{{n}_1}}(\theta)
    ).
  \]
  Thus conditioning on $\filt_{{n}_1}$
  and
  using the standard Gaussian tail bound,
  \begin{equation}\label{eq:2024_4}
    \begin{aligned}
      &\Exp\biggl[
        \one\{{\mathscr{R}}^{1.1}_{j}({n}^+_1)\}
        \one\{\operatorname{REsc}_j^+\}
        \one\{\mathscr{O}_j^\Psi \cap \mathscr{G}_{n}\}
        \sum_{\theta \in I_j}
        \one\{\mathscr{L}(\theta)\}
        \one\{
          \widehat{\mathscr{R}}(\theta)
        \}
        ~\vert~
        \filt_{k_2}
      \biggr] \\
      &\leq
      \frac{{n}_1}{n}
      \Exp\biggl[
        \sum_{\theta \in I_j}
        e^{ \sqrt{2}\mathfrak{Z}_{H_{{n}_1}}(\theta)
	      -2 m_{{n}_1}}
        \one\{\operatorname{REsc}_j^+\}
	      \one\{{\mathscr{R}}_{j}^{1.1}({n}_1^+)\}
        \one\{\mathscr{G}_{n}\}
        ~\vert~
        \filt_{k_2}
      \biggr].
    \end{aligned}
  \end{equation}
  From monotonicity of the Pr\"ufer phases, 
  for $\theta \in I_j$
  \(\varphi_{\widehat{n}_1}(\theta) \leq \varphi_{\widehat{n}_1}(\theta_j) + o_{k_1}(1)\),
  and hence on $\mathscr{G}_n$ we have the same estimate for $\mathfrak{Z}$ up to a constant depending on $k_6$. 
  Thus
  \[
    \begin{aligned}
      \limsup_{n \to \infty}
      &\sum_{j \in \mathcal{D}_{n/k_1}}
      \Exp\biggl[
        \one\{{\mathscr{R}}^{1.1}_{j}({n}^+_1)\}
        \one\{\operatorname{REsc}_j^+\}
        \one\{\mathscr{O}_j^\Psi \cap \mathscr{G}_{n}\}
        \sum_{\theta \in I_j}
        \one\{\mathscr{L}(\theta)\}
        \one\{
          \widehat{\mathscr{R}}(\theta)
        \}
        ~\vert~
        \filt_{k_2}
      \biggr]
      \\
      \leq
      \limsup_{n \to \infty}
      &\sum_{j \in \mathcal{D}_{n/k_1}}
      \frac{C k_5k_1{n}_1e^{C(k_6)}}{n}
      \Exp\biggl[
        e^{ \sqrt{2}\mathfrak{Z}_{H_{{n}_1}}(\theta_j)
	      -2 m_{{n}_1}}
        \one\{\operatorname{REsc}_j^+\}
	      \one\{{\mathscr{R}}_{j}^{1.1}({n}_1^+)\}
        \one\{\mathscr{G}_{n}\}
        ~\vert~
        \filt_{k_2}
      \biggr].
    \end{aligned}
  \]
  
  Hence it suffices to estimate above the $\filt_{n_1^+}$--conditional probability of $\operatorname{REsc}_j^+$, which requires that the Brownian motion $\mathfrak{Z}_t$ for $t \in [H_{n_1^+}, H_{\widehat{n}_1}]$ exceeds a barrier.  This barrier can be bounded below by the linear barrier
  \begin{equation}\label{eq:2025_1}
    \mathfrak{Z}_t \leq \sqrt{2}(t-\tfrac 34 \log\log n) - (\log k_1)^{\tfrac{19}{20}\tfrac{1}{4p+2}},
  \quad t \in [H_{n_1^+}, H_{\widehat{n}_1}]
  \end{equation}
  with $p=2.$  We also have that $\mathfrak{Z}_t$ cannot exceed the barrier
  \begin{equation}\label{eq:2025_2}
    \mathfrak{Z}_t \leq \sqrt{2}(t-\tfrac 34 \log\log n) + 2(\log k_1)^{1/100},  
    \quad t \in [H_{n_1^+}, H_{\widehat{n}_1}]
  \end{equation}
  which is implied by the event $\mathscr{G}_n$ for all $k_1$ sufficiently large.

  The $\filt_{n_1^+}\vee \sigma(\mathfrak{Z}_{H_{n_1}})$--conditional probability that $\mathfrak{Z}_{H_{\widehat{n}_1}} > \sqrt{2} A_{H_{\widehat{n}_1}}^{1,+}$ is vanishingly small as the variance of the Brownian bridge is at least $\tfrac12(\log k_1)^{19/20}$ at this time. 
  Moreover for a standard Brownian bridge, the probability that it stays below a straight line with intercepts $a,b > 0$ is $1-e^{-2ab}$ (\cite[(3.40)]{KaratzasShreve}).  It follows that
  \[
    \Pr( \eqref{eq:2025_1} \mid \filt_{n_1^+}\vee \sigma(\mathfrak{Z}_{H_{\widehat{n}_1}}))
    = 1-\exp\biggl(-2 \tfrac{(\sqrt{2}m_{n_1^+}-\mathfrak{Z}_{H_{n_1^+}}-(\log k_1)^{\tfrac{19}{20}\tfrac{1}{10}})(\sqrt{2}m_{\widehat{n}_1}-\mathfrak{Z}_{H_{\widehat{n}_1}}-(\log k_1)^{\tfrac{19}{20}\tfrac{1}{10}})}{H_{\widehat{n}_1}-H_{{n}_1^+}}\biggr).
  \]
  Uniformly over allowed $\mathfrak{Z}_{H_{n_1}^+}$ (which are restricted by the event $\mathscr{R}^{(1.1)}_j(n_1^+)$) and uniformly over $\mathfrak{Z}_{H_{\widehat{n}_1}} \in [\sqrt{2} A_{H_{\widehat{n}_1}}^{1,-},\sqrt{2} A_{H_{\widehat{n}_1}}^{1,+}]$
  \[
    \Pr( \eqref{eq:2025_1} \mid \filt_{n_1^+}\vee \sigma(\mathfrak{Z}_{H_{\widehat{n}_1}}))
    = 1-\exp\biggl(-2 \tfrac{(\sqrt{2}m_{n_1^+}-\mathfrak{Z}_{H_{n_1^+}})(\sqrt{2}m_{\widehat{n}_1}-\mathfrak{Z}_{H_{\widehat{n}_1}})}{H_{\widehat{n}_1}-H_{{n}_1^+}}(1+o_{k_1}(1))\biggr).
  \]
  The same holds for 
  $\Pr( \eqref{eq:2025_2} \mid \filt_{n_1^+}\vee \sigma(\mathfrak{Z}_{H_{\widehat{n}_1}}))$, and so we conclude that
  \[
    \Pr( \eqref{eq:2025_2} \setminus \eqref{eq:2025_1}  \mid \filt_{n_1^+}\vee \sigma(\mathfrak{Z}_{H_{\widehat{n}_1}}))
    =o_{k_1}(1),
  \]
  as either both probabilities are close to $1$ (and so no cancellation within the exponentials is needed), or they cancel.  We conclude that
  \[
    \begin{aligned}
      \limsup_{n\to\infty}
    &\Exp\biggl[
    e^{ \sqrt{2}\mathfrak{Z}_{H_{{n}_1}}(\theta_j)
    -2 m_{{n}_1}}
    \one\{\operatorname{REsc}_j^+\}
    \one\{{\mathscr{R}}_{j}^{1.1}({n}_1^+)\}
    \one\{\mathscr{G}_{n}\}
    ~\vert~
    \filt_{n_1^+}\vee \sigma(\mathfrak{Z}_{H_{\widehat{n}_1}})
    \biggr] \\
    &\leq
    o_{k_1}(1)\cdot 
    \limsup_{n\to\infty}
    \frac{\widehat{n}_1}{n_1}
    \Exp\biggl[
      e^{ \sqrt{2}\mathfrak{Z}_{H_{\widehat{n}_1}}(\theta_j)
      -2 m_{\widehat{n}_1}}
      \one\{{\mathscr{R}}_{j}^{1.1}({n}_1^+)\}
      ~\vert~
      \filt_{n_1^+}\vee \sigma(\mathfrak{Z}_{H_{\widehat{n}_1}})
      \biggr],
    \end{aligned}
\]
so that we have shown
\[
    \begin{aligned}
      \limsup_{n \to \infty}
      &\sum_{j \in \mathcal{D}_{n/k_1}}
      \Exp\biggl[
        \one\{{\mathscr{R}}^{1.1}_{j}({n}^+_1)\}
        \one\{\operatorname{REsc}_j^+\}
        \one\{\mathscr{O}_j^\Psi \cap \mathscr{G}_{n}\}
        \sum_{\theta \in I_j}
        \one\{\mathscr{L}(\theta)\}
        \one\{
          \widehat{\mathscr{R}}(\theta)
        \}
        ~\vert~
        \filt_{k_2}
      \biggr]
      \\
      \leq
      o_{k_1}(1)\cdot 
      \limsup_{n \to \infty}
      &\sum_{j \in \mathcal{D}_{n/k_1}}
      \frac{C k_5k_1{n}_1^+e^{C(k_6)}}{n}
      \Exp\biggl[
        \exp\bigl(\sqrt{2}\mathfrak{Z}_{H_{{n}_1^+}}(\theta_j)
	      -2 m_{{n}_1^+}\bigr)
	      \one\{{\mathscr{R}}_{j}^{1.1}({n}_1^+)\}
        ~\vert~
        \filt_{k_2}
      \biggr].
    \end{aligned}
  \]
  This conditional expectation we then evaluate, giving 
  \[
    \begin{aligned}
      \limsup_{n \to \infty}
      &\sum_{j \in \mathcal{D}_{n/k_1}}
      \Exp\biggl[
        \one\{{\mathscr{R}}^{1.1}_{j}({n}^+_1)\}
        \one\{\operatorname{REsc}_j^+\}
        \one\{\mathscr{O}_j^\Psi \cap \mathscr{G}_{n}\}
        \sum_{\theta \in I_j}
        \one\{\mathscr{L}(\theta)\}
        \one\{
          \widehat{\mathscr{R}}(\theta)
        \}
        ~\vert~
        \filt_{k_2}
      \biggr]
      \\
      \leq
      o_{k_1}(1)\cdot 
      \limsup_{n \to \infty}
      &\sum_{j \in \mathcal{D}_{n/k_1}}
      \frac{C k_5k_1e^{C(k_6)}}{n}
        \exp\bigl(\sqrt{\beta/2}\varphi_{k_2}(\theta_j)
	      -\log k_2\bigr)
        \bigl(\sqrt{2}\log k_2 - \sqrt{\tfrac{\beta}{4}}\varphi_{k_2}(\theta_j) + k_6 \bigr)_+.
      \end{aligned}
  \]
  Hence on sending $k_1 \to \infty$ and using Theorem \ref{thm:Bj},  this tends to $0$.

  Recalling the definition below \eqref{eq:2024_2}, we turn to the case of $\operatorname{REsc}_j^-$. On this event, for some $t \in [H_{n_1^+}, H_{\widehat{n}_1}]$ we have $\mathfrak{Z}_t(\theta_j) \leq \sqrt{2}(t - \tfrac34 \log\log n - (\log n - t)^{\tfrac{9}{10}}).$
  From monotonicity, it follows that for all $\theta \in  I_j$ and some $t \in [H_{n_1^+}, H_{\widehat{n}_1}]$
  \[
    \mathfrak{Z}_t(\theta) \leq \sqrt{2}(t - \tfrac34 \log\log n - (\log n - t)^{\tfrac{9}{10}}) + C(\beta,k_6).
  \]
  On the other hand, on the event $\widehat{R}(\theta),$ we have
  \[
    \mathfrak{Z}_t(\theta) \geq \sqrt{2}(t - \tfrac34 \log\log n - (\log n - t)^{\tfrac{5}{6}}).
  \]
  Hence these are incompatible and so we have that trivially
  \begin{equation}\label{eq:2024_6}
    \limsup_{k_1,n \to \infty}
      \sum_{j \in \mathcal{D}_{n/k_1}}
      \Exp\biggl[
        \one\{{\mathscr{R}}^{1.1}_{j}({n}^+_1)\}
        \one\{\operatorname{REsc}_j^-\}
        \one\{\mathscr{O}_j^\Psi \cap \mathscr{G}_{n}\}
        \sum_{\theta \in I_j}
        \one\{\mathscr{L}(\theta)\}
        \one\{
          \widehat{\mathscr{R}}(\theta)
        \}
        ~\vert~
        \filt_{k_2}
      \biggr]
      \Prto[k_5,k_4,k_2]
      0.      
  \end{equation}

  We finish by briefly commenting on the proof of \eqref{eq:2024_1}.  We begin by taking a conditional probability of the event $\mathscr{L}(\tfrac{\pi j}{2k_5 n}) \cap \widehat{\mathscr{R}}(\tfrac{\pi j}{2k_5 n}) \cap \widehat{\mathscr{O}}(\tfrac{\pi j}{2k_5 n})^c$.  From Corollary \ref{cor:osc}, 
  \begin{align*}
    &  \Pr[
      \widehat{\mathscr{R}}(\tfrac{\pi j}{2k_5 n}) 
      \cap \widehat{\mathscr{O}}(\tfrac{\pi j}{2k_5 n})^c
         \cap \mathscr{G}_n
       ~\vert~ \filt_{k_2}, \mathfrak{Z}_{H_{\widehat{n}_1}}(\theta_j)]\nonumber\\
    & \qquad  \leq
       C_\beta
       \frac{
         \left( \sqrt{2}\log k_2 - \sqrt{\tfrac{\beta}{4}} \varphi_{k_2}(\theta_j)+k_6 \right)_+
         \left( \sqrt{2}m_n - \mathfrak{Z}_{H_{\widehat{n}_1}}(\theta_j) \right)_+
       }{(\log k_1)^{50}\log(\widehat{n}_1/k_2)}.
  \end{align*}
  This is the same control that we had in the real part (compare to \eqref{eq:oscillation}), and so the proof can be completed in a similar way to Proposition \ref{prop:trunkray2}.
\end{proof}

\subsection{Decoupling and a diffusion approximation}
\label{sec:decoupling}

Our main tool for showing a Poisson approximation will be to modify the process $\varphi_k$ for $k \geq n_1^+$ which shows both that in a sense $\varphi_k(\theta)-\varphi_{n_1^+}(\theta)$ stabilizes as $n\to \infty$ and simultaneously gives actual $(\filt_{n_1^+})$--conditional independence of these processes for sufficiently separated $\theta.$  Set $T_{+} = \log k_1$ and $T_- = \log(k_1/k_1^+).$ We shall construct a family of standard complex Brownian motions
\begin{equation}
  \{ \mathfrak{W}_{t}^j : T_- \leq t \leq T_+, j \in \mathcal{D}_{n/k_1}\},
  \label{eq:Wtheta}
\end{equation}
which will have the property that they are independent from one another when the relevant arcs are separated by a small power of $n.$
Now with respect to these Brownian motions we define the complex diffusions $(\mathfrak{L}^j_t : t \in [T_-,T_+], \theta \in \R,j \in \mathcal{D}_{n/k_1})$ as the (strong) solution of the stochastic differential equation
\begin{equation}\label{eq:LU}
  \begin{aligned}
    d \mathfrak{L}^{j}_t(\theta)
    =
    d \mathfrak{L}_t(\theta)
    &=
    {i\theta e^t}{k_1^{-1}} dt
    +\sqrt{\tfrac{4}{\beta}} e^{i \Im \mathfrak{L}_t(\theta)} d \mathfrak{W}_t^j, \text{ and }\\
    \mathfrak{U}^j_t(\theta)
    &=
    -\Re\bigl( \sigma \bigl(\mathfrak{L}^j_t(\theta) - i\theta k_1^{-1} e^t
    \bigr)\bigr)
    -\sqrt{\tfrac{8}{\beta}}m_n.\\
    \mathfrak{L}^j_{T_-}(\theta)&=
    - 2\log \Phi_{n_1^+}^*(\exp(i(\theta_j + \tfrac{\theta}{n} )))+\tfrac{i\theta}{k_1}e^{T_-},
    \quad
    \text{for}
    \quad
    \theta \in \R. 
  \end{aligned}
\end{equation}
The diffusion $\mathfrak{L}^j(\theta)$ will serve as a proxy for the evolution of 
$-2\log \Phi^*_{k(t)}(\theta_j+\theta/n)+i\theta e^t/k_1$
where $k(t) \approx n_1 e^t$ up to rounding errors, and in particular its 
imaginary part will mimick the evolution of the Pr\"{u}fer phases.
When $\sigma=1$, the diffusion $\mathfrak{U}^j_t(\theta)$ 
is designed to be a proxy for $2\log |\Phi^*_{k(t)}(\theta_j+\theta/n)|$. 

Later, we shall consider a different
initial conditions for the SDE in \eqref{eq:LU}, which will be enforced at $t=T_-$; Namely, we shall enforce
constant initial conditions (the latter process will be denoted $\mathfrak{L}_t^{o,j}$ or $\mathfrak{L}_t^o$, see e.g. \eqref{eq:dSDE} below).
We note that if the initial conditions do not depend on $n$ then the law of the process in \eqref{eq:LU} does not depend on $n$; moreover if the initial conditions are constant in $\theta$, then the law of this diffusion is just a translation of a $0$--initial condition process by the initialization.

\begin{proposition}\label{prop:uberdecoupling}
  Fix 
  $C>0$. Then, 
  there exists $\delta_0>0$ so that for any $\delta\in (0,\delta_0)$ there
  exists  a family of Brownian motions $\{ \mathfrak{W}_{t}^j : T_- \leq t \leq T_+, j \in \mathcal{D}_{n/k_1}\}$ so that the following holds.
  Let $j \in \mathcal{D}_{n/k_1}.$
  Let $(\mathfrak{L}^j_t : t \in [T_-,T_+])$ solve \eqref{eq:LU}.
For all $n$ sufficiently large (with respect to $k_1$),
  \[
  \sup_{|\theta| \leq n^{8\delta}}\!
    \Pr
    \biggl[
      \mathscr{T}_{n_1^+}
      \cap \biggl\{
	\sup_{t \in [T_-,T_+]}
	|\mathfrak{L}^j_{t}(\theta)\!- i\tfrac{\theta}{k_1} e^t\!+\!2\log \Phi_{k(t)}^*(\exp(i(\theta_j + \tfrac{\theta}{n} )))
	|\! >\! n^{-\delta}
      \biggr\}\!
      ~\bigg\vert~ \!\filt_{n_1^+}
    \biggr] \leq n^{-C} \As,
  \]
  where $\log k(t)= \log k_n(t)$ is a linear function with $k(T_-) = n_1^+$ and $k(T_+) = n.$
  Moreover for any $j \in \mathcal{D}_{n/k_1}$,
  \[
    \sigma\left\{  \mathfrak{W}_{t}^i : t \in [T_-,T_+], d_\T(\theta_i,\theta_j) \leq n^{-1+8\delta} \right\}
    \quad
    \text{and}
    \quad
    \sigma\left\{ \mathfrak{W}_{t}^i : t \in [T_-,T_+], d_\T(\theta_i,\theta_j) \geq 4n^{-1+8\delta} \right\}
  \]
  are conditionally independent given $\filt_{n_1^+}.$
\end{proposition}
\noindent This allows the comparison of $2\log \Phi_{k(t)}^*(\exp(i(\theta_j + \tfrac{\theta}{n} )))$ to different $\mathfrak{L}^{j'}$ (for $j'\neq j$) 
provided that $|\theta_j-\theta_{j'}|$ is small enough.
We give the proof in Section \ref{sec:diffusion}.

\subsubsection{Simplifying the decoration}
Now, since  the probability is so high that $|\mathfrak{U}^j_t(\theta) - \varphi_{k(t)}(\theta_j + \tfrac{\theta}{n}) +\sqrt{\tfrac{8}{\beta}}m_n| < n^{-\delta}$ for all $t \in [T_-,T_+]$,
we may just work on the event
\begin{equation}\label{eq:g1}
  \mathscr{G}^1_{n} =
  \biggl\{
    \sup_{j \in \mathcal{D}_{n/k_1}}
    \,\,
    \sup_{\substack{\theta \in \tfrac{2\pi}{4k_5}\Z  \\ |\theta| \leq n^{8\delta} }}
    \,\,
    \sup_{t \in [T_-,T_+]}
  |\mathfrak{U}^j_t(\theta) - \varphi_{k(t)}(\theta_j + \tfrac{\theta}{n})
   +\sqrt{\tfrac{8}{\beta}}m_n| < n^{-\delta}
\biggr\}.
\end{equation}
We will also introduce another barrier event, now specific to $\mathfrak{U}^j$ and only concerning $k \geq n_1^+.$
We define a decoration--ray event with a (once more) enlarged barrier function.
These are given by for $t \in [T_-,T_+]$
\begin{equation}
  \mathcal{A}_t^{\pm}
      \coloneqq
      (t-T_{+}) - (T_+-t)^{1/2 \mp 3/7}.
  \label{eq:decorationbarrierA}
\end{equation}
Note the exponents coincide with those of $A_{t}^{3,\pm}$.
So we introduce a decoration ray event
\begin{equation}
    \begin{aligned}
    \mathscr{P}_j'(\theta)
    &\coloneqq
    \left\{
      \forall~t \in [T_-,T_+-k_4]
      :
      \sqrt{\tfrac{8}{\beta}}\mathcal{A}_{t}^{-}
      \leq
      \mathfrak{U}^j_t(\theta)
      \leq \sqrt{\tfrac{8}{\beta}}\mathcal{A}_{t}^{+}
    \right\} \\
     &\cap
     \left\{
      \forall~t \in [T_+-k_4,T_+]
      :
      \mathfrak{U}^j_t(\theta)
      \leq \sqrt{\tfrac{8}{\beta}}\bigl(t-T_+ +
      \bigl( {T_+ - t + (\log k_5)^{50}}\bigr)^{1/50}\bigr)
    \right\}, \\
    \mathscr{P}_j'
    &\coloneqq
    \bigcap_{\theta \in I_j}
    \bigl(
    \left\{
      \mathfrak{U}_{T_+}(\theta)
      <  -k_6 \right\} \cup\mathscr{P}_j'(\theta)\bigr),\\
  \end{aligned}
  \label{eq:decorationbarrier}
\end{equation}
the latter of which states that all lattice--point near-leaders ($\theta \in I_j$ for which $\mathfrak{U}^j_{T_+}(\theta)$ is large) reached that height by staying within (an enlarged) banana--like tube.  On the event $\mathscr{G}^1_n \cap \mathscr{G}_n,$ if the event $\widehat{\mathscr{R}_j}$ holds (see \eqref{eq:Rhatj} and \eqref{eq:Rhat}) then deterministically $\mathscr{P}_j'$ holds as well (once $k_1$ and $k_4$ are sufficiently large).

Instead of looking at the local maximum of $\varphi_n$
we may look at the local maximum of $\mathfrak{U}^j_{T^+}.$
Moreover, we define the analogues of $W_j$ and $V_j$ (see \eqref{eq:localmax} and  \eqref{eq:leafheight}) for this new process
\begin{equation}
    W_j' \coloneqq \max_{\substack{
      \theta \in [-2\pi k_1,0]\cap \tfrac{2\pi}{4k_5}\Z \\
      \one{\mathscr{P}_j'(\theta)}=1
    }} \mathfrak{U}^j_{T_+}( \theta)
    \quad
    \text{and}\quad
    V_j' \coloneqq
    \sqrt{2}m_{n_1^+}
    -\mathfrak{Z}_{H_{n_1^+}}{(\theta_j)}.
  \label{eq:localmax1}
\end{equation}
We also define a new decoration process $D_j'$ by first defining it on a mesh by the formula
\begin{equation}\label{eq:decorationD_j}
  D_j'\coloneqq
  \left\{
  \begin{aligned}
    &\biggl( e^{ -\mathfrak{L}^j_{T_+}\!\!(\theta)-\sqrt{\tfrac{8}{\beta}}m_{n}
  +i(\theta_j+\tfrac{\theta}{n})(n+1)}\times \one{\mathscr{P}_j'(\theta)} : \theta \in [-2\pi k_1,0] \cap \tfrac{2\pi}{4k_5}\Z\biggr),
  \quad\text{if $\sigma=1$ or if not,}\quad \\
  &\biggl( e^{ \mathfrak{U}^j_{T_+}\!\!(\theta)}\times \one{\mathscr{P}_j'(\theta)} : \theta \in [-2\pi k_1,0] \cap \tfrac{2\pi}{4k_5}\Z \biggr).
\end{aligned}
  \right.
\end{equation}
We extend $D_j'$ to be piecewise linear between the mesh points.

\subsubsection{Continuity of the decoration processes around near maxima}

We need to show that near local maxima, the decoration process is well--behaved, and moreover that the meshing accurately captures the continuum process.  Define the good event:
\begin{definition}\label{def:N}
  We let ${\mathscr{N}}_j$ be the event that the following holds.
  \begin{enumerate}
    \item The maximum over the grid, $W_j,$ is in $[-k_6,k_6].$
    \item For some $\epsilon_\beta$ (to be specified), in the neighborhood of a near-maximal grid point $\theta \in I_j$ the field $\varphi_n(\theta)$ does not vary too much.  Specifically: at any point $\theta \in I_j$ at which $\varphi_n(\theta) \in \sqrt{\tfrac{8}{\beta}}m_{n} + [-k_6,\infty)$ (i.e.~$\mathscr{L}(\theta)$ holds), all points $\theta'$ with $|\theta-\theta'| \leq \frac{2\pi}{4k_5 n}$ satisfy $|\varphi_n(\theta)-\varphi_n(\theta')| \leq k_5^{-\epsilon_\beta}.$
    \item For any point $\theta \in \widehat{I}_j$ at which $\varphi_n(\theta) \geq \sqrt{\tfrac{8}{\beta}}m_{n} - k_6$, there is a point $\theta' \in I_j$ with $|\theta-\theta'| \leq \frac{2\pi}{4k_5 n}$ at which $|\varphi_n(\theta)-\varphi_n(\theta')| \leq k_5^{-\epsilon_\beta}.$
  \end{enumerate}
\end{definition}

By a combination of a first moment argument and interpolation theorems for polynomials, we show that in any arc where $\widehat{W}_j$ is large, we also have $\mathscr{N}_j.$
\begin{proposition}  Any interval in which $\widehat{W}_j$ is large must also have $\mathscr{N}_j$ occur in that
  \begin{align*}
   & \limsup_{k_1,n \to \infty}
    \sum_{j \in \mathcal{D}_{n/k_1}}
    \Exp[
      \one\{ \widehat{W}_j \in [-k_7,k_7]\}
      \one\{\mathscr{N}_j^c\}
      \one\{\mathscr{R}_{j}^2({n}_1^+)\}
      \one\{\mathscr{P}_j'\}
      \one\{\mathscr{O}_j^\Psi\}
      \one\{\mathscr{O}_j\}
      \one\{\mathscr{G}_{n}\}
      ~\vert~
      \filt_{k_2}
    ]\\
    &
    \qquad \qquad\Prto[k_6,k_5,k_2]
    0.
  \end{align*}
  \label{prop:interpolation}
\end{proposition}
\noindent We give the proof in Section \ref{sec:meshing}.

\subsection{Summary of first--moment reductions}
\label{sec:summary}

The reductions made in Section \ref{sec:milieu} will now be summarized.
Define for $j \in \mathcal{D}_{n/k_1}$
\begin{equation}
  \begin{aligned}
    \mathscr{A}_{j} =
    \mathscr{R}_{j}^2({n}_1^+)
    \cap \mathscr{O}_j^+
    \cap \mathscr{O}_j^{\Psi}\quad\text{and}\quad
    \mathscr{A}'_{j} =
    \mathscr{N}_j
    \cap \mathscr{P}_j'
    \cap
    \mathscr{A}_{j}
  \end{aligned}
  \label{eq:qj}
\end{equation}
Using these events,
define for Borel subsets of $[0,2\pi] \times \R \times \mathcal{C}( [-2\pi k_1,0], \C),$
the following approximation
\begin{equation}\label{eq:Extr}
  \Extr_n
  =\Extr_n^{k_1,k_2,k_4,k_5,k_6}
  \coloneqq
  \sum_{j \in \mathcal{D}_{n/k_1}}
  \delta_{(\theta_j, V_j', D_j'e^{\sqrt{4/\beta}V_j'})}
  \one[{\mathscr{A}_j}]
\end{equation}
\begin{proposition}\label{prop:milieu}
  For any $k_7$ the restrictions of
  $\Extr_n$ and $\Ext_n$ to $\Gamma_{k_7}$ satisfy
  \[
    \limsup_{k_1,n \to \infty}
    \partial_2(
      \Extr_n \cap \Gamma_{k_7}
      ~\vert~ \filt_{k_2}
      ,
      \Ext_n \cap \Gamma_{k_7}
      ~\vert~ \filt_{k_2}
    )
    \Prto[k_6,k_5,k_4,k_2]
    0.
  \]
\end{proposition}
\begin{proof}
  By definition of the metric $\partial_2$ we may restrict the point processes to good events and conclude
  \[
    \partial_2(
      \Extr_n \cap \Gamma_{k_7}
      ~\vert~ \filt_{k_2}
      ,
      \{\Extr_n \cap \Gamma_{k_7} \one[{\mathscr{G}^1_n \cap \mathscr{G}_n}]\}
      ~\vert~ \filt_{k_2}
      )
      \leq \Pr (( \mathscr{G}^1_n \cap \mathscr{G}_n)^c
      ~\vert~ \filt_{k_2}),
  \]
  and hence from Lemma
  \ref{lem:goodstuff1}
  and Proposition \ref{prop:uberdecoupling},
  this tends to $0$ in that
  \[
    \limsup_{k_1,n\to\infty}
      \partial_2(
      \Extr_n \cap \Gamma_{k_7}
      ~\vert~ \filt_{k_2}
      ,
      \{\Extr_n \cap \Gamma_{k_7} \one[{\mathscr{G}^1_n \cap \mathscr{G}_n}]\}
      ~\vert~ \filt_{k_2}
      )
      \Prto[k_2] 0.
   \]
   The same holds for $\Extr_n$ replaced by $\Ext_n,$ and so we may as well restrict to the event ${\mathscr{G}^1_n \cap \mathscr{G}_n}.$

   The point process $\Ext_n$ we then further thin by defining
   \[
     \Ext_n'
     \coloneqq
     \sum_{j \in \mathcal{D}_{n/k_1}}
     \delta_{(\theta_j, V_j, D_je^{\sqrt{4/\beta}V_j})}
     \one[{\mathscr{A}'_j}]\]
and
\[
     \Ext_n''
     \coloneqq
     \sum_{j \in \mathcal{D}_{n/k_1}}
     \delta_{(\theta_j, V_j, D_je^{\sqrt{4/\beta}V_j})}
     \one[{\mathscr{A}'_j}]
     \one[    \mathscr{R}_{j}^2(\widehat{n}_1)
     \cap \mathscr{O}_j].
   \]
  We  have that as point processes
   $
     \Ext_n
     \geq
     \Ext_n'
     \geq
     \Ext_n''.
   $
   Hence from the definition of the $\partial_2$-distance it suffices to show that $\Ext_n(\Gamma_{k_7})-\Ext_n''(\Gamma_{k_7})\Prto 0$ as it then follows that all of these point processes are close on restriction to $\Gamma_{k_7}$.

  Recall that the events
   $\mathscr{R}_{j}^2(\widehat{n}_1) \subseteq \mathscr{R}_{j}^2({n}_1^+)$
   and
   $\mathscr{O}_j \subseteq \mathscr{O}_j^+.$ We claim that the combination of
   Propositions
   \ref{prop:interpolation0},
   \ref{prop:trunkray1},
   \ref{prop:trunkray3},
   and \ref{prop:interpolation} together with Corollary \ref{cor:trunkray2}
   show that on the good event ${\mathscr{G}^1_n \cap \mathscr{G}_n}$ the expected number of points lost by performing this thinning tends to $0,$ i.e.
   \begin{equation}\label{eq:goodthins}
     \limsup_{k_1,n\to\infty}
     \Exp\bigl[
       \bigl(
     \Ext_n(\Gamma_{k_7})
     -\Ext_n''(\Gamma_{k_7})
     \bigr)
     \one[{\mathscr{G}^1_n \cap \mathscr{G}_n}]
     ~\vert~ \filt_{k_2}
   \bigr]
   \end{equation}
   Proposition \ref{prop:interpolation0} implies that for any triple $(\theta_j,V_j,D_je^{\sqrt{4/\beta}V_j}) \in \Gamma_{k_7}$,
   we may include the good event that ${W_j} \in [-k_6,k_6]$.
Proposition \ref{prop:trunkray1}
 implies for $\sigma=1$ (the real case) that for any ${W_j} \in [-k_6,k_6]$ we may include the good events both $\mathscr{R}_{j}^2(\widehat{n}_1)$ and $\mathscr{O}_j$.
  Corollary \ref{cor:trunkray2} implies for $\sigma=1$ (the real case) that we may further add the good event that $\mathscr{O}_j^{\Psi}$ occurs.
   Proposition \ref{prop:trunkray3} implies the analogue of the previous two statements in the imaginary case.
   Proposition \ref{prop:interpolation} finally adds the $\mathscr{N}_j$ event.

   We also introduce a thinning
   \[
     \Extr_n'
     \coloneqq
     \sum_{j \in \mathcal{D}_{n/k_1}}
     \delta_{(\theta_j, V_j', D_j'e^{\sqrt{4/\beta}V_j'})}
     \one[{\mathscr{A}'_j}]
     \leq
     \Extr_n.
   \]
   On the events $\mathscr{G}^1_n \cap \mathscr{G}_n$ we have
   \[
     \max_{j \in \mathcal{D}_{n/k_1}} |V_j'-V_j| \leq n^{-\delta}
     \quad
     \text{and}
     \quad
     \max_{j \in \mathcal{D}_{n/k_1}} |D_j'-D_j| \leq e^{k_6+n^{-\delta}}n^{-\delta}.
   \]
   In particular,
   \[
     \limsup_{k_1,n\to\infty}
     \partial_2(
     \Ext_n'' \cap \Gamma_{k_7} \one[{\mathscr{G}^1_n \cap \mathscr{G}_n}]
     ~\vert~ \filt_{k_2}
     ,
     \{\Extr_n' \cap \Gamma_{k_7} \one[{\mathscr{G}^1_n \cap \mathscr{G}_n}]\}
     ~\vert~ \filt_{k_2}
     )
     \Prto[k_2] 0.
   \]
   Finally we undo the thinning on $\Extr_n.$   We just note that
   since on the event $\mathscr{G}^1_n \cap \mathscr{G}_n$, the processes $(V_j',V_j)$ and $(D_j', D_j)$ are sufficiently close, we have that
   \[
     (\Extr_n-\Extr_n')(\Gamma_{k_7})
     \leq(\Ext_n-\Ext_n'')(\Gamma_{2k_7}).
   \]
   As \eqref{eq:goodthins} holds for any $k_7$, the result follows.
\end{proof}

\subsection{Second approximation: removing the oscillations at level $n_1^+$}
\label{sec:2ndapprox}


The first--moment reductions allow us to assume that the profile at level $n_1^+$ in the neighborhoods of high decorations are flat.  However, they still carry $n$--dependence, and we still need to show that these small oscillations at level $n_1^+$ propagate to small oscillations at the level $n.$
So, we introduce a new diffusion which solves \eqref{eq:LU} after an intermediate time $T_\dagger = \log (k_1/\widehat{k}_1)$, has $0$ initial conditions at time $T_{-}$, and between $[T_-,T_\dagger]$ it is a single Brownian motion.
\begin{equation}\label{eq:dSDE}
  \begin{aligned}
    d \mathfrak{L}^{o,j}_t(\theta)
    =
    d \mathfrak{L}^{o}_t(\theta)
    &=
    {i\theta e^t}{k_1^{-1}}
    \one[t \geq T_\dagger] dt
    +\sqrt{\tfrac{4}{\beta}} e^{i \Im \mathfrak{L}^{o,j}_t(\theta)}
    e^{i \Im \mathfrak{L}^{j}_{T_-}(0)}
    d \mathfrak{W}_t^j, \\
    \mathfrak{U}^{o,j}_t(\theta)
    =
    \mathfrak{U}^{o}_t(\theta)
    &=
    -\Re\bigl( \sigma \bigl(\mathfrak{L}^{o}_t(\theta) - i\theta k_1^{-1} (e^t-1)\bigr)\bigr) -\sqrt{\tfrac{8}{\beta}} \log k_1^+,\text{ and } \\
    \mathfrak{L}^{o}_{T_-}(\theta)
    &= 0.
  \end{aligned}
\end{equation}
These initial conditions are flat, in that they do not vary over the interval $\theta \in [-2\pi k_1,0]$ (note that the value of  the initial condition, when flat, does not change the law of the increments).  We note that there is a fixed phase $\Im \mathfrak{L}^{j}_{T_-}(0)$ that appears in the diffusion, which does not affect the law of the process as it rotates the complex white noise.  We shall show that $\mathfrak{L}^{o,j}_t(\theta) + \mathfrak{L}^{j}_{T_-}(0)$ is a good approximation for $\mathfrak{L}^j_{t}(\theta)$, and so we need to carry this $\filt_{T_-}$--measurable constant $\mathfrak{L}^{j}_{T_-}(0)$ whenever we make a comparison.

We also change the barrier event by dropping the barrier on $[T_-,T_\dagger]$; we also include a shift $h$, which measures the additional (positive) amount the process $\mathfrak{U}_t^{o,j}$ will need to climb (recall \eqref{eq:decorationbarrierA}):
 \begin{equation} 
  \begin{aligned}
    \mathscr{P}_j(\theta,h)
    &\coloneqq
    \left\{
      \forall~t \in [T_\dagger,T_+-k_4]
      :
      \sqrt{\tfrac{8}{\beta}}\mathcal{A}_{t}^{-}
      \leq
      \mathfrak{U}^{o,j}_t(\theta)-\sqrt{\tfrac{4}{\beta}}h
      \leq \sqrt{\tfrac{8}{\beta}}\mathcal{A}_{t}^{+}
    \right\}, \\
     &\cap
     \left\{
      \forall~t \in [T_+-k_4,T_+]
      :
      \mathfrak{U}^{o,j}_t(\theta)-\sqrt{\tfrac{4}{\beta}}h
      \leq \sqrt{\tfrac{8}{\beta}}\bigl(t-T_+ +
      \bigl( {T_+ - t + (\log k_5)^{50}}\bigr)^{1/50}\bigr)
    \right\}. \\
  \end{aligned}
  \label{eq:decorationbarrier2}
\end{equation}
In analogy with \eqref{eq:decorationD_j}, we define a decoration process in which we replace the $\mathfrak{L}^{j}_t$ by $\mathfrak{L}^{o,j}_t$:
\begin{equation}\label{eq:decoration}
  D_j^{o}(h) \coloneqq
  \left\{
  \begin{aligned}
    &\biggl( e^{- \mathfrak{L}^{o,j}_{T_+}\!\!(\theta)-\sqrt{\tfrac{8}{\beta}}\log k_1^+ - i\Im \mathfrak{L}_{T_-}^j(0)
    +i\theta_j(n+1)}\times \one{\mathscr{P}_j(\theta, h)} : \theta \in [-2\pi k_1,0] \cap \tfrac{2\pi}{4k_5}\Z\biggr),
  \quad\text{if $\sigma=1$,}\quad \\
  &\biggl( e^{ \mathfrak{U}^{o,j}_{T_+}\!\!(\theta)}\times \one{\mathscr{P}_j(\theta,h)} : \theta \in [-2\pi k_1,0] \cap \tfrac{2\pi}{4k_5}\Z \biggr),\quad \text{if $\sigma\neq 1$}
\end{aligned}
  \right.
\end{equation}
and where we linearly interpolate between these $\theta$ to give continuous functions on $[-2\pi k_1,0].$ We introduce $W_j^o$ as
\[
W_j^o = \max_\theta \log | D_j^o(\theta)|  - \sqrt{\frac{4}{\beta}}V_j',
\] 
which plays the same role as $W_j'$.
We also define the Markov kernel $\mathfrak{s}(\cdot,\cdot)=\mathfrak{s}_{k_1,k_5}(\cdot,\cdot)$
\begin{equation}\label{eq:decorationlaw}
  \mathfrak{s}(h,\cdot) \coloneqq
  \begin{cases}
    \Pr(  \exp( i\Im \mathfrak{L}_{T_-}^j(0))D_j^{o}(h) \in \cdot ),& \quad\text{if $\sigma=1$}, \\
    \Pr(  D_j^{o}(h) \in \cdot ),& \quad\text{otherwise},
  \end{cases}
\end{equation}
which is deterministic and depends neither on $j$ nor on $n$.

\begin{remark}
  By construction, the decoration takes as input at time $T_\dagger$ the process $\mathfrak{L}^{o,j}_{T_\dagger}$, which is flat. The time 
$T_\dagger$
is negative by a sublogarithmic time in $k_1$.  This we show to be interchangeable with a decoration that uses  a process that solves \eqref{eq:LU} with
 flat initial condition at time $T_-.$  Indeed, we could as well start the initial condition at $-\infty$, but we do not pursue this.
\end{remark}
\begin{remark}
 The effect of restricting the processes to a grid is 
 for technical convenience.
 In fact, we could now remove the linear interpolation at this stage of the proof if we so desire: see Remark \ref{rem:continuity}.  However, this is not necessary to complete the main results, and so we do not pursue it.
\end{remark}
We introduce a new point process in which the decoration has been replaced by this one, and unnecessary events have been dropped.
\begin{equation}\label{eq:Extre}
  \Extre_n
  =\Extre_n^{k_1,k_2,k_4,k_5,k_6}
  \coloneqq
  \sum_{j \in \mathcal{D}_{n/k_1}}
  \delta_{(\theta_j, V_j', D_j^o)}
  \one[{\mathscr{R}^2_j(n_1^+)}].
\end{equation}
Here we use the shorthand $D_j^o = D_j^o(V_j').$

\begin{proposition}\label{prop:milieu2}
  For any $k_7$ the restrictions of
  $\Extr_n$ and $\Extre_n$ to $\Gamma_{k_7}$ satisfy
  \[
    \limsup_{k_1,n \to \infty}
    \partial_2(
      \Extr_n \cap \Gamma_{k_7}
      ~\vert~ \filt_{k_2}
      ,
      \Extre_n \cap \Gamma_{k_7}
      ~\vert~ \filt_{k_2}
    )
    \Prto[k_6,k_5,k_4,k_2]
    0.
  \]
\end{proposition}
\begin{proof}
 Recall \eqref{eq:decorationD_j} and let $E_j$ be the event that
  \[
    \sup_{\theta \in [-2\pi k_1,0]} | D_j^oe^{-\sqrt{4/\beta}V_j'} - D_j'|
    \geq 3e^{-(\log k_1)^{1/100}}.
  \]
  This event is typical in that for $\theta$ for which both $D_j^oe^{-\sqrt{4/\beta}V_j'}$ and $D_j'$ are below $e^{-(\log k_1)^{1/100}}$ it holds trivially, whereas if a single one is above this level, we do control the difference (and moreover the difference is much smaller). By Proposition \ref{prop:coic}, by a simple union bound, on the event $\mathscr{O}_j^{+} \cap \mathscr{O}_j^{\Psi}$ 
  (recall \eqref{eq:upsilon}), for all $k_1$ sufficiently large
  \[
    \begin{aligned}
    &\Pr( E_j \cap \bigl\{\delta_{(\theta_j,V_j',D_j'e^{\sqrt{4/\beta}V_j'})}(\Gamma_{k_7}^+) > 0\bigr\}
    ~|\filt_{n_1^+})
    \leq Ck_5(k_1/k_1^+) e^{-\delta(\log k_1)^{19/20}}, \\
    &\Pr( E_j \cap \bigl\{\delta_{(\theta_j,V_j',D_j^o)}(\Gamma_{k_7}^+) > 0\bigr\}
    ~|\filt_{n_1^+})
    \leq Ck_5(k_1/k_1^+) e^{-\delta(\log k_1)^{19/20}}. \\
    \end{aligned}
  \]
  The probability that $E_j^c$ is sufficiently high that we can essentially assume it always holds. In other words, if we define the process
  \begin{equation}
    \Extre_n^*
      \coloneqq
      \sum_{j \in \mathcal{D}_{n/k_1}}
      \biggl(
      \delta_{(\theta_j, V_j', D_j^o)}
      +
      \delta_{(\theta_j, V_j', D_j'\exp({\sqrt{4/\beta}V_j'}))}
      \biggr)
      \one[{\mathscr{R}_j^2(n_1^+) \cap E_j}],
    \label{eq:Extrestar}
  \end{equation}
  then
  \[
    \Extre_n^*(\Gamma_{k_7}^+)
    \Prto[k_1,n]
    0.
  \]
  To see this just note using the estimate
  Lemma \ref{lem:bananadensity}
  and the control on $\varphi_{n_1^+}$ given by
  $\mathscr{R}_j^2(n_1^+)$,
  \[
    \begin{aligned}
    \Exp [\Extre_n^*(\Gamma_{k_7}^+) ~|~ \filt_{k_2}]
    \leq
    \frac{c(k_5)}{n_1}
    \sum_{j \in \mathcal{D}_{n/k_1}}
    &e^{\sqrt{\tfrac{\beta}{2}}(\varphi_{k_2}(\theta_j)) -\log(k_2)}
    \bigl(\sqrt{2}\log k_2 - \sqrt{\tfrac{\beta}{4}} \varphi_{k_2}(\theta)\bigr) \\
    &\times
    \exp\bigl(O( (\log k_1)^{\tfrac{9}{10}})-\delta (\log k_1)^{19/20}
    \bigr).
  \end{aligned}
  \]
  This tends to $0$ almost surely on taking $n \to \infty$ followed by $k_1 \to \infty.$

  As in the proof of Proposition \ref{prop:milieu}, we introduce comparison point processes,
  \begin{equation}\label{eq:Extrep}
    \begin{aligned}
      &\Extre_n'
      \coloneqq
      \sum_{j \in \mathcal{D}_{n/k_1}}
      \delta_{(\theta_j, V_j', D_j^o)}
      \one[{\mathscr{R}^4_j(\widehat{n}_1) \cap \mathscr{O}_j}]  \\
      &\Extre_n''
      \coloneqq
      \sum_{j \in \mathcal{D}_{n/k_1}}
      \delta_{(\theta_j, V_j', D_j^o)}
      \one[{\mathscr{R}^4_j(\widehat{n}_1) \cap \mathscr{O}_j} \cap \mathscr{O}_j^{\Psi}]  \\
    \end{aligned}
\end{equation}
We note that all of these processes are thinnings of one another in that
\[
  \Extre_n
  \geq
  \Extre_n'
  \geq
  \Extre_n''
\]
and so if we show that the number of points in $\Gamma_{k_7}$ that are thinned tends to $0$ in probability between each of these.
We justify each of these thinnings below.
\paragraph{The first thinning.}
Note that since $\Extre_n^*(\Gamma_{k_7})$ converges to $0$, we may assume that $E_j^c$ whenever $\delta_{(\theta_j, V_j', D_j^o)}(\Gamma_{k_7}) = 1$.
On the event $\mathscr{G}_n \cap \mathscr{G}_n^1 \cap E_j^c$,
if $\delta_{(\theta_j, V_j', D_j^o)}(\Gamma_{k_7}) = 1,$ then it follows that for some $\theta \in I_j$, $\mathscr{P}_j(\theta,V_j')$ held (because, if not, by definition $D_j^o$ is identically $0$, contradicting the $\Gamma_{k_7}$ condition). As $E_j^c$ holds, we also have that
$\mathscr{P}_j'(\theta)$ holds.
By Proposition \ref{prop:trunkray1} and the fact that $\mathscr{R}^2_j(n_1^+)$ holds, we therefore have that deterministically
$\mathscr{R}_j^4(\widehat{n}_1) \cap \mathscr{O}_j$ holds.

\paragraph{The second thinning.}
We have that the difference
\[
  \begin{aligned}
    &\Exp [(\Extre_n'-\Extre_n'')(\Gamma_{k_7})\one[{\mathscr{G}_n \cap \mathscr{G}_n^1}]~|~\filt_{\widehat{n}_1}] \\
  &\leq
  \sum_{j \in \mathcal{D}_{n/k_1}}
    \Exp[
      \one[{(\theta_j, V_j', D_j^o) \in \Gamma_{k_7}}]
      \one\{{\mathscr{R}}^4_{j}(\widehat{n}_1)\}
      \one\{\mathscr{O}_j\}
      (1-\one\{\mathscr{O}_j^\Psi\})
      \one\{\mathscr{G}_{n}\}
      ~\vert~
      \filt_{\widehat{n}_1}
    ].
  \end{aligned}
\]
We can dominate the event
\[
  \one[{(\theta_j, V_j', D_j^o) \in \Gamma_{k_7}}]
  \leq
  \sum_{\theta \in I_j}
  \one[{\mathscr{P}_j(\theta,V_j')}].
\]
As we work on ${\mathscr{G}_n \cap \mathscr{G}_n^1}$,
for one of these rays to succeed, we must have that a Gaussian of variance $\log(n/\widehat{n}_1)$ climbs distance
\[
  -\sqrt{\tfrac{\beta}{4}}\varphi_{\widehat{n}_1}(\theta_j) + \sqrt{2}m_{\widehat{n}_1}
  +\sqrt{2}\log(n/\widehat{n}_1)+O_{k_4}(1).
\]
Using that $\log (n/\widehat{n}_1)=\log \widehat{k}_1\gg \sqrt{\tfrac{\beta}{8}}\varphi_{\widehat{n}_1}(\theta_j) - m_{\widehat{n}_1}+O_{k_4}(1)$ on our event, the probability of the Gaussian
exceedance  is bounded above by 
\[ \exp\bigl({-\bigl( -\sqrt{\tfrac{\beta}{4}}\varphi_{\widehat{n}_1}(\theta_j)\! +\!\sqrt{2}m_{\widehat{n}_1}
 \! +\!\sqrt{2}\log(n/\widehat{n}_1)\!+\!O_{k_4}(1)\bigr)^2/2\log \widehat{ k}_1}\bigr)\!\leq \!\frac{C(k_4)}{\widehat{k}_1}\exp\bigl(
  \sqrt{\tfrac{\beta}{2}}\varphi_{\widehat{n}_1}(\theta_j) - 2m_{\widehat{n}_1}
  \bigr).
\]
There are $O_{k_5}(k_1)$ of these Gaussians, and so by a union bound
\[
  \Pr({(\theta_j, V_j', D_j^o) \in (\Gamma_{k_7})}~|~\filt_{\widehat{n}_1})
  \leq
  C(k_4)
  \frac{
  k_1}{\widehat{k}_1}
  \times
  \exp\bigl(
  \sqrt{\tfrac{\beta}{2}}\varphi_{\widehat{n}_1}(\theta_j) - 2m_{\widehat{n}_1}
  \bigr).
\]
The claim now follows from Proposition \ref{prop:trunkray2} (or \ref{prop:trunkray3} in the imaginary case).

\paragraph{The third thinning.}
If we introduce now
\[
  \Extre_n^\dagger
  \coloneqq
  \sum_{j \in \mathcal{D}_{n/k_1}}
  \delta_{(\theta_j, V_j', D_j^o)}
  \one[{\mathscr{A}_j}],
\]
then we have that on $\Gamma_{k_7}$
\[
  \Extre_n
  \geq
  \Extre_n^\dagger
  \geq
  \Extre_n'',
\]
and so
\[
  \limsup_{k_1,n \to \infty}
  \partial_2(
  \Extre_n \cap \Gamma_{k_7}
  ~\vert~ \filt_{k_2}
  ,
  \Extre_n^\dagger \cap \Gamma_{k_7}
  ~\vert~ \filt_{k_2}
  )
  \Prto[k_6,k_5,k_4,k_2]
  0.
\]

Finally by direct computation on the event $\Extre_n^*(\Gamma_{k_7})=0$
\[
  \partial_2(
  \Extre_n^\dagger
  \cap \Gamma_{k_7},
  \Extr_n
  \cap \Gamma_{k_7}
  )
  \leq
  3e^{-(\log k_1)^{1/100}},
\]
as this holds with probability tending to $1,$ the proof is complete.
\end{proof}

\subsection{Initial Poisson approximation}
\label{sec:poisson}

\noindent We are now in a position to introduce the first Poisson process approximation that we make.
Define the $(\filt_{n_1^+})$--measurable random measures on $(\T, \R, \mathcal{C}([-2\pi k_1,0], \mathbb{C}))$
\begin{equation}\label{eq:2ndintensity}
  \begin{aligned}
    \mathfrak{m}_j(d\theta,dv, df)
    &\coloneqq
    \begin{cases}
      \one[{\mathscr{R}^2_j({n}_1^+)}]
      \delta_{\theta_j, V_j'}(d\theta,dv)\mathfrak{s}_{k_1,k_5}\bigl( V_j',
      \exp\bigl({i\Psi_{n_1^+}(\theta_j)-i(n_1^++1)\theta_j}\bigr)
      df \bigr)
      & \text{if $\sigma=1$}, \\
      \one[{\mathscr{R}^2_j({n}_1^+)}]
      \delta_{\theta_j, V_j'}(d\theta,dv)\mathfrak{s}_{k_1,k_5}\bigl( V_j', 
      df \bigr)
      & \text{if $\sigma=i$}, \\
    \end{cases} \\
    \mathfrak{m}
    &\coloneqq
    \sum_{j \in {\mathcal{D}}_{n/k_1}}
    \mathfrak{m}_j.
  \end{aligned}
\end{equation}
This is the intensity of $\Extre_n$ conditioned on $\filt_{n_1^+}$.
We show that $\Extre_n$ can be compared to a Poisson process of the same intensity. Here and throughout, we write $\Pi(\Lambda)$ for a Poisson process of intensity $\Lambda$.

To execute the proof, we will need to use some second moment machinery for events depending on the behavior of pairs of rays between times $k_2$ and $n_1^+$.  To do this, we leverage an important technical tool from \cite{CMN}.
Specifically \cite{CMN} introduces another, auxiliary Gaussian process $k\mapsto {Z}_{2^k}^{k_2}(\theta)$, defined for each $k_2$, which is shown to be close to $\varphi$. The process ${Z}_{2^k}^{k_2}(\theta)$ is similar to the Gaussian random walk $G_k$, albeit with some changes to make the process simpler on short blocks.  We will not need the exact form of the process; it is given by \cite[(5.2)]{CMN}.
Define the event
\begin{equation}\label{eq:G3}
  \mathscr{Z}_{k_2}
  =\left\{
    \forall~ \log_2 k_2 \leq k \leq \log_2 {n}^+_1, \theta \in [0,2\pi]
    :
    |{Z}_{2^k}^{k_2}(\theta) -
    2(\log \Phi^*_{2^k}(e^{i\theta})-\log \Phi^*_{k_2}(e^{i\theta}))|
    < \frac{1}{k_3}
  \right\}.
\end{equation}
Using results of \cite{CMN}, this event is typical:
\begin{lemma}
  For any $k_3$
  \[
    \liminf_{k_2,k_1, n \to \infty} \Pr ( \mathscr{Z}_{k_2} ~\vert~\filt_{k_2}) = 1.
  \]
  \label{lem:goodstuff}
\end{lemma}
\begin{proof}
  See \cite[Proposition 5.2]{CMN} (note that the notation differs in that their $\left\{ Z_k \right\}$ is our $\left\{ G_k \right\}$ and their process $\{Z^{(2^r,\Delta)}_k\}$ is our ${Z}_{2^k}^r(\theta)$)
\[
  \lim_{k_2 \to \infty} \sup_{2^n \geq k_2} \sup_{\theta \in [0,2\pi]} |{Z}_{2^n}^{k_2}(\theta) - \log \Phi^*_{2^n}(e^{i\theta})-\log \Phi^*_{k_2}(e^{i\theta})| = 0 \quad \As
\]
\end{proof}

Using this process \cite{CMN} are able to get good two-ray estimates that mimic branching random walk behavior.  We need slightly different estimates, but at its heart they are small (albeit not easily verified) modifications to the estimates of \cite{CMN}.  We summarize the estimates we need in the following.  Define a function, with $x_j = \sqrt{2}H_{k_2}-\mathfrak{Z}_{H_{k_2}}(\theta_j)$,
\begin{equation}
  \begin{aligned}
  &\mathscr{Q}(\theta_j,x_j,z_j;\theta_\ell,x_\ell,z_\ell)
  =\mathscr{Q}^{(k_2,k_3,n_1^+)}(\theta_j,x_j,z_j;\theta_\ell,x_\ell,z_\ell) \\
  &=
  \Pr(
  \mathscr{Z}_{k_2}
  \cap {\mathscr{R}^2_j({n}_1^+)}
  \cap {\mathscr{R}^2_\ell({n}_1^+)}
  \cap \operatorname{EP}_j \cap \operatorname{EP}_l
  ~\vert~ \filt_{k_2})
  \end{aligned}
\end{equation}
where we define the event
\[
  \operatorname{EP}_j \coloneqq \{\sqrt{2}m_{n_1^+} -
  \mathfrak{Z}_{H_{n_1^+}}(\theta_j)
  \in [z_j,z_j+\tfrac{1}{\sqrt{k_3}}]\}.
\]
The second moment estimates we import in the following.  We show how these can be derived from modifications of \cite{CMN} in  Appendix \ref{section:lower_bound}.
\def\k{\mathtt k}
\begin{proposition}
  \label{prop:2ray}
  The two ray estimate satisfies the three following upper bounds.  Let $\k$ be the time of branching between $\theta_j$ and $\theta_\ell,$ which we can take to be $\k \coloneqq \lfloor-\log_2 |e^{i\theta_j}-e^{i\theta_\ell}|\rfloor.$  Let $\k_+=\k_+(k_2,n)$ be defined by
  \[
    \k_+ \coloneqq
    \begin{cases}
      \k + 3( e^{\sqrt{\log\log_2 k_2}} + 100(\log^2 \k))
      &\text{if } \k \leq (\log_2 n)/2, \\
      \k + 3(\log k_2/100) + 100(\log^2 (\log_2 n-\k))
      &\text{if } \k > (\log_2 n)/2.
    \end{cases}
  \]
  The following second moment estimates hold
  \begin{enumerate}
    \item (Time of branching $\k \leq (\log_2 k_2)/2$)  For any $k_3$ and all $k_2$ large enough if $\k \leq (\log_2 k_2)/2,$
      \[
	\mathscr{Q}(\theta_j,x_j,z_j;\theta_\ell,x_\ell,z_\ell)
	\leq
	(1+\eta_{k_2,k_3})
	\frac{2}{\pi}
	\frac{z_jx_j k_2}{n_1^+}
	\frac{z_\ell x_\ell k_2}{n_1^+}
        \exp\bigl(\sqrt{2}(z_j-x_j + z_\ell-x_\ell)\bigr)
      \]
      where $\eta_{k_2,k_3} \to 0$ as $k_2 \to \infty$ followed by $k_3$.
      See Appendix Lemma \ref{lemma:ktoutDebut}.
    \item (Time of branching $\k \leq (\log_2 n)/2$) For all $k_2$ sufficiently large and all $n \gg k_2$,
      \begin{align*}
	&\mathscr{Q}(\theta_j,x_j,z_j;\theta_\ell,x_\ell,z_\ell)\\
	&\leq
	c(k_3)
	\begin{cases}
	\frac{z_jx_j k_2}{n_1^+}
	\frac{z_\ell x_\ell k_2}{n_1^+}
        \exp\bigl(\sqrt{2}(z_j-x_j + z_\ell-x_\ell)\bigr)
	&\text{if } \k_+ \leq \log_2 k_2, \\
	\frac{z_j z_\ell x_j k_2}{n_1^+}
	\frac{2^{\k}}{n_1^+}
        \exp\bigl(\sqrt{2}(z_j-x_j + z_\ell)\bigr)
	e^{-c(\k_+)^{1/10}}
	&\text{if } \k_+ \geq \log_2 k_2.
	\end{cases}
      \end{align*}
      See Appendix Lemma \ref{lemma:kDebut}.
    \item (Time of branching $(\log_2 n)/2 \leq \k \leq \log_2 n_1^+$) For all $k_2$ sufficiently large and all $n \gg k_2$,
      \[
	\mathscr{Q}(\theta_j,x_j,z_j;\theta_\ell,x_\ell,z_\ell)
	\leq c(k_3)
	\frac{ x_jz_j k_2}{n_1^+}
	\frac{2^{\k}}{n_1^+}
	e^{\sqrt{2}(z_j-x_j+z_\ell)}
	e^{-c(\log_2 n - \k_+)^{1/10}}
      \]
      See Appendix Lemma \ref{lemma:kMil}.
    \item (Time of branching $\k \geq \log_2 n_1^+$)
      While useful for the range of $\k$ described, this holds for all $\k$:
      \[
	\mathscr{Q}(\theta_j,x_j,z_j;\theta_\ell,x_\ell,z_\ell)
	\leq
	c(k_3)
	\frac{z_jx_j k_2}{n_1^+}
        \exp\bigl(\sqrt{2}(z_j-x_j)\bigr).
      \]
      This is a triviality which follows from Lemma \ref{lem:bananadensity} and bounding above the two-ray event by a one-ray event.
  \end{enumerate}
\end{proposition}
\begin{remark}\label{rk:holybarrier}
  All these upper bounds also hold if in $\mathscr{Q}$ we replace the ray event $\mathscr{R}_j^p(n_1^+)$ by one which only holds at times $H_{2^k}$ (instead of a continuous barrier), see \eqref{eq:RD}.  Likewise, if we further replace the process 
$\mathfrak{Z}_{H_{2^k}}$ by $Z_{2^k}^{k_2} + \mathfrak{Z}_{H_{k_2}}$, in the discrete barrier event,
 the bound holds.  In this case, we no longer need to work on the good event $\mathscr{Z}_{k_2}.$
\end{remark}

Using these second moment estimates, we turn to the first Poisson approximation.
\begin{proposition}\label{prop:ppp1}
  For any $k_4,\dots,k_7$ the restrictions of $\Extre_n$ and $\Pi(\mathfrak{m})$ to $\Gamma_{k_7}$ satisfy
  \[
    \limsup_{k_1,n \to \infty}
    \Exp
    [
      \partial_2(
      \Extre_n \cap \Gamma_{k_7}
      ~\vert~ \filt_{{n}_1^+}
      ,
      \Pi(\mathfrak{m}) \cap \Gamma_{k_7}
      ~\vert~ \filt_{{n}_1^+}
      )
      \one[{\mathscr{Z}_{k_2}}]
      ~\vert~
      \filt_{k_2}
    ]
    \Prto[k_2]
    0.
  \]
\end{proposition}

To give the proof, we need to develop some first and second moment estimates for the process near the end of the ray.
\subsection{Decoration Process estimates}
We turn to giving some conditional first and second moment estimates for rays.
\begin{lemma}[Coarse intensity bound]\label{cor:qbnd}
  The intensity measure $\mathfrak{m}_j$ satisfies
  \[
    \mathfrak{m}_j(\Gamma_{k_7})
    \leq
    \mathfrak{m}_j(\Gamma_{k_7}^+)
    =
    c(k_4)(1+o_{k_4})
    e^{T_-}
    \frac{V_j'}{(T_+-T_-)^{3/2}}
    e^{-\sqrt{2}V_j' -(V_j')^2/2(T_+-T_-)}.
  \]
 \end{lemma}
\begin{proof}
  To estimate $\mathfrak{m}_j$ we first observe that (recalling $W_j^o$ from \eqref{eq:decoration})
  \[
    \mathfrak{m}_j(\Gamma_{k_7}^+)
    =
    \Pr(  W_j^o \in [-k_7,\infty) ~|~ \filt_{n_1^+})
    \one[\mathscr{R}_j^2(n_1^+)].
  \]
  Define the first moment
  \[
    m_1
    =
    \Exp \biggl[
    \sum_{\substack{
      \theta \in [-2\pi k_1,0]\cap \tfrac{2\pi}{4k_5}\Z
    }}
    \one[{\bigl\{ \mathfrak{U}^{o}_{T_+}(\theta) - \sqrt{\tfrac{4}{\beta}}V_j' \in
    [-k_7,\infty) \bigr\} \cap \mathscr{P}_j(\theta)}]
    ~\big\vert~\filt_{n_1^+} \biggr].
  \]
  Hence we can bound by first moment on the event $\mathscr{R}_j^2(n_1^+)$,
  \[
    \mathfrak{m}_j(\Gamma_{k_7}^+)
    \leq
    m_1
    =
    \sum_{\substack{
      \theta \in [-2\pi k_1,0]\cap \tfrac{2\pi}{4k_5}\Z
    }}
    \Pr\biggl( \bigl\{ \mathfrak{U}^{o}_{T_+}(\theta) - \sqrt{\tfrac{4}{\beta}}V_j' \in
    [-k_7,\infty) \bigr\} \cap \mathscr{P}_j(\theta) ~\big\vert~\filt_{n_1^+} \biggr).
  \]
  We then apply Lemma \ref{lem:SDE1st2nd} to each of these summands.
  Hence we arrive at
    \[
      m_1=(1+o_{k_4})
    k_1k_5
    e^{-(T_+-T_-)}
    \frac{c(k_7,\infty) V_j' \sqrt{k_4}}{(T_+-T_-)^{3/2}}
    e^{-\sqrt{2}V_j' -(V_j')^2/2(T_+-T_-)},
  \]
  which simplifies to the claimed result.
\end{proof}

\begin{proof}[Proof of Proposition \ref{prop:ppp1}]
  We shall use the Poisson process machinery stated in Theorem \ref{thm:PP}, which we shall translate into this context.  Due to the nature of the Poisson approximation, we shall use the nonatomicity of the derivative martingale proved in Theorem \ref{thm:Bj}.  We let $\eta > 0$ be a parameter which shall be taken to $0$ after $k_2$ to establish the convergence in probability.
  \paragraph{Step 1: restricting the measures.}
  As the measure $\mathscr{D}_\infty$ is finite and nonatomic, we have that there is $m \in \N$ sufficiently large that
  \begin{equation}
\label{eq-totalmassininterval}
    \Pr( \max_{j=1,\dots,m} \mathscr{D}_\infty([ \tfrac{2\pi (j-1)}{m},\tfrac{2\pi (j+1)}{m}]) > \eta) \leq \eta,
  \end{equation}
  where for the case $j=m$, we consider the arc as part of the torus.
  Hence for all $k_2$ sufficiently large, we have
  \[
    \Pr( \max_{j=1,\dots,m} \mathscr{D}_{k_2}([ \tfrac{2\pi (j-1)}{m},\tfrac{2\pi (j+1)}{m}]) > \eta) \leq 2\eta.
  \]
 Let  $\mathcal{E}$  denote the event  that there exists a $j=1,\dots,m$ such that
  \[
    \mathscr{D}_{\infty}([ \tfrac{2\pi (j-1)}{m},\tfrac{2\pi (j+1)}{m}]^c)
    \leq \eta.
  \]
 From \eqref{eq-totalmassininterval} we have that
  \[
    \Pr(\mathcal{E} \cap \{ \mathscr{D}_{\infty}([0,2\pi])> 2\eta \})
    \leq \eta.
  \]
  As we can relate $\mathscr{D}_{\infty}([0,2\pi])$ to intensity of both point processes, this in effect is saying that neither point process has any points.
  Indeed, for either process $\Xi= \Extre_n$ or $\Pi(\mathfrak{m})$
  \[
    \partial_2( \Xi \cap \Gamma_{k_7}, 0)
    \leq
    \Pr( \Xi(\Gamma_{k_7}) \geq 1 ~|~ \filt_{n_1^+})
    \leq
    \mathfrak{m}(\Gamma_{k_7}^+).
  \]
  Using Lemmas \ref{lem:bananadensity} and
  \ref{cor:qbnd},
  \[
    \Exp (\mathfrak{m}(\Gamma_{k_7}^+~|~\filt_{k_2}))
    \leq
    \frac{c(k_4)}{n_1}
    \sum_{j \in \mathcal{D}_{n/k_1}}
    e^{\sqrt{\tfrac{\beta}{2}}(\varphi_{k_2}(\theta_j)) -\log(k_2)}
    \bigl(\sqrt{2}\log k_2 - \sqrt{\tfrac{\beta}{4}} \varphi_{k_2}(\theta_j)\bigr)_+
  \]
  On taking $n \to \infty$, we conclude
  \[
    \limsup_{n \to \infty}
    \Exp (\mathfrak{m}(\Gamma_{k_7}^+~|~\filt_{k_2}))
    \leq c(k_4) \mathscr{D}_{k_2}([0,2\pi]).
  \]
  And hence 
we conclude
  \[
    \limsup_{k_2,k_1,n \to \infty}
    \Pr( \mathcal{E} \cap  \{\partial_2( \Extre_n \cap \Gamma_{k_7}, \Pi(\mathfrak{m})) > 4c(k_4)\eta\} )
    \leq \eta.
  \]
  Hence it suffices to work on the event $\mathcal{E}^c$.

  \paragraph{Step 2: setting up the Poisson approximation.}

  We let $\delta > 0$ be the same constant as in \eqref{eq:g1} (so that Proposition \ref{prop:uberdecoupling} applies), and define for $j \in \mathcal{D}_{n/k_1}$
\[
  \mathcal{B}_j = \left\{ k  \in \mathcal{D}_{n/k_1} : d_\T(\theta_j,\theta_k) \leq n^{-1+8\delta} \right\},
  \quad{\text{and}}\quad
  B_j = \left\{ k  \in \mathcal{D}_{n/k_1} : d_\T(\theta_j,\theta_k) \leq 4n^{-1+8\delta} \right\},
\]
with $d_\T$ the distance in the quotient space $\R/(2\pi \Z).$  Recall \eqref{eq:localmax1} and \eqref{eq:decoration}.
By construction, for all $j \in \mathcal{D}_{n/k_1}$,
\[
  \{(V_k',D_k^o) : k \in \mathcal{B}_j\} \quad \text{is $\bigl(\filt_{{n}_1^+}\bigr)$--conditional independent of}\quad 
\{(V_k',D_k^o) : k {\not\in} B_j\}.
\]
Define for any $j \in \mathcal{D}_{n/k_1},$
\[
  \Xi_j \coloneqq
  \delta_{(\theta_j,V_j', D_j^o)}
  \times
  \one[ \mathscr{R}_j^2(n_1^+) ],
  \,
  P_j \coloneqq \Xi_j(\Gamma_{k_7}^+),
  \,
  S_j \coloneqq \sum_{i \in \mathcal{B}_j \setminus \{j\}} P_i,
  \,
  \text{and}
  \quad
  L_j \coloneqq  \sum_{i \in \mathcal{D}_{n/k_1}\setminus B_j}
  \mathfrak{m}_i(\Gamma_{k_7}^+).
\]
We note that $\Exp( P_i ~|~ \filt_{n_1^+}) = \mathfrak{m}_i(\Gamma_{k_7}^+).$
Theorem \ref{thm:PP} shows that there is a numerical constant $C>0$ so that with $\Pi=\Pi(\mathfrak{q}),$
\begin{equation}\label{eq:ppp1}
  \begin{aligned}
&  \partial_2(
      \Extr_n \cap \Gamma_{k_7}
      ~\vert~ \filt_{{n}_1^+},
      \Pi(\mathfrak{m}) \cap \Gamma_{k_7}
      ~\vert~ \filt_{{n}_1^+}
      )
  \\
    \leq & \partial_2(
      \Extr_n \cap \Gamma_{k_7}^+
      ~\vert~ \filt_{{n}_1^+},
      \Pi(\mathfrak{m}) \cap \Gamma_{k_7}^+
      ~\vert~ \filt_{{n}_1^+}
      )\\
 \leq& C
  [\Var(\Extre_n ( \Gamma_{k_7}^+)
~|~ \filt_{n_1^+}) + 3\mathfrak{m}(\Gamma_{k_7}^+)]^{3/2} \\
  &\times \sum_{j \in \mathcal{D}_{n/k_1}}
  \biggl( \frac{ \Exp(P_j ~|~ \filt_{n_1^+})\Exp(S_j ~|~ \filt_{n_1^+}) + \Exp( S_j P_j ~|~ \filt_{n_1^+}) + (\mathfrak{m}_j(\Gamma_{k_7}^+))^2}{ L_j^2 }\biggr).
\end{aligned}
\end{equation}
We also note that the left hand side of \eqref{eq:ppp1} is bounded by $1,$ and hence it suffices to bound the right hand side of \eqref{eq:ppp1} on any $\bigl(\filt_{{n}_1^+}\bigr)$--measurable event with probability tending to $1$.  Nonetheless, as input, we need estimates for the $\bigl(\filt_{n_1^+}\bigr)$--conditional expectation of $P_j$ and for pairs $P_j$ and $P_i$ where $i \in \mathcal{B}_j.$  We note that
$\Exp[ P_j ~|~\filt_{n_1^+}] = \mathfrak{m}_j( \Gamma_{k_7}^+).$

To complete the proof of the proposition, we claim that there are events $\mathcal{F}=\mathcal{F}(n,(k_j))$ so that on $\mathcal{F} \cap \mathcal{E}^c$ the following hold
\begin{enumerate}
  \item $\mathfrak{m}(\Gamma_{k_7}^+) = O_{k_2}(1)$,
  \item $\Var(\Extr_n ( \Gamma_{k_7}^+)~|~ \filt_{n_1^+}) = O_{k_2}(1)$,
  \item $L_j^2 \geq c(k_4)\eta^2$,
\end{enumerate}
and so that $\one[\mathcal{F}] \Prto[k_2,k_1,n] 1$.

\paragraph{Step 3: completing the proof on the good event $\mathcal{F}$.}
Assuming the claim holds, we have reduced the problem to showing that
on the event $\mathcal{F}$,
\begin{equation}\label{eq:ppp1a}
  \begin{aligned}
  \sum_{j \in \mathcal{D}_{n/k_1}}
  \biggl(
  \underbracket{\mathfrak{m}_j(\Gamma_{k_7}^+)\Exp(S_j ~|~ \filt_{n_1^+})}_{(i)}
  +
  \underbracket{\Exp( S_j P_j ~|~ \filt_{n_1^+})}_{(ii)}
  +\underbracket{\mathfrak{m}_j(\Gamma_{k_7}^+)^2}_{(iii)}
  \biggr)
  \Prto[k_2,k_1,n] 0.
\end{aligned}
\end{equation}
For the term $(iii),$ we note that $\mathfrak{m}_j(\Gamma_{k_7}^+),$ which goes to $0$ uniformly in $j$ faster than any power of $\log k_1$ (from Lemma \ref{cor:qbnd}).  Hence by the boundededness of $\mathfrak{m}(\Gamma_{k_7}^+)$ on $\mathcal{F}$, this term tends to $0.$

For term $(i),$ we will use  the second moment machinery. To prepare for it, we use Lemma \ref{cor:qbnd} and, bounding the sum on $z_j,z_\ell$ by an integral, which holds up to a constant depending on $k_3$, arrive at
\begin{equation}\label{zrh:3}
  \Exp[~ (i) ~|~ \filt_{k_2}]
  \leq
  c(k_3)
  \bigl(\frac{k_1}{k_1^+}\bigr)^2
  \sum_{\ell \in \mathcal{B}_j}
  \int_{z_j,z_\ell}
  \mathscr{Q}(\theta_j,x_j,z_j;\theta_\ell,x_\ell,z_\ell)
  \frac{z_j z_\ell}{(\log k_1)^3}
  e^{-\sqrt{2}(z_j+z_\ell)}
  dz_jdz_\ell.
\end{equation}
We divide the angles $\theta_\ell$ according to whether or not
$(\log_2 n_1^+-\k_+) \geq q$ or $(\log_2 n_1^+-\k_+) \leq q$ for a $q$ chosen below as $(\log k_1)^{1/100}$.
For the former case we use the third case of Proposition \ref{prop:2ray}.
For the latter case, we just bound $\mathscr{Q}(\theta_j,x_j,z_j;\theta_\ell,x_\ell,z_\ell)$ by the probability of the $\theta_j$--ray event.
We also note that the number of $\theta_j$ in this latter case is $(k_1^+/k_1)e^{q + 100(\log\log k_1)^2}$.
Applying the bound from this case, we arrive at
\begin{equation}\label{eq:ppp_i}
  \begin{aligned}
\Exp[~ (i) ~|~ \filt_{k_2}]
  &\leq
  \frac{c(k_2)}{n_1}
  e^{\sqrt{\tfrac{\beta}{2}}(\varphi_{k_2}(\theta_j)) -\log(k_2)}
  \bigl(\sqrt{2}\log k_2 - \sqrt{\tfrac{\beta}{4}} \varphi_{k_2}(\theta)\bigr) \\
  &\times
  (\log k_1)^3
  \biggl\{
  \sum_{\varkappa = q}^{\delta \log n}
  \bigl\{
  \exp\bigl(-c \varkappa^{1/10}\bigr)
\bigr\}
  +
  \exp\bigl(q+ C(\log\log k_1)^2-c (\log k_1)^{1/10}\bigr)
  \biggr\}.
\end{aligned}
\end{equation}
Note that the stretched exponential gain in the second term is simply from the entropic envelope.
As the event we consider restricts the location of $\varphi_{k_2}(\theta_j)$ to be positive, we may use that
\[
  \sum_{j \in \mathcal{D}_{n/k_1}}
  \frac{1}{n_1}
  e^{\sqrt{\tfrac{\beta}{2}}(\varphi_{k_2}(\theta_j)) -\log(k_2)}
  \bigl(\sqrt{2}\log k_2 - \sqrt{\tfrac{\beta}{4}} \varphi_{k_2}(\theta)\bigr)_+
\]
converges on taking $n \to \infty$.  Taking $q = (\log k_1)^{1/100}$, the sum is on the order of $e^{-\Omega( (\log k_1)^{1/1000})}$.
Hence on taking $k_1\to\infty$,
\[
\sum_{j \in \mathcal{D}_{n/k_1}}
\Exp[~ (i) ~|~ \filt_{k_2}]
\Prto[k_1,n] 0.
\]

The analysis for
 $\Exp[~ (ii) ~|~ \filt_{k_2}]$ is similar, but with one important modification.
 For ``bushes''
 $\theta_j$ and $\theta_\ell$ that branch at or before $n_1^+$ (i.e. $\k \leq \log_2 n_1^+$) we claim that the bound in \eqref{zrh:3} holds as well:
 we bound each indicator $\one[{W_j^o \in [-k_7,\infty)}]$ above by a sum of indicators of ray event and then use Lemma \ref{lem:SDE1st2nd}.
 For $\theta_\ell$ which are close to $\theta_j$, we use the same strategy, although we instead lose the precise dependence on $z_\ell$ and instead have just the factor $e^{-\Omega( \log k_1)^{1/10}}.$  For $\ell$ so that $\k \geq \log_2 n_1^+$, this produces the bound:
 \[
  \sum_{j \in \mathcal{D}_{n/k_1}}
   c(k_3)
  \bigl(\frac{k_1}{k_1^+}\bigr)^2
  \sum_{\ell}
  \mathscr{Q}(\theta_j,x_j,z_j;\theta_\ell,x_\ell,z_\ell)
  \frac{z_j}
  {(\log k_1)^{3/2}}
  e^{-\sqrt{2}(z_j) -c(\log k_1)^{1/10}}.
 \]
Using the final case of Proposition \ref{prop:2ray} gives an estimate which is $e^{-\Omega( (\log k_1)^{1/10})}$.
\paragraph{Step 4: Proof of claim.}

\emph{Point 1.} We have that
\[
  \mathfrak{m}(\Gamma_{k_7}^+)
  =
  \sum_{j \in \mathcal{D}_{n/k_1}}
  \mathfrak{m}_j(\Gamma_{k_7}^+),
\]
and so using Lemmas \ref{cor:qbnd}
and  \ref{lem:bananadensity},
\[
  \Exp[\mathfrak{m}(\Gamma_{k_7}^+)~|~\filt_{k_2}]
  \leq
  \frac{c(k_4)}{n_1}
  \sum_{j \in \mathcal{D}_{n/k_1}}
  e^{\sqrt{\tfrac{\beta}{2}}(\varphi_{k_2}(\theta_j)) -\log(k_2)}
  \bigl(\sqrt{2}\log k_2 - \sqrt{\tfrac{\beta}{4}} \varphi_{k_2}(\theta)\bigr)_+.
\]
This converges on taking $n\to \infty$ almost surely to
\begin{equation}\label{eq:mupper}
  \limsup_{n \to \infty}
  \Exp[\mathfrak{m}(\Gamma_{k_7}^+)~|~\filt_{k_2}]
  \leq
  c(k_4) \int_0^{2\pi} \mathscr{D}_{k_2}(\theta)\,d\theta,
\end{equation}
and so remains bounded almost surely by Theorem \ref{thm:Bj}.

\emph{Point 2.} It suffices to bound the conditional second moment, which is to say
\[
  \Var(\Extre_n ( \Gamma_{k_7}^+)~|~ \filt_{n_1^+})
  \leq
  \Exp[
    \Extre_n^2 ( \Gamma_{k_7}^+)~|~ \filt_{n_1^+}
  ].
\]
We then have
\[
  \Exp[
    \Extre_n^2 ( \Gamma_{k_7}^+)~|~ \filt_{n_1^+}
  ]
  \leq
  \sum_{j \in \mathcal{D}_{n/k_1}}
  \mathfrak{m}_j( \Gamma_{k_7}^+)
  +
  \sum_{j \in \mathcal{D}_{n/k_1}}
  \Exp[ P_jS_j ~|~ \filt_{n_1^+}]
  +
  \sum_{j \in \mathcal{D}_{n/k_1}}
  \sum_{\ell \not\in \mathcal{B}_j}
  \mathfrak{m}_j( \Gamma_{k_7}^+)
  \mathfrak{m}_\ell( \Gamma_{k_7}^+)
\]
The first term is nothing but $\mathfrak{m}(\Gamma_{k_7}^+)$ which we have already controlled. The second term we controlled earlier in part $(ii)$ above.
 The third term we estimate in the same fashion as $(i)$ in \eqref{eq:ppp_i}.

\emph{Point 3.} We let $\text{Arc}_{i}$ for $i=1,\dots,m$ be
a Lipschitz function of the torus $\R / 2\pi\Z$ which is $1$ on the complement of $[\tfrac{i-1}{m}2\pi,\tfrac{i+1}{m}2\pi]^c$ and $0$ on $[\tfrac{i-0.5}{m}2\pi,\tfrac{i+0.5}{m}2\pi]$.
Extend $\text{Arc}_{i}$ to a function of $\Gamma_{k_7}^+$ by setting $\text{Arc}_{i}(x,y,z)=\text{Arc}_{i}(x).$
Then for each $j \in \mathcal{D}_{n/k_1}$, for all $n$ sufficiently large there is an $i\in 1,\dots,m$ so that
\(
  L_j \geq \mathfrak{m}( \text{Arc}_{i} ).
\)
Hence it suffices to show that each of these $\mathfrak{m}( \text{Arc}_{i} )$ satisfies the claimed bound.

Now we claim that in fact
\(
  \mathfrak{m}( \text{Arc}_{i} )
\)
concentrates around its $\filt_{k_2}$--conditional mean and that its conditional mean has the claimed lower bound.  In fact, we shall need this concentration argument at a later point as well, and so we formulate a general statement here.
\begin{lemma}
  Let $f$ be a non-negative, Lipschitz function from the torus $\R / (2\pi \Z) \to \R$ bounded above by $1$.  Extend it to a function of $\R / (2\pi \Z) \times \R_+ \times \mathcal{C}([-2\pi k_1,0],\C) \cap \Gamma_{k_7}^+$ by setting $f(x,y,z) = f(x)$.
  Then there is positive constant $\mathcal{H}=\mathcal{H}(k_1,k_4,k_5,k_7)$ which  is bounded   (above and away from $0$, for $k_7$ large), uniformly
  in $k_1$,  so that
  \[
    \mathfrak{m}(f) - \mathscr{D}_{k_2}(f)\times \mathcal{H} \Prto[{k_3,k_2,k_1,n}]
    0.
  \]
  The constant $\mathcal{H}$ is given in \eqref{zrh:H}.
  \label{zrh:2moment}
\end{lemma}
This completes the proof of the proposition.
\end{proof}

\begin{proof}[Proof of Lemma \ref{zrh:2moment}]

If $\mathscr{D}_\infty(f) = 0$ then it suffices to show $\mathfrak{m}(f)$ tends to $0$.
This follows directly from \eqref{eq:mupper}, and so it suffices to consider the case that $\mathscr{D}_\infty(f) > 0.$

The proof follows from a second moment method, restricted to the event $\mathscr{Z}_{k_2}$,
which complicates the first moment estimate.
We first compute the first moment without the restriction to $\mathscr{Z}_{k_2}$.
The conditional first moment of $\mathfrak{m}(f)$ is given by
\[
  \Exp[ \mathfrak{m}(f) ~|~ \filt_{k_2}]
  =\sum_{j \in \mathcal{D}_{n/k_1}}
  f(\theta_j)
  \Exp[
    ~
    \mathscr{R}_{j}^2(n_1^+)
    p_{j}
  ~|~ \filt_{k_2}]
\]
where $p_j = p(V_j') = \Pr( W_j^o \in [-k_7,\infty)~|~\filt_{n_1^+}).$
The function $p$ is inexplicit and depends on $k_4,k_5,k_7$.  Using Lemma \ref{lem:bananadensity},
\[
  \Exp[
    ~
    \mathscr{R}_{j}^2(n_1^+)
    p_{j}
  ~|~ \filt_{k_2}]
  \leq
  \frac{1}{n_1}
  \bigl(\sqrt{2}H_{k_2}-\sqrt{\tfrac{\beta}{4}}\varphi_{k_2}(\theta_j)\bigr)
  e^{\sqrt{\beta/2}\varphi_{k_2}(\theta_j) - \log k_2 + o_{k_2}}
  \sqrt{\frac{2}{\pi}}
  \int_J
  \frac{k_1^+}{k_1}
  e^{\sqrt{2}z}p(z)
  z\,dz
\]
The interval $J$ is $[(\log k_1^+)^{1/10}, (\log k_1^+)^{9/10}]$.  The constant $\mathcal{H}$ is given by
\begin{equation}
  \mathcal{H}
  \coloneqq
  \int_J
  \biggl(
  \sqrt{\frac{2}{\pi}}
  z
  e^{\sqrt{2}z}
  \biggr)
  \times
  \biggl(
  \frac{k_1^+}{k_1}
  p(z)
  \biggr)
  \,dz.
  \label{zrh:H}
\end{equation}
Note that the density that appears is
bounded by Lemma \ref{cor:qbnd} (using $T_-=\log(k_1/k_1^+)$ and $T_+=\log k_1$),  and we obtain
\begin{equation}
  \mathcal{H}
=
  c(k_4)(1+o_{k_4})
  \int_J
  e^{-z^2/(2\log k_1^+)} \frac{z^2 dz}{(\log k_1^+)^{3/2}}
  \leq c'(k_4).
\label{zrh:Hupper}
\end{equation}
We return to this bound in a moment, but note that by combining the above, we have the upper bound for the conditional first moment:
\begin{equation}
  \limsup_{n\to\infty}
  \Exp[ \mathfrak{m}(f) ~|~ \filt_{k_2}]
  \leq
  \int_0^{2\pi}
  f(\theta)
  \mathscr{D}_{k_2}(\theta)
  d\theta
  \times
  \mathcal{H}.
  \label{zrh:m1u}
\end{equation}
Before proceeding we 
show that $\mathcal{H}$ is bounded below uniformly in $n,k_1$, at least for $k_7$ large enough. Assume not. Take $f=1$.
Then necessarily,  
$\liminf_{k_1\to \infty} \limsup_{n\to\infty}
  \mathfrak{m}(f)=0$, in probability. 
  Using Proposition  \ref{prop:milieu}, we conclude then
  that, at least for $k_7$ large,  
  \[
    \liminf_{k_1\to \infty} \limsup_{n\to\infty}
    \Ext_n \cap \Gamma_{k_7}=0, \; \textrm{in probability}.
  \]
  But this contradicts the tightness in Theorem \ref{thm:tightness}.

To produce a lower bound for the first moment, we introduce two comparison measures $\mathfrak{m}_j'$ and $\mathfrak{m}_j''$, which serve as comparisons to $\mathfrak{m}_j$, and which are modified by slightly adjusting the ray event $\mathscr{R}_j^2(n_1^+)$.  We introduce two ray events, which will only be used for this argument,
\begin{equation}\label{eq:RD}
  \begin{aligned}
    &\mathscr{V}_j'
    = \mathscr{U}(\theta_j) \bigcap \{ \forall~t \in [H_{k_2},H_{n_1^+}] \cap \{H_{2^k} : k \in \N\},
    ~:~
    \mathfrak{Z}_t(\theta_j)
    \in \sqrt{2}[A_{t}^{5/2,-},A_{t}^{5/2,+}]
  \} \\
  &\qquad\qquad\bigcap \{ -V_j' \in [ (\log k_1)^{0.49},(\log k_1)^{0.51} \},\quad \text{ and }\\
    &\mathscr{V}_j''
    = \mathscr{U}(\theta_j) \bigcap \{ \forall~t \in [H_{k_2},H_{n_1^+}] \cap \{H_{2^k} : k \in \N\},
    ~:~
    Z^{k_2}_{2^n}(\theta_j)
    \in \sqrt{2}[A_{t}^{2,-},A_{t}^{2,+}]
  \} \\
  &\qquad\qquad\bigcap \{ -V_j' \in [ (\log k_1)^{0.49},(\log k_1)^{0.51} \}.
  \end{aligned}
\end{equation}
Let $\mathfrak{m}'$ and $\mathfrak{m}''$ be the sum of all the $\mathfrak{m}_j'$ and $\mathfrak{m}_j''$ respectively (see \eqref{eq:2ndintensity}).
From Lemma \ref{lem:bananadensity} and a direct first moment estimate
\[
  \liminf_{n \to \infty}
  \Exp[
  |\mathfrak{m}'( f )
  -
  \mathfrak{m}( f )|~|~\filt_{k_2}
  ]
  \leq
  o_{k_2}
  \mathscr{D}_{k_2}(f),
\quad \As
\]
and hence on taking $k_2 \to \infty,$ this tends to $0$ almost surely.
Moreover, on the event $\mathscr{Z}_{k_2}$, we have that $\mathfrak{m}' \geq \mathfrak{m}''$, and hence combining this with the display above
\[
  \liminf_{n \to \infty}
  \Exp[\mathfrak{m}(f)~|~\filt_{k_2}]
  \geq
  \liminf_{n \to \infty}
  \Exp[\mathfrak{m}''(f) \one[\mathscr{Z}_{k_2}]~|~\filt_{k_2}].
\]
To evaluate this expectation, we lower bound the expectation without the restriction to the event $\mathscr{Z}_{k_2}$ and then argue the event can be removed.

Using Lemma \ref{lem-onerayLB} and the Girsanov transformation,
we also can give a sharp estimate for the ray probability.
However, this is not given as a density estimate, but rather with a restriction of the endpoint $V_j'$ to land in an interval of length $1/\sqrt{k_3}$ (as in Proposition \ref{prop:2ray}).  Hence we partition the interval $J$ into a grid $\bar J$ of separation $1/\sqrt{k_3}$, whose elements are parametrized by $z$:
\[
  \begin{aligned}
  \Exp[\mathfrak{m}''(f)~|~\filt_{k_2}]
  =
  &\sum_{j \in \mathcal{D}_{n/k_1}}
  \Exp
  \bigl(
  \one[{
  \mathscr{V}_{j}''
  }]
  p_j
  ~
  \big\vert
  ~\filt_{k_2}
  \bigr) \\
  \geq
  &\sum_{j \in \mathcal{D}_{n/k_1}}
  \frac{1}{n_1}
  \bigl(\sqrt{2}H_{k_2}-\sqrt{\tfrac{\beta}{4}}\varphi_{k_2}(\theta_j)\bigr)
  e^{\sqrt{\beta/2}\varphi_{k_2}(\theta_j) - \log k_2 + o_{k_2}} \\
  &\hspace{3em}\times
  (1+\eta_{k_2,k_3})
  \frac{k_1^+}{k_1}
  \sum_{z\in \bar J}
  \sqrt{\frac{2}{\pi}}
  ze^{\sqrt{2}z}
  \times
  \min_{x \in [0,1/\sqrt{k_3}]}\frac{p(z+x)}{\sqrt{k_3}},
\end{aligned}
\]
where the factor $\eta_{k_2,k_3} \to 0$ as $k_2 \to \infty$ followed by $k_3\to\infty$.
In Lemma \ref{lem:kregular} we show that the function $p(z)$ satisfies the estimate
\begin{equation}
  \label{eq-Add1}
  p(w+x) = (e^{\sqrt{2} x}+o_{k_1})p(w)\quad \text{uniformly over $|x| \leq 1$ and $w\in J$. }
\end{equation}
 Hence, we can uniformly approximate the sum over $z$ by an integral and arrive at
\begin{align*}
&  \Exp[\mathfrak{m}''(f)~|~\filt_{k_2}]\\
&  \geq
  (1+o_{k_3})
  \sum_{j \in \mathcal{D}_{n/k_1}}
  \frac{1}{n_1}
  \bigl(\sqrt{2}H_{k_2}-\sqrt{\tfrac{\beta}{4}}\varphi_{k_2}(\theta_j)\bigr)
  e^{\sqrt{\beta/2}\varphi_{k_2}(\theta_j) - \log k_2 }
  \frac{k_1^+}{k_1}
  \int_J
  \sqrt{\frac{2}{\pi}}
  ze^{\sqrt{2}z}
  p(z)\,dz.
\end{align*}
This is exactly $\mathcal{H}$, and so we have
\begin{equation}\label{zrh:mpp}
  \liminf_{n \to \infty}
  \Exp[\mathfrak{m}''(f)~|~\filt_{k_2}]
  \geq
  (1+o_{k_3})
  \int_0^{2\pi}
  \mathscr{D}_{k_2}(\theta)
  \,d\theta
  \times
  \mathcal{H}.
\end{equation}

To complete the first moment analysis, it suffices to show that
\[
  \limsup_{k_2 \to \infty}
  \limsup_{k_1,n \to \infty}
  \Exp[
  \mathfrak{m}''( f )
  \one[{\mathscr{Z}_{k_2}^c}]
  ~|~\filt_{k_2}
  ]
  \Prto[k_3] 0,
\]
as then from \eqref{zrh:mpp} and \eqref{zrh:m1u},
\begin{equation}\label{zrh:1stultimo}
  \Exp[
    \mathfrak{m}( f )
    \one[{\mathscr{Z}_{k_2}}]
    ~|~\filt_{k_2}
  ]
  -
  \int_0^{2\pi}
  \mathscr{D}_{k_2}(\theta)
  \,d\theta
  \times
  \mathcal{H}
  \Prto[{k_3,k_2,k_1,n}]
  0.
\end{equation}

We shall show using the second moment machinery in Proposition \ref{prop:2ray} (and see Remark \ref{rk:holybarrier}), that
\begin{equation}\label{zrh:5}
  \limsup_{k_2 \to \infty}
  \limsup_{k_1,n \to \infty}
  \Exp [\bigl(\mathfrak{m}''( \text{Arc}_{i} )\bigr)^2~|~\filt_{k_2} ]
  <
  \infty
  \quad \As
\end{equation}
And hence it follows by Cauchy--Schwarz and Lemma \ref{lem:goodstuff}
\[
  \limsup_{k_2 \to \infty}
  \limsup_{k_1,n \to \infty}
  \Exp[
  \mathfrak{m}''( \text{Arc}_{i} )
  \one[{\mathscr{Z}_{k_2}^c}]
  ~|~\filt_{k_2}
  ]
  \Prto[k_3] 0.
\]

We turn to the conditional second moment, and note that this argument will give the same upper bound as was needed for the second moment of $\mathfrak{m}''$ in \eqref{zrh:5}
(see Remark \ref{rk:holybarrier}):
\begin{equation}\label{zrh:4}
  \Exp[
  \mathfrak{m}( f )^2
  \one[ \mathscr{Z}_{k_2}]
  ~|~\filt_{k_2}]
  \leq \sum_{j,\ell}
  f(\theta_{j})
  f(\theta_{\ell})
  \Exp
  \biggl(
  \one[ \mathscr{Z}_{k_2}]
  \one[{
  \mathscr{R}_{j}^2(n_1^+)
  \cap
  \mathscr{R}_{\ell}^2(n_1^+)
  }]
  p_{j}p_{\ell}
  ~
  \bigg\vert
  ~\filt_{k_2}
  \biggr),
\end{equation}
where $p_j = p(V_j') = \Pr( W_j^o \in [-k_7,\infty)~|~\filt_{n_1^+})$ and the sum is over all $j,\ell \in \mathcal{D}_{n/k_1}.$
We divide the angle pairs $(\theta_j,\theta_\ell)$ into two classes $\operatorname{NR}$ and $\operatorname{FR}$, where the latter means $\theta_{j}$ and $\theta_{\ell}$ have branchpoint $\k$ (in the notation of Proposition \ref{prop:2ray}) less than $\tfrac12 \log_2 k_2$.  We let $\operatorname{NR}$ be all other pairs.
For the well--separated angles $\operatorname{FR}$, we can afford only multiplicative errors of the form $1+o_{k_3}$.  Now Proposition \ref{prop:2ray} gives such an estimate, albeit only after asking the endpoint to be in a bin of size $(1/\sqrt{k_3})$.

Thus by partitioning the two-ray expectation in \eqref{zrh:4} according to the values of $V_{j}'$ and $V_{\ell}'$ (using bins of size $(1/\sqrt{k_3})$ so to apply Proposition \ref{prop:2ray}), we have
that  for $(\theta_j,\theta_\ell) \in \operatorname{FR}$ the contribution to \eqref{zrh:4} is bounded above by
\[
  \sum_{z_j,z_\ell}
  \mathscr{Q}(\theta_j,x_j,z_j; \theta_\ell,x_\ell,z_{\ell})
  \times \max_{u_j,u_{\ell} \leq 1/\sqrt{k_3}}
  p(z_j+u_j)p(z_\ell+u_\ell).
\]
By \eqref{eq-Add1},
the maximum value of $p$ over a short interval of length $1/\sqrt{k_3}$ is its average value over the same interval up to an error $o_{k_3}$.
Using the first case of Proposition \ref{prop:2ray}, we conclude with that notation
\[
  \begin{aligned}
    &\sum_{\theta_j,\theta_\ell \in \operatorname{FR}}
    \Exp
    \biggl(
    \one[ \mathscr{Z}_{k_2}]
    \one[{
      \mathscr{R}_{j_1}^2(n_1^+)
      \cap
      \mathscr{R}_{j_2}^2(n_1^+)
    }]
    p_{j_1}p_{j_2}
    ~
    \bigg\vert
    ~\filt_{k_2}
    \biggr) \\
    &\leq
    (1+\eta_{k_2,k_3}+o_{k_3})
    \sum_{\theta_j,\theta_\ell}
    \int_J
    \int_J
    \frac{2}{\pi}
    \frac{z_jx_j k_2}{n_1^+}
    \frac{z_\ell x_\ell k_2}{n_1^+}
    \exp\bigl(\sqrt{2}(z_j-x_j + z_\ell-x_\ell)\bigr)
    p(z_1)p(z_2)
    dz_1dz_2,
  \end{aligned}
\]
where the sum is over all $\theta_j,\theta_\ell$ in the arc and $J$ is as before.
Thus on taking $n\to\infty,$
\begin{equation}
  \begin{aligned}
    \limsup_{n \to \infty}
    \sum_{\theta_j,\theta_\ell \in \operatorname{FR}}
    &\Exp
    \biggl(
    \one[ \mathscr{Z}_{k_2}]
    \one[{
      \mathscr{R}_{j_1}^2(n_1^+)
      \cap
      \mathscr{R}_{j_2}^2(n_1^+)
    }]
    p_{j_1}p_{j_2}
    ~
    \bigg\vert
    ~\filt_{k_2}
    \biggr) \\
    &\leq
    (1+o_{k_3})
    \biggl(
    \int_0^{2\pi}
    f(\theta)
    \mathscr{D}_{k_2}(\theta)
    d\theta
    \times
    \mathcal{H}
    \biggr)^2.
  \end{aligned}
  \label{zrh:2ndfar}
\end{equation}

For the near terms, we need to use all the cases of the 2-ray bound Proposition \ref{prop:2ray}:
\[
  \begin{aligned}
    \sum_{\theta_j,\theta_\ell \in \operatorname{NR}}
    &\Exp
    \biggl(
    \one[ \mathscr{Z}_{k_2}]
    \one[{
      \mathscr{R}_{j_1}^2(n_1^+)
      \cap
      \mathscr{R}_{j_2}^2(n_1^+)
    }]
    p_{j_1}p_{j_2}
    ~
    \bigg\vert
    ~\filt_{k_2}
    \biggr) \\
    &\leq
    c(k_3)
    \sum_{\theta_j,\theta_\ell \in \operatorname{NR}}
    \int_J\int_J
    \mathscr{Q}(\theta_j,x_j,z_j;\theta_\ell,x_\ell,z_\ell)
    p(z_j)p(z_\ell)
    dz_j dz_\ell.
    \end{aligned}
\]
We then partition these near terms into which case of Proposition \ref{prop:2ray} the pairs land.  That is we define the sets:
\begin{enumerate}
  \item $\operatorname{NR}_1(\theta_j)$ are all those $\ell$ so that $\k \geq (\log_2 k_2)/2$ but $\k_+ \leq \log_2 k_2$.
  \item $\operatorname{NR}_2(\theta_j)$ are all those $\ell$ not in $\operatorname{NR}_1(\theta_j)$ but so that $\k \leq (\log_2 n)/2$.
  \item $\operatorname{NR}_3(\theta_j)$ are all those $\ell$ not in $\operatorname{NR}_i(\theta_j)$ for $i\in\{1,2\}$ but so that $\k \leq \log_2 n_1^+$.
  \item $\operatorname{NR}_4(\theta_j)$ are all those $\ell$ not in $\operatorname{NR}_i(\theta_j)$ for $i\in\{1,2,3\}$.
\end{enumerate}
In all cases, we have sharp dependence on $x_j$ and $z_j$, and so we can integrate over $z_j$ to give exactly $\mathcal{H}$ up to an error depending on $k_3$. For the terms in $\operatorname{NR}_1(\theta_j)$, we also have sharp dependence on $x_\ell$ and $z_\ell$ up to a multiplicative error, and hence we have
\[
  \begin{aligned}
    \limsup_{n\to\infty}
    &\sum_{j\in\mathcal{D}_{n/k_1}}
    \sum_{\ell \in \operatorname{NR}_1(\theta_j)}
    \Exp
    \biggl(
    \one[ \mathscr{Z}_{k_2}]
    \one[{
      \mathscr{R}_{j_1}^2(n_1^+)
      \cap
      \mathscr{R}_{j_2}^2(n_1^+)
    }]
    p_{j_1}p_{j_2}
    ~
    \bigg\vert
    ~\filt_{k_2}
    \biggr) \\
    &\leq
    c(k_3)
    \biggl(
    \int_0^{2\pi}
    \mathscr{D}_{k_2}(\theta)
    \,d\theta
    \times
    \mathcal{H}
    \biggr)
    \times
    \max_{\theta}
    \mathscr{D}_{k_2}([\theta,\theta+1/\sqrt{k_2}]).
  \end{aligned}
\]
By nonatomicity of $\mathscr{D}_\infty$,
\[
\max_{\theta} \mathscr{D}_{k_2}([\theta,\theta+1/\sqrt{k_2}])
\Asto[k_2] 0.
\]

For the terms of the second type, the bound in Proposition \ref{prop:2ray} loses the dependence in $x_\ell,$ but we gain a factor due to the entropic barrier, which is summable
\[
  \begin{aligned}
    \limsup_{n\to\infty}
    &\sum_{j\in\mathcal{D}_{n/k_1}}
    \sum_{\ell \in \operatorname{NR}_2(\theta_j)}
    \Exp
    \biggl(
    \one[ \mathscr{Z}_{k_2}]
    \one[{
      \mathscr{R}_{j_1}^2(n_1^+)
      \cap
      \mathscr{R}_{j_2}^2(n_1^+)
    }]
    p_{j_1}p_{j_2}
    ~
    \bigg\vert
    ~\filt_{k_2}
    \biggr) \\
    &\leq
    c(k_3)
    \biggl(
    \int_0^{2\pi}
    \mathscr{D}_{k_2}(\theta)
    \,d\theta
    \times
    \mathcal{H}
    \biggr)
    \times
    e^{-c(\log k_2)^{1/10}}.
  \end{aligned}
\]
The same bound holds for terms of the third type, using the entropic barrier gain, but now gaining a factor $e^{-c(\log k_1)^{1/10}}$.  Finally for the terms of fourth type, we use Lemma \ref{cor:qbnd}, due to which we gain the same $e^{-c(\log k_1)^{1/10}}$.  In all, we have that the sum over all $\operatorname{NR}$ pairs satisfies
\[
  \begin{aligned}
    \sum_{\theta_j,\theta_\ell \in \operatorname{NR}}
    &\Exp
    \biggl(
    \one[ \mathscr{Z}_{k_2}]
    \one[{
      \mathscr{R}_{j_1}^2(n_1^+)
      \cap
      \mathscr{R}_{j_2}^2(n_1^+)
    }]
    p_{j_1}p_{j_2}
    ~
    \bigg\vert
    ~\filt_{k_2}
    \biggr)
    \leq
    o_{k_2}
    \biggl(
    \int_0^{2\pi}
    \mathscr{D}_{k_2}(\theta)
    \,d\theta
    \times
    \mathcal{H}
    \biggr).
    \end{aligned}
\]
Combining this with \eqref{zrh:2ndfar}, we have that
\[
  \limsup_{k_1,n\to\infty}
  \Var( \mathfrak{m}(f)
  \one[ \mathscr{Z}_{k_2}]
  ~|~\filt_{k_2})
  \Prto[k_3,k_2] 0.
\]
This together with \eqref{zrh:1stultimo} proves the lemma.
\end{proof}

\subsection{Third Poisson approximation: concentration of the intensity}
\label{sec-thirdppp}

The final Poisson approximation replaces the intensity $\mathfrak{m}$ by (essentially) its $\filt_{k_2}$--conditional expectation.  This is done by a first and second moment computation.
The measure to which we compare $\mathfrak{m}$ is the one given in the introduction \eqref{eq:trueintensity}, namely with $I(v)$ as in \eqref{eq:trueintensity}
and $\mathfrak{s}$ as in \eqref{eq:decorationlaw}:
\begin{equation}\label{mtl:intensity}
  \mathfrak{n}(\theta,v,f) \coloneqq
  \mathscr{D}_{k_2}(\theta)d\theta 
  \times I(v)dv 
  \times \mathfrak{p}(v,df)
  \quad\text{where}\quad
  \mathfrak{p}(v,f)
  \coloneqq
  \begin{cases}
    \int_0^{2\pi} \mathfrak{s}_{k_1,k_5}(v,e^{i\xi} df)\tfrac{d\xi}{2\pi},
   &
   \text{ if $\sigma=1$,} \\
   \mathfrak{s}_{k_1,k_5}(v,df),
   &\text{ otherwise. }
  \end{cases}
\end{equation}
\begin{proposition}\label{prop:ppp3}
  For any $k_7$ the restrictions of $\Pi(\mathfrak{m})$ and $\Pi(\mathfrak{n})$ to $\Gamma_{k_7}$ satisfy
  \[
      \partial_2(
      \Pi(\mathfrak{m}) \cap \Gamma_{k_7}
      ~\vert~ \filt_{k_2}
      ,
      \Pi(\mathfrak{n}) \cap \Gamma_{k_7}
      ~\vert~ \filt_{k_2}
      )
    \Prto[k_3,k_2,k_1,n]
    0.
  \]
\end{proposition}
\begin{proof}
  In the first step, we replace the measure $\mathfrak{m}$ by one in which the decoration is averaged, that is, we compare to
  \begin{equation}\label{eq:2ndintensityprime}
    \overline{\mathfrak{m}}(\theta,v, f) \coloneqq
    \sum_{j \in {\mathcal{D}}_{n/k_1}}
    \one[{\mathscr{R}^2_j({n}_1^+)}]
    \delta_{\theta_j, V_j'}(d\theta,dv)\mathfrak{p}\bigl(v,
    df \bigr).
  \end{equation}
Note that this is identical to $\mathfrak{m}$ in the case $\sigma = i$.
The intensities $\mathfrak{m}(\Gamma_{k_7})$ and $\overline{\mathfrak{m}}(\Gamma_{k_7})$ are tight in all parameters (see \eqref{eq:mupper}) $\{n,k_1,\dots,k_6\}$.
Using Theorem \ref{thm:PPcom}, it thus suffices to show $d_{\operatorname{BL}}(\mathfrak{m},\overline{\mathfrak{m}}) \to 0$ to conclude that
\[
 \partial_2(
      \Pi(\mathfrak{m}) \cap \Gamma_{k_7}
      ~\vert~ \filt_{k_2}
      ,
      \Pi(\overline{\mathfrak{m}}) \cap \Gamma_{k_7}
      ~\vert~ \filt_{k_2}
      )
    \Prto[k_3,k_2,k_1,n]
    0.
\]
Letting $F$ be the functions in $\mathcal{C}( (-2\pi k_1,0), \C)$ with max modulus in $[e^{-k_7},e^{k_7}]$, we can bound
\[
  \begin{aligned}
    d_{\operatorname{BL}}(\mathfrak{m},\overline{\mathfrak{m}})
    &\leq
    (\mathfrak{m}(\Gamma_{k_7}) + \overline{\mathfrak{m}}(\Gamma_{k_7})) \\
    &\times\sum_{j \in {\mathcal{D}}_{n/k_1}}
    \one[{\mathscr{R}^2_j({n}_1^+)}]
    d_{\operatorname{BL}}\biggl(
    \int_0^{2\pi} \mathfrak{s}_{k_1,k_5}(V_j',e^{i\xi-\sqrt{4/\beta}V_j'} df)\tfrac{d\xi}{2\pi}
    ,
    \mathfrak{s}_{k_1,k_5}(V_j',
    e^{i\alpha_j-\sqrt{4/\beta}V_j'} df)
    \biggr),
  \end{aligned}
\]
where $\alpha_j \coloneqq \Psi_{n_1^+}(\theta_j)-(n_1^++1)\theta_j$ and the bounded Lipschitz norm is restricted to $F$.
We show in Corollary \ref{cor:dtv} that this distance is bounded by $\mathcal{O}(e^{-(\log k_1)^{19/20}})p(V_j')$ where $p(V_j') = \Pr( W_j^o \in [-k_7,\infty)~|~\filt_{n_1^+}).$  Thus like in \eqref{eq:mupper}
\[
    d_{\operatorname{BL}}(\mathfrak{m},\overline{\mathfrak{m}})
    \Prto[k_1,n] 0.
\]

For the final step, we show
\[
 \partial_2(
      \iota\#(\Pi(\mathfrak{n}) \cap \Gamma_{k_7})
      ~\vert~ \filt_{k_2}
      ,
    \iota\#( \Pi(\overline{\mathfrak{m}}) \cap \Gamma_{k_7})
      ~\vert~ \filt_{k_2}
      )
    \Prto[k_3,k_2,k_1,n]
    0.
\]
We again use Theorem \ref{thm:PPcom} to reduce this to controlling the bounded Lipschitz norm on $\Gamma_{k_7}$. Let $\bar\iota\#\mathfrak{n}$ and $\bar \iota\#\overline{\mathfrak{m}}$  be the restriction of $\iota\#(\mathfrak{n}\cap \Gamma_{k_7})$ and $ \iota\# (\overline{\mathfrak{m}}\cap \Gamma_{k_7})$  to $[0,2\pi] \times \mathcal{C}( [-2\pi k_1,0], \C)$,
 and note that the restriction to $\mathcal{C}( [-2\pi k_1,0], \C)$ 
is identical for both processes except for  for a additive random variable, taking value in a compact ($k_7$-dependent) set. 
From Corollary \ref{cor:PPcom} and the fact that after the push forward by $\iota$, the second coordinate in $\Gamma$ is continuously determined by the third, there is a finite list of non-negative Lipschitz functions $\{f_j\}$, depending on $k_7$, for which it suffices to show that
\begin{equation}
\label{eq-210822}
 \bar \iota\# \mathfrak{n}(f_j) - \bar \iota\#\overline{\mathfrak{m}}(f_j)
  \Prto[k_3,k_2,k_1,n]
  0.
\end{equation}
(We note that the use of $\iota$ was precisely to reduce the collections of functions $\{f_j\}$ to a collection that does not depend on $k_1$.)
Now, \eqref{eq-210822}  follows from Lemma \ref{zrh:2moment}.
\end{proof}

\section{Interpolation--based regularity arguments}
\label{sec:meshing}

In this section we give a proofs of Proposition \ref{prop:interpolation0} and Proposition \ref{prop:interpolation}; both rely on certain \emph{a priori} regularity properties of $\varphi_n,$ but the second one additionally uses a substantial probabilistic input, Proposition \ref{prop:lsmcontinuity} in the case $\sigma=1$. The $\sigma=i$ case is substantially simpler; in effect the regularity is much better for imaginary $\sigma$ owing to the monotonicity of the Pr\"ufer phases.  We introduce the following notation for working with the case of real $\sigma.$  For any $\theta \in \R,$ define (with $\delta_\beta$ given by half of $\delta$ from Proposition \ref{prop:lsmcontinuity})
\begin{equation}
  J(\theta)
  \coloneqq
  \left\{
    \theta' \in \frac{2\pi}{4k_5 n}\Z
    :
    |\theta' - \theta| \leq \frac{\pi}{n k_5^{1-\delta_\beta}}
  \right\}.
\end{equation}

The deterministic result we need is the following:
\begin{lemma}\label{lem:interpolation}
  For $\sigma = i,$
  and any $\theta' \leq \theta$
  \[
    \varphi_n(\theta') \leq
    \varphi_n(\theta) + (n+1)(\theta-\theta').
  \]
  For $\sigma = 1$ if $\sup_{\theta} \varphi_n(\theta) \leq \sqrt{\tfrac{8}{\beta}}m_n+k_6,$ then there is an absolute constant $C$ so that for all $n,k_5$ sufficiently large with respect to $k_6,$ and all $\theta \in [0,2\pi]$
  \[
      \begin{aligned}
	&
	\max_{|\theta'-\theta| <  \frac{2\pi}{4n k_5^{1-\delta_\beta}}}
      \biggl\{e^{\varphi_n(\theta')-\sqrt{\tfrac{8}{\beta}}m_n} \biggr\}
      \leq \biggl(\frac{2k_5}{2k_5-1}\biggr)\cdot\max_{\theta' \in J(\theta)}
      \biggl\{e^{\varphi_n(\theta')-\sqrt{\tfrac{8}{\beta}}m_n} \biggr\}
      +\frac{Ce^{k_6} }{k_5^{\delta_\beta}}
      \quad\text{and}\\
	&\min_{|\theta'-\theta| <  \frac{2\pi}{4n k_5^{1-\delta_\beta}}}
      \biggl\{e^{\varphi_n(\theta')-\sqrt{\tfrac{8}{\beta}}m_n} \biggr\}
      \geq \biggl(\frac{2k_5}{2k_5-1}\biggr)\cdot\min_{\theta' \in J(\theta)}
      \biggl\{e^{\varphi_n(\theta')-\sqrt{\tfrac{8}{\beta}}m_n} \biggr\}
      -\frac{Ce^{k_6} }{k_5^{\delta_\beta}}.
  \end{aligned}
  \]
\end{lemma}
\begin{proof}
  For the first claim, we use that the Pr\"ufer phases $\theta \mapsto \Psi_k(\theta)$ are monotone increasing in $\theta$, see the discussion in the beginning of Section
\ref{subsec-softprufer}, and recall that $\varphi_n(\theta)=\Psi_n(\theta)-(n+1)\theta$.
 For the second claim, we apply Theorem \ref{thm:interpolation} to the polynomial $Q$ such that $|Q(e^{i\theta})|^2 = e^{\varphi_n(\theta)}$ for all real $\theta$.  This has degree $n$ and the theorem applies with $m=2k_5$ and $b\asymp k_5^{1-\delta_\beta}.$
\end{proof}

To take advantage of this deterministic result, we then use some first and second moment estimates which when combined with Lemma \ref{lem:interpolation} imply Proposition \ref{prop:interpolation}. The first of these is essentially a triviality that shows that we can disregard near maxima that occur near the boundary of an interval $\widehat{I}_j,$ more specifically those mesh points that are not contained in $K = \cup_{j \in \mathcal{D}_{n/k_1}} I_j$ (recall the definition of $I_j$ in \eqref{eq:Ij}).
\begin{lemma}\label{lem:inter1}
For all $k_2,k_4,k_5,k_6$
  \[
    \limsup_{n \to \infty}
    \sum_{j = 1}^{4k_5 n}
    \one\{ \tfrac{\pi j}{2k_5 n} \not\in K\}
    \Exp[
      \one\{\mathscr{L}(\tfrac{\pi j}{2k_5 n})\}
      \one\{\widehat{\mathscr{R}}(\tfrac{\pi j}{2k_5 n})\}
      \one\{\mathscr{G}_{n}\}
      ~\vert~
      \filt_{k_2}
    ]
    \Prto[k_1]
    0.
  \]
\end{lemma}
\begin{proof}
  For any $I_j$ the fraction of angles $\theta \in I_j$ so that $\theta \not\in K$ vanishes like $\tfrac{\log k_1}{k_1}$ as $k_1 \to \infty.$  In particular, the left-hand-side of the display tends to $0$ like $O_{k_p, p \geq 2}( (\log k_1)/k_1)$  (compare with the proof of Proposition \ref{prop:nearmaximarays}).
\end{proof}

From here we can also give a quick proof of Proposition \ref{prop:interpolation0}:
\begin{proof}[Proof of Proposition \ref{prop:interpolation0}] We focus on the harder case $\sigma=1$, since the proof for $\sigma=i$ is immediate.
  We first observe that on the event $\mathscr{G}_{n},$ globally $\theta \mapsto \varphi_n(\theta)$ is bounded by $k_6$, and hence we always have $W_j \leq k_6$ for any $j \in \mathcal{D}_{n/k_1}.$

  Suppose that $\widehat{W}_j \geq -k_7$, and hence there is a $\theta$ in a neighborhood of $\widehat{I}_j$ at which $\varphi_n(\theta) - \sqrt{\tfrac{8}{\beta}}m_n \geq -k_7.$
  Using the first conclusion of Lemma \ref{lem:interpolation}, there must be a $\theta' \in J(\theta)$ at which
  \[
    e^{\varphi_n(\theta') - \sqrt{\tfrac{8}{\beta}}m_n} \geq
    \biggl(\frac{2k_5-1}{2k_5}\biggr)
    \biggl(
    e^{-k_7} - \frac{Ce^{k_6} }{k_5^{\delta_\beta}}
    \biggr).
  \]
  By virtue of Lemma \ref{lem:inter1}, we may assume that $\theta' \in I_j.$  Making $k_5$ large, we conclude
  \[
    \varphi_n(\theta') - \sqrt{\tfrac{8}{\beta}}m_n
    \geq -k_7 - o_{k_5}(1).
  \]
  As $k_6$ is much larger that $k_7,$ this concludes the proof.
\end{proof}

The second probabilistic input we need in order to prove Proposition \ref{prop:interpolation} is that on the finite sets $J(\theta)$ for $\theta \in I_j$ at which $\mathscr{L}(\theta)$ occurs, the process $\varphi$ can be taken nearly constant (or more specifically its oscillation is no more than $k_5^{-\delta_\beta}$) simultaneously for all $\theta$ for which $\mathscr{L}(\theta)$ holds.  Define the event
\begin{equation}\label{eq:J}
  \mathscr{J}(\theta)
  \coloneqq
  \biggl\{
    \bigl(
    \max_{\theta' \in J(\theta)}
    \{\varphi_n(\theta')\}
    -\min_{\theta' \in J(\theta)} \{\varphi_n(\theta')\}
    \bigr)\leq k_5^{-\delta_\beta}
  \biggr\}.
\end{equation}
Thus we show:
\begin{lemma}\label{lem:inter2}
For all $k_6$
  \begin{align*}
&    \limsup_{n \to \infty}
    \sum_{j = 1}^{4k_5 n}
    \Exp[
      \one\{\mathscr{L}(\tfrac{\pi j}{2k_5 n})
      \cap \mathscr{R}_{j}^2({n}_1^+)
      \cap \mathscr{P}_j'
	    \cap \mathscr{O}(\tfrac{\pi j}{2k_5 n})
	    \cap \mathscr{O}^{\Psi}(\tfrac{\pi j}{2k_5 n})
      \}
      \one\{\mathscr{J}^c(\tfrac{\pi j}{2k_5 n})\}
      \one\{\mathscr{G}_{n}\}
      ~\vert~
      \filt_{k_2}
    ]\\
&\qquad \qquad\qquad\qquad
    \Prto[k_5,k_4,k_2,k_1]
    0.
  \end{align*}
\end{lemma}
\noindent Proposition \ref{prop:interpolation} follows immediately from Lemmas \ref{lem:inter2} and \ref{lem:interpolation} in the case $\sigma =1$.
\begin{proof}
  We may additionally work on the event $\mathscr{G}_n^1,$ using which we may replace $\varphi_n(\theta_\ell + \tfrac{\theta}{n})$ by $\mathfrak{U}_{T_+}(\theta)$.
  The main technical work is contained in Proposition \ref{prop:lsmcontinuity}.
  Under this proposition, we have an estimate
  \[
    \begin{aligned}
     &\Exp[
      \one\{\mathscr{L}(\tfrac{\pi j}{2k_5 n})
      \cap \mathscr{R}_{j}^2({n}_1^+)
      \cap \mathscr{P}_j'
	    \cap \mathscr{O}(\tfrac{\pi j}{2k_5 n})
	    \cap \mathscr{O}^{\Psi}(\tfrac{\pi j}{2k_5 n})
      \}
      \one\{\mathscr{J}^c(\tfrac{\pi j}{2k_5 n})\}
      \one\{\mathscr{G}_{n}\}
      ~\vert~
      \filt_{n_1^+}
    ] \\
    &\leq
    C(k_6)
    k_5^{-(1+2\delta_\beta)(1-\delta_\beta)}
      \frac{
	\bigl(
	\sqrt{\tfrac{8}{\beta}}m_{n_1^+}
	-
	\varphi_{n_1^+}(\tfrac{\pi j}{2k_5 n})
	\bigr)
	\exp\bigl(\sqrt{\tfrac{\beta}{2}}
	\varphi_{n_1^+}(\tfrac{\pi j}{2k_5 n})
	-2
	m_{n_1^+}
	\bigr)
      }{k_1^+(\log k_1)^{3/2}}
        \Exp[
          \one\{\mathscr{R}_{j}^2({n}_1^+)\}
      ~\vert~
      \filt_{n_1^+}
    ].
  \end{aligned}
  \]
  We have used here that $\varphi_{n_1^+}(\tfrac{\pi j}{2k_5 n}) = \mathfrak{U}_{T_-}^j(\tfrac{\pi j}{2k_5}-n\theta_j)$, the initial conditions from \eqref{eq:LU}.
  Bounding above the probability, which proceeds in the same fashion as \eqref{eq:nma4a}--\eqref{eq:nmr3}, we have
    \begin{equation*}
    \begin{aligned}
      \limsup_{k_1,n \to \infty}
      &\sum_{j = 1}^{4k_5 n}
      \Exp[
      \one\{\mathscr{L}(\tfrac{\pi j}{2k_5 n})
      \cap \mathscr{R}_{j}^2({n}_1^+)
      \cap \mathscr{P}_j'
	    \cap \mathscr{O}(\tfrac{\pi j}{2k_5 n})
	    \cap \mathscr{O}^{\Psi}(\tfrac{\pi j}{2k_5 n})
      \}
      \one\{\mathscr{J}^c(\tfrac{\pi j}{2k_5 n})\}
      \one\{\mathscr{G}_{n}\}
      ~\vert~
      \filt_{k_2}
      ] \\
      &\leq
      C(k_6)
      k_5^{-\delta_\beta + 2\delta_\beta^2}
      \int_0^{2\pi}
      e^{\sqrt{\tfrac{\beta}{2}}( 6k_6 + \varphi_{k_2}(\theta)) -\log(k_2)}
      \bigl(C_\beta k_6+\sqrt{2}\log k_2 - \sqrt{\tfrac{\beta}{4}} \varphi_{k_2}(\theta)\bigr)
      d\theta \\
      &+
      k_5^{-\delta_\beta + 2\delta_\beta^2}
      \int_0^{2\pi}
      e^{\sqrt{\tfrac{\beta}{2}}(  \varphi_{k_2}(\theta) + k_6) -\log(k_2)}
      \one[{ \varphi_{k_2}(\theta) \not\in [s_{k_2}^{-},s_{k_2}^+] }]
      \bigl|c_\beta k_6+\sqrt{2}\log k_2 - \sqrt{\tfrac{\beta}{4}} \varphi_{k_2}(\theta)\bigr|
      d\theta.
    \end{aligned}
  \end{equation*}
  Using Theorem \ref{thm:Bj} we have on sending $k_2 \to \infty$ that this upper bound converges almost surely.
  On afterward sending $k_5$ to infinity, the result follows.
\end{proof}

\section{Bessel bridges}
\label{sec:bessels}
In what follows $\theta_0$ will be a fixed angle in $[0,2\pi].$
Recall from \eqref{eq:GkZt} that for $\sigma \in \left\{ 1, i \right\},$
\begin{equation*}
  G_{k+1}(\theta_0)=G_{k}(\theta_0) - 2\Re \{ \sigma \bigl(\tfrac{Z_k}{\beta_k}\bigr)e^{i\Psi_k(\theta_0)}\}
  ,\quad G_0(\theta_0) = \theta_0.
\end{equation*}
Recall further, see \eqref{eq:Z}, \eqref{eq:GkZt},
the complex Brownian motion $\mathfrak{Z}^\C_t(\theta_0)$ with the 
normalization $[\mathfrak{Z}_t^\C(\theta_0),\overline{\mathfrak{Z}_t^\C(\theta_0)}] = 2t,$ and the standard real Brownian motion $\mathfrak{Z}_t=\Re(\sigma \mathfrak{Z}_t^\C(\theta_0))$ for any $t \geq 0,$ so that $G_k(\theta_0) = \sqrt{\tfrac{4}{\beta}}\mathfrak{Z}_{H_k}(\theta_0)$ for any $k \geq 0.$

We shall work conditionally on the endpoints $\mathfrak{Z}_t$ of various intervals $[t_0,t_1]$, after which the process is a standard Brownian bridge.  Further conditioning this bridge to lie below the line $t\mapsto \alpha t,$ on the interval $[t_0, t_1]$ the process $\alpha t - \mathfrak{Z}_t$ has the law of a \emph{$3$--dimensional Bessel process bridge}.   It remains a semimartingale after this conditioning (with respect to the appropriate augmented filtration), and moreover it is a strong solution to an SDE.  We record these facts and some distributional facts about this Bessel bridge in the following lemma.
\begin{lemma}\label{lem:besselbridge}
  Let $(B_t : t \geq 0)$ be a standard real Brownian motion.
  For any $\alpha, c_0,c_1 > 0$ and any $ 0 < t_0 < t_1,$
  let $(\mathfrak{X}_t : t \in [t_0,t_1])$ be the strong solution of
  \[
    \mathfrak{X}_t
    -
    \mathfrak{X}_{t_0}
    =
    \alpha (t-t_0) +  (B_t - B_{t_0}) + \int_{t_0}^t \left( \frac{1}{\mathfrak{X}_s-\alpha s} - \frac{\mathfrak{X}_s-(\alpha s-c_1)}{t_1-s} \right)\,ds,
    \quad
    \mathfrak{X}_{t_0}
    = \alpha t_0 - c_0.
  \]
  Then $(\mathfrak{Z}_t : t \in [t_0,t_1])$ has the same law
  as $(\mathfrak{X}_t : t \in [t_0,t_1])$
  when conditioning on
  \[
    \mathfrak{Z}_{t_0} = \alpha t_0 -c_0,
    \quad
    \mathfrak{Z}_{t_1} = \alpha t_1 -c_1,
    \quad{\text{and}}\quad
    \mathfrak{Z}_t \leq \alpha t \quad \text{for all} \quad t \in [t_0,t_1]
    .
  \]
  For any $t \in (t_0,t_1),$ the density of 
  $y_t = \alpha t-\mathfrak{X}_t$
  is given by
  \[
    f(u)
    =
    Z(t_1-t_0,c_0,c_1)
    \sinh\biggl(
    \frac{c_0 u}{t-t_0}
    \biggr)
    \sinh\biggl(
    \frac{u c_1}{t_1-t}
    \biggr)
    \exp\biggl(
    -\frac{u^2}{2(t-t_0)}
    -\frac{u^2}{2(t_1-t)}
    \biggr),
    \quad u \geq 0,
  \]
  where $Z(\cdot,\cdot,\cdot)$ is a normalizing constant (given explicitly in the proof text).
  Provided $(c_0^2+c_1^2) \leq (t_1 - t_0)$ we therefore have the estimates, with $s=\min\{t-t_0,t_1-t\}$, for some absolute constant $C$
  \[
    \Pr(\alpha t - \mathfrak{X}_t \leq x) \leq \frac{C x^3}{s^{3/2}}
    \quad \text{for all}
    \quad x \leq \sqrt{s}.
  \]
  Furthermore, with $i \in \{0,1\}$ depending on 
  whichever achieves the minimum of $\min\{|t-t_i|\}$ in the definition of $s$,
  \begin{equation}
    \label{eq-290324}
    \Pr(\alpha t - \mathfrak{X}_t \in du) \leq
    C
    e^{-(u-c_i)^2/(2s)}\frac{u^2 du}{s^{3/2}}
    \quad \text{for all}
    \quad u \geq 0.
  \end{equation}
\end{lemma}
\begin{proof}
  See \cite[Chapter XI]{RevuzYor} for a good exposition.
  Setting $X_t = \alpha t - \mathfrak{Z}_t,$ the SDE reduces to a standard result on a Brownian motion conditioned to stay positive (c.f.\ \cite[Exercise XI.3.11.2]{RevuzYor}, \cite{Roberts}).
  The density $f$
  is given in \cite[Chapter XI.3]{RevuzYor}. The normalizing constant $Z$ is given by
  \[
    Z
    =
    \sqrt{\frac{2}{\pi}}
    \sqrt{\frac{t_1-t_0}{(t-t_0)(t_1-t)}}
    \exp\biggl(
    -\frac{c_0^2}{2(t-t_0)}
    +\frac{c_0^2}{2(t_1-t_0)}
    -\frac{c_1^2}{2(t_1-t)}
    +\frac{c_1^2}{2(t_1-t_0)}
    \biggr)\!\operatorname{csch}\biggl(\frac{c_0c_1}{t_1-t_0}\biggr).
  \]
  Under the assumption on $c_0^2+c_1^2$ and using the numerical inequality $x \leq \sinh(x)$ for all $x \geq 0$ we may bound above the normalizing constant $Z$
  \[
    Z \leq C\biggl(
      \frac{t_1-t_0}{c_0c_1}\biggr)
      \biggl(\frac{t_1-t_0}{(t-t_0)(t_1-t)}\biggr)^{1/2}
      \exp\biggl(
      -\frac{c_0^2}{2(t-t_0)}
      -\frac{c_1^2}{2(t_1-t)}
      \biggr)
 .
  \]
  Using $\sinh(x) \leq xe^x$ for all $x \geq 0$, we arrive at the density bound for another absolute constant $C>0$
  \[
  f(u) \leq 
  Cu^2\biggl(\frac{t_1-t_0}{(t-t_0)(t_1-t)}\biggr)^{3/2}
  \exp\biggl(
    -\frac{(u-c_0)^2}{2(t-t_0)}
    -\frac{(u-c_1)^2}{2(t_1-t)}
    \biggr).
  \]
  Hence with the same constant $C$, we have 
  \begin{equation}
    \label{eq-290324a}
    f(u) \leq \frac{2^{3/2}Cu^2}{s^{3/2}}
    \quad\text{for all}\quad u \leq \sqrt{s}.
  \end{equation}
  The final conclusion follows similarly.
\end{proof}

We also note that a Bessel bridge has good oscillation properties on short time windows, provided its endpoints are not brought close to the barrier:
\begin{lemma}
  Suppose that $(\mathfrak{X}_t : t \in [0,1])$ has the law of a $3$-dimensional Bessel bridge with $\mathfrak{X}_0 = x_0$ and $\mathfrak{X}_1 = x_1$.  Then if $x_1,x_0 \geq 1$, there is an absolute constant $c > 0$ so that
  \[
    \Pr\biggl(
    \sup_{t \in [0,1]}
    |\mathfrak{X}_t - (x_1-x_0)t - x_0| \geq s+c
    \biggr)
    \leq \exp(-s^2/c).
  \]
  \label{lem:bbridgeosc}
\end{lemma}
\begin{proof}
  We can realize the Bessel bridge by taking a Brownian Bridge $\mathfrak{Z}_t$ with the same endpoints and conditioning it to remain positive.  Let $\mathcal{H}$ be the event that
  this Brownian bridge is positive.  It then suffices to prove that
  \[
    \Pr\biggl(
    \biggl\{
    \sup_{t \in [0,1]}
    |\mathfrak{Z}_t - (x_1-x_0)t - x_0| \geq s+c
  \biggr\}
  ~\vert~
  \mathcal{H}
    \biggr)
    \leq \exp(-s^2/c).
  \]
  This is a standard fact for $\mathfrak{Z}_t$ without the conditioning, i.e.\ there is a $c>0$ sufficiently large that
    \[
    \Pr\biggl(
    \biggl\{
    \sup_{t \in [0,1]}
    |\mathfrak{Z}_t - (x_1-x_0)t - x_0| \geq s+c
  \biggr\}
    \biggr)
    \leq \exp(-s^2/c).
  \]
  The event $\mathcal{H}$ has probability bounded below by some absolute constant (indeed it is possible to compute it), but note that by monotonicity it is bounded below by the same probability where we take $x_1=x_0=1,$ which in any case is some number $p > 0$, and so
    \[
    \Pr\biggl(
    \biggl\{
    \sup_{t \in [0,1]}
    |\mathfrak{Z}_t - (x_1-x_0)t - x_0| \geq s+c
  \biggr\}
  ~\vert~
  \mathcal{H}
    \biggr)
    \leq \exp(-s^2/c)/p.
  \]
  Then increasing $c$, the claimed bound follows.
\end{proof}

As our first application, we show that conditionally on the process lying slightly below a concave barrier, the process in fact tends to stay in the entropic envelope. Recall the event 
$\mathscr{G}_n$ from Lemma \ref{lem:goodstuff1},
 and the event $ \widehat{\mathscr{R}}(\theta)$ from \eqref{eq:Rhat}.

\begin{lemma}\label{lem:allhailbanana}
  On the event that
  $\mathfrak{Z}_{H_{k_2}} \in \sqrt{2}(\log k_2 - [(\log k_2)^{0.49}, (\log k_2)^{0.51}])$
  and
  that
  $\mathfrak{Z}_{H_n} \in \sqrt{2}m_n + [-k_6,k_6],$
  there is a constant $c > 0$ sufficiently small
  so that for all $n \gg k_2 \gg k_4 \gg k_5$ sufficiently large
  \[
    \Pr(
    \widehat{\mathscr{R}}(\theta)^c
    \cap \mathscr{G}_n
    ~\vert~ \mathfrak{Z}_{H_{k_2}},\mathfrak{Z}_{H_{n}}
    )
    \leq
    \frac{
      (\sqrt{2}\log k_2 - \mathfrak{Z}_{H_{k_2}})
      e^{-c(\log k_5)^{2}}
    }{\log n}
    .
  \]
\end{lemma}
\begin{proof}
  We will work conditionally on $\mathfrak{Z}_{H_{k_2}}$ and $\mathfrak{Z}_{H_{n}}$ throughout the proof, and we will work on the $\sigma(\mathfrak{Z}_{H_{k_2}},\mathfrak{Z}_{H_{n}})$--measurable event in the statement of the lemma.

  We recall that on the event $\mathscr{G}_n$
 we have (see \eqref{eq:upperbarrier})
  \begin{equation}\label{eq:aha-1}
    \mathfrak{Z}_{H_{2^k}}
    \leq \sqrt{\beta} k_6 + \sqrt{2}A_{2^k}^{\ll}
    \quad
    \text{for all }
    \log_2 k_2 \leq k \leq \log_2 n.
  \end{equation}
  From \cite[Lemma A.5]{CMN}, the probability of this event is bounded above by
  \begin{equation}\label{eq:aha-1a}
    \Pr ( \eqref{eq:aha-1} | \mathfrak{Z}_{H_{k_2}}, \mathfrak{Z}_{H_{n}})
    \leq C_\beta \frac{
      (\sqrt{2}\log k_2 - \mathfrak{Z}_{H_{k_2}})_+
      k_6
    }{\log(n/k_2)}.
  \end{equation}

  We can fill in this barrier and slightly increase its height
  and define the event
  \begin{equation}\label{eq:aha0}
    \mathfrak{Z}_{t} \leq
    \sqrt{2}(1-\tfrac{3}{4}\tfrac{\log \log n}{\log n})t
    +g(t),
    \quad
    g(t)\coloneqq
    \biggl( \frac{t (H_n - t + (\log k_5)^{50})}{H_n}\biggr)^{1/50}
    \quad
    \text{for all }
    H_{k_2} \leq t \leq H_n.
  \end{equation}
  By a simple union bound estimate over each interval $[H_{2^{k-1}},H_{2^k}]$
  for all $k_5$ sufficiently large
  \begin{equation}\label{eq:aha0a}
    \Pr ( \eqref{eq:aha0}^c | \eqref{eq:aha-1}, \mathfrak{Z}_{H_{k_2}}, \mathfrak{Z}_{H_{n}})
    \leq e^{-c(k_6)(\log k_5)^{2}}.
  \end{equation}
  Note that we therefore have $\Pr ( \eqref{eq:aha0}^c \cap \eqref{eq:aha-1} \vert \mathfrak{Z}_{H_{k_2}}, \mathfrak{Z}_{H_{n}})$ is the same order as the bound claimed in the lemma, and it suffices to, going forward, bound $\Pr (\widehat{\mathscr{R}}(\theta)^c \cap \eqref{eq:aha0} \vert \mathfrak{Z}_{H_{k_2}}, \mathfrak{Z}_{H_{n}})$ (which will be lower order for $k_4 \gg k_5$).


  Recall that under $\Pr,$ conditional on $\mathfrak{Z}_{H_{k_2}},\mathfrak{Z}_{H_{n}},$ the process
  \[
    \mathfrak{B}_t = \mathfrak{Z}_t - \mathfrak{Z}_{H_{k_2}} - \frac{\mathfrak{Z}_{H_{n}}- \mathfrak{Z}_{H_{k_2}}} {H_n - H_{k_2}}(t-H_{k_2})
  \]
  is a bridge starting and ending an $0.$
  Define a change of measure by
  \begin{equation}\label{eq:aha1}
    \begin{aligned}
      \frac{d\Q}{d\Pr}
      &\coloneqq \exp\left(
      \int_{H_{k_2}}^{H_{n}} g'(t) d\mathfrak{Z}_t -
      \frac{1}{2}\int_{H_{k_2}}^{H_{n}} (g'(t))^2 dt
      \right) \\
      &= \exp\left(
      F( \mathfrak{Z}_{H_{n}}, \mathfrak{Z}_{H_{k_2}})
      -\int_{H_{k_2}}^{H_{n}}
      \bigl\{g''(t) \mathfrak{B}_t
      +\tfrac{(g'(t))^2}{2}\bigr\} dt
      \right),
    \end{aligned}
  \end{equation}
  for some (linear) functional $F.$
  Then, without conditioning on its endpoints, under $\Q,$ the process
  \(
  \mathfrak{Z}_t
  - g(t)
  \)
  is a standard Brownian motion for $t \in [H_{k_2},H_n],$ and after conditioning this process is a Brownian bridge.  Moreover, we see that the conditional Radon--Nikodym derivative is just given by \eqref{eq:aha1} with $F=0.$

  Estimating the probability $\Q( \widehat{\mathscr{R}}^c(\theta) ~|~ \eqref{eq:aha0})$ is straightforward, given that there is an explicit probabilistic description for the conditional process.  So, our goal is to show that $\Pr( \widehat{\mathscr{R}}^c(\theta) \cap \eqref{eq:aha0})$ is not much larger.
  \begin{equation}\label{eq:ahb0}
    \begin{aligned}
      \Pr( \widehat{\mathscr{R}}^c(\theta) \cap \eqref{eq:aha0}
      ~|~\vert
      \mathfrak{Z}_{H_{k_2}}, \mathfrak{Z}_{H_{n}}
      )
      &=
      \Q\bigl[
        \tfrac{d\Pr}{d\Q}
        \one[{
          \widehat{\mathscr{R}}^c(\theta) \cap \eqref{eq:aha0}
        }]
        ~\vert~
        \mathfrak{Z}_{H_{k_2}}, \mathfrak{Z}_{H_{n}}
      \bigr]
      \\
      &=
      \Q\bigl[
        \tfrac{d\Pr}{d\Q}
        \one[{
          \widehat{\mathscr{R}}^c(\theta)
        }]
        ~\vert~
        \eqref{eq:aha0},
        \mathfrak{Z}_{H_{k_2}}, \mathfrak{Z}_{H_{n}}
      \bigr]
      \Q\bigl[
        \eqref{eq:aha0}
        ~\vert~
        \mathfrak{Z}_{H_{k_2}}, \mathfrak{Z}_{H_{n}}
      \bigr].
      \\
    \end{aligned}
  \end{equation}
  Conditioning on \eqref{eq:aha0}, $\mathfrak{Z}_{H_{k_2}},$ and $\mathfrak{Z}_{H_{n}},$ $\mathfrak{Z}_t-g(t)$ has the law of a Bessel bridge under $\Q$.
  Thus under $\Q$ we can use the estimates in Lemma \ref{lem:besselbridge},
  to estimate the probability of $\widehat{\mathscr{R}}(\theta)$ under $\Q$.

  We first establish that under the conditional measure $\Q(\cdot~|~\eqref{eq:aha0}, \mathfrak{Z}_{H_{k_2}}, \mathfrak{Z}_{H_{n}})$, the process $\mathfrak{Z}_t$ is not outside the entropic barrier of half the size at \emph{integer times} up to $H_n-k_4+1$, (c.f.\ \eqref{eq:barrier})
  \begin{equation}\label{eq:tlv1}
    \tilde{A}_t^{n,\pm} \coloneqq
    t
    +
    \begin{cases}
      - 0.5 t ^{1/2 \mp 2/5} & \text{ if } t \leq \frac{1}{2}\log(n), \\
      -0.5 (\log n - t)^{1/2 \mp 2/5} - \frac{3}{4}\log\log n& \text{ if }  \frac{1}{2}\log(n)  < t \leq \log(n).
    \end{cases}
  \end{equation}
  For bounding the event of exceeding the entropic barrier we use the summability of the powers in Lemma \ref{lem:besselbridge}, to conclude
  \begin{align*}
   & \Q( \mathfrak{Z}_t \notin
    [\tilde{A}_t^{n,-},\tilde{A}_t^{n,+}]
    \text{ for some integer $t \leq H_n-k_4+1$}~|~\eqref{eq:aha0}, \mathfrak{Z}_{H_{k_2}}, \mathfrak{Z}_{H_{n}})\\
  &  \leq C\bigl( (\log k_2)^{-1/5} + k_4^{-1/5}\bigr).
  \end{align*}
  We may then fill in the gaps by again using union bound together with the oscillation estimate Lemma \ref{lem:bbridgeosc} over each interval to conclude for some larger constant $C>0$
  \begin{equation}\label{eq:aha2}
    \Q( \widehat{\mathscr{R}}(\theta) | \eqref{eq:aha0}, \mathfrak{Z}_{H_{k_2}}, \mathfrak{Z}_{H_{n}}) \geq e^{-C(\log k_2)^{-1/5} - Ck_4^{-1/5}}
  \end{equation}
  for all $k_2, k_4$ sufficiently large.
  We will show that we can reduce to this case by bounding the expression 
  \(
  \Q\bigl[
    \bigl(\tfrac{d\Pr}{d\Q}\bigr)^2
    ~\vert~
    \eqref{eq:aha0},
    \mathfrak{Z}_{H_{k_2}}, \mathfrak{Z}_{H_{n}}
  \bigr].
  \)
  Let $\alpha = \sqrt{2}(1-\tfrac{3}{4}\tfrac{\log\log n}{\log n}) - \frac{\mathfrak{Z}_{H_{n}}- \mathfrak{Z}_{H_{k_2}}} {H_n - H_{k_2}}=O(\tfrac{(\log k_2)^{0.51} + k_6} {\log n}),$ and which is positive for all $k_2$ and $n$ sufficiently large.
  Under $\Q$ and conditioned on $\mathfrak{Z}_{H_{k_2}}, \mathfrak{Z}_{H_{n}}$, the process
  \[
    \begin{aligned}
      \mathfrak{X}_t
      &= -\mathfrak{B}_t + g(t) + \alpha (t-H_{k_2})
      +\sqrt{2}(1-\tfrac{3}{4}\tfrac{\log\log n}{\log n})H_{k_2}
      -\mathfrak{Z}_{H_{k_2}}
      \\
      &= -\mathfrak{Z}_t + g(t) + \sqrt{2}(1-\tfrac{3}{4}\tfrac{\log\log n}{\log n})t
    \end{aligned}
  \]
  is a Brownian bridge
  and the event \eqref{eq:aha0} becomes
  \[
    \mathfrak{X}_t \geq 0
    \quad
    \text{for all }
    \quad
    t \in [H_{k_2},H_n].
  \]
  Thus, conditionally on \eqref{eq:aha0}, $(\mathfrak{X}_t : t \in [H_{k_2},H_n])$ has the law of a Bessel bridge (c.f.\ Lemma \ref{lem:besselbridge}).

  The conditional Radon--Nikodym derivative can be expressed in terms of $\mathfrak{X}$ by
  \begin{equation}\label{eq:ahb2}
    \frac{d\Pr}{d\Q}
    =
    \exp\left(
    \int_{H_{k_2}}^{H_{n}}
    \bigl\{-g''(t) \mathfrak{X}_t
    +g''(t)(\alpha(t- H_{k_2})+\sqrt{2}(1-\tfrac{3}{4}\tfrac{\log\log n}{\log n})H_{k_2}
    -\mathfrak{Z}_{H_{k_2}})
    -\tfrac{(g'(t))^2}{2}\bigr\} dt
    \right).
  \end{equation}
  We can estimate the function $g$ for a sufficiently large absolute constant $L>0,$
  \begin{equation}\label{eq:ahag}
    |g''(t)| \leq L(t^{-99/50} \wedge (H_n-t+\log k_5)^{-99/50}),
  \end{equation}
  and we also observe that $g$ is concave, recall \eqref{eq:aha0}.  
  From the 
  concavity of $g$ and the convexity of the exponential, we can therefore bound
  \begin{equation}\label{eq:ahb3}
    \biggl(\frac{d\Pr}{d\Q}\biggr)^2
    \leq
    \exp\left(
    \int_{H_{k_2}}^{H_{n}}
    -2 g''(t)\mathfrak{X}_t
    dt
    \right).
  \end{equation}
  From Lemma \ref{lem:besselbridge}, we have a subgaussian estimate on $\mathfrak{X}_t$.
  Recall the subgaussian norm
  \[
    \| X \|_{\psi_2} = \inf\{ t \geq 0 : \Exp[ e^{X^2/t^2}~|~\eqref{eq:aha0}, \mathfrak{Z}_{H_{k_2}}, \mathfrak{Z}_{H_{n}}] \leq 2 \},
  \]
  such that
  \[
    \| \mathfrak{X}_t \|_{\psi_2}
    \leq
    \begin{cases}
      (\log k_2)^{0.51} + \sqrt{t} & t \leq \tfrac12 H_n \\
      k_6 + \sqrt{H_n - t} & t \leq H_n - k_4 \\
    \end{cases}
  \]
  Hence  using Jensen's inequality,
  \[
    \biggl\|
    \int_{H_{k_2}}^{H_{n}}
    -2 g''(t)\mathfrak{X}_t
    dt
    \biggr\|_{\psi_2}
    \leq
    \int_{H_{k_2}}^{H_{n}}
    2g''(t)
    \|\mathfrak{X}_t\|_{\psi_2} dt.
  \]
  Splitting the integral and using the bound on $g''$ and $\|\mathfrak{X}_t\|_{\psi_2}$,
  we have
  \[
    \int_{H_{k_2}}^{H_{n}}
    2g''(t)
    \|\mathfrak{X}_t\|_{\psi_2} dt
    \leq C(k_6,L)\biggl( (\log k_2)^{-47/100} + k_4^{-48/100}\biggr).
  \]
  In summary to control of the Radon--Nikodym derivative we conclude for all $k_4$ and $k_2$ sufficiently large there is an absolute constant $C>0$ so that
  \begin{equation}\label{eq:ahb7}
    \Q\biggl(\biggl(\frac{d\Pr}{d\Q}\biggr)^2~\vert~\eqref{eq:aha0}, \mathfrak{Z}_{H_{k_2}}, \mathfrak{Z}_{H_{n}}
    \biggr)
    \leq
    e^{C k_4^{-24/25} + C(\log k_2)^{-47/50} }.
  \end{equation}

  Hence 
  using Cauchy-Schwarz, we get that for all $k_2$ and $k_4$ sufficiently large
  \begin{equation}\label{eq:ahb8}
    \Q\bigl[
      \tfrac{d\Pr}{d\Q}
      \one[{
        \widehat{\mathscr{R}}^c(\theta)
      }]
      ~\vert~
      \eqref{eq:aha0},
      \mathfrak{Z}_{H_{k_2}}, \mathfrak{Z}_{H_{n}}
    \bigr]
    \leq
    (\log k_2)^{-1/5} + k_4^{-1/5}.
  \end{equation}
  Finally, the $\Q$ probability of \eqref{eq:aha0}, conditionally on $\mathfrak{Z}_{H_{k_2}},\mathfrak{Z}_{H_{n}},$ is the probability that the Brownian bridge $\mathfrak{X}_t$ stays positive for all time.
  This probability is in fact explicit \cite[(3.40)]{KaratzasShreve}:
  \begin{equation}\label{eq:aha3}
    \begin{aligned}
      \Q( \eqref{eq:aha0} | \mathfrak{Z}_{H_{k_2}}, \mathfrak{Z}_{H_{n}})
      &=
      1-\exp\left(-\frac{2\mathfrak{X}_{H_{k_2}}\mathfrak{X}_{H_{n}}}{H_{n} - H_{k_2}}
      \right) \\
      &\leq C \frac{(\sqrt{2}\log k_2- \mathfrak{Z}_{H_{k_2}})\log k_5}{\log n },
    \end{aligned}
  \end{equation}
  for all $n,k_2$ sufficiently large and some $C>0$.
  For $n$ large, 
  this contribution is negligible in comparison to \eqref{eq:aha0a}.
  Hence by \eqref{eq:ahb0} and \eqref{eq:ahb8}, the proof is complete. 
\end{proof}

By a similar argument, we can actually estimate the density of the Bessel process endpoint killed when it exits an entropic envelope.

\begin{lemma}\label{lem:bananadensity}
  For notational convenience, let $s = H_{n_1^+}$.
  Fix $p \geq 2$.
  On the event that
  $\mathfrak{Z}_{H_{k_2}} \in \sqrt{2}(\log k_2 - [(\log k_2)^{0.49}, (\log k_2)^{0.51}])$
  and
  $y \in \sqrt{2}(s-\frac{3}{4}\log\log n -
  [A_{s}^{p-1,-}, A_{s}^{p-1,+}]),$
  \begin{equation}
    \label{eq-300324star}
    \begin{aligned}
    &\Pr(
    \mathscr{R}^p_{j}(n_1^+)
    \cap
    \{\sqrt{2}m_{n_1^+} - \mathfrak{Z}_{s} \in [y,y+dy]\}
    ~\vert~ \mathfrak{Z}_{H_{k_2}}
    ) \\
    &=\sqrt{\frac{2}{\pi}}\frac{y}{n_1^+}
    \bigl(\sqrt{2}H_{k_2}-\mathfrak{Z}_{H_{k_2}}\bigr)
    \exp\bigl(\sqrt{2}y
    +\sqrt{2}\mathfrak{Z}_{H_{k_2}} - \log k_2 + o_{k_2}\bigr)dy
  \end{aligned}.
\end{equation}
  Furthermore for all $y \in \sqrt{2}(s-\frac{3}{4}\log\log n-[A_{s}^{p,-}, A_{s}^{p,+}])$, the density is upper bounded by the right hand side.
  With $y \in \sqrt{2}(s-\frac{3}{4}\log\log n-[A_{s}^{p-1,-}, A_{s}^{p-1,+}])$,
  if we introduce $\mathscr{V}_j'$ (see \eqref{eq:RD}) as the alteration of \eqref{eq:rayevent}, where we only put the entropic envelope restriction at $t$ in the set $\{H_{2^k} : k \in \N\}$,
  we furthermore have
  \begin{equation}
    \label{eq-300324dstar}
    \begin{aligned}
    &\Pr(
    \bigl(\mathscr{R}^2_{j}(n_1^+)\bigr)^c
    \cap
    \mathscr{V}_j'
    \cap
    \{\sqrt{2}m_{n_1^+} - \mathfrak{Z}_{s} \in [y,y+dy]\}
    ~\vert~ \mathfrak{Z}_{H_{k_2}}
    ) \\
    &\leq
    o_{k_2}
    \times
    \sqrt{\frac{2}{\pi}}\frac{y}{n_1^+}
    \bigl(\sqrt{2}H_{k_2}-\mathfrak{Z}_{H_{k_2}}\bigr)
    \exp\bigl(\sqrt{2}y
    +\sqrt{2}\mathfrak{Z}_{H_{k_2}} - \log k_2\bigr)dy.
    \end{aligned}
  \end{equation}
\end{lemma}
In the lemma and also below, we write $P(dy)\leq fdy$ for a (sub)-probability measure $P$ when we mean that
the density of $P$ with respect to Lebesgue measure is bounded by $f$.
\begin{proof}
Define the event $\mathscr{E}$ that
  \begin{equation}\label{zrh:1}
    \mathfrak{X}_t \coloneqq \sqrt{2}\bigl(1-\tfrac{3}{4}\tfrac{\log s}{s}\bigr)t -  \mathfrak{Z}_t \geq 0
    \quad \text{for all}
    \quad t \in [H_{k_2},s],
  \end{equation}
  and note that $\mathscr{R}^p_{j}(n_1^+) \subset \mathscr{E}$.

Define the change of measure
\[
  \frac{d\Q}{d\Pr} =
  \exp\bigl(
  \sqrt{2}\bigl(1-\tfrac{3}{4}\tfrac{\log s}{s}\bigr)(\mathfrak{Z}_{s}-\mathfrak{Z}_{H_{k_2}})
  -\bigl(1-\tfrac{3}{4}\tfrac{\log s}{s}\bigr)^2(s-H_{k_2})
  \bigr).
\]
Under $\Q$, $\mathfrak{X}_t$ is a standard Brownian motion. Therefore,
 \begin{equation}\label{zrh:2}
  \begin{aligned}
    &\Pr(
    \mathscr{R}^p_{j}(n_1^+)
    \cap
    \{\sqrt{2}m_{n_1^+} - \mathfrak{Z}_{s} \in [y,y+dy]\}
    ~\vert~ \mathfrak{Z}_{H_{k_2}}
    ) \\
    &=
    \Q(
    \mathscr{R}^p_{j}(n_1^+)
    \cap \{\mathfrak{X}_s \in [y,y+dy]\}
    ~\vert~ \mathfrak{Z}_{H_{k_2}}
    ) \\
    &\times
    \exp\bigl(-s + \tfrac32 \log s+\sqrt{2}y + \sqrt{2}\mathfrak{Z}_{H_{k_2}} - H_{k_2} + o_n\bigr).
  \end{aligned}
\end{equation}

Define for convenience
 $t = s-H_{k_2}$
 and
 $x = \sqrt{2}\bigl(1-\tfrac{3}{4}\tfrac{\log s}{s}\bigr)H_{k_2}-\mathfrak{Z}_{H_{k_2}}.$
Under $\Q((\cdot)~\vert~\mathfrak{Z}_{H_{k_2}})$, $\mathfrak{X}_t$ has a gaussian density of mean $\mathfrak{X}_{H_{k_2}}$ and variance $t$.  However by the restriction on $y$, the gaussian trivializes, and so
\[
  \Q(
  \mathscr{R}^p_{j}(n_1^+)
  \cap \{\mathfrak{X}_s \in [y,y+dy]\}
  ~\vert~ \mathfrak{Z}_{H_{k_2}}
  )
  =
  \Q(
  \mathscr{R}^p_{j}(n_1^+)
  ~\vert~ \mathfrak{Z}_{H_{k_2}},
  \mathfrak{X}_s = y)
  \sqrt{\frac{1}{2\pi t}}
  (1+o_n).
\]
To produce the upper bound on the density in \eqref{eq-300324star}, 
we use the inclusion $\mathscr{R}^p_{j}(n_1^+)
\subset \mathscr{E}$ and then bound, using \eqref{eq:aha3},
\begin{equation}
  \label{eq-120924c}
  \Q(
  \mathscr{R}^p_{j}(n_1^+)
  ~\vert~ \mathfrak{Z}_{H_{k_2}},
  \mathfrak{X}_s = y)
  \leq
  \Q(
  \mathscr{E}
  ~\vert~ \mathfrak{Z}_{H_{k_2}},
  \mathfrak{X}_s = y)
  =1-\exp(-2xy/t).
\end{equation}
Hence we conclude, from combining everything with \eqref{zrh:2},
 \[
   \begin{aligned}
   &\Pr(
    \mathscr{R}^p_{j}(n_1^+)
    \cap
    \{\sqrt{2}m_{n_1^+} - \mathfrak{Z}_{s} \in [y,y+dy]\}
    ~\vert~ \mathfrak{Z}_{H_{k_2}}
    ) \\
    &\leq
    \sqrt{\frac{2}{\pi}}\frac{y x}{t^{3/2}}
    \exp\bigl(-s + \tfrac32 \log s+\sqrt{2}y
    -(x^2 + y^2)/2t
    +\sqrt{2}\mathfrak{Z}_{H_{k_2}} - H_{k_2} + o_n\bigr) dy \\
    &\leq
    \sqrt{\frac{2}{\pi}}\frac{yx}{n_1^+}
    \exp\bigl(\sqrt{2}y
    +\sqrt{2}\mathfrak{Z}_{H_{k_2}} - \log k_2 + o_{k_2}\bigr) dy.
  \end{aligned}
 \]
 For the lower bound in \eqref{eq-300324star}, we condition on $\mathscr{E}$ and express
\[
  \Q(
  \mathscr{R}^p_{j}(n_1^+)
  ~\vert~ \mathfrak{Z}_{H_{k_2}},
  \mathfrak{X}_s = y)
  =
  \Q(
  \mathscr{E}
  ~\vert~ \mathfrak{Z}_{H_{k_2}},
  \mathfrak{X}_s = y)
  \Q(
  \mathscr{R}^p_{j}(n_1^+)
  ~\vert~ \mathfrak{Z}_{H_{k_2}},
  \mathfrak{X}_s = y, \mathscr{E}).
\]
Thus it suffices to give a vanishing upper bound on
\begin{equation}
  \label{eq-120924a}
  \Q\bigl(
  \bigl(\mathscr{R}^p_{j}(n_1^+)\bigr)^c
  ~\vert~ \mathfrak{Z}_{H_{k_2}},
  \mathfrak{X}_s = y, \mathscr{E}\bigr) = o_{k_2}.
\end{equation}
Conditionally on $\mathscr{E}$, $\mathfrak{Z}_{H_{k_2}}$ and
 $ \mathfrak{X}_s = y$,
$\mathfrak{X}_t$ is a Bessel-3 bridge on $[H_{k_2},s]$.  
As this determines the Radon--Nikodym derivative, and using
the extra assumption from the statement of the lemma that $y$ is far from the 
edge of the 
entropic window, the argument
is now similar to what is shown in Lemma \ref{lem:allhailbanana}:
one starts by bounding the probability that the Bessel bridge escapes the entropic envelope at times in $H_{2^k}$ with $k \in \N$, and then between two integer times use a gaussian tail bound for the oscillation from Lemma \ref{lem:bbridgeosc}. We omit further details.

To see \eqref{eq-300324dstar}, note that
the same oscillation lemma shows that with $p=2$,
\begin{equation}
  \label{eq-120924b}
  \Q\bigl(
  \mathscr{E}^c
  ~\vert~ \mathfrak{Z}_{H_{k_2}},
  \mathfrak{X}_s = y,
  \mathscr{V}_j'
  \bigr) =
  o_{k_2},\;
  \Q\bigl(
  \mathscr{E}^c\cap\mathscr{V}_j'
  ~\vert~ \mathfrak{Z}_{H_{k_2}},
  \mathfrak{X}_s = y
  \bigr) =
  o_{k_2}\cdot\Q\bigl(\mathscr{E}~\vert~ \mathfrak{Z}_{H_{k_2}},
  \mathfrak{X}_s = y\bigr) 
\end{equation}
which implies the  claim \eqref{eq-300324dstar}
since
\[
  \begin{aligned}
    &\Q\bigl(
    \bigl(\mathscr{R}^{2}_{j}(n_1^+)\bigr)^c
    \cap  \mathscr{V}_j'
    ~\vert~ \mathfrak{Z}_{H_{k_2}},
    \mathfrak{X}_s = y
    \bigr) \\
    &= 
    \Q\bigl(
    \bigl(\mathscr{R}^{2}_{j}(n_1^+)\bigr)^c
    \cap   \mathscr{V}_j'
    \cap \mathscr{E}
    ~\vert~ \mathfrak{Z}_{H_{k_2}},
    \mathfrak{X}_s = y
    \bigr)
    +\Q\bigl(
    \bigl(\mathscr{R}^{2}_{j}(n_1^+)\bigr)^c
    \cap   \mathscr{V}_j'
    \cap \mathscr{E}^c
    ~\vert~ \mathfrak{Z}_{H_{k_2}},
    \mathfrak{X}_s = y
    \bigr)\\
    &\leq 
    \Q\bigl(\mathscr{E} ~\vert~ \mathfrak{Z}_{H_{k_2}},
    \mathfrak{X}_s = y
    \bigr)\cdot
    \Q\bigl(
    \bigl(\mathscr{R}^{2}_{j}(n_1^+)\bigr)^c
    ~\vert~ \mathfrak{Z}_{H_{k_2}},
    \mathfrak{X}_s = y,
     \mathscr{E}
    \bigr)
    +\Q\bigl(
     \mathscr{V}_j'
    \cap \mathscr{E}^c
    ~\vert~ \mathfrak{Z}_{H_{k_2}},
    \mathfrak{X}_s = y
    \bigr)\\
    &\leq o_{k_2} \Q\bigl(\mathscr{E} ~\vert~ \mathfrak{Z}_{H_{k_2}},
    \mathfrak{X}_s = y
    \bigr),
  \end{aligned}
\]
where we used \eqref{eq-120924a} and \eqref{eq-120924b} in the last inequality.
The conclusion follow from the explicit expression for the right hand side,
see e.g. \eqref{eq-120924c}.
\end{proof}


\subsection{The coarse oscillation bound}

In this section, we consider a single $j \in \mathcal{D}_{n/k_1}$, and hence we
write simply the barrier event
\begin{equation}\label{eq:oscR}
  \begin{aligned}
    {\mathscr{R}}
    &=
    \left\{
      \forall~ \log k_2  \leq t \leq \log \widehat{n}_1
      :
      \sqrt{\tfrac{8}{\beta}}A_{t}^{4,-}
      \leq
      \sqrt{\tfrac{4}{\beta}} \mathfrak{Z}_t(\theta_0)
      \leq \sqrt{\tfrac{8}{\beta}}A_{t}^{4,+}
    \right\}.
  \end{aligned}
\end{equation}
In this section, our goal is to estimate the probability of the oscillation event

\begin{equation}\label{eq:oscO}
  \begin{aligned}
    &\widehat{\mathscr{O}}
    \coloneqq
    \left\{
      \max_{k_2 \leq k \leq n_1^+}
    |\Psi_{k}(\theta^{-}) - \Psi_{k}(\theta^{+})| 
    \leq \tfrac{k_1(\log k_1)^{50}}{k_1^+}
    \right\},
  \end{aligned}
\end{equation}
with $\theta^{\pm} = \theta_0+\frac{\pm 2\pi k_1}{n},$ when we condition on $\mathscr{R}.$
Let $(\Gfilt_t : t \geq 0)$ be the join of the filtrations
(recalling \eqref{eq:gaussians})
\[
  \bigl( (\mathfrak{Z}^\C_s : 0 \leq s \leq t) :  t \geq 0 \bigr)
  \quad\text{and}\quad
  \bigl( ({X}_k,{Y_k},\Gamma_k^a: \forall~k~, H_k \leq t) : t \geq 0\bigr),
\]
where we further augment by $\mathfrak{Z}_{\log{\widehat{n}_1}}.$
Note that $G_k$ and $\Psi_k$ are adapted to $\Gfilt_{H_k}$ for any $k \in \N.$
Let $\tau$ be any $\Gfilt$ stopping time such that for all $\log k_2 \leq t \leq \tau,$  $\mathfrak{Z}_t$ satisfies
\begin{equation}\label{eq:continuousescape}
  \sqrt{\tfrac{4}{{\beta}}}
  \mathfrak{Z}_t
  \in
  \sqrt{\tfrac{8}{{\beta}}}
  [
    {A}_t^{4,-},
    {A}_t^{4,+}
  ].
\end{equation}
Then (see\ \eqref{eq:oscR}), if $\tau$ is just the first time \eqref{eq:continuousescape} fails, then on $\mathscr{R}$ it follows $\tau > \log \widehat{n}_1.$

We will let $\mathcal{B}$ be the event that $\mathfrak{Z}_t \leq \sqrt{2}t$ for all $t \geq H_{k_2}$ and $\mathfrak{Z}_{\log{\widehat{n}_1}} \in [
  \mathcal{A}_{\log \widehat{n_1}}^{4,-},
  \mathcal{A}_{\log \widehat{n}_1}^{4,+}
].$
Then conditionally on $\mathcal{B},$ we can use Lemma \ref{lem:besselbridge} to compute the behavior of increments of $\mathfrak{Z}.$
\begin{lemma}
  Set
  \(
  \Delta_{k+1}
  \coloneqq
  \mathfrak{Z}_{\tau \wedge H_{k+1}}
  -
  \mathfrak{Z}_{\tau \wedge H_k}
  .
  \)
  For $k_2 \leq k \leq \log n_1^+,$ there is a constant $C>0$ so that for all $k_2$ and $k_1$ sufficiently large,
  \[
    \begin{aligned}
      &\Exp[ \Delta_{k+1} \vert \Gfilt_{H_k}, \mathcal{B}]
      =
      (\sqrt{2} \pm c(k_1,k_2))\Exp[  (\tau \wedge H_{k+1} - \tau \wedge H_{k}) \vert \Gfilt_{H_k}, \mathcal{B}], \\
      &
        \left|\Exp[ \Delta_{k+1}^2 \vert \Gfilt_{H_k},\mathcal{B} ]
      -
      \Exp[  (\tau \wedge H_{k+1} - \tau \wedge H_{k}) \vert \Gfilt_{H_k},\mathcal{B} ]
        \right|
        \leq 
      C(\Exp[  (\tau \wedge H_{k+1} - \tau \wedge H_{k})^2 \vert \Gfilt_{H_k},\mathcal{B} ],
    \end{aligned}
  \]
  where the meaning of the first line is that the LHS is bounded above and below by the RHS, choosing signs appropriately, and where
  \[
    c(k_1,k_2) \leq C((\log k_2)^{-1/18} + (\log k_1)^{-1/100}).
  \]
  for some constant $C>0.$
  \label{lem:Gtau_k}
\end{lemma}
\begin{proof}
  We can express the increment $\Delta_{k+1},$ using Lemma \ref{lem:besselbridge}, as
  \begin{equation}\label{eq:Gtau0}
    \begin{aligned}
      \Delta_{k+1}
      &=
      \sqrt{2}(\tau \wedge H_{k+1} - \tau \wedge H_{k})
      +B_{\tau \wedge H_{k+1}} - B_{\tau \wedge H_{k}} +\Delta_{k+1}',\\
      &\mbox{\rm where\  }  \Delta_{k+1}'\eqqcolon\int_{\tau \wedge H_{k}}^{\tau \wedge H_{k+1}}
      \left( \frac{1}{\mathfrak{Z}_s-\sqrt{2} s} - \frac{\mathfrak{Z}_s-\mathfrak{Z}_{\log \widehat{n}_1 } - \sqrt{2}(s-\log \widehat{n}_1 )}{\log \widehat{n}_1-s} \right)ds
    \end{aligned}
  \end{equation}
  and $B_{(\cdot)}$ is a $\Gfilt$--adapted standard Brownian motion.
  The increment $\Delta_{k+1}'$ we can control using the barriers.  By construction, for $s \leq \tau$ for all $k_1$ and $n$ sufficiently large
  \[
    -\sqrt{2} (\log k_2)^{17/18}
    \leq \mathfrak{Z}_s - \sqrt{2}s
    \leq -\sqrt{2}(\log k_2)^{1/18}.
  \]
  For $\tau \geq s \geq \frac{1}{2} \log n,$ we can also bound
  \[
    -\sqrt{2} (\log n - s)^{17/18}
    \leq \mathfrak{Z}_s - \sqrt{2}(s-\tfrac 34 \log\log n)
    \leq -\sqrt{2}(\log n - s)^{1/18}.
  \]
  In particular, for $k \leq \log n_1^+$ and all $n$ sufficiently large
  \begin{equation}\label{eq:Gtau1}
    |\Delta_{k+1}'|
    \leq (\tau \wedge H_{k+1} - \tau \wedge H_{k})
    \biggl(  (\log k_2)^{-1/18} + 2(\log k_1)^{-1/100}\biggr).
  \end{equation}

  As $B_{\tau \wedge (\cdot)}$ is a martingale,
  \[
    \Exp[ \Delta_{k+1} \vert \Gfilt_{H_k},\mathcal{B}  ]
    =
    \sqrt{\tfrac{8}{\beta}}\Exp[  (\tau \wedge H_{k+1} - \tau \wedge H_{k}) \vert \Gfilt_{H_k},\mathcal{B}  ]
    +\Exp[ \Delta_{k+1}' \vert \Gfilt_{H_k},\mathcal{B}  ],
  \]
  and using the bound \eqref{eq:Gtau1}, the claim concerning the first moment follows.
  For the second moment, we use
  \begin{equation}\label{eq:Gtau2}
    \left|
    \sqrt{\Exp[ \Delta_{k+1}^2 \vert \Gfilt_{H_k},\mathcal{B}  ]}
    -
    \sqrt{\Exp[ (\Delta_{k+1}-\Delta_{k+1}')^2 \vert \Gfilt_{H_k},\mathcal{B}]}
    \right|
    \leq
    \sqrt{\Exp[ (\Delta_{k+1}')^2 \vert \Gfilt_{H_k},\mathcal{B}  ]}.
  \end{equation}
  Meanwhile, we have the exact formula
  \[
    \Exp[ (\Delta_{k+1}-\Delta_{k+1}')^2 \vert \Gfilt_{H_k},\mathcal{B}  ]
    =
    2\Exp[  (\tau \wedge H_{k+1} - \tau \wedge H_{k})^2 \vert \Gfilt_{H_k},\mathcal{B}  ]
    +
    \Exp[ \tau \wedge H_{k+1} - \tau \wedge H_{k} \vert \Gfilt_{H_k},\mathcal{B}  ].
  \]
  Hence using \eqref{eq:Gtau1} and \eqref{eq:Gtau2}, the bound for the second moment follows.
\end{proof}

\begin{proposition}\label{prop:osc1st}
  There is a deterministic constant
  $C_\beta$ so that for
  all $k_1$ sufficiently large, depending on $k_2,$
  and
  all $n$ sufficiently large, depending on $k_1$,
  it holds uniformly in $\theta_0$ that
  \[
    \Pr[ \widehat{ \mathscr{O}}^c
      \cap \mathscr{R}
    \cap \mathscr{G}_n ~\vert~ \Gfilt_{H_{k_2}},\mathcal{B}]
    \leq
    C_\beta (\log k_1)^{-50}
    \quad
    \As
  \]
\end{proposition}
As the conditional probability of $\mathcal{B}$ is explicit, we conclude:
\begin{corollary}\label{cor:osc}
  There is a constant $C_\beta$ so that for all $n,k_1$ sufficiently large 
  it holds uniformly in $\theta_0$ that
  \[
    \Pr[ \widehat{ \mathscr{O}}^c
      \cap \mathscr{R}
    \cap \mathscr{G}_n ~\vert~ \Gfilt_{H_{k_2}}]
    \leq
    C_\beta
    \frac{
      \left( \sqrt{2}\log k_2 - \sqrt{\tfrac{\beta}{4}} \varphi_{k_2}(\theta_0)+k_6 \right)_+
      \left( \sqrt{\tfrac{8}{\beta}}m_n
      - G_{\widehat{n}_1}(\theta_0)
      \right)_+
      }{(\log k_1)^{50}\log(\widehat{n}_1/k_2)}, \quad a.s.
  \]
\end{corollary}
\begin{proof}
  We multiply the result of Proposition \ref{prop:osc1st} by \(
  \Pr[  \mathcal{B} ~\vert~ \Gfilt_{H_{k_2}}],
  \)
  which is the probability that a Brownian bridge stays below a straight barrier (see \eqref{eq:aha3}).
\end{proof}

\begin{proof}[Proof of Proposition \ref{prop:osc1st}]
  With some abuse of notation, let
  $\psi_{k}(\theta-\theta_0) \coloneqq \Psi_k(\theta) - \Psi_k(\theta_0)$ for any $k\in \N$ for any $\theta \geq \theta_0.$ 
  We will show that for all $k_1$ and $n$ sufficiently large (depending on $k_2$) there is a constant $C_\beta$ so that
  \begin{equation}\label{eq:o1-2}
    \Pr[ \{
      \max_{k_2 \leq k \leq n_1^+}
      \psi_{k}(\theta_+-\theta_0)
      >
      \tfrac{k_1(\log k_1)^{50}}{4k_1^+ }
    \}
  \cap \mathscr{R} \cap \mathscr{G}_n ~\vert~ \Gfilt_{H_{k_2}},\mathcal{B}]
  \leq
  C_\beta (\log k_1)^{-50}
  \quad
  \As
\end{equation}
The analogous bound for
\(
|\psi_{n_1^+}(\theta_--\theta_0)|
=-\psi_{n_1^+}(\theta_--\theta_0)
\)
holds by the same argument after making appropriate sign changes.  We just show \eqref{eq:o1-2}. 
The proof will use tail estimates and computations
contained in Section \ref{sec-prufest}.

We will just write $\theta$ for $\theta_+ - \theta_0.$
Further, since our estimates will not depend on $\theta_0$, in the rest of the proof we take $\theta_0=0$.
Let $\tau$ be the first time $t \geq \log k_2$ that either \eqref{eq:continuousescape} fails, or for some $H_k \leq t$ where $k_2 \leq k$
either $\psi_{H_k}(\theta) \geq \pi$
or
\begin{equation}\label{eq:o1-1}
  \max_{t \in [H_k,H_{k+1}]}
  |\mathfrak{Z}_{t}^\C
  -\mathfrak{Z}_{H_{k}}^\C|^2
  >
  \frac{32 \log k}{k+1}
  \quad\text{or}
  \quad
  |\sqrt{\Gamma^a_{k}} - \beta_{k}|> 4\sqrt{\log H_k},
\end{equation}
recall \eqref{eq:betagamma}.
Define
\[
  Z_k^\tau \coloneqq
  -e^{-i\Psi_k(\theta)} \sqrt{\frac{k+1}{2}}
  \bigl(\mathfrak{Z}_{\tau \wedge H_{k+1}}^\C
  -\mathfrak{Z}_{\tau \wedge H_{k}}^\C\bigr)
  \quad \text{and $\gamma_k^\tau = \frac{Z_k^\tau}{\sqrt{ |Z_k^\tau|^2 + \Gamma_k^a}}.$}
\]
For any $\theta > 0,$ let $\psi_k^\tau(\theta)$ solve (c.f.\ \eqref{eq:prufer})
\[
  \psi_{k+1}^\tau(\theta)=\psi_{k}^\tau(\theta) +
  \begin{cases}
    \theta - 2\Im\left( \log(1-\gamma_k^\tau e^{i\psi_k^\tau(\theta)}) - \log(1-\gamma_k^\tau) \right) & \text{if } H_k < \tau, \\
    0 & \text{else},
  \end{cases}
\]
and $\psi_0^\tau(\theta) = \theta.$
Then on the event $\mathscr{R} \cap \mathscr{G}_n$ (compare with
\eqref{eq:good}) if $\{\max_{k_2 \leq k \leq n_1^+} \psi^\tau_k(\theta) < \pi\}$ we must actually have $\tau > H_{ n_1^+}.$  Thus for any $t > 0$
\begin{equation}\label{eq:o1-1a}
  \Pr[ \{
    \max_{k_2 \leq k \leq n_1^+}
    \psi_{k}(\theta)
    > t\}
    \cap \mathscr{R}
    \cap \mathscr{G}_n
  ~\vert~ \Gfilt_{H_{k_2}}, \mathcal{B}]
  \leq
  \Pr[
    \{\max_{k_2 \leq k \leq n_1^+} \psi^\tau_k(\theta) >t \}
    ~\vert~ \Gfilt_{H_{k_2}}, \mathcal{B}
  ]
  \quad
  \As
\end{equation}
On the event $\tau > H_{n_1^+},$ $\psi_{n_1^+}^\tau(\theta) = \psi_{n_1^+}(\theta).$
By construction, $\psi^\tau_k(\theta) \geq 0$ almost surely (see Lemma \ref{lem:psikmonotone}).

From Lemma \ref{lem:calculus2} we have that for $k_2$ sufficiently large,
the relative Pr\"ufer phases (recall \eqref{eq:prufer}) satisfy for some absolute constant $C$ and any $H_k < \tau$
\begin{equation}\label{eq:o10}
  \begin{aligned}
    \psi_{k+1}^\tau(\theta)
    -\psi_k^\tau(\theta)
    &\leq \theta
    +2\Im\biggl(\frac{ (e^{i\psi_k^\tau(\theta)}-1)Z_k^\tau}{\sqrt{\Gamma_k^a}} 
    +
    \frac{ (e^{2i\psi_k^\tau(\theta)}-1)(Z_k^\tau)^2}{2{\Gamma_k^a}}
    \biggr)
    +
    \frac{C\psi_k^\tau(\theta)|Z_k^\tau|^3}{{(\Gamma_k^a)^{3/2}}}. 
  \end{aligned}
\end{equation}

\noindent \emph{The case $\sigma=1$:}
We will estimate the conditional expectation of \eqref{eq:o10} given $\Gfilt_{H_k}$
using Lemma \ref{lem:Gtau_k}.
We note using Lemma \ref{lem:Gtau_k}, we have for $H_k < \tau$
\[
  \Exp[
    Z^\tau_k 
~\vert~\Gfilt_{H_k}, \mathcal{B}]
  = \frac{(-\sqrt{2} + O( (\log k_2)^{-1/18} + (\log k_1)^{-1/100}))}{\sqrt{k+1}},
\]
and moreover the imaginary part of the expectation is $0$.  In the same way
\[
  \Exp[
    (Z^\tau_k)^2
~\vert~\Gfilt_{H_k}, \mathcal{B}]
  = 
  \Exp[
    \Re \left((Z^\tau_k)^2\right)
~\vert~\Gfilt_{H_k}, \mathcal{B}]
  \leq 
  \frac{C}{k+1},
\]
Hence using that $\psi_k^\tau(\theta) \in [0, \pi)$ for some $C_\beta>0$, which implies that the $O(1/k)$ terms have a negative sign,
  \begin{equation}\label{eq:o12}
    \begin{aligned}
      \Exp[\psi_{k+1}^\tau(\theta)
        -\psi_k^\tau(\theta)
      ~\vert~\Gfilt_{H_k}, \mathcal{B}]
      \leq \theta
      +\frac{C_\beta\psi_k^\tau(\theta)(\log k)^3}{(k+1)^{3/2}}
      \quad
      \text{for}
      \quad k_2 \leq  k \leq n_1^+.
    \end{aligned}
  \end{equation}
  The remainder of the real case will be covered by the argument for the imaginary cases, but the argument for the real case is simpler and given below.
  If we define the increasing function $k \mapsto P_k$ by the recurrence
  \begin{equation}\label{eq:o13}
    P_{k+1} \coloneqq P_k + \theta
    +\frac{C_\beta(\log k)^3 P_k}{(k+1)^{3/2}},
    \quad
    \text{for all $k \geq 0$ and}
    \quad
    P_0 \coloneqq \theta,
  \end{equation}
  then $\psi_k^\tau(\theta) - P_k$ is a supermartingale started at $0.$
  Hence for all $t \geq 0,$
  \begin{equation}\label{eq:o14}
    \begin{aligned}
      \Pr[ \max_{k_2 \leq k \leq n_1^+} \psi_k^\tau(\theta) > t + P_{n_1^+}
        ~\vert~
      \Gfilt_{H_{k_2}}, \mathcal{B}]
      &\leq
      \Pr[ 
        \max_{k_2 \leq k \leq n_1^+} (\psi_k^\tau(\theta) - P_k) > t
        ~\vert~
      \Gfilt_{H_{k_2}}, \mathcal{B}] \\
      &\leq \frac{\psi_{k_2}^\tau(\theta) + P_{n_1^+}}{t},
    \end{aligned}
  \end{equation}
  with the final inequality following from the same argument 
  as Doob's inequality applied to the super-martingale $\psi_k^\tau-P_k$.
  The recurrence for $P_k$ \eqref{eq:o13} is easily solved, and it can be checked that
  \begin{equation}\label{eq:o15}
    P_k \leq  C_\beta (k+1)\theta
    \qquad
    \text{for any}
    \quad
    k \leq n_1^+.
  \end{equation}
  for some other sufficiently large $C_\beta.$  Hence as $\theta \leq Cn_1^{-1},$ we have $P_{n_1^+} \leq C_\beta e^{-(\log k_1)^{(29/30)}}.$  Thus \eqref{eq:o14} and \eqref{eq:o1-1a} imply \eqref{eq:o1-2}.

  \noindent \emph{The imaginary case $\sigma=  \pm  i$:}
  In this case, to estimate \eqref{eq:o10},
  we first use that for $H_k < \tau$
  We note using Lemma \ref{lem:Gtau_k}, we have for $H_k < \tau$
  \[
    \Exp[
      Z^\tau_k 
  ~\vert~\Gfilt_{H_k}, \mathcal{B}]
    = \pm i \frac{(\sqrt{2} + O( (\log k_2)^{-1/18} + (\log k_1)^{-1/100}))}{\sqrt{k+1}},
  \]
  and moreover the \emph{real} part of the expectation is $0$; the $\pm$ depends on the sign of $\sigma.$  In the same way
  \[
    \Exp[
      (Z^\tau_k)^2
  ~\vert~\Gfilt_{H_k}, \mathcal{B}]
    = 
    \Exp[
      \Re \left((Z^\tau_k)^2\right)
  ~\vert~\Gfilt_{H_k}, \mathcal{B}]
  \quad \text{and} \quad 
  \left|\Exp[
    \Re \left((Z^\tau_k)^2\right)
~\vert~\Gfilt_{H_k}, \mathcal{B}]\right|
    \leq 
    \frac{C}{k+1},
  \]We note that there will also be a sign change in the $(1-\cos)$ term.  For either of $\sigma = \pm i$ it will be necessary to consider the (less advantageous) case considered on account of needing to consider the case $\theta_{-}$ (in which $\theta < 0$).

  Applying Lemma \ref{lem:Gtau_k} and bounding the cosine,
  \begin{equation}\label{eq:o17}
    \begin{aligned}
      \Exp[
        \psi_{k+1}^\tau(\theta)
        -\psi_k^\tau(\theta)
      ~\vert~\Gfilt_{H_k}, \mathcal{B}]
      \leq \theta
      +\frac{C_\beta(\psi_k^\tau(\theta))^2}{(k+1)}
      +\frac{C_\beta\psi_k^\tau(\theta)(\log k)^3}{(k+1)^{3/2}}
      \quad
      \text{for}
      \quad k_2 \leq  k \leq n_1^+.
    \end{aligned}
  \end{equation}
  Let $\eta(x) = (\log \tfrac{1}{x})^{-100}e^{(\log \tfrac{1}{x})^{29/30}}$ for $x \in (0,1)$
  and define a stopping time $\vartheta$ as the first $k \geq k_2$ such that
  \[
    \psi_k^\tau(\theta) >
    \eta( \tfrac{k+1}{n_1})
    (k+1)\theta.
  \]
  Then the stopped process $\psi_{k\wedge \vartheta}^\tau$ satisfies
   for
      $k_2 \leq  k \leq n_1^+$,

  \begin{equation}\label{eq:o18}
      \Exp[
        \psi_{\vartheta \wedge (k+1)}^\tau(\theta)
        -\psi_{\vartheta \wedge k}^\tau(\theta)
      ~\vert~\Gfilt_{H_k}, \mathcal{B}]
      \leq \theta
      +C_\beta
      \eta^2( \tfrac{k+1}{n_1})
      (k+1) \theta^2 
      +\frac{C_\beta\psi_{\vartheta \wedge k}^\tau(\theta) (\log k)^{3}}{(k+1)^{3/2}}.
  \end{equation}
  If we define the increasing function $k \mapsto P_k$ by the recurrence
  \begin{equation}\label{eq:o19}
    P_{k+1} \coloneqq P_k + \theta
    +C_\beta
    \eta^2( \tfrac{k+1}{n_1})
    (k+1) \theta^2
    +
    \frac{C_\beta P_k (\log k)^{3}}{(k+1)^{3/2}},
    \quad
    \text{for all $k \geq 0$ and}
    \quad
    P_0 \coloneqq \theta,
  \end{equation}
  then $\psi_{k\wedge \vartheta}^\tau(\theta) - P_{k\wedge \vartheta}$ is a supermartingale started at $0.$
  The recurrence for $P_k$ is easily solved explicitly, and in particular there is a constant $C_\beta$ sufficiently large that for any $0 \leq k \leq n_1^+,$
  \begin{equation}\label{eq:pbnd}
    k \theta
    \leq
    P_{k}-P_0 \leq C_\beta \sum_{\ell=0}^{k-1}
    \bigl\{ \theta + \eta^2( \tfrac{\ell+1}{n_1}) (\ell+1) \theta^2 \bigr\}
  \end{equation}
  for all $k_1$ sufficiently large.
  If we take $c$ as the maximum of $x \eta^2(x)$ on $[0,1],$ then we can further bound
  \begin{equation}\label{eq:pbnd1}
    P_{k}-P_0 \leq (1+2c)C_\beta k\theta.
  \end{equation}
  Moreover, for all $t \geq 0$ and any $m \leq n_1^+,$ by the same argument as in Doob's inequality
  \begin{equation}\label{eq:o20}
    \begin{aligned}
      \Pr[ \max_{k_2 \leq k \leq m} \psi_{k \wedge \vartheta}^\tau(\theta) >  t + P_m
        ~\vert~
      \Gfilt_{H_{k_2}}, \mathcal{B}]
      &\leq
      \Pr[ \max_{k_2 \leq k \leq m} (\psi_{k \wedge \vartheta}^\tau(\theta) - P_{k\wedge \vartheta}) > t
        ~\vert~
      \Gfilt_{H_{k_2}}, \mathcal{B}]
      \\
      &\leq \frac{(\psi_{k_2}^\tau(\theta) - P_{k_2}) + P_{m}}{t}.
    \end{aligned}
  \end{equation}

  We can use this bound to control the probability that $\vartheta \in [2^{\ell-1},2^{\ell}].$  For this to happen for some $\ell \geq \log_2 k_2$ we must have for some $c_\beta$ sufficiently small
  \[
    \max_{k_2 \leq k \leq 2^\ell} \psi_{k \wedge \vartheta}^\tau(\theta)
    >P_{2^{\ell}}(1+c_\beta \eta(\tfrac{2^{\ell-1}}{n_1})).
  \]
  Hence summing over $\log_2 k_2 \leq \ell \leq \log_2 n_1^+,$ and using \eqref{eq:pbnd1} and \eqref{eq:o20}
  and increasing $C_\beta$ as needed between the inequalities,
  \[
    \Pr( \vartheta \leq n_1^+)
    \leq C_\beta\sum_{\ell=\log_2 k_2 }^{ \log_2(n_1^+)}
    \frac{\psi_{k_2}^\tau(\theta) + 2^\ell \theta}{2^\ell \theta\eta(\tfrac{2^{\ell-1}}{n_1}) }
    \leq C_\beta\biggl(
    \frac{\psi_{k_2}^\tau(\theta)}{k_2\theta \eta(\tfrac{k_2}{n_1})}
    +\frac{(\log \tfrac{n_1^+}{n_1})^{\tfrac{1}{30}} }{\eta(\tfrac{n_1^+}{n_1})}
    \biggr).
  \]
  On the event $\{\vartheta > n_1^+\} \cap \mathscr{G}_n,$ we thus have that $\tau > n_1^+$ and
  \[
    \psi_{n_1^+}(\theta)
    =\psi_{n_1^+}^\tau(\theta)
    \leq \eta (\tfrac{n_1^++1}{n_1})(n_1^++1)\theta.
  \]
  Hence if we apply \eqref{eq:o20} again with $t = (\log k_1)^{50}P_{n_1^+},$ we
  conclude that	
  for all $k_1$ sufficiently large
  \[
    \Pr[
      \max_{k_2 \leq k \leq n_1^+}
      \psi_{k }^\tau(\theta) >  (\log k_1)^{50}P_{n_1^+}
      ~\vert~
      \Gfilt_{H_{k_2}}, \mathcal{B}
    ]
    \leq
    C_\beta\biggl(
    \frac{\psi_{k_2}^\tau(\theta)}{k_2\theta \eta(\tfrac{k_2}{n_1})}
    +\frac{(\log \tfrac{n_1^+}{n_1})^{\tfrac{1}{30}}}{\eta(\tfrac{n_1^+}{n_1})}
    \biggr)
    +
    2\frac{\psi_{k_2}^\tau(\theta) + P_{n_1^+}}{(\log k_1)^{50}P_{n_1^+}}.
  \]

  To conclude the proof, we observe that $\theta \mapsto \psi_{k_2}^\tau(\theta)$ is a continuously differentiable function, and
  \[
    |\psi_{k_2}^\tau(\theta)|
    \leq
    \bigl\{\sup_{\theta_0 \in [0,2\pi]} |(\psi_{k_2}^\tau)'(\theta_0)|\bigr\}
    \theta.
  \]
  Hence on taking $n\to \infty$ the $\psi_{k_2}^\tau$ terms tend to $0$ almost surely, and we conclude
  \[
    \limsup_{n \to \infty}
    \Pr[
      \max_{k_2 \leq k \leq n_1^+}
      \psi_{k }^\tau(\theta) >  (\log k_1)^{50}P_{n_1^+}
      ~\vert~
      \Gfilt_{H_{k_2}}, \mathcal{B}
    ]
    \leq
    C_\beta\biggl(
    \frac{(\log \tfrac{n_1^+}{n_1})^{\tfrac{1}{30}}}{\eta(\tfrac{n_1^+}{n_1})}
    \biggr)
    +2(\log k_1)^{-50},
  \]
  which completes the proof by how $\eta$ was chosen.
\end{proof}

\subsection{The fine oscillation bound}

In this section, we develop an estimate of continuity for the real part of the field, where we consider a high value of the field $\mathfrak{U}^j_{T_+}(\theta')$ and then give a continuity estimate for $\mathfrak{U}^j_{T_+}(\theta)$ for $\theta \in [\theta',\theta'+\theta_0]$ for small $\theta_0.$ Without loss of generality we will take $\theta'=0.$

\begin{proposition}\label{prop:lsmcontinuity}
  ($\sigma=1$).
  We suppose that for $\alpha>4,\theta_0 > 0$ are given and satisfy
  \begin{equation}\label{eq:lsmtheta}
    (\log k_1)^{-\alpha/2} \leq \theta_0 \leq \varepsilon
  \end{equation}
  where the constant $\varepsilon$ is a small positive constant,
  to be determined,  that depends on $\beta > 0$.
  We will condition on
  $(\mathfrak{L}_{T_-}(\theta): \theta)$.
  We assume that the oscillation of the initial conditions are small in that
  \begin{equation}\label{eq:lsmalpha}
    \max_{ |\theta| \leq \theta_0}
    |
    \mathfrak{L}^j_{T_-}(\theta) -
    \mathfrak{L}^j_{T_-}(0)
    |
    \leq (\log k_1)^{-\alpha}.
  \end{equation}
    We consider large endpoints in the sense that
  \begin{equation}\label{eq7:lb}
    - k_6 \leq \mathfrak{U}^j_{T_+}(0) \leq 
    k_6.
  \end{equation}
  Then there is $\delta=\delta(\beta) >0$
  and a constant $C_\beta$ sufficiently large
  so that for any fixed set $S$ in $[0,\theta_0]$ of cardinality at most $e^{\theta_0^{-\delta}}$,
  the event
  \begin{equation*}
    \mathscr{O}_{*}
    \coloneqq
    \{
    \sup_{\theta \in S}
      |\mathfrak{U}^j_{T_+}(\theta) - \mathfrak{U}^j_{T_+}(0)|
      \leq \theta_0^{\delta}
    \quad\text{and}\quad
    \Im\bigl(\mathfrak{L}^j_{T_+}(\theta_0) - \mathfrak{L}^j_{T_+}(0)\bigr)
    \leq \theta_0^{\delta}
    \}
  \end{equation*}
  satisfies,
  on the event \eqref{eq:lsmalpha},
  \[
    \begin{aligned}
      &\Pr\bigl( \mathscr{O}_{*}^c \cap  \mathscr{P}_j'(0)
      \cap \eqref{eq7:lb} ~\vert~ (\mathfrak{L}^j_{T_-}(\theta): \theta)\bigr)
      \\
      &\qquad\leq C_\beta(k_6)
      \frac{
      \theta_0^{1+\delta}
	(\log k_5)\bigl(
    -\sqrt{\tfrac{8}{\beta}}(T_+-T_-)
	-
	\mathfrak{U}_{T_-}^j(0)
	\bigr)
	\exp\bigl(\sqrt{\tfrac{\beta}{2}}
	\mathfrak{U}_{T_-}^j(0)
	+2(T_+-T_-)
	\bigr)
      }{k_1^+(\log k_1)^{3/2}}.
    \end{aligned}
  \]
\end{proposition}
\begin{remark}\label{rem:continuity} An extension of the argument shows that in the case that $\alpha=\infty$ (which is to say the initial conditions are $0$) we may in fact bound the $\delta$--H\"older exponent on the interval by a constant with the same probability (up to constants).  This is by applying a chaining argument, and effectively applying this proposition repeatedly to control the process on intervals of size $\theta_0 {2^{-k}}$ for $k=0,1,2,\dots$.  Step $1$ would remain the same (albeit with $\mathscr{O}_*$ being the H\"older continuity event).  Steps 2-5 (control on the difference of the imaginary part) should be generalized from the difference of the imaginary part over the interval $[0,\theta_0]$ to the interval $[\theta_1,\theta_2]$, but essentially no details change.  Step $6$ would change similarly.  Finally, the chaining would be done.
\end{remark}
\begin{proof}
  We let $d\mathfrak{X}_t = \sqrt{\tfrac{4}{\beta}}\Re (e^{i \Im \mathfrak{L}^{j}_t(0)} d \mathfrak{W}_t^j)$ and let $d\mathfrak{B}_t = \sqrt{\tfrac{4}{\beta}}\Im (e^{i \Im \mathfrak{L}_t^{j}(0)} d \mathfrak{W}_t^j),$ for $t \in [T_-,T_+]$.  We also set the initial conditions $\mathfrak{B}_{T_-} =0$ and $\mathfrak{X}_{T_-}=\mathfrak{-U}_{T_-}^{j}(0)$. Hence the process
  $\mathfrak{X}$ equals $-\mathfrak{U}^j$.

  We let $\Gfilt$ be the filtration generated by $\mathfrak{X}$ and $\mathfrak{B}$ with $\mathfrak{X}_{T_+}$ and $(\mathfrak{L}^j_{T_-}{(\theta)} : \theta)$ adjoined.
  For a large integer absolute constant $k_*$ we set $T_* = T_+-k_*$ and we condition on the event that we remain below the concave barrier
  \begin{equation}\label{eq7:ray}
    \mathfrak{U}^j_{t}(0) \leq
    \sqrt{\tfrac{8}{\beta}}\bigl(t-T_{+}
    +2\bigl( {T_+ - t + (\log k_5)^{50}}\bigr)^{1/50}
    \bigr)
    \eqqcolon \ell(t)
    \quad
    \text{for all }
    t \in [T_-,T_*].
  \end{equation}
  This barrier is above the concave barrier in $\mathscr{P}_j'(0).$  

  As in the definitions and manipulations between \eqref{eq:aha1}--\eqref{eq:ahb3} from Lemma \ref{lem:allhailbanana}, we define a measure $\mathbb{Q}$ which is mutually absolutely continuous with respect to $\mathbb{P}$ and which flattens the curvature in \eqref{eq7:ray}.  For this measure, there is a $(\mathbb{Q}, \Gfilt)$--Brownian motion $(X_t:t \in [T_-,T_+])$
  such that
  \begin{equation}\label{eq:tlv3}
    d\mathfrak{X}_t
    =
     d X_t
    +
    \biggl(
    -\ell'(t)
    +
    \frac{\one[t \leq T_*]}{\ell(t) + \mathfrak{X}_t}
    + \frac{\mathfrak{X}_{T_+}-\mathfrak{X}_t + \ell(T_*) - \ell(t)}{T_+-t} \biggr) dt,
    \quad
    \mathfrak{X}_{T_-}
    =-\mathfrak{U}^{j}_{T_-}(0).
  \end{equation}
  In the above, we have extended the definition of $\ell$ for $t \geq T_*$ by taking it constant and equal to $\ell(T_*)$.
  Under $\mathbb{Q}$, $\mathfrak{B}$ remains a $(\mathbb{Q}, \Gfilt)$--Brownian motion.
  Moreover from \eqref{eq:ahb3}, we have a bound on the Radon--Nikodym derivative
  \begin{equation}\label{eq:tlv2}
    \frac{d \mathbb{P}}{d \mathbb{Q}}
    \leq
    \exp
    \biggl(
    \int_{T_-}^{T_*}
    \!\!\!\!\!\!
    -\ell''(t)(\ell(t) + \mathfrak{X}_t) dt
    \biggr).
  \end{equation}
  We also introduce $\mathfrak{D}_t = \ell(t \wedge T_*)+\mathfrak{X}_t$, which under $\mathbb{Q}$ is an $\Gfilt$--adapted Bessel bridge with endpoints
  \[
    \mathfrak{D}_{T_-} \leq  2\sqrt{\tfrac{8}{\beta}}\exp\bigl( \tfrac{9}{10}(\log k_1)^{29/30}\bigr)
    \quad
    \text{and}
    \quad
    \mathfrak{D}_{T_+} \leq \sqrt{\tfrac{8}{\beta}}\bigl( ((\log k_5)^{50} + k_*)^{1/50} + k_6\bigr).
  \]
  The process $\mathfrak{D}$ describes the distance of $\mathfrak{X}$ from
  the barrier.
  \paragraph{Step 1: reduction to a conditional probability.}

  We will show in this section that the problem can be reduced to
  showing
  \begin{equation}\label{eq:tlvreduce1}
    \Q\bigl( \mathscr{O}_{*}^c \cap  \mathscr{P}_j'(0) ~\vert~ \mathfrak{U}^j_{T_+}(0),(\mathfrak{L}^j_{T_-}(\theta): \theta), \eqref{eq7:ray}\bigr)
    \leq\theta_0^{1+\delta}
  \end{equation}
  for some
 $\delta> 0$ (note that $\delta$  also hides in the definition of $\mathscr{O}_{*}$) depending on $\beta$ and all $k_4,k_5$ bigger than some constant also
depending only $\beta$.
  Suppose we have established \eqref{eq:tlvreduce1}.
  Using the change of measure, we have
  \begin{align*}
   & \Pr\bigl( \mathscr{O}_{*}^c \cap  \mathscr{P}_j'(0) ~\vert~ \mathfrak{U}^j_{T_+}(0),(\mathfrak{L}^j_{T_-}(\theta): \theta), \eqref{eq7:ray}\bigr)\\
  &  =
    \Q\left[
    \frac{d \mathbb{P}}{d \mathbb{Q}}
    \one[
    \mathscr{O}_{*}^c \cap  \mathscr{P}_j'(0)]
  ~\middle\vert~ \mathfrak{U}^j_{T_+}(0),(\mathfrak{L}^j_{T_-}(\theta): \theta), \eqref{eq7:ray}\right].
  \end{align*}
  Applying H\"older's inequality, there is $\lambda > 1$ sufficiently large that
  \[
    \Pr\bigl( \mathscr{O}_{*}^c \cap  \mathscr{P}_j'(0) ~\vert~ \mathfrak{U}^j_{T_+}(0),(\mathfrak{L}^j_{T_-}(\theta): \theta), \eqref{eq7:ray}\bigr)
    \leq
    \Q\left(
    \exp
    \biggl(
    \int_{T_-}^{T_*}
    \!\!\!\!\!\!
    -\lambda \ell''(t)\mathfrak{D}_t dt
    \biggr)
    \right)^{1/\lambda}
    \theta_0^{1+\delta/2}.
  \]
  Controlling the Radon--Nikodym derivative is the same argument as the argument between \eqref{eq:ahb3} and \eqref{eq:ahb7}.
  We have for some absolute constant $L>0$
  \[
    -\ell''(t)
    \leq L(T_+ - t + (\log k_5)^{50})^{-99/50}
  \]
  and from Lemma \ref{lem:besselbridge} for some constant $C_\beta$
  \[
    \|\mathfrak{D}_t\|_{\psi_2} \leq C_\beta ((\log k_5) + \sqrt{T_+-t}),
  \]
  where $\|\cdot\|_{\psi_2}$ is the subgaussian norm with respect to the conditional probability measure given as $\mathbb{Q}( \cdot ~|~\mathfrak{U}^j_{T_+}(0),(\mathfrak{L}^j_{T_-}(\theta): \theta), \eqref{eq7:ray})$.
  Hence by convexity of the norm
  \[
    \biggl\|
    \int_{T_-}^{T_*}
    \!\!\!\!\!\!
    -\ell''(t)\mathfrak{D}_t dt
    \biggr\|_{\psi_2}
    \leq
    \int_{T_-}^{T_*}
    \!\!\!\!\!\!
    \bigl\|
    -\ell''(t)\mathfrak{D}_t dt
    \bigr\|_{\psi_2}
    \leq
    C_\beta (\log k_5)^{-24}.
  \]
  Thus we have after increasing $C_\beta$ for all $\lambda \geq 1$
  \[
     \Q\left(
    \exp
    \biggl(
    \int_{T_-}^{T_*}
    \!\!\!\!\!\!
    -\lambda \ell''(t)\mathfrak{D}_t dt
    \biggr)
    \right)^{1/\lambda}
    \leq
    \exp\bigl(
    \lambda C_\beta (\log k_5)^{-24}
    \bigr),
  \]
  which is bounded by $2$ for all $k_5$ small (depending on $\lambda$, which in turn depends only on $\beta$).

  Thus we have shown
  \begin{equation}\label{eq:tlvreduce2}
    \Pr\bigl( \mathscr{O}_{*}^c \cap  \mathscr{P}_j'(0) ~\vert~ \mathfrak{U}^j_{T_+}(0),(\mathfrak{L}^j_{T_-}(\theta): \theta), \eqref{eq7:ray}\bigr)
    \leq\theta_0^{1+\delta}
  \end{equation}
  for some constant $\delta > 0$ depending only on $\beta$ and all $k_4,k_5$ bigger than some constant depending only on $\beta$.
  The event \eqref{eq7:ray} contains $\mathscr{P}_j'(0)$, and so
  \[
     \Pr\bigl( \mathscr{O}_{*}^c \cap  \mathscr{P}_j'(0) ~\vert~ \mathfrak{U}^j_{T_+}(0),(\mathfrak{L}^j_{T_-}(\theta): \theta)\bigr)
    \leq\theta_0^{1+\delta}
    \Pr\bigl(
    \eqref{eq7:ray}
     ~\vert~ \mathfrak{U}^j_{T_+}(0),(\mathfrak{L}^j_{T_-}(\theta): \theta)
    \bigr).
  \]
  We then take expectation on both sides of the equation over $\mathfrak{U}^j_{T_+}(0)$ satisfying \eqref{eq7:lb}.
  Let $\mathbb{M}$ be the change of measure that flattens the linear part of $\ell$ (with respect to the conditional measure $\Pr( \cdot ~|~(\mathfrak{L}^j_{T_-}(\theta): \theta))$)
  \[
    \frac{d\mathbb{M}}{d\Pr}
    =
    \exp
    \biggl(
    \sqrt{\tfrac{\beta}{2}}
    (\mathfrak{U}_{T_+}^j(0)
    -
    \mathfrak{U}_{T_-}^j(0))
    -
    (T_+-T_-)
    \biggr),
  \]
  under which
  \[
    t\mapsto
    \mathfrak{U}_{t}^j(0)
    -
    \mathfrak{U}_{T_-}^j(0)
    -\sqrt{\tfrac{8}{\beta}}(t - T_-)
  \]
  is a speed--$(\tfrac{4}{\beta})$ Brownian motion.
  Note that on \eqref{eq7:lb} the change of measure $\frac{d\mathbb{M}}{d\Pr}$ is controlled up to constants by
  \[
    \frac{d\mathbb{M}}{d\Pr}
    \geq
    C(k_6)
    \exp
    \biggl(
    -\sqrt{\tfrac{\beta}{2}}
    \mathfrak{U}_{T_-}^j(0)
    -(T_+-T_-)
    \biggr).
  \]

  In particular, we have (using $T_+ - T_- = \log k_1^+$)
  \begin{align*}
 &   \Pr\bigl(
    \eqref{eq7:ray}
    \cap
    \eqref{eq7:lb} ~\vert~(\mathfrak{L}^j_{T_-}(\theta): \theta)\bigr)\\
&     \leq
     \frac{
     C(k_6)}
     {k_1^+
     }
     \exp\biggl(\sqrt{\tfrac{\beta}{2}}
   \mathfrak{U}_{T_-}^j(0)
   +2(T_+-T_-)
 \biggr)
     \mathbb{M}\bigl( \eqref{eq7:ray}
    \cap
    \eqref{eq7:lb} ~\vert~ (\mathfrak{L}^j_{T_-}(\theta): \theta)\bigr).
  \end{align*}
  Using barrier estimates as in \cite[Corollary A.6]{CMN}
  we can bound the $\mathbb{M}$--probability
  for values of $\mathfrak{U}^j_{T_-}(0)$ given by $\mathscr{P}_j'(0)$
  \[
    \mathbb{M}\bigl( \eqref{eq7:ray}
    \cap
    \eqref{eq7:lb} ~\vert~ (\mathfrak{L}^j_{T_-}(\theta): \theta)\bigr)
    \leq
    \frac{C(k_6)(\log k_5)
    \bigl(
      -\mathfrak{U}_{T_-}^j(0)
      -\sqrt{\tfrac{8}{\beta}}(T_+-T_-)
    \bigr)
    }{(\log k_1)^{3/2}}.
  \]
  This concludes the reduction of the lemma to \eqref{eq:tlvreduce1}.

  \paragraph{Step 2: Finding a good event on which the slope of the Bessel bridge is tame.}
   We introduce a random variable to control the slopes of the Bessel bridge:
  \[
    \mathfrak{u}
    \coloneqq
    \max_{k \geq k_* : T_+-k \geq T_-}
    k^{0.49}
    \times
    \biggl|
    \frac{
    \mathfrak{D}_{T_+} -
    \mathfrak{D}_{T_+-k}
    }
    {k}
    \biggr|.
  \]
  Conditioning on $\mathfrak{D}_{T_*}$, by the Gaussian tail bound for the Brownian bridge increment $\mathfrak{D}_{T_+} - \mathfrak{D}_{T_*}$ and the Gaussian tail bound for the Bessel bridge increment (c.f.\ Lemma \ref{lem:besselbridge}) $\mathfrak{D}_{T_*} - \mathfrak{D}_{T_+-k}$, there is a Gaussian tail bound for $\mathfrak{u}$ of the form
  \[
    \mathbb{Q}
    \left(
    \biggl|
    \frac{
    \mathfrak{D}_{T_+} -
    \mathfrak{D}_{T_+-k}
    }
    {k}
  \biggr|
   > C\sqrt{\tfrac{4}{\beta}}(1+x)
   ~\middle\vert~ \Gfilt_{T_-}
   \right) \leq Ce^{-k x^2},
  \]
  for some $C$ sufficiently large and all $k \geq k_*.$
  Thus, summing in $k$ we may assume that the event
  \begin{equation}\label{eq:lhslope}
    E_0
    \coloneqq
    \{
      \mathfrak{u}
      \leq
      \tfrac{\widehat{\delta}}{3}
      \bigl(
      \log(1/\theta_0)
      \bigr)^{0.51}
    \}
  \end{equation}
  occurs with $(\mathbb{Q}~|~\mathscr{H}_{T_-})$--conditional probability at least $1-C_{\beta} e^{-c_\beta\widehat{\delta}^2(\log(1/\theta_0))^{1.01}}$ for all $\widehat{\delta}, \theta_0> 0$ sufficiently small.

  \paragraph{Step 3: Finding a good event with small oscillations.}

  The oscillations of $\mathfrak{X}_t$ must be controlled, especially near the endpoint $t=T_+,$
  to control its contribution to the diffusions $d\mathfrak{X}_t
  +i
  d\mathfrak{B}_t.$
  These oscillations are ultimately controlled by those of the underlying Brownian motion ${X}_t$.  
  Define the event
  $E_{1}$, for $\widehat{\delta}$ to be determined as a function of $\beta$,
  \[
    \max \biggl\{
      \frac{|X_u-X_s|}{|u-s|^{1/4}}
      ,
      \frac{|\mathfrak{X}_{u}-\mathfrak{X}_s|}{|u-s|^{1/4}}
    \biggr\}
    \leq 
    \tfrac{\widehat{\delta}}{3}
    \bigl(
    \log(1/\theta_0)^{0.51}
    (T_+-u+1)^{0.49} \bigr),
    \forall~
    T_+ \geq u \geq s \geq T_-,
    |u-s| \leq 1.
  \]
  We may bound this probability by reducing the statement to a union bound over sets
  \[
    \operatorname{CP}_k \coloneqq \left\{ (u,s) : |T_+ - u - k| \leq 2, |T_+ - s - k| \leq 2  \right\}
  \]
  for integer $k.$

  First we will need a basic input: from oscillation theory of Brownian motion, there is an absolute constant $C>0$ so that for all $x>0,$
  \begin{equation}\label{eq:lh0}
    \mathbb{Q}\biggl(
    \sup_{ (u,s) \in \operatorname{CP}_k}
    \frac{|X_u-X_s|}{|u-s|^{1/4}}
    > \sqrt{\tfrac{4}{\beta}}C(1+x)
    ~\bigg\vert~
    \mathscr{H}_{T_-}
    \biggr)
    \leq e^{-x^2}.
  \end{equation}
This follows e.g. from the Borel--Tsirelson-Ibragimov-Sudakov inequality, together with \cite[Theorem I.2.1]{RevuzYor}.
  Taking a union bound over $k$,
  we conclude that with $(\mathbb{Q}~|~\mathscr{H}_{T_-})$--conditional probability $1-C_{\beta} e^{-c_\beta \widehat{\delta}^2 (\log(1/\theta_0))^2}$ for all $\widehat{\delta}, \theta_0> 0$ sufficiently small,
  \begin{equation}\label{eq:lh1}
      \frac{|X_u-X_s|}{|u-s|^{1/4}}
    \leq
    \tfrac{\widehat{\delta}}{3}
     \bigl(
    \log(1/\theta_0)^{0.51}
    (T_+-u+1)^{0.49} \bigr),
    \quad
    \forall~
    T_+ \geq u \geq s \geq T_-,
    |u-s| \leq 1.
  \end{equation}

  To control the oscillations of $\mathfrak{X},$ it suffices to control the oscillations of $\mathfrak{D}$ in its place, as for any $T_+ \geq u \geq s \geq T_-$ with $|u-s| \leq 1$
  \[
    |\mathfrak{X}_{u}-\mathfrak{X}_s|
    \leq
    |\mathfrak{D}_{u}-\mathfrak{D}_s|
    +
    \bigl(\max_{T_- \leq t \leq T_+} |\ell'(t)|\bigr) |u-s|,
  \]
  and $\ell'$ remains bounded.

  For the control on $\mathfrak{D}$ we consider separately the cases of $k \leq k_*-2$ and $k \geq k_*-2$.  In the former case, if we condition on the value of $\mathfrak{X}_{T_*}$ then the process $\bigl(\mathfrak{X}_{t} : t \in [T_*,T_+]\bigr)$ is a speed--$({\tfrac{4}{\beta}})$ Brownian bridge.  Its slope can be controlled by $\mathfrak{u}$, and after removing its slope, the same bound \eqref{eq:lh0} holds, with possibly different constants, and so
  taking a union bound over $k$, we control the probability of $E_1$ failing for these $k$.

  For $k \geq k_*-2$, integrating the differential representation for $\mathfrak{D}_t$, we have that for $(u,s) \in \operatorname{CP}_k$
  \[
    \begin{aligned}
      \mathfrak{D}_u
      -\mathfrak{D}_s
      =
      X_u-X_s
      +\int_s^u
      \biggl( \frac{\one[t \leq T_*]}{\mathfrak{D}_t} \biggr) dt
      +\int_s^u
      \biggl( \frac{\mathfrak{D}_{T_+}-\mathfrak{D}_t}{T_+-t} \biggr) dt.
    \end{aligned}
  \]
  On the event $\mathscr{P}_j'(\theta)$, we may bound $\mathfrak{D}_t \geq \sqrt{\tfrac{8}{\beta}}(\log k_5)$ (which is due to the concave barrier in $\mathscr{P}_j'(\theta)$ being shifted by a constant factor from the conditioning \eqref{eq7:ray}), for all $t \in (s,u)$ for which $t \leq T_*$.  To the second 
integral, we add and subtract $\mathfrak{D}_{T_+-k+2}$ in the following way
  \[
\int_s^u
      \biggl( \frac{\mathfrak{D}_{T_+}-\mathfrak{D}_t}{T_+-t} \biggr) dt
      =
\int_s^u
\biggl( \frac{\mathfrak{D}_{T_+}-\mathfrak{D}_{T_+-k+2}}{T_+-t} \biggr) dt
+
\int_s^u
      \biggl( \frac{\mathfrak{D}_{T_+-k+2}-\mathfrak{D}_t}{T_+-t} \biggr) dt.
  \]
  The first part we control in the same way as above.  As for the second, we may bound it above to create the following implicit bound
    \begin{align*}
    &  \sup_{ (u,s) \in \operatorname{CP}_k}
      \frac{|\mathfrak{D}_u
      -\mathfrak{D}_s|}{|u-s|^{1/4}}\\
  &   \quad \leq
      \sup_{ (u,s) \in \operatorname{CP}_k}
      \frac{|X_u-X_s|}{|u-s|^{1/4}}
      +
      (\mathfrak{u}+C_\beta)
      +
      \frac{C}{k_*-4}
      \biggl(
      \max_{\kappa=k_*-2, \dots, k_*+2}
      \sup_{ (u,s) \in \operatorname{CP}_\kappa}
      \frac{|\mathfrak{D}_u
      -\mathfrak{D}_s|}{|u-s|^{1/4}}
      \biggr).
    \end{align*}
  Here $C_\beta$ controls the Bessel generator term and $C$ is an absolute constant.  Hence, on taking maxima over both sides, we conclude for all $k_*$ larger than an absolute constant
  \[
          \sup_{ (u,s) \in \operatorname{CP}_k}
      \frac{|\mathfrak{D}_u
      -\mathfrak{D}_s|}{|u-s|^{1/4}}
      \leq 2
      \biggl(
      \max_{\kappa=k_*-2, \dots, k_*+2}
      \sup_{ (u,s) \in \operatorname{CP}_\kappa}
      \frac{|{X}_u
      -{X}_s|}{|u-s|^{1/4}}
      +\mathfrak{u}+C_\beta
      \biggr).
  \]
  Using \eqref{eq:lh1} and \eqref{eq:lhslope}, we conclude that there is some $C_\beta, c_\beta > 0$ such that for all sufficiently small $\theta_0$ and $\widehat{\delta}$,
  \begin{equation}\label{eq:lh2}
        \Q( E_1^c \cap E_0 \cap \mathscr{P}_j'(0) ~\vert~
	\Gfilt_{T_-})
    \leq
    C_\beta e^{-c_\beta \widehat{\delta}^2 (\log(1/\theta_0))^{1.01}}.
  \end{equation}

  \paragraph{Step 4: Finding a good event on which the imaginary part can not explode.}
  This will turn out to be the most probabilistically expensive part.
  Let $E_{2}$ be the event, for $C_\beta, \delta_\beta$ to be determined,
  \begin{equation}\label{eq:wzb}
    \mathfrak{X}_t
    -\mathfrak{X}_s
    \leq
    \left( 1+\frac{2}{\beta} - 3\widehat{\delta}\right)
    (t-s) -
    \delta_\beta
    \bigl(
    \log(\theta_0)
    \bigr)
    +
    \log_+(T_+ - t)
    ,
    \quad
    \forall\,
    T_+ \geq t \geq s \geq T_-.
  \end{equation}
  As we will work on the event $E_1$ it suffices to control the above on integer--valued points, hence it instead suffices to bound the probability
    \[
      \begin{aligned}
    \mathfrak{X}_t
    -\mathfrak{X}_s
    &\leq
    \left( 1+\frac{2}{\beta}
    - 3\widehat{\delta}\right)
    (t-s) -
    (\delta_\beta-2\widehat{\delta})
    \bigl(
    \log(\theta_0)
    \bigr)
    +
    (1-2\widehat{\delta})\log_+(T_+ - t), \\
    &\forall\
    T_+ \geq t \geq s \geq T_-
    \quad\text{for which}
    \quad t,s \in \Z.
  \end{aligned}
  \]
  We can furthermore formulate a sufficient bound in terms of $\mathfrak{D}$ using that
  $|\ell'(t) - \sqrt{\tfrac{8}{\beta}}| \leq \widehat{\delta}$ (which holds provided $k_5$ is chosen sufficiently large with respect to $\widehat{\delta}$), so that  \begin{equation}\label{eq:tlv4}
      \begin{aligned}
    \mathfrak{D}_t
    -\mathfrak{D}_s
    &\leq
    \left( 1
    +\frac{2}{\beta}
    +\sqrt{\frac{8}{\beta}}
    - 4\widehat{\delta}\right)
    (t-s) -
    (\delta_\beta-2\widehat{\delta})
    \bigl(
    \log(\theta_0)
    \bigr)
    +
    (1-2\widehat{\delta})\log_+(T_+ - t), \\
    &\forall\
    T_+ \geq t \geq s \geq T_-
    \quad\text{for which}
    \quad t,s \in \Z.
  \end{aligned}
\end{equation}

  If $s\geq T_*$, then as in the lead-up to \eqref{eq:lhslope}, conditioning on the value of $\mathfrak{D}_{T_*}$ the process
  $
   (\mathfrak{D}_t
   -\mathfrak{D}_{T_*} : t \in [T_*,T_+])
  $
  is a speed--$({\tfrac{4}{\beta}})$ Brownian bridge under $\Q$ with slope at most $\mathfrak{u}$.
  Hence we have a tail bound that for any $t \geq s$ with $s \geq T_*$
  \begin{equation}\label{eq:wza}
    \begin{aligned}
      &\Q\left(
      \bigl\{
	\mathfrak{D}_t
	-
	\mathfrak{D}_s
	>
	\delta
	+\tfrac{\widehat{\delta}}{3}(t-s)
	      \bigr\} \cap E_0
      ~\middle\vert~
      \Gfilt_{T_-}
      \right)
      \leq
      \exp\left(
      -
      \frac{\beta \widehat{\delta}^2}{8(t-s)}
      \right).
    \end{aligned}
  \end{equation}
  Using this, we have that the portion of \eqref{eq:tlv4}
  for which $t,s$ with $T_* \leq s \leq t \leq T_+$ can be controlled by $\widehat{\delta}\log(1/\theta_0)$, which is less than the bound stated in \eqref{eq:tlv4}, with
  probability
   $1-C_{\beta} e^{-c_\beta \widehat{\delta}^2 (\log(1/\theta_0))^2}$ for all $\widehat{\delta}, \theta_0> 0$ sufficiently small.
   Hence, we may reduce the problem to showing
    \begin{equation}\label{eq:tlv5}
      \begin{aligned}
    \mathfrak{D}_t
    -\mathfrak{D}_s
    &\leq
    \left( 1
    +\frac{2}{\beta}
    +\sqrt{\frac{8}{\beta}}
    - 4\widehat{\delta}\right)
    (t-s) -
    (\delta_\beta-3\widehat{\delta})
    \bigl(
    \log(\theta_0)
    \bigr)
    +
    (1-2\widehat{\delta})\log_+(T_+ - t), \\
    &\forall\
    T_* \geq t \geq s \geq T_-
    \quad\text{for which}
    \quad t,s \in \Z.
  \end{aligned}
\end{equation}

  For all $s \leq t \leq T_*$, we begin by recalling that the Bessel bridge SDE has strong solutions:
    \[
    d\mathfrak{D}_t
    =
     d X_t
    +
    \biggl(
    \frac{\one[t \leq T_*]}{\mathfrak{D}_t} + \frac{\mathfrak{D}_{T_+}-\mathfrak{D}_t}{T_+-t} \biggr) dt,
    \quad
    \mathfrak{D}_{T_-} + \ell(T_-)
    =-\mathfrak{U}^{j}_{T_-}(0).
  \]
  For all $k_5$ sufficiently large (with respect to $\widehat{\delta}$) on $\mathscr{P}_j'(0)$, we may bound above the Bessel generator term by $\widehat{\delta}$.
  The slope we bound on $E_0 \cap E_1 \cap \mathscr{P}_j'(0)$ by comparing $t$ to its (integer) ceiling
  \begin{equation}\label{eq:tlv6}
    \biggl|
    \frac{\mathfrak{D}_{T_+}-\mathfrak{D}_t}{T_+-t}
    \biggr|
    \leq
    \frac{2\widehat{\delta}}{3}
    \frac{ \log(1/\theta_0)^{0.51}}
    {(T_+-t)^{0.51}}.
  \end{equation}
  When $T_+-t \geq \log(1/\theta_0)$ we can just as well bound the above by $\frac{2\widehat{\delta}}{3}$.  We conclude that we have the bound
  on
  $E_0 \cap E_1 \cap \mathscr{P}_j'(0)$
  \begin{equation}\label{eq:tlv7}
    \mathfrak{D}_t
    -
    \mathfrak{D}_s
    \leq
    X_t - X_s
    + \frac{5\widehat{\delta}}{3}
    (t-s)
    + \frac{5\widehat{\delta}}{3} \log(1/\theta_0)
    \quad
    \text{for all}
    \quad
     T_- \leq s \leq t \leq T_*.
  \end{equation}

  Thus for any $t  \leq T_*$ and $w>0$ we have the bound, using the probability that Brownian motion hits a line,
  \[
    \begin{aligned}
      &
      \Q
      \left( \exists~s \in [T_-,t]:
       \mathfrak{D}_t
    	-
\mathfrak{D}_s
      >
      w
      + \biggl(1 + \frac{2}{\beta}+\sqrt{\frac{8}{\beta}} - 4\widehat{\delta}
      \biggr)(t-s)
      ~\bigg\vert~
      \Gfilt_{T_{-}}
      \right) \\
      &
      \leq
      \exp\left(
      -
      \frac{\beta}{2}
      \biggl(1 + \frac{2}{\beta}
      +\sqrt{\frac{8}{\beta}} - \frac{17\widehat{\delta}}{3}\biggr)
      \biggl(w -
      \frac{5\widehat{\delta}}{3}\log(1/\theta_0)\biggr)
      \right).
    \end{aligned}
  \]
  We pick a $\delta_\beta \in (0,1)$ and an $\eta_\beta$ so that
  \begin{equation}\label{eq:lsmetabeta}
   \frac{\beta}{2}\biggl(1 + \frac{2}{\beta} + \sqrt{\frac{8}{\beta}}\biggr)\delta_\beta
   > \eta_\beta
    >1.
  \end{equation}
  We apply the tail bound just above with $w =(\delta_\beta - 3\widehat{\delta})
    \bigl(
    \log(1/(\theta_0))
    \bigr)
    +
    \bigl(1 - 2\widehat{\delta}\bigr)\log_+(T_+ - t),$ from which it follows that for all $\widehat{\delta}$ sufficiently small (as a function of $\beta$),
 \[
    \begin{aligned}
      &
      \Q
      \left( \biggl\{\exists~s \in [T_-,t]:
      \text{ \eqref{eq:tlv5} fails }
    \biggr\}
    \cap E_0 \cap E_1 \cap \mathscr{P}_j'(0)
      ~\bigg\vert~
      \Gfilt_{T_{-}}
      \right)
      \leq
      (T_+ - t)^{-(1/2)\log(1/\theta_0)}
      \theta_0^{\eta_\beta + \widehat{\delta}}
      \end{aligned}
  \]
  Summing in $t$, it is seen that for $\widehat{\delta}$ sufficiently small and $\theta_0$ sufficiently small to absorb the constants,
  \begin{equation}\label{eq:lhE2}
    \Q( (E_2)^c \cap E_0 \cap E_1 \cap \mathscr{P}_j'(0) ~|~ \Gfilt_{T_-})
    \leq \theta_0^{\eta_\beta + \widehat{\delta}/2}.
  \end{equation}

  \paragraph{Step 5: A tail bound for the change in the imaginary part.}

  We estimate the imaginary part of
  \[
    \Delta_t
    =\Delta_t(\theta_0)
    = \Im(\mathfrak{L}_{t}(\theta_0)-\mathfrak{L}_t(0)).
  \]
  This satisfies the SDE
  \[
    d\Delta_t
    =
    \theta_0 e^t k_1^{-1} dt
    +
    \sqrt{\tfrac{4}{\beta}} \Im \biggl((e^{i \Delta_t}-1)e^{i \Im \mathfrak{L}_t(0)} d \mathfrak{W}_t^j\biggr),
  \]
  which has almost surely non-negative solutions.
  Then we express
  \[
    d\Delta_t
    =
    \theta_0 e^tk_1^{-1} dt
    + \Delta_t d\mathfrak{X}_t
    + \xi_1(\Delta_t) d\mathfrak{X}_t
    + \xi_2(\Delta_t) d\mathfrak{B}_t
    ,
  \]
  where $\xi_1(x)$ and $\xi_2(x)$ are bounded by $C_\beta x^2.$

  Let $M_t = \exp\left( \mathfrak{X}_t - \tfrac{2}{\beta}t \right).$
  Then we have from It\^o's Lemma
  \begin{equation}\label{eq:lsm0}
    d\biggl( \frac{
      \Delta_t
    }
    {M_t}
    \biggr)
    =
    \frac{1}{M_t}
    \biggl(
    \theta_0
    e^tk_1^{-1} dt
    + \xi_1(\Delta_t) (d\mathfrak{X}_t-\tfrac{4}{\beta}dt)
    + \xi_2(\Delta_t) d\mathfrak{B}_t\biggr).
  \end{equation}
  This we can integrate to conclude for $t \in [T_-,T_+],$
  \begin{equation}\label{eq:lsm1}
    \Delta_{t} =
    \frac{M_t}{M_{T_-}}
    \Delta_{T_-}
    +
    \int_{T_{-}}^{t}
    \frac{M_t}{M_s}
    \bigl(\theta_0 e^s k_1^{-1}ds
    + \xi_1(\Delta_s) (d\mathfrak{X}_s-\tfrac{4}{\beta} ds)
    + \xi_2(\Delta_s) d\mathfrak{B}_s\bigr).
  \end{equation}
  Now by definition we have
  \begin{equation}\label{eq:lsmMM}
    \frac{M_t}{M_s}
    =
    \exp\left(
    \mathfrak{X}_t
    -\mathfrak{X}_s
    -\frac{2}{\beta}(t-s)
    \right).
  \end{equation}
  Recalling \eqref{eq:wzb}, on the event $E_0 \cap E_1 \cap E_2 \cap \mathscr{P}_j'(0)$,
  \begin{equation}\label{eq:lsmMM3}
    \frac{M_t}{M_s}
    \leq
    \exp\left(
    \bigl( 1- 3\widehat{\delta}\bigr)
    (t-s) -
    \delta_\beta
    \bigl(
    \log(\theta_0)
    \bigr)
    +
    \log_+(T_+ - t)
    \right).
  \end{equation}

  Let
  \[
    L_t =
    \int_{T_-}^{t}
    \frac{
    M_t
    }{M_s}
    \theta_0 e^s k_1^{-1} ds,
    \quad
    \text{for any}
    \quad
    t\in [T_-,T_+].
  \]
  On the event
  $E_0 \cap E_1 \cap E_2 \cap \mathscr{P}_j'(0),$
  we have from \eqref{eq:lsmMM3} the bound 
  \begin{equation}\label{eq:lsmbound}
    |L_t| \leq
    \widehat{\delta}^{-1}(\theta_0)^{1-\delta_\beta}
    e^{-(T_+-t) +\log_+(T_+ - t)
    }
    \eqqcolon
    (\theta_0)^{1-\delta_\beta}\mathcal{W}(t),
  \end{equation}
  for $t \in [T_-,T_+].$

  Let $\tau$ be the first time $s$ greater than $T_-$ that
  \[
  |\Delta_s-L_s
  -
  M_{s}
  M_{T_-}^{-1}\Delta_{T_-}
  | \geq
  (\theta_0)^{1-\delta_\beta}\bigl( \mathcal{W}(t) + (\log k_1)^{-\alpha/2}\bigr),
  \]
  and note that from \eqref{eq:lsmbound}, for all $t \in [T_-,T_+],$
  and on the event
  $E_0 \cap E_1 \cap E_2 \cap \mathscr{P}_j'(0)$
  and the conditions on $\alpha, \theta_0$ (\eqref{eq:lsmalpha} and \eqref{eq:lsmtheta}),
  \begin{equation}\label{eq:lsm1a}
    \begin{aligned}
    \Delta_{t\wedge \tau}
    =
    |\Delta_{t\wedge \tau}|
    &\leq |\Delta_{t\wedge \tau}-L_{t\wedge \tau}
    - M_{t\wedge \tau}
    M_{T_-}^{-1}\Delta_{T_-}|
    +|M_{t\wedge \tau}
    M_{T_-}^{-1}\Delta_{T_-}|
    +|L_{t\wedge \tau}| \\
    &\leq 2
    (\theta_0)^{1-\delta_\beta}
    \bigl(
    \mathcal{W}(t \wedge \tau)
    + (\log k_1)^{-\alpha/2}
    \bigr).
    \end{aligned}
  \end{equation}

  Returning to the SDE for $\Delta,$ we now write
  \[
    d\Delta_t
    =
    \theta_0 e^tk_1^{-1} dt
    +\Delta_t dU_t
    \quad
    \text{where}
    \quad
    dU_t
    =
    d\mathfrak{X}_t
    + \tfrac{\xi_1(\Delta_t)}{\Delta_t} d\mathfrak{X}_t
    + \tfrac{\xi_2(\Delta_t)}{\Delta_t} d\mathfrak{B}_t
    ,
  \]
  Letting $N_t = \exp\bigl(U_t - \tfrac12 \langle U_t \rangle\bigr)$, we therefore have
  \begin{equation}\label{eq:lsm2}
    \Delta_t
    -L_t
    -
    \frac{M_{t}}
    {M_{T_-}}
    \Delta_{T_-}
    =
    \biggl(
    \frac{N_t}{N_{T_-}}
    -\frac{M_t}{M_{T_-}}
    \biggr)
    \Delta_{T_-}
    +
    \int_{T_{-}}^{t}
    \biggl(
    \frac{N_t}{N_s}
    -\frac{M_t}{M_s}
    \biggr)
    \bigl(\theta_0 e^s k_1^{-1}ds\bigr).
  \end{equation}

  This essentially reduces the problem to an estimate on the ratios of integrating factors that holds up to the stopping time.  First we observe that we have the representation
  \[
    \log\biggl(
    \frac{N_t}{N_{T_-}}
    \frac{M_{T_-}}{M_t}
    \biggr)
    =
    \int_{T_-}^t
    \bigl(\tfrac{\xi_1(\Delta_u)}{\Delta_u} d\mathfrak{X}_u
    + \tfrac{\xi_2(\Delta_u)}{\Delta_u} d\mathfrak{B}_u\bigr)
    -
    \frac2\beta
    \int_{T_-}^t
    \bigl(
    2\tfrac{\xi_1(\Delta_u)}{\Delta_u}
    +
    \bigl(\tfrac{\xi_1(\Delta_u)}{\Delta_u}\bigr)^2
    +
    \bigl(\tfrac{\xi_2(\Delta_u)}{\Delta_u}\bigr)^2
    \bigr)
    \,du.
  \]
  We bound the right hand side uniformly over $T_{-} \leq t \leq \tau$.  There are three types of terms to control: the finite variation terms in the second integral $(i)$, the martingale terms in the first integral $(ii)$, and the finite variation terms in the first integral $(iii)$.
  Then for the first terms, using the bound on $\Delta$ in \eqref{eq:lsm1a}:
  \[
    (i)
    \coloneqq
    \int_{T_-}^{T_+ \wedge \tau}
    \bigl|
    2\tfrac{\xi_1(\Delta_u)}{\Delta_u}
    +
    \bigl(\tfrac{\xi_1(\Delta_u)}{\Delta_u}\bigr)^2
    +
    \bigl(\tfrac{\xi_2(\Delta_u)}{\Delta_u}\bigr)^2
    \bigr|
    du
    \leq
    C(\beta)
    (\theta_0)^{1-\delta_\beta}.
  \]
  For the second terms, we need a stochastic control, and we have by bounding the quadratic variation for some constant
  \[
    \begin{aligned}
    &\Q
    \left(
    \biggl\{
    \max_{s \leq T_+ \wedge \tau}
    \biggl|
    \int_{T_-}^{s}
    \tfrac{\xi_1(\Delta_u)}{\Delta_u} d{X}_u
    + \tfrac{\xi_2(\Delta_u)}{\Delta_u} d\mathfrak{B}_u
    \biggr|
    > x
    \biggr\}
    \cap
    E_0 \cap E_1 \cap E_2 \cap \mathscr{P}_j'(0)
    ~\middle\vert~ \Gfilt_{T_-}
    \right) \\
    &\leq
    \exp\biggl(-\frac{x^2}
    {C(\beta, \widehat{\delta}) (\theta_0)^{4(1-\delta_\beta)}}
    \biggr).
  \end{aligned}
  \]
  In particular, we may assume with probability $1-e^{O( \theta_0^{-2(1-\delta_\beta)})},$ that
  \begin{equation}\label{eq:lsm3a}
    (ii)\coloneqq
    \max_{s \leq t \wedge \tau}
    \biggl|
    \int_{T_-}^{s}
    \tfrac{\xi_1(\Delta_u)}{\Delta_u} d{X}_u
    + \tfrac{\xi_2(\Delta_u)}{\Delta_u} d\mathfrak{B}_u
    \biggr|
    \leq \theta_0^{1-\delta_\beta}.
  \end{equation}
  Finally for the third terms,
  \[
    (iii)\coloneqq
    \int_{T_-}^{T_+ \wedge \tau}
    \Delta_s
    \biggl(
    \biggl| \frac{\one[s \leq T_*]}{\ell(s) + \mathfrak{X}_s} \biggr|
    +
    \biggl| \frac{\mathfrak{X}_{T_+}-\mathfrak{X}_s}{T_+ - s}\biggr|
    \biggr)
    ds.
  \]
   On the event $\mathscr{P}_j'(0),$
  \[
    \biggl| \frac{\one[s \leq T_*]}{\ell(s) + \mathfrak{X}_s} \biggr|
    \leq 2.
  \]
   On the event $E_0 \cap E_1 \cap \mathscr{P}_j'(0),$ using \eqref{eq:tlv6} and the control for $\ell'$, we have
  \begin{equation}\label{eq:lsmptwise}
    \biggl|
    \frac{\mathfrak{X}_{T_+}-\mathfrak{X}_s}{T_+ - s}
    \biggr|
    \leq
    \widehat{\delta} \log(1/\theta_0).
  \end{equation}
  Hence applying these bounds, we have for some $C(\beta,\widehat{\delta})$ sufficiently large
  \[
    (iii) \leq C(\beta,\widehat{\delta})\theta_0^{1-\delta_\beta} \log(1/\theta_0).
  \]
  Combining all of these, we conclude that for some $C(\beta,\widehat{\delta})$ sufficiently large
  \begin{equation}\label{eq:lsm3}
    \max_{T_- \leq s \leq t \leq T_+\wedge \tau}
    \biggl|
    \log\biggl(
    \frac{N_t}{N_{s}}
    \frac{M_{s}}{M_t}
    \biggr)
    \biggr|
    \leq C(\beta,\widehat{\delta})\theta_0^{1-\delta_\beta} \log(1/\theta_0).
  \end{equation}
 Hence we conclude from \eqref{eq:lsm2} that for $t \leq \tau$, for some $C(\beta,\widehat{\delta})$ and all $\theta_0$ sufficiently small (depending on $\beta, \widehat{\delta}$),
  \[
    |\Delta_{t}
    -L_t
    -
    \frac{M_{t}}
    {M_{T_-}}
    \Delta_{T_-}
    |
    \leq
    \biggl(
    \exp\bigl(
    C\theta_0^{1-\delta_\beta} \log(1/\theta_0)
    \bigr)-1\biggr)
    \cdot
    \biggl(
    \frac{M_t}{M_{T_-}}
    \Delta_{T_-}
    +
    \int_{T_{-}}^{t}
    \frac{M_t}{M_s}
    \bigl(\theta_0 e^s k_1^{-1}ds\bigr)
    \biggr).
  \]
  We conclude as in \eqref{eq:lsm1a} that for some $C=C(\beta,\widehat{\delta})$,
  \[
    |\Delta_{T_+\wedge \tau}
    -L_{T_+ \wedge \tau}
    -
    \frac{M_{T_+\wedge \tau}}
    {M_{T_-}}
    \Delta_{T_-}
    |
    \leq
    C\theta_0^{2(1-\delta_\beta)} \log(1/\theta_0)
    \bigl( \mathcal{W}(t) + (\log k_1)^{-\alpha/2}\bigr).
  \]
  By the definition of $\tau,$ we conclude that on the events considered, $E_0 \cap E_1 \cap E_2 \cap \mathscr{P}_j'(0)$, as well as on the event in which \eqref{eq:lsm3a} holds, that $\tau > T_+$, and hence \eqref{eq:lsm1a} holds for all $t \in [T_-,T_+].$
  In summary, we conclude there is an $\eta_\beta>1$ (see \eqref{eq:lsmetabeta}), a $\delta_\beta \in (0,1)$, a $C_\beta > 0$ and an $\epsilon >0$ so that for all $\theta_0$ sufficiently small,
  \begin{align}\label{eq:lsm5}
      &\Q\Big(
      \{
        \exists~t \in [T_-,T_+] ~:~\\
&\nonumber\qquad  |\Delta_t| > C_\beta
	(\theta_0)^{1-\delta_\beta}\bigl(e^{-(1-\widehat{\delta})(T_+-t)}
	+\log^{-\alpha/2} (k_1)
	\bigr)
      \}
      \cap E_0 \cap E_1 \cap E_2 \cap \mathscr{P}_j'(0)
      ~\vert~ \mathscr{H}_{T_-}
     \Big ) 
      \leq
      \theta_0^{\eta_\beta + \widehat{\delta}}.
  \end{align}

  \paragraph{Step 6: Control for the real part.}
  Going forward, we work on the event $F_{\theta_0}$ on which
  \[
    \forall~t \in [T_-,T_+] ~:~ |\Delta_t(\theta_0)| \leq C_\beta
    (\theta_0)^{1-\delta_\beta}
    \bigl(
    e^{-(1-\widehat{\delta})(T_+-t)}
    + (\log k_1)^{-\alpha/2}
    \bigr),
  \]
  the probability of which is estimated in  \eqref{eq:lsm5}.
  The real part of $\mathfrak{L}_t$ we will represent by
  \[
    \Delta_t^r(\theta)
    =\Re(
    \mathfrak{L}_t(\theta) - \mathfrak{L}_t(0)
    ).
  \]
  Using the SDE for $\mathfrak{L}$ in \eqref{eq:LU} 
  \[
    d\Delta_t^r
    =
    (\cos(\Delta_t(\theta))-1) d \mathfrak{X}_t
    -\sin(\Delta_t(\theta)) d\mathfrak{B}_t.
  \]
  Since $0 \leq \Delta_t(\theta) \leq \Delta_t(\theta_0)$
  we can estimate
  the finite variation parts on $F_{\theta_0} \cap E_1 \cap E_2 \cap \mathscr{P}_j'(0)$
  using a similar analysis as in \eqref{eq:lsm3}
  by
  \begin{equation}\label{eq:lsm8a}
    \int_{T_-}^{T_+ \wedge \tau}
    \Delta_s^2
    \biggl(
    \biggl| \frac{\one[s \leq T_*]}{\ell(s) + \mathfrak{X}_s} \biggr|
    +
    \biggl| \frac{\mathfrak{X}_{T_+}-\mathfrak{X}_s}{T_+ - s}\biggr|
    \biggr)
    ds
    \leq
    C(\beta,\widehat{\delta})
    \log(1/\theta_0)
    |\theta_0|^{2-2\delta_\beta}
    .
  \end{equation}
  The quadratic variation is dominated  on the event $F_{\theta_0}$ by, for some sufficiently large $C_\beta,$
  \[
    \langle \Delta_{T_+}^r(\theta) \rangle
    \leq
    C_\beta
    \int_{T_-}^{T_+}
    \Delta_t^2(\theta_0)
    dt
    \leq
    C_\beta^2 \bigl(\theta_0\bigr)^{2-2\delta_\beta}.
  \]
  Therefore, we have a tail bound that for some $C_\beta$ sufficiently large and for all $x > 0,$
  \[
    \Q\left(
    \{
      |\Delta_{T_+}^r(\theta_0)|
      > C_\beta(1+x)|\theta_0|^{1-\delta_\beta}
    \}
    \cap
    F_{\theta_0} \cap E_0 \cap E_1 \cap E_2 \cap \mathscr{P}_j'(0)
    ~\vert~ \mathscr{H}_{T_-}
    \right)
    \leq
    e^{-x^2}
  \]
  Thus taking $x = \theta_0^{-\varepsilon}$ for $\varepsilon < 1-\delta_\beta$
  we conclude that
  \[
    \Q\bigl( \mathscr{O}_{*}^c \cap
    F_{\theta_0} \cap E_0 \cap E_1 \cap E_2 \cap \mathscr{P}_j'(0) ~\vert~
    \Gfilt_{T_-}
    \bigr)
    \leq e^{-\theta_0^{-\delta}},
  \]
  which finally, using \eqref{eq:lhslope},\eqref{eq:lh2},\eqref{eq:lhE2} and \eqref{eq:lsm5}, concludes \eqref{eq:tlvreduce1}.
\end{proof}

\section{The diffusion approximation}\label{sec:diffusion}

In this section, we prove Proposition \ref{prop:uberdecoupling}.
The proof is given by a series of short lemmas.  A summary proof is given in the penultimate section.  In the final section, we develop a general tail bound which we use at multiple points in the development.

It is convenient in this section to work conditionally on the event $\mathscr{T}_{n_1^+}$ from \eqref{eq:good}.
We let, for the event $\mathscr{T}_{n_1^+}$ and $\sigma$--filtration $\filt_{n_1^+}$,
\[
  \Pr_M( \cdot ) \coloneqq \Pr ( \cdot ~|~ \mathscr{T}_{n_1^+}, \filt_{n_1^+}, (\Gamma_j^{a}: j > n_1^+)),
\]
and let $\Exp_M$ denote the associated expectation.
Under the law $\Pr_M,$  $(X_j,Y_j)$ (see \eqref{eq:gaussians}) are no longer independent of one another for $j \geq n_1^+$.  They do however still satisfy uniform Gaussian integrability, in that
\begin{equation}\label{eq:GI}
  \Exp_M\left[
    \exp(\lambda F(X_j,Y_j))
  \right]
  \leq
  \exp\left(
  \lambda \Exp_M \left[ F(X_j,Y_j)\right]
  +{\lambda^2 \|\nabla F(X_j,Y_j)\|^2_{\text{L}^{\infty}(\Pr_M)}}
  \right), \quad \text{ for all } \lambda \in \R,
\end{equation}
for all Lipschitz $F$ on the (convex) support of $(X_j,Y_j)$, see
\cite[Proposition 3.1]{BLe}.  We will use this subgaussian concentration for the family of functions that appear in the definitions of $\log \Phi^*_k(e^{i\theta})$, and in particular to control the linearization error.

\subsection{Locally linear processes}
\label{sec:llp}

In a similar fashion to \cite{CMN}, we introduce a process which in short windows of $k$ evolves linearly.
Let $\varkappa$ be a parameter, to be chosen later as a power of $n$, which will be the block length within which $\psi_k$ will evolve linearly.
Define a new recurrence, recalling $\beta_j = \sqrt{\frac{\beta}{2}(j+1)}$ for all $j \geq 0$
\begin{equation}\label{eqd:llp}
  \begin{aligned}
    &\lambda_{k+1}(\theta)=\lambda_{k}(\theta) + i\theta + 2\frac{Z_k e^{i\Im\lambda^*_k(\theta)}}{\beta_k}
  \quad \text{for all } k \geq {n_1^+}, \quad \text{where} \quad Z_k=X_k+iY_k,\\
  &\lambda_{{n_1^+}}(\theta) =
  -2\log \Phi^*_{{n_1^+}}(e^{i\theta})+i\theta n_1^+,
  \; \lambda^*_k(\theta) = \lambda_{n^*(k)}(\theta) +i (k-n^*(k))\theta,\\
  &
  \;  \text{and}\; n^*(k) = n_1^+ +\varkappa \lfloor \tfrac{k-{n_1^+}}{\varkappa}\rfloor
  \; \forall\, k \geq {n_1^+} .
  \end{aligned}
\end{equation}
As we will see below, the variables $\lambda_k(\theta)$
are good approximations for
$-2\log \Phi_k^*(\theta)+i\theta k$.

We begin with a simple observation that for $\varkappa$ sufficiently small with respect to ${n_1^+},$ the difference between $\lambda_k$ and $\lambda_k^*$ can be controlled.
\begin{lemma}
  For all $C>0$ there is constant $D >0$ so that for all ${n_1^+}$ sufficiently large (depending on $\beta$ and $M$),
  \[
    \Pr_M\left[
      \max_{{n_1^+} \leq k < n} |\lambda_k(\theta) - \lambda_k^*(\theta)|
      \geq D \sqrt{\tfrac{\varkappa \log n}{\beta {n_1^+}}}
    \right] \leq n^{-C}.
  \]
  \label{lemd:llstar}
\end{lemma}
\begin{proof}
  We show the proof for ${n_1^+} \leq k < {n_1^+} + \varkappa.$  For larger $k,$ we have the same estimate (and indeed it only improves).
  We have that
  \[
    \lambda_k(\theta) - \lambda_k^*(\theta) = \sum_{j = {n_1^+}}^{k-1} \frac{Z_j e^{i\Im\lambda^*_j(\theta)}}{\beta_j}.
  \]
  Under $\Pr$, this is Gaussian with variance $\tfrac{4}{\beta}\sum_{j={n_1^+}}^{k-1} \frac{1}{j+1} \leq \tfrac{4 \varkappa}{\beta {n_1^+}}.$  Hence there is an absolute constant so that for all $t \geq 0$ and all $\theta$
  \[
    \Pr_M\left[
      \max_{{n_1^+} \leq k < n_1^++\varkappa} |\lambda_k(\theta) - \lambda_k^*(\theta)|
      \geq t
    \right]
    \leq
    \frac{2}{\Pr(\mathscr{T}_{{n_1^+}})}
    \exp\left(
    -\beta (n_1^+ t)^2/(C\varkappa)
    \right).
  \]
  In particular for $t = D\sqrt{\tfrac{\varkappa \log n}{\beta {n_1^+}}}$ with $D$ sufficiently large, the claim follows.  
\end{proof}

We turn to comparing the differences
$R_k+i\Delta_k \coloneqq -2\log \Phi^*_{k} +ik\theta- \lambda_k$ (for real-valued $R$ and $\Delta$).
We begin by observing that the difference satisfies a recurrence for $k \geq {n_1^+}$:
\begin{equation}
  \begin{aligned}
    R_{k+1}(\theta)
    +i\Delta_{k+1}(\theta)
    &=
    R_{k}(\theta)
    +i\Delta_{k}(\theta)
    + 2\left( -\log(1-\gamma_ke^{i\Psi_k(\theta)}) - \left( \tfrac{Z_k e^{i\Im\lambda^*_k(\theta)}}{\beta_k} \right)\right) \\
    &=R_{k}(\theta)
    +i\Delta_{k}(\theta)
    + 2\left( -\log(1-\tilde{\gamma_k}e^{i\Delta_k(\theta)+i\Im(\lambda_k(\theta)-\lambda_k^*(\theta))})  - \left( \tfrac{\tilde{Z_k}}{\beta_k} \right)\right),
  \end{aligned}
  \label{eqd:deltak}
\end{equation}
where $\left\{ (\tilde{\gamma_k},\tilde{Z_k}) \right\}$ have the same law as $\left\{ (\gamma_k,Z_k) \right\}.$
The proof of the following lemma uses calculus computations
contained in Section
\ref{sec-prufest}.
\begin{lemma}
  Suppose $\varkappa \geq \sqrt{{n_1^+}}.$  For all $\delta > 0$ and all $C>0$.
\[
  \Pr_M\biggl[
    \max_{{n_1^+} \leq k \leq n} |R_k+i\Delta_k| \geq n^{\delta} \sqrt{\varkappa/{n_1^+}}
  \biggr]
  \leq n^{-C}
\]
for all $n$ sufficiently large.
  \label{lemd:delta}
\end{lemma}
\begin{proof}
  We show the bound for the imaginary part.
  We can express
  \[
    \Exp_M\left[ e^{\mu (\Delta_{k+1} - \Delta_k)} \middle \vert \filt_k \right]
    =
    \Exp_M\left[ e^{\mu F(X_j,Y_j)} \right],
  \]
  where in the notation of Lemma \ref{lem:calculus},
  \[
    F(x,y) = \Re\left\{ i\log(1-u(x+iy) e^{i\Delta_k(\theta)+i\Im(\lambda_k(\theta)-\lambda_k^*(\theta))})+i(x+iy)/\beta_k\right\}.
  \]
  Hence by \eqref{eq:GI} and Lemma \ref{lem:calculus}, there is a constant $C>0$ so that for any $k \geq {n_1^+},$ and any $\mu \in \R$
  \[
    \Exp_M\left[ e^{\mu \Delta_{k+1}} \middle \vert \filt_k \right]
    \leq
    \exp\left(\mu \Delta_k + \frac{C\mu^2| e^{i\Delta_k(\theta)+i\Im(\lambda_k(\theta)-\lambda_k^*(\theta))} - 1|^2}{\beta_k^2} + \frac{C\mu^2\log k}{\beta_k^3}\right).
  \]
  Let $T$ be the first $k \geq {n_1^+}$ such that $|\Im(\lambda_k(\theta) - \lambda_{k}^*(\theta))|$ is larger than $D\sqrt{\tfrac{\varkappa \log n}{\beta {n_1^+}}}$ for some $D > \sqrt{\beta}.$ 
Let $\Delta_k^T=\Delta_{k\wedge T}$. We have supposed that $\varkappa \geq \sqrt{{n_1^+}},$
  and therefore there is an absolute constant $C>0$ so that for any $k \geq {n_1^+}$ and any $\mu \in \R,$
  \[
    \Exp_M\left[ e^{\mu \Delta^T_{k+1}} \middle \vert \filt_k \right]
    \leq
    \exp\left(\mu \Delta^T_k + \frac{C\mu^2((\Delta^T_k)^2+ D^2\tfrac{\varkappa \log n}{\beta {n_1^+}})}{\beta k}\right).
  \]

  In preparation to use Proposition~\ref{prop:logladder}, we observe from \eqref{eqd:deltak} and Lemma \ref{lem:calculus} that there is an absolute constant $C>0$ so that for all ${n_1^+}$ sufficiently large,
  \[
    \Delta^T_{k+1}(\theta)-\Delta^T_k(\theta) \leq C\Delta^T_k\sqrt{ \frac{\log({n_1^+})}{\beta {n_1^+}}} + CD\sqrt{\frac{\varkappa \log n}{\beta {n_1^+}}}.
  \]
  Hence by Proposition~\ref{prop:logladder},  there is an absolute constant $C>0$ so that for all $x \geq C\log(n/{n_1^+})/\beta$ $  + C,$
  \[
    \Pr_M\left[ \max_{{n_1^+} \leq k \leq n} \Delta^T_k \geq x D\sqrt{\varkappa \log(n)/(\beta {n_1^+})} \right]
    \leq
    \exp\left(
    -\frac{1}{C}\frac{\log(x)^2}{\log(n/{n_1^+})}
    \right).
  \]
  The same bound holds for $-\Delta_k^T,$ and therefore by Lemma \ref{lemd:llstar}, the lemma follows.

  To control the real part $R_k$, we again use Lemma \ref{lem:calculus}, although there is no longer a need for the ladder.  We suppress the details.
\end{proof}

\subsection{Band--resampled approximation}
Recall from \eqref{eq:gaussians}, that $Z_\ell = \sqrt{E_\ell} e^{i\Theta_\ell}$ is a complex Gaussian for each $\ell.$  For any $r \in \N,$ define $k_r = {n_1^+} + r\varkappa.$  The family of Gaussians $\left\{ Z_\ell : k_{r-1} \leq \ell < k_r \right\}$ are i.i.d.  Hence, we can represent for any $r \in \N$, these Gaussians through their discrete Fourier transform:
\begin{equation}
  Z_{\ell+k_{r-1}} = \frac{1}{\sqrt{\varkappa}}\sum_{p=1}^{\varkappa} e(\tfrac{-p\ell}{\varkappa}) \hat{Z}_{p}^{(r)}, \quad \text{for } 0 \leq \ell \leq \varkappa-1
\end{equation}
where $e(x) = e^{2\pi i x}$ and $\left\{ \hat Z_p^{(r)} : r \in \N, 1 \leq p \leq \varkappa \right\}$ are another family of i.i.d.\ complex Gaussians.

We shall estimate the effect on the recurrence $\lambda_k(\theta)$ wherein we resample some of the $\{\hat Z_p\}$ corresponding to modes that are far from $\theta.$
Let $\omega \leq \varkappa$ be a positive integer parameter, which will be the bandwidth.
For a fixed $j \in {\mathcal{D}}_{n/k_1},$ define for any $r \in \N$
\begin{equation}\label{eqd:bandresampled}
  Z_{\ell+k_{r-1}}^{(j)} = \frac{1}{\sqrt{\varkappa}}\sum_{p=1}^{\varkappa} e(\tfrac{-p\ell}{\varkappa}) (\hat{Z}_{p}^{(r)} + (\check{Z}_{p}^{(r,j)}- \hat{Z}_{p}^{(r)}) \one[|e(\tfrac{p}{\varkappa})-e(\tfrac{\theta_j}{2\pi})|> \tfrac{\omega}{\varkappa}]), \quad \text{for } 0 \leq \ell \leq \varkappa-1,
\end{equation}
where $\left\{ \check Z_p^{(r,j)} : r \in \N, 1 \leq p \leq \varkappa,j \in {\mathcal{D}}_{n/k_1} \right\}$ is another family of standard complex Gaussians. 

We also define locally linear processes driven by these band--resampled Gaussians:
\begin{equation}\label{eqd:llpre}
  \begin{aligned}
    &\lambda_{k+1}^{(j)}(\theta)=\lambda_{k}^{(j)}(\theta)
    + i\theta
    + 2
    \frac{Z_k^{(j)} e^{i\Im \lambda^{*,(j)}_k(\theta)}}{\beta_k}
   \quad \text{for all } k \geq {n_1^+}, \quad \text{where}\\
   &\lambda_{{n_1^+}}^{(j)}(\theta) = \lambda_{n_1^+}(\theta),
   \quad \text{ and } \quad
  \lambda^{*,(j)}_k(\theta) =
  \lambda^{(j)}_{n^*(k)}(\theta)
  +i (k-n^*(k))\theta
   \quad \text{for all } k \geq {n_1^+} .
  \end{aligned}
\end{equation}
We then define $R^{(j)}_k(\theta) + i\Delta^{(j)}_k(\theta) \coloneqq \lambda_k(\theta) - \lambda_k^{(j)}(\theta).$

\begin{lemma}\label{lemd:bandvariance}
  Suppose that $j \in \mathcal{D}_{n/k_1}$ and that $\theta \in \R$ satisfies
  \[
    |
    e(\tfrac{\theta_j}{2\pi})
    -e(\tfrac{\theta}{2\pi})|
    \leq \frac{\omega}{2\varkappa}.
  \]
There is an absolute constant $C>0$ so that for any $r \in \N$ and all $\mu \in \R$
\[
    \Exp_M\left[ e^{\mu \Delta_{k_r}^{(j)}} \middle \vert \filt_{k_{r-1}}  \right]
    \leq
    \exp\left(
    \mu \Delta_{k_{r-1}}^{(j)}
    +\frac{\mu^2 C\varkappa |\Delta^{(j)}_{k_{r-1}}|^2}{\beta k_{r-1}}
    +
    \frac{\mu^2 C\varkappa }{ \beta \omega k_{r-1}}
    \right),
\]
and so that for all $t \geq 0$
\[
  \Pr_M\left[
    \max_{k_{r-1} \leq k < k_r}
    |
    \Delta_k^{(j)}-\Delta_{k_{r-1}}^{(j)}
    | \geq t
     \middle \vert \filt_{k_{r-1}}
  \right]
  \leq 2\exp\left(
  -\frac{t^2k_{r-1}\omega \beta}{C\varkappa(\omega |\Delta_{k_{r-1}}^{(j)}|^2 + 1) }
  \right).
\]
\end{lemma}
\begin{proof}
  For any $r \in \N,$ we have that conditionally on $\filt_{k_{r-1}},$ $\left\{ \Delta_k^{(j)} : k_{r-1} \leq k < k_r \right\}$ are jointly Gaussian under $\Pr.$  Moreover, for such $k,$ we can write
  \[
    \Delta_k^{(j)}
    -\Delta_{k_{r-1}}^{(j)}
    =
    2\Im
    \sum_{\ell=0}^{k-k_{r-1}}
    \left\{
      \frac{Z_{\ell+k_{r-1}}e^{i\ell\theta + i\lambda_{k_{r-1}}}}{\beta_{\ell+k_{r-1}}}
      -\frac{Z_{\ell+k_{r-1}}^{(j)}e^{i\ell\theta + i\lambda^{(j)}_{k_{r-1}}}}{\beta_{\ell+k_{r-1}}}
    \right\}.
  \]
  We give an upper bound for the variance of $\Delta_k^{(j)}-\Delta_{k_{r-1}}^{(j)},$ which we do by separately bounding the variance of
  \[
    A=  2\Im\sum_{\ell=0}^{k-k_{r-1}}
    \left\{
      \frac{Z_{\ell+k_{r-1}}e^{i\ell\theta + i\lambda_{k_{r-1}}}}{\beta_{\ell+k_{r-1}}}
      -\frac{Z_{\ell+k_{r-1}}^{(j)}e^{i\ell\theta + i\lambda_{k_{r-1}}}}{\beta_{\ell+k_{r-1}}}
    \right\},
  \]
  and of
  \[
    B=
     2\Im\sum_{\ell=0}^{k-k_{r-1}}
    \left\{
      \frac{Z_{\ell+k_{r-1}}^{(j)}e^{i\ell\theta + i\lambda_{k_{r-1}}}}{\beta_{\ell+k_{r-1}}}
      -\frac{Z_{\ell+k_{r-1}}^{(j)}e^{i\ell\theta + i\lambda^{(j)}_{k_{r-1}}}}{\beta_{\ell+k_{r-1}}}
    \right\}.
  \]
  For the second one, we observe that
  \begin{equation}
    \Var(B \vert \filt_{k_{r-1}} ) \leq \frac{8\varkappa |\Delta^{(j)}_{k_{r-1}}|^2}{\beta \cdot {k_{r-1}}}.
    \label{eqd:varb}
  \end{equation}
  The main work is to control the variance of $A.$ 
  By rotation invariance, we may drop the $e^{i\lambda_{k_{r-1}}}$ from both terms and write
  \[
    \begin{aligned}
    A&\lawequals 2\Im \sum_{\ell=0}^{k-k_{r-1}}
    \frac{e^{i\ell\theta}}{\sqrt{\varkappa}\beta_{\ell+k_{r-1}}}
    \sum_{p=1}^\varkappa
    e(\tfrac{-p\ell}{\varkappa})(\check{Z}_{p}^{(r,j)}- \hat{Z}_{p}^{(r)}) \one[|e(\tfrac{p}{\varkappa})-e(\tfrac{\theta_j}{2\pi})|> \tfrac{\omega}{\varkappa}]) \\
    &= 2\Im
    \sum_{p=1}^\varkappa
    c_p (\check{Z}_{p}^{(r,j)}-\hat{Z}_{p}^{(r)}) \one[|e(\tfrac{p}{\varkappa})-e(\tfrac{\theta_j}{2\pi})|> \tfrac{\omega}{\varkappa}]),
  \end{aligned}
  \]
  where we have set
  \[
    c_p = \sum_{\ell=0}^{k-k_{r-1}}
    \frac{ e(\ell(\tfrac{\theta}{2\pi} - \tfrac{p}{\varkappa}))}{\sqrt{\varkappa}\beta_{\ell+k_{r-1}}}.
  \]
  Moreover,
  \begin{equation}\label{eqd:varA0}
    \Var(A \vert \filt_{k_{r-1}})
    =4
    \sum_{p=1}^\varkappa
    |c_p|^2 \one[|e(\tfrac{p}{\varkappa})-e(\tfrac{\theta_j}{2\pi})|> \tfrac{\omega}{\varkappa}],
  \end{equation}
  and so it remains to estimate $|c_p|^2.$
  Note that we have the simple bound
  \[
    \left|  {\textstyle \sum_{\ell=0}^{k-k_{r-1}} e(\ell(\tfrac{\theta}{2\pi} - \tfrac{p}{\varkappa})) }\right|
    \leq
    \frac{2}{|e(\tfrac{\theta}{2\pi}) - e(\tfrac{p}{\varkappa})|}.
  \]
  We have that under the assumptions on $\theta$, when $|e(\tfrac{p}{\varkappa})-e(\tfrac{\theta_j}{2\pi})| > \tfrac{\omega}{\varkappa}$ then $|e(\tfrac{p}{\varkappa})-e(\tfrac{\theta}{2\pi})|> \tfrac12|e(\tfrac{p}{\varkappa})-e(\tfrac{\theta_j}{2\pi})|.$ Hence using summation-by-parts, there is an absolute constant $C>0$ so that
  \[
    |c_p|
    \leq
    \frac{C}{|e(\tfrac{\theta_j}{2\pi}) - e(\tfrac{p}{\varkappa})|\sqrt{\varkappa} \beta_{k_{r-1}}}.
  \]
  Thus turning to \eqref{eqd:varA0}, we can bound for some absolute constant $C>0$
  \[
    \Var(A \vert \filt_{k_{r-1}})
    \leq
    \frac{C}{\beta_{k_{r-1}}^2}
    +
    \sum_{p=\omega}^\infty
    \frac{C\varkappa}{p^2\beta_{k_{r-1}}^2}.
  \]
  Hence, we conclude from this equation and \eqref{eqd:varb} that there is an absolute constant $C >0$ so that
  \[
    \Var(\Delta_k^{(j)}-\Delta_{k_{r-1}}^{(j)} \vert \filt_{k_{r-1}})
    \leq
    \frac{C\varkappa |\Delta^{(j)}_{k_{r-1}}|^2}{\beta k_{r-1}}
    +
    \frac{C\varkappa }{ \beta \omega k_{r-1}}.
  \]
  As we condition on an event of probability at least $1/2$ for $M$ sufficiently large, it follows that $\Delta_k^{(j)}$ remains subgaussian conditionally on $\filt_{k_{r-1}},$ with subgaussian constant only an absolute constant more than the unconditioned standard deviation of $\Delta_{k}^{(j)}$ given $\filt_{k_{r-1}}.$  As $\Delta_k^{(j)}$ remains centered under $\Pr_M,$ using standard manipulations (see for example \cite[Proposition 2.5.2]{Vershynin}) it follows that there is another absolute constant $C >0$ so that for all $\mu \in \R$
  \[
    \Exp_M\left[ e^{\mu \Delta_{k_r}^{(j)}} \middle \vert \filt_{k_{r-1}}  \right]
    \leq
    \exp\left(
    \mu \Delta_{k_{r-1}}^{(j)}
    +\frac{\mu^2 C\varkappa |\Delta^{(j)}_{k_{r-1}}|^2}{\beta k_{r-1}}
    +
    \frac{\mu^2 C\varkappa }{ \beta \omega k_{r-1}}
    \right).
  \]
  Likewise, the desired concentration inequality follows for the maximum.
\end{proof}

\begin{lemma}
    Suppose that $j \in \mathcal{D}_{n/k_1}$ and that $\theta \in \R$ satisfies
  \[
    |
    e(\tfrac{\theta_j}{2\pi})
    -e(\tfrac{\theta}{2\pi})|
    \leq \frac{\omega}{2\varkappa}.
  \]
  For all $\delta$ sufficiently small and all $C>0$, if
  $n_1^+/\varkappa \geq n^{\delta}$ and $\omega \geq n^{\delta}$
  then for all $n$ sufficiently large
  \[
    \Pr_M\left[
      \max_{{n_1^+} \leq k \leq n} |R_{k}^{(j)}(\theta)+i\Delta_{k}^{(j)}(\theta)|
      \geq n^{\delta}/\sqrt{\omega}
    \right] \leq n^{-C}.
  \]
  \label{lemd:linearization}
\end{lemma}
\begin{proof}

  Using Lemma \ref{lemd:bandvariance}, for any $t \geq 0$
  \[
  \Pr_M\left[
    \max_{k_{r-1} \leq k < k_r}
    |
    \Delta_k^{(j)}-\Delta_{k_{r-1}}^{(j)}
    | \geq  t\bigl(\sqrt{\omega} |\Delta_{k_{r-1}}^{(j)}| +1\bigr)
     \middle \vert \filt_{k_{r-1}}
  \right]
  \leq 2\exp\left(
  -\frac{t^2k_{r-1}\omega \beta}{C\varkappa}
  \right).
  \]
  Hence taking $t=1/\sqrt{\omega}$, if we let $\mathcal{E}$ be the event that
  \[
    |\Delta_k^{(j)}-\Delta_{k_{r-1}}^{(j)}
    | \leq  |\Delta_{k_{r-1}}^{(j)}| + (1/\sqrt{\omega}),
    \quad
    \text{for all $r \geq 1$ such that $k_{r-1} \leq n$},
  \]
  then this event holds with probability $1-e^{-\Omega(n^{\delta})}$.
  We turn to controlling $\Delta_{k_{r}}^{(j)}$ on the event $\mathcal{E}$ using Proposition \ref{prop:logladder}.  From Lemma \ref{lemd:bandvariance}, $A_{r+\lfloor {n_1^+}/\varkappa \rfloor} = \Delta_{k_{r-1}}^{(j)}$ satisfies \eqref{eqd:loglaplacegrowth} with $V=C/\beta$, $W = C/(\beta \omega)$, $\epsilon = 1$ and $E=1/\sqrt{\omega}$.
  \[
    \Pr\left[
      \bigl\{
      \max_{r : k_{r-1} \leq n} |A_{r+\lfloor {n_1^+}/\varkappa \rfloor}| \geq x\sqrt{1/\omega}
    \bigr\} \cap \mathcal{E}
    \right]
    \leq 2\exp\left(
    -\frac{1}{C}\frac{\beta(\log x)^2}{\log(n/{n_1^+})}
    \right).
  \]
  Hence taking $x = n^{\delta}$ completes the proof for the imaginary part.

  Once more, for the real part the proof is simpler: having controlled the difference of imaginary parts, the difference of real parts admits a block martingale structure.  We suppress the details.
\end{proof}

\subsection{Coupling to Brownian motions}
We augment the probability space by creating a family of complex Brownian motions $\{\widehat{(\mathfrak{W}}_t^j : t \geq 0) : j \in \mathcal{D}_{n/k_1}\}$ having
\[
  \sqrt{\tfrac{2}{k+1}}
  Z_k^{(j)} =
  \widehat{\mathfrak{W}}^j_{H_{k+1}}
  -\widehat{\mathfrak{W}}^j_{H_{k}}
\]
for all $k \geq n_1^+$.
They may be constructed so that conditionally on all $\{Z_k^{(j)} : k,j\}$, the bridges
\[
  \biggl\{
  (\widehat{\mathfrak{W}}^j_{t}
  -\widehat{\mathfrak{W}}^j_{H_{k}}
  : t \in [H_k,H_{k+1}]), j \in \mathcal{D}_{n/k_1}, k \geq n_1^+
  \biggr\}
\]
are independent.
By construction we may extend $\lambda^{(j)}$ to a continuous function of time by setting (c.f.\,\eqref{eqd:llpre})
\[
  \lambda_{t}^{(j)}(\theta)=\lambda_{k}^{(j)}(\theta)
    + i\theta(t-k)
    + \sqrt{\tfrac{4}{\beta}}
    (\widehat{\mathfrak{W}}^j_{H_t}
  -\widehat{\mathfrak{W}}^j_{H_{k}})
    e^{i\Im \lambda^{*,(j)}_k(\theta)}
    \quad \text{for } t \in [k,k+1],
\]
and where $H_t = H_k + \tfrac{t-k}{k+1}.$
We make a time change by setting
\[
  k_n(t) = n_1^+\exp\biggl(  \frac{\log(n/n_1^+)}{T+-T_-} \bigl(t-T_-\bigr)\biggr)
  \quad
  t \in [T_-,T_+].
\]
In terms of this time change, we set
\[
  \widehat{\mathfrak{L}}^j_t(\theta) =
  \lambda_{k_n(t)}^{(j)}(\theta_j + \tfrac{\theta}{n})-i(k_n(t)+1)\theta_j.
\]
Finally, we define the Brownian motion $\mathfrak{W}_t^j$ by the identity
\begin{equation}\label{eqd:dW}
  d\mathfrak{W}_t^j =
  \frac{
    e^{i(k_n(t)+1)\theta_j}
  d\widehat{\mathfrak{W}}_{H_{k_n(t)}}^j
}{\sqrt{\tfrac{d}{dt} (H_{k_n(t)})}}
  \quad
  \text{on}
  \quad
  t \in [T_-,T_+].
\end{equation}

Recall $\mathfrak{L}_t^j$
which solves \eqref{eq:LU}.
The function $\widehat{\mathfrak{L}}^j_t(\theta)$ is an approximate solution of the same equation, and we can compare the two solutions.
\begin{lemma}
  For any $C>0$ and any $\delta > 0$ sufficiently small, for all $n$ sufficiently large
  \[
    \sup_{|\theta| \leq n^{1-\delta}}
    \Pr_M
    \bigl[
      \sup_{t \in [T_-,T_+]}
      |\widehat{\mathfrak{L}}^j_t(\theta)-\mathfrak{L}^j_{t}(\theta)| > n^{-\delta/2}
    \bigr] \leq n^{-C} \quad \As,
  \]
  \label{lemd:final}
\end{lemma}
\begin{proof}
  We begin by posing a stochastic differential equation for $\widehat{\mathfrak{L}}^j_t(\theta)$. We have that it is a strong solution of the differential equation
  \begin{equation}
    \label{eqd:sde1}
    d\widehat{\mathfrak{L}}^j_t(\theta)
    =
    i\frac{\theta k_n'(t)}{n}
    +\sqrt{\tfrac{4}{\beta}}
    d\bigl(\widehat{\mathfrak{W}}_{H_{k_n(t)}}^j\bigr)
    e^{i\Im \lambda^{*,(j)}_{k_n(t)}(\theta)}.
  \end{equation}
  We note the derivative $k_n'(t)$ satisfies
  \[
    \tfrac{1}{n} k_n'(t) = e^{t}k_1^{-1}\bigl(1+O_{k_1}(1/n)\bigr).
  \]
  Similarly, almost everywhere,
  \[
    \tfrac{d}{dt}\bigl( H_{k_n(t)} \bigr)
    =\frac{k_n'(t)}{\lfloor k_n(t) \rfloor +1} = t + O_{k_1}(1/n).
  \]
  Thus we can express the SDE as
    \begin{equation}
    \label{eqd:sde2}
    d\widehat{\mathfrak{L}}^j_t(\theta)
    =
    i\theta (e^{t}k_1^{-1} + E_1(t))dt
    +\sqrt{\tfrac{4}{\beta}}
    d\mathfrak{W}^j_t
    e^{i \Im\widehat{\mathfrak{L}}^j_t(\theta) + iE_3(t)}
    (1+E_2(t)),
  \end{equation}
  for deterministic errors $E_1,E_2$ which are $O_{k_1}(1/n)$ and a random error $E_3(t)$ which is controlled by Lemma \ref{lemd:llstar}.

  Hence if we form the difference $D_t \coloneqq \widehat{\mathfrak{L}}^j_t(\theta)-{\mathfrak{L}}^j_t(\theta)$, we have that
   \begin{equation}
    \label{eqd:sde3}
    dD_t
    =
    i\theta E_1(t)dt
    +
    \sqrt{\tfrac{4}{\beta}}
    d\mathfrak{W}^j_t e^{i \Im{\mathfrak{L}}^j_t(\theta)}
    (e^{i\Im D_t}-1
    +
    (e^{i \Im D_t + iE_3(t)}-e^{i \Im D_t})
    +E_2(t)
    ).
  \end{equation}
  We can furthermore check that $f_t = \log(1 + n^{2\delta} |D_t|^2)$ has both drift and diffusion coefficient which are bounded above by $O_{k_1}(1)$ uniformly in $\theta$ with probability at least $1-n^{-C}$ by Lemma \ref{lemd:llstar}.  It follows that we have a Gaussian tail bound for the difference with a variance which is $O_{k_1}(1)$, which implies the claim.
\end{proof}

The lemmas assembled give a proof of Proposition \ref{prop:uberdecoupling}.
\begin{proof}[Proof of Proposition \ref{prop:uberdecoupling}]
  We briefly survey the role of each lemma and combine them for the proof of the Proposition.
  We suppose $C$ is given and let $\delta > 0$ be as in the Proposition.
  We apply these lemmas with $\varkappa = n^{1-4\delta}$ and $\omega = 2n^{4\delta}$.
  Lemma \ref{lemd:final} connects the SDE $\mathfrak{L}_t^j(\theta)$
  to an approximate solution
  \(
  \widehat{\mathfrak{L}}^j_t(\theta) =
  \lambda_{k_n(t)}^{(j)}(\theta_j + \tfrac{\theta}{n})-i(k_n(t)+1)\theta_j
  \)
: for all $\delta>0$ sufficiently small
  \[
    \sup_{|\theta| \leq n^{1-2\delta}}
    \Pr_M
    \bigl[
      \sup_{t \in [T_-,T_+]}
      |\widehat{\mathfrak{L}}^j_t(\theta)-\mathfrak{L}^j_{t}(\theta)| \geq n^{-\delta}
    \bigr] \leq n^{-C}/3, \quad \As
  \]
  This approximate solution is a time--changed and spatially scaled version of the process $\lambda_k^{(j)}(\theta)$.
  Lemma \ref{lemd:linearization} bounds the difference between $\lambda_k^{(j)}$ and $\lambda_k$
  (and all $n$ sufficiently large and $\delta < \tfrac18$) as
  \[
    \sup_{|\theta| \leq n^{8\delta}}
    \Pr_M\left[
      \max_{{n_1^+} \leq k \leq n} |\lambda_k^{(j)}(\theta_j + \tfrac{\theta}{n})-\lambda_k(\theta_j + \tfrac{\theta}{n})|
      \geq n^{\delta-2\delta}
    \right] \leq n^{-C}/3,\quad \As
  \]  
  This shows that we can replace the driving Gaussian noise by band--resampled Gaussians, for which Fourier modes that are far from those $\theta_j$ are resampled.  Note that if we take $\delta < \tfrac{1}{10}$ the constraint on $\theta$ that $|\theta| \leq n^{8\delta}$ is more restrictive than $|\theta| \leq n^{1-2\delta}.$  
 
  Lemma \ref{lemd:delta} now shows that $\lambda_k(\theta)$, which is a locally linearized (in time) version of 
  a shift of $-2\log \Phi^*(k(e^{i\theta})$, is indeed close to it, 
  i.e.,
  for all $\delta > 0$ sufficiently small, $n$ sufficiently large 
  \[
    \sup_{\theta}
    \Pr_M\biggl[
    \max_{{n_1^+} \leq k \leq n} |-2\log \Phi^*_{k}(\theta_j + \tfrac{\theta}{n}) +i(\theta_j + \tfrac{\theta}{n}) k- \lambda_k(\theta_j + \tfrac{\theta}{n})| \geq k_1^+ n^{\delta-2\delta}
    \biggr]
    \leq n^{-C}/3,\quad \As
  \]
  Finally, this construction holds for every $j$, and for $j_1$ and $j_2$, if the sets 
  $\{ \theta : |e(\theta_{j_p}) - e(\theta)| \leq 2n^{8\delta-1}\}$ are disjoint for $p\in\{1,2\}$ then the processes $\mathfrak{L}_t^{j_p}$ are $\Pr_M$--independent.  
\end{proof}

\subsection{Logarithmic ladder}

We suppose that $\left\{ A_k \right\}$ is a sequence of random variables which roughly has the type of multiplicative recurrence structure of the Pr\"ufer phases.  This is to say, we let $\filt_k = \sigma(A_1,A_2, \dots, A_k)$ and we suppose there are constants $V$ and $W$ so that for all $\lambda \in \R$ and all $k \geq {n_1^+}$ for some ${n_1^+} \in \N,$
\begin{equation}\label{eqd:loglaplacegrowth}
  \Exp\left[ e^{\lambda A_{k+1}} \middle\vert \filt_k \right]
  \leq
  e^{\lambda A_k + \frac{\lambda^2}{k}( V A_k^2 + W)}.
\end{equation}

\begin{proposition} \label{prop:logladder}
  Let $\epsilon,E> 0$ be given and suppose that $E \leq \sqrt{W/V}.$
  Let $\mathcal{E}$ be the event such that
  \[
    A_{k+1} \leq (1+\epsilon) A_k + E \quad \text{for all }  {n_1^+} \leq k \leq n \quad \text{ and } A_{{n_1^+}} \leq E.
  \]
  There is an absolute constant $C>0$ so that for all $x \geq \max\{(2+\epsilon)^2,CV(2+\epsilon)^3\log(2+\epsilon)\log(n/{n_1^+})\},$
\[
  \Pr\bigl[
    \bigl\{
      \max_{{n_1^+} \leq k \leq n} A_k \geq x \sqrt{W/V}
    \bigr\} \cap \mathcal{E}
  \bigr]
  \leq
  \exp\left( -\tfrac{1}{C}
  \tfrac{ (\log x)^2}{V(2+\epsilon)^3\log(2+\epsilon)^2 \log(n/{n_1^+})}
  \right).
\]
\end{proposition}

\begin{proof}

  Set $\eta = 1+\epsilon$.
  We define stopping times $\{\tau_p\}$, for $p \in \N$, as the first times $k \geq {n_1^+}$ that $A_{k}$ exceeds $\eta^p E$ or that $k=n$.  Then it follows that on $\mathcal{E}$
\[
  A_{\tau_{p}}
  \leq (1+\epsilon)\eta^{p}E + E.
\]
Define, for any $p \in \N$ the process
\[
  M_j = e^{\lambda A_{j}(\theta) - \lambda^2 (V\eta^{2p+2}E^2+W)H_j}, \quad\quad \text{where } H_j = \sum_{k=1}^j \frac{1}{k}.
\]
When stopped at $\tau_{p+1},$ $\left\{ M_j \right\}$ is a supermartingale, and so
\[
  \Exp \left[ M_{\tau_{p+1}} \middle\vert \filt_{\tau_p} \right] \leq M_{\tau_p}.
\]
It follows that
\[
  \Exp \left[
    e^{\lambda(\eta^{p}E(\eta-1-\epsilon) - E) - (H_{\tau_{p+1}}-H_{\tau_p})\lambda^2 (V\eta^{2p+2}E^2+W) }
    \one[\mathcal{E}]
    \middle\vert
    \filt_{\tau_p}
  \right]
  \leq 1.
\]
Now $H_{\tau_{p+1}} - H_{\tau_p} \leq \log( \tau_{p+1}/\tau_p).$
On the event that $\tau_{p+1}/\tau_p \leq t_p$ for some $t_p \geq 1,$ we conclude that
\[
  \Pr\left[
    \{\tau_{p+1}/\tau_p \leq t_p\} \cap \mathcal{E}
  \middle\vert
  \filt_{\tau_p}
  \right]
  \leq
  e^{-\lambda(\eta^{p}(\eta-1-\epsilon) - 1)E + \log(t_p)\lambda^2 (V\eta^{2p+2}E^2+W)}.
\]
Finally, optimizing in $\lambda,$ it follows that
\begin{align*}
  \Pr\left[
    \{\tau_{p+1}/\tau_p \leq t_p\} \cap \mathcal{E}
  \middle\vert
  \filt_{\tau_p}
  \right]
 & \leq
  \exp\left(
  -\frac{(\eta^{p}(\eta-1-\epsilon) - 1)^2E^2}{4\log(t_p)(V\eta^{2p+2}E^2+W)}
  \right)\\
&  \leq
   \exp\left(
  -\frac{\eta^{2p-1}E^2}{4\log(t_p)(V\eta^{2p+2}E^2+W)}
  \right),
\end{align*}
where in the final equality we have used that $\eta = 2+\epsilon$ and $p \geq 1.$

Let $r_0 \geq {n_1^+}$ be the smallest integer such that $V \eta^{2r_0}E^2 \geq W.$  Then by iterating the previous conditional expectation,
\[
  \begin{aligned}
  \Pr\left[
    \cap_{p=r_0}^r \{\tau_{p+1}/\tau_p \leq t_p\}
    \cap \mathcal{E}
  \middle\vert
  \filt_{\tau_{r_0}}
  \right]
  &\leq
   \exp\left(
   -\sum_{p=r_0}^r\frac{\eta^{2p-1}E^2}{4\log(t_p)(V\eta^{2p+2}E^2+W)}
  \right) \\
  &\leq
   \exp\left(
   -\sum_{p=r_0}^r\frac{1}{8V\eta^3\log(t_p)}
  \right) \\
 &\leq
 \Pr\left[
    \cap_{p=r_0}^r \biggl\{ \tfrac{1}{8V\eta^3 X_p} \leq \log( t_p ) \biggr\}
  \right],
  \end{aligned}
\]
where $\left\{ X_p \right\}$ are a family of independent $\Exponential(1)$ random variables.
Hence, we may couple $\left\{ \tau_p\! :\! p \geq r_0 \right\}$ with $\left\{ X_p \right\}$ in such a way that on $\mathcal{E}$
\[
  \tau_{p+1}/\tau_p \geq e^{\tfrac{1}{8V\eta^3 X_p}}, \quad \text{ for all } p \geq r_0.
\]
Moreover, we conclude that for any $t \geq 0,$
\[
  \begin{aligned}
    \Pr\left[ \{\tau_r/\tau_{r_0} \leq t\} \cap \mathcal{E} \right] \leq
  \Pr\left[
    \exp\left({\textstyle \sum_{p=r_0}^{r-1} }\tfrac{1}{8V\eta^3 X_p}\right) \leq t
  \right]
  &=
  \Pr\left[
    {\textstyle \sum_{p=r_0}^{r-1} }\tfrac{1}{X_p}
    \leq 8V\eta^3 \log(t)
  \right]. \\
  \end{aligned}
\]
Using the harmonic--mean--arithmetic--mean inequality,
\[
  \frac{r-r_0}{ \sum_{p=r_0}^{r-1} X_p} \leq \frac{1}{r-r_0}  \sum_{p=r_0}^{r-1} \tfrac{1}{X_p}.
\]
Hence we arrive at, under the assumption that $\tfrac{ (r-r_0)^2}{8V\eta^3 \log(t)} \geq 2(r-r_0),$
\begin{equation}\label{eqd:laddertail}
  \Pr\left[ \{\tau_r/\tau_{r_0} \leq t\} \cap \mathcal{E} \right]
  \leq
  \Pr\left[
    {\textstyle \sum_{p=r_0}^{r-1} X_p}
    \geq \tfrac{ (r-r_0)^2}{8V\eta^3 \log(t)}
  \right]
  \leq
  \exp\left( -\frac{1}{C}
  \min \left\{
    \tfrac{ (r-r_0)^2}{V\eta^3 \log(t)},
    \tfrac{ (r-r_0)^3}{V^2\eta^6 \log(t)^2}
  \right\}
  \right),
\end{equation}
using Bernstein's inequality for subexponential random variables \cite[Theorem 2.8.1]{Vershynin}.
Observe that under the assumption the minimum is always attained by the first term.

Finally, we observe that for $r$ such that $V \eta^{2r}E^2 \geq W,$
\[
  \Pr\bigl[
    \bigl\{
    \max_{{n_1^+} \leq k \leq n} A_k \geq \eta^{r}E
  \bigr\} \cap \mathcal{E}
  \bigr]
  \leq
  \Pr\bigl[ \{\tau_{r}/\tau_{r_0} \leq n/{n_1^+} \} \cap \mathcal{E}\bigr]
  \leq
  \exp\bigl( -\tfrac{1}{C}
    \tfrac{ (r-r_0)^2}{V\eta^3 \log(n/{n_1^+})}
  \bigr),
\]
provided $r-r_0 \geq 16V \eta^3 \log(n/{n_1^+}).$
Moreover, for $x \geq 1,$ if we take $r = r_0 + \lfloor\log(x)/\log(\eta)\rfloor - 1$
\[
  \begin{aligned}
  \Pr\bigl[
    \bigl\{
    \max_{{n_1^+} \leq k \leq n} A_k \geq x \sqrt{W/V}
  \bigr\} \cap \mathcal{E}
  \bigr]
  &\leq
  \Pr\bigl[
    \bigl\{
    \max_{{n_1^+} \leq k \leq n} A_k \geq \eta^{r-r_0+1}\sqrt{W/V}
  \bigr\} \cap \mathcal{E}
  \bigr]\\
  &\leq
  \Pr\bigl[
    \bigl\{
    \max_{{n_1^+} \leq k \leq n} A_k \geq \eta^{r}E
  \bigr\} \cap \mathcal{E}
  \bigr].
\end{aligned}
\]
Hence for $x \geq \eta^2,$ we conclude there is an absolute constant $C>0$ such that
\[
    \Pr\bigl[
    \bigl\{
    \max_{{n_1^+} \leq k \leq n} A_k \geq x \sqrt{W/V}
  \bigr\} \cap \mathcal{E}
  \bigr]
  \leq
  \exp\bigl( -\tfrac{1}{C}
    \tfrac{ (\log x)^2}{V\eta^3\log(\eta)^2 \log(n/{n_1^+})}
  \bigr),
\]
provided $\log(x) \geq \log(\eta)+16V\eta^3\log(\eta) \log(n/{n_1^+}).$
This completes the proof.
\end{proof}

\section{Changing the initial condition}
\label{sec-init}

In our application, we will want to consider changing the initial conditions of $\mathfrak{L}_{T_-}.$  The real part of $(\mathfrak{L}_t : t)$ does not influence the evolution of the diffusion, and therefore any initial condition specified for $\Re \mathfrak{L}_{T_{-}}$ will simply appear as an additive perturbation to solution of $(\mathfrak{L}_t : t)$ with $\Re \mathfrak{L}_{T_{-}}(\theta)=0.$  On the other hand, we wish to show that for the imaginary part, the probability that a small perturbation of initial condition grows in magnitude \emph{and} the random walk performs an unusual growth (as is needed to be relevant for the maximum) is small.  In fact, it will be important in the real case $\sigma=1$ that having a large real part tends to compress the relative Pr\"ufer phase.  

\begin{proposition}\label{prop:coic}
  Fix some $j \in \mathcal{D}_{n/k_1}$ and some $\theta \in [-2\pi k_1,0].$
  Let $(\mathfrak{L}_t:t \in [T_-,T_+])$ solve \eqref{eq:LU}
  and
  let $(\mathfrak{L}_t^o:t \in [T_-,T_+])$ solve \eqref{eq:dSDE}.
  Suppose
  $|\Im(\mathfrak{L}_{T_-}(\theta)-\mathfrak{L}_{T_-}(0))| \leq \frac{k_1(\log k_1)^{50}}{k_1^+}.$
  Set $\Delta_t \coloneqq \Im\mathfrak{L}_{t}(\theta)-\Im\mathfrak{L}_{t}^o(\theta)-\Im \mathfrak{L}_{T_-}(0).$
  On the event
  \[
    \sqrt{\tfrac{8}{\beta}}
    \mathcal{A}_{T_-}^{-}
    \leq
    \mathfrak{U}_{T_-}(\theta)
    \leq
    \sqrt{\tfrac{8}{\beta}}
    \mathcal{A}_{T_-}^{+},
  \]
  there is a $\delta>0$ so that for all $k_1$ sufficiently large
  \[
    \Pr(
    |\Delta_{T_+}| > \sqrt{\tfrac{k_1}{\hat{k}_1}},
    \left\{ \mathfrak{U}_{T_+}(\theta)
    \in [-(\log k_1)^{1/100},3k_6] \right\}
    ~\vert~ \filt_{n_1^+})
    \leq \frac{1}{k_1^+}
    e^{-\delta(\log k_1)^{19/20}}.
  \]
  The same holds if we replace $\mathfrak{U}$ by $\mathfrak{U}^o+\mathfrak{U}(0)$ in the above two equations.
  If $\sigma=1$, we also have the conclusion in the last display with
  $\Delta^r_t \coloneqq \Re\mathfrak{L}_{t}(\theta)-\Re\mathfrak{L}_{t}^o(\theta)
  -(\Re\mathfrak{L}_{T_-}(\theta)-\Re\mathfrak{L}_{T_-}^o(\theta))$ replacing $\Delta_t$.
\end{proposition}
\begin{proof}
  We show the case $\sigma=1$ first.  We shall show how to modify the argument for $\sigma=i$ after completing the $\sigma=1$ case.
  \paragraph{Step 1: change of measure.}
  In this case, we have
  \(
  d\mathfrak{U}_t =
  \sqrt{\tfrac{4}{\beta}}
  \Re (
  e^{i \Im \mathfrak{L}_t(\theta)}
  d \mathfrak{W}_t^j
  ).
  \)
  Let $d\mathfrak{B}_t = \sqrt{\tfrac{4}{\beta}}
  \Im (
  e^{i \Im \mathfrak{L}_t(\theta)}
  d \mathfrak{W}_t^j
  ),$ which (for fixed $\theta$) is an independent Brownian motion of $\mathfrak{U}$.
  Define a change of measure
  \[
    \frac{d\Q}{d\Pr} =
    \exp\left(
    \sqrt{\frac{\beta}{2}} (\mathfrak{U}_{T_+}(\theta)-\mathfrak{U}_{T_-}(\theta)) - (T_+-T_-)
    \right)
  \]
  then under $\Q,$ $d\mathfrak{U}_t = d\mathfrak{X}_t + \sqrt{\tfrac{8}{\beta}}dt$ on $[T_-,T_+]$ for a $\Q$--Brownian motion $(\mathfrak{X}_t: t \in [T_-,T_+])$ (with quadratic variation $\tfrac{4}{\beta}(t-T_{-})$).  Under $\Q,$ $\mathfrak{B}$ remains an independent Brownian motion with the same quadratic variation as $\mathfrak{X}.$

   On the event given for $\mathfrak{U}_{T_+}(\theta)$ in the statement of the Lemma, this Radon--Nikodym derivative is also in control and is given by $e^{(\log k_1^+) + O( \log k_1)^{9/10}}.$  Hence it suffices to show that
   \[
     \mathbb{Q}(
     |\Delta_{T_+}| > \sqrt{\tfrac{k_1}{\hat{k}_1}}
    ~\vert~ \filt_{n_1^+})
    \leq
    e^{-\delta(\log k_1)^{19/20}}.
   \]

   To prove the statements with $\mathfrak{U}^o$, we instead need to use the change of measure
   \[
         \frac{d\Q^o}{d\Pr} =
    \exp\left(
    \sqrt{\frac{\beta}{2}} (\mathfrak{U}^o_{T_+}(\theta)-\mathfrak{U}^o_{T_-}(\theta)) - (T_+-T_-)
    \right),
   \]
   but everything proceeds with obvious changes.  We continue with the case $\mathfrak{U}.$

  The difference $\Delta_t$ satisfies the SDE
  \begin{equation}\label{eq:sbd1}
    \begin{aligned}
      d\Delta_t
      &=
      \theta e^tk_1^{-1}\one[t \leq T_\dagger]dt
      +
      \sqrt{\tfrac{4}{\beta}} \Im \biggl((e^{i \Im \mathfrak{L}_t(\theta)}-e^{i \Im (\mathfrak{L}_t^o(\theta) + \mathfrak{L}_{T_-}^j(0) )}) d \mathfrak{W}_t^j	      \biggr) \\
      &=
      \theta e^tk_1^{-1}\one[t \leq T_\dagger]dt
      +
      \sqrt{\tfrac{4}{\beta}}
      \biggl(
      \Im (
      e^{i \Im \mathfrak{L}_t(\theta)}
      d \mathfrak{W}_t^j
      )
      (1-\cos \Delta_t)
      +
      \Re (
      e^{i \Im \mathfrak{L}_t(\theta)}
      d \mathfrak{W}_t^j
      )
      \sin \Delta_t
      \biggr) \\
      &=
      \theta e^tk_1^{-1}\one[t \leq T_\dagger]dt
      +
      d\mathfrak{B}_t
      (1-\cos \Delta_t)
      -
      d\mathfrak{U}_t
      \sin \Delta_t.
    \end{aligned}
  \end{equation}
  In the case that $\Delta_{T-} > 0$, the process remains nonnegative for all time.

  \paragraph{Step 2: no movement before $T_\dagger$.}
  Recall that $T_\dagger=-(\log k_1)^{19/20}<0$.
  Set $\mathfrak{e}(t) = \exp(\sqrt{\tfrac{8}{\beta}}(t-T_-)).$
  Note that
  \begin{align*}
  &  d(\Delta_t \mathfrak{e}(t))\\
  &    =
      \mathfrak{e}(t) \theta e^tk_1^{-1}\one[t \leq T_\dagger]dt
      +
      \mathfrak{e}(t)
      d\mathfrak{B}_t
      (1-\cos \Delta_t)
      -
      \mathfrak{e}(t)
      d\mathfrak{X}_t
      \sin \Delta_t
      -
      \sqrt{\tfrac{8}{\beta}}
      \mathfrak{e}(t)
      dt
      (\sin \Delta_t-\Delta_t).
  \end{align*}
  Let $\vartheta$ be the first time in $[T_-,T_\dagger]$ that $|\Delta_{t}| >
  10\exp(T_\dagger)$.
  Then
  \[
    \Delta_t = \frac{\Delta_{T_-}}{\mathfrak{e}(t)}
    +\frac{\theta}{k_1 \mathfrak{e}(t)}\int_{T_-}^t \mathfrak{e}(s)e^s\,ds
    + \frac{1}{\mathfrak{e}(t)}
    \bigl(M(t) + \mathcal{Y}(t)
    \bigr)
  \]
  for a martingale $M$ and a finite variation term $\mathcal{Y}$.  The first two terms, prior to $T_\dagger$ are bounded by $7\exp(T_\dagger)$.  The final term is bounded, before $T_\dagger \wedge \vartheta$, by $Ce^{3T_\dagger}$.  The martingale, up to time $T_\dagger \wedge \vartheta$, has quadratic variation bounded above by $\mathfrak{e}^2(t)e^{4T_\dagger}$.  Thus, summing over integer times between $[T_-,T_\dagger]$, the $\mathbb{Q}$-probability that it reaches height $\mathfrak{e}(t)e^{1.9T_\dagger}$ is at most $\exp(-c\exp(-cT_\dagger))$ for some $c> 0$. As $\Delta_t$ is continuous we conclude that with probability $1-\exp(-c\exp(-cT_\dagger))$, $\vartheta = \infty,$ which is to say
  $|\Delta(T_\dagger)| \leq 10\exp(T_\dagger).$

  \paragraph{Step 3: self-stabilizing after $T_\dagger$.}

  From time $t > T_\dagger$ the SDE \eqref{eq:sbd1} becomes
  \[
    d\Delta_t
    =
      d\mathfrak{B}_t
      (1-\cos \Delta_t)
      -
      d\mathfrak{U}_t
      \sin \Delta_t.
  \]
  From the vanishing of the drift and diffusion terms, this equation cannot cross any multiple of $2\pi \Z$.
  By passing to its negative if necessary, we may assume $\Delta_{T_\dagger} \in (0,2\pi)$ 
  We let $\tau$ be the first hitting time of $\Delta_{t}$ to $\delta \pi.$  Then by comparison, before $\tau,$ $\Delta_{t}$ for any $\epsilon >0$ there is $\delta > 0$ sufficiently small that $\Delta_t$ is dominated by the solution to
  \[
    d
    \Delta_t' =
    d\mathfrak{B}_t
    (1-\cos \Delta_t')
    -
    d\mathfrak{X}_t
    \sin \Delta_t'
    -(\sqrt{\tfrac{8}{\beta}}-\epsilon)\Delta_t'dt,
  \]
  at for all $t \leq \tau$, where $\Delta_{T_\dagger}'=\Delta_{T_\dagger}$.
  Taking logarithms, we have from It\^o's Lemma,
  \[
    d
    \log(\Delta_t') =
    d\mathfrak{B}_t
    \frac{(1-\cos \Delta_t')}{\Delta'_t}
    -
    d\mathfrak{X}_t
    \frac{\sin \Delta_t'}{\Delta'_t}
    -\biggl(\sqrt{\frac{8}{\beta}}-\epsilon\biggr)dt
    -\frac{4}{\beta}
    \frac{(1-\cos \Delta_t')}{(\Delta'_t)^2}
    dt.
  \]
  The stopped martingale part has uniformly bounded quadratic variation.
  The drift is bounded as well, using $\cos(x) \leq 1 - x^2(\tfrac12-\epsilon)$ before the stopping time $\tau',$ the first time $\Delta_t'$ reaches $\pi\delta'$ for $\delta'$ sufficiently small.
  In particular for $t \leq \tau'$,
  \[
    d
    \log(\Delta_t')
    \leq
    dM_t
    -
    \left( \sqrt{\frac{8}{\beta}}-(1+\frac4\beta)\epsilon+\frac{2}{\beta} \right)
    dt,
  \]
  for a martingale $M_t$ with $d\langle M \rangle_t \leq \frac{C}{\beta}dt$ and $M_{T_\dagger}=0.$
  Hence to bound the $\Q$-probability that
  $\log(\Delta_{T_+}') \geq \log(\Delta_{T_\dagger}') +0.4 (\log k_1)^{19/20},$
  we can instead bound the probability that there is a $t \leq \tau'$ such that
  \[
    M_t \geq
    0.4(\log k_1)^{19/20}
    +(1-\epsilon)\left( \sqrt{\frac{8}{\beta}}-(1+\frac4\beta)\epsilon+\frac{2}{\beta} \right) (t-T_\dagger),
  \]
  noting that
  on the complement of this event $\tau' > (T_+-T_\dagger)$ and so also $\tau > T_+-T_\dagger.$
  From a time-change, this probability is dominated above by the probability that a $\frac{\beta}{C}t$--quadratic variation Brownian motion crosses the same linear barrier, and so
  \[
    \begin{aligned}
      \Q(\log(\Delta_{T_+\wedge \tau'}') \geq \log(\Delta_{T_\dagger}')+
      0.4(\log k_1)^{19/20})
      &\leq
      \exp\left( -(\log k_1)^{19/20}\frac{0.4\sqrt{\beta}}{\sqrt{C}}\left( \sqrt{\frac{8}{\beta}}-(1+\frac4\beta)\epsilon+\frac{2}{\beta} \right)   \right)\\
      &=e^{-c (\log k_1)^{19/20}}
      .
    \end{aligned}
  \]
  for some $c(\beta) >0$ and
  for all $k_1$ sufficiently large.

  \paragraph{Step 4: Control of the real part.}
  Having controlled the difference of imaginary parts, we can then control the difference of real parts $\Delta^r.$  We derive the diffusion for $\Delta^r,$ whose behavior is determined entirely by that of $\Delta:$
  \begin{equation*}
    \begin{aligned}
      d\Delta_t^r
      &= \sqrt{\tfrac{4}{\beta}} \Re \biggl((e^{i \Im \mathfrak{L}_t(\theta)}-e^{i \Im (\mathfrak{L}_t^o(\theta) + \mathfrak{L}_{T_-}^j(\theta_j) )}) d \mathfrak{W}_t^j	      \biggr) \\
      &=
      \sqrt{\tfrac{4}{\beta}}
      \biggl(
      \Re (
      e^{i \Im \mathfrak{L}_t(\theta)}
      d \mathfrak{W}_t^j
      )
      (1-\cos \Delta_t)
      -
      \Im (
      e^{i \Im \mathfrak{L}_t(\theta)}
      d \mathfrak{W}_t^j
      )
      \sin \Delta_t
      \biggr) \\
      &=
      d\mathfrak{U}_t
      (1-\cos \Delta_t)
      -
      d\mathfrak{B}_t
      \sin \Delta_t \\
      &=
      dt
      \sqrt{\tfrac{8}{\beta}}
      (1-\cos \Delta_t)
      +
      d\mathfrak{X}_t
      (1-\cos \Delta_t)
      -
      d\mathfrak{B}_t
      \sin \Delta_t.
    \end{aligned}
  \end{equation*}
  The $\Q$--probability of the event that $|\Delta_t| \leq e^{-0.51(\log k_1)^{19/20} - \epsilon(t-T_-)}$ for all $t \in [T_-,T_+]$ is $1-e^{-\delta (\log k_1)^{19/20}}$ for some $\delta,\epsilon > 0$ and all $k_1$ sufficiently large.
  On that event, both the drift and the quadratic variation of the
  martingale part of $\Delta_t^r$ are bounded by $O( e^{- 1.01(\log k_1)^{19/20}})$ for all $k_1$ sufficiently large.  Thus the probability that this reaches $e^{-0.5(\log k_1)^{19/20}}$ has the claimed probability.

  \paragraph{Step 5: The imaginary case.}
  We now have
  \[
    d\mathfrak{U}_t =
    -
    \sqrt{\tfrac{4}{\beta}}
    \Re ( \sigma
    e^{i \Im \mathfrak{L}_t(\theta)}
    d \mathfrak{W}_t^j
    )
    =
    \sqrt{\tfrac{4}{\beta}}
    \Im (
    e^{i \Im \mathfrak{L}_t(\theta)}
    d \mathfrak{W}_t^j
    ),
  \]
  and we let
  \[
    d\mathfrak{B}_t
    =
    \sqrt{\tfrac{4}{\beta}}
    \Im ( \sigma
    e^{i \Im \mathfrak{L}_t(\theta)}
    d \mathfrak{W}_t^j
    )
    =
    \sqrt{\tfrac{4}{\beta}}
    \Re (
    e^{i \Im \mathfrak{L}_t(\theta)}
    d \mathfrak{W}_t^j
    )
  \]
  which again is an independent Brownian motion.
  The difference $\Delta_t$ satisfies the SDE (see the first two lines of \eqref{eq:sbd1})
  \[
    \begin{aligned}
      d\Delta_t
      &=
      \theta e^tk_1^{-1}\one[t \leq T_\dagger]dt
      +
      d\mathfrak{U}_t
      (1-\cos \Delta_t)
      +
      d\mathfrak{B}_t
      \sin \Delta_t
      .
    \end{aligned}
  \]
  Step 1 and 2 proceed in exactly the same way.
  For step 3, after the change of measure
  we have the SDE for $t \geq T_\dagger$,
  \[
    \begin{aligned}
      d\Delta_t
      =
      \sqrt{\tfrac{8}{\beta}}
      (1-\cos \Delta_t)
      dt
      +
      d\mathfrak{X}_t
      (1-\cos \Delta_t)
      +
      d\mathfrak{B}_t
      \sin \Delta_t.
    \end{aligned}
  \]
  In the case $\sigma=i$, the drift is positive, but weak (as it is quadratic in $\Delta_t$).  In particular before $\tau,$ we can dominate the solution by
  \[
    \begin{aligned}
      d\Delta_t'
      =
      \sqrt{\tfrac{8}{\beta}}
      \delta
      \Delta_t'
      dt
      +
      d\mathfrak{X}_t
      (1-\cos \Delta_t')
      +
      d\mathfrak{B}_t
      \sin \Delta_t'.
    \end{aligned}
  \]
  The proof now continues the same way as in the real case.
\end{proof}

\section{Calculus estimates for the Pr\"ufer phases}
\label{sec-prufest}
\begin{lemma}
  Let $\lambda_1,\lambda_2 \in \mathbb{C}$ and let $\alpha,\Gamma \in \R$ with $\Gamma > 1$ be fixed real numbers and define, for $z=x+iy$ with $x,y \in \mathbb{R},$
  \[
    \begin{aligned}
      &F(x,y)=F(z)
      =
      \Re\left\{
        \lambda_1\log(1-u(z)e^{i\alpha}) - \lambda_1\log(1-u(z))
        +\lambda_2\log(1-u(z))
      \right\}, \\
      &\text{where}
      \quad
      u=u(z) = \frac{z}{\sqrt{ |z|^2 + \Gamma}}.
    \end{aligned}
  \]
  Then there is an absolute constant $C> 0$ so that for $x^2+y^2 = r^2 \leq \Gamma/2$
  \begin{equation}
    \label{eq-partialF}
    \begin{aligned}
      &|\partial_x F(x+iy) + \Re\{\lambda_1(e^{i\alpha}-1) + \lambda_2\}\Gamma^{-1/2}|
      \leq
      \frac{C(|\lambda_1(e^{i\alpha}-1)| + |\lambda_2|)r}
      {\Gamma}
      , \text{ and } \\
      &|\partial_y F(x+iy) - \Im\{\lambda_1(e^{i\alpha}-1) + \lambda_2\}\Gamma^{-1/2}|
      \leq
      \frac{C(|\lambda_1(e^{i\alpha}-1)| + |\lambda_2|)r}
      {\Gamma},
    \end{aligned}
  \end{equation}
  \label{lem:calculus}
\end{lemma}
\begin{proof}
  We begin by computing the partial of $F$ with respect to $u,$ giving
  \begin{equation}\label{eqd:FU}
    \frac{d}{du}F
    =
    -
    \Re\left\{
      \lambda_1\frac{e^{i\alpha}-1}{(1-ue^{i\alpha})(1-u)}
      +\lambda_2\frac{1}{1-u}
    \right\}
  \end{equation}
  Further
  \[
    \partial_x(u(x+iy))
    =\frac{y^2+\Gamma - ixy}{(x^2+y^2+\Gamma)^{3/2}}
    \quad
    \text{and}
    \quad
    \partial_y(u(x+iy))
    =\frac{i(x^2+\Gamma) - xy}{(x^2+y^2+\Gamma)^{3/2}}.
  \]
  We have that $\partial_x(u(x+iy))$ is nearly real, as is  $\partial_y(u(x+iy))$ nearly imaginary, when $x$ and $y$ are much smaller than $\Gamma:$
    \begin{equation}\label{eqd:ducross}
    \max\{
      |
      \Im \partial_x(u(x+iy))
      |,
      |
      \Re \partial_y(u(x+iy))
    |\}
    \leq
    \frac{|xy|}{\Gamma^{3/2}} \leq \frac{r^2}{\Gamma^{3/2}}.
  \end{equation}
  As for the principal terms, using $r^2 \leq \Gamma,$
  \[
    |\Re \partial_x(u(x+iy)) - \Gamma^{-1/2}|
    \leq \frac{(x^2+y^2+\Gamma)^{3/2} - \Gamma^{1/2}(y^2+\Gamma)}{(x^2+y^2+\Gamma)^{3/2}\Gamma^{1/2}}
    \leq \frac{2r^2}{\Gamma^{3/2}}.
  \]
  Likewise,
  \[
    |\Im \partial_y(u(x+iy)) - i\Gamma^{-1/2}|
    \leq \frac{2r^2}{\Gamma^{3/2}}.
  \]
  Hence using $|u(x+iy)| \leq r\Gamma^{-1/2} \leq \frac{1}{\sqrt{2}}$, we conclude \eqref{eq-partialF} for
  an absolute constant $C>0$
by combining the previous displays.
\end{proof}

By a similar second order expansion, we arrive at:
\begin{lemma}\label{lem:calculus2}
  Let $\lambda_1,\lambda_2 \in \mathbb{C}$ and let $\alpha,\Gamma \in \R$ with $\Gamma > 1$ and let $F$ be as in Lemma \ref{lem:calculus}.
  Then there is an absolute constant $C> 0$ so that for $x^2+y^2 = r^2 \leq \Gamma/2$
  \[
    \left|
    F(x+iy) + 
    \Re\left\{
      (\lambda_1(e^{i\alpha}-1) + \lambda_2)
      \frac{x+iy}{\Gamma^{1/2}}
      +
      (\lambda_1(e^{2i\alpha}-1) + \lambda_2)
      \frac{(x+iy)^2}{2\Gamma}
      \right\}
    \right|
    \leq
    \frac{C(|\lambda_1(e^{i\alpha}-1)| + |\lambda_2|)r^3}{\Gamma^{3/2}}.
  \]
\end{lemma}
\begin{proof}
  Differentiating \eqref{eqd:FU} and composing with $u(x+iy)$ and its derivatives, the claimed bound is easily checked.
\end{proof}

\section{Polynomial interpolation}
\label{sec:interpolation}

We will use some results for the \emph{a priori} stability of polynomials.  The first of these is a classical inequality due to Bernstein:
\begin{theorem}\label{thm:Bernstein} For any polynomial $Q$ of degree $k \geq 1,$
  \[
    \max_{|z|=1} |Q'(z)| \leq k \cdot \max_{|z|=1} |Q(z)|.
  \]
\end{theorem}
\noindent See \cite[Chapter 14]{RahmanSchmeisser}.

We also need a quantitative interpolation result for polynomials of a given degree.  The following is in some sense a generalization of \cite[Lemma 4.3]{CMN}.  Related inequalities have been published before, see especially \cite{RakhmanovShekhtman} and \cite[Theorem 8]{FrappierRahmanRuscheweyh}.
\begin{theorem}
  For any polynomial $Q$ of degree $k \geq 1,$ and any natural number $m \geq 2,$
  \[
    \max_{|z|=1} |Q(z)|^2 \leq \frac{m}{m-1} \cdot \max_{\omega : \omega^{2mk} = 1}|Q(\omega)|^2.
  \]
  Furthermore, if for any $b > 0$ we partition the $(2mk)$-th roots of unity into $\mathcal{N}$ and $\mathcal{F}$ so that $\mathcal{N}$ are all those roots of unity $\omega$ so that $|\omega-1| \leq \frac{2b}{k},$ then there is an absolute constant $C>0$ so that
  \[
    \begin{aligned}
      &\max_{\substack{|z-1|\leq \frac{b}{k},\\ |z|=1}} |Q(z)|^2 \leq \frac{m}{m-1}\cdot \max_{\omega \in \mathcal{N}} |Q(\omega)|^2 + \frac{C}{b(m-1)}\cdot\max_{\omega \in \mathcal{F}} |Q(\omega)|^2 \quad\text{and}\\
    &\min_{\substack{|z-1|\leq \frac{b}{k}, \\ |z|=1}} |Q(z)|^2 \geq \frac{m}{m-1}\cdot \min_{\omega \in \mathcal{N}} |Q(\omega)|^2 - \biggl(1+\frac{C}{b}\biggr)\frac{1}{(m-1)}\cdot\max_{\omega : \omega^{2mk}=1} |Q(\omega)|^2. \\
  \end{aligned}
  \]
  \label{thm:interpolation}
\end{theorem}
We give a proof of this fact.
For any $m \in \N$ let $F_m$ be the Fej\'er kernel, which for $|z|=1$ has the representation
\begin{equation}\label{eq:Fejer}
  F_m(z)
  =\frac{1}{m}\sum_{r=0}^{m-1}\sum_{s=-r}^r z^s
  =\frac{1}{m}
  \biggl[\sum_{s=0}^{m-1} z^s \biggr]
  \biggl[\sum_{s=0}^{m-1} z^{-s} \biggr]
  =\frac{1}{m}\frac{|1-z^m|^2}{|1-z|^2}.
\end{equation}
We will need the following identity. In what follows we use the shorthand notation $e(t)=e^{i2\pi t}$.
\begin{lemma}\label{lem:Fejersums}
  For all $m,r \in \N$ and all $t \in \R,$
  \[
    \sum_{j=1}^{rm}
    F_{m}(e(t+j/(rm))
    =rm.
  \]
\end{lemma}
\noindent See \cite{Hofbauer} for a discussion of this.  We give a proof below:
\begin{proof}
  Observe that using \eqref{eq:Fejer} we can write
  \[
    F_{m}(e(t+j/(rm))
    =
    \frac{1}{m}
    \frac{ \sin({\pi t m}+\tfrac{\pi j}{r})^2}{\sin({\pi t}+\tfrac{\pi j}{rm})^2}.
  \]
  We use the well-known identity that for any $x \in \mathbb{C} \setminus \pi\Z$
  \begin{equation}\label{eq:sineresidue}
    \frac{1}{\sin(x)^2} = \sum_{k \in \Z} \frac{1}{(x + k\pi)^2}.
  \end{equation}
  By continuity, it suffices to establish the identity for irrational $t.$
  We have, by grouping the terms in the sum over $j$ according to their residue class $\ell$ modulo $r$
  \[
    \begin{aligned}
      \sum_{j=1}^{rm}
      F_{m}(e(t+j/(rm))
      &=
      \sum_{\ell=0}^{r-1}
      \sum_{p=1}^{m}
      \frac{1}{m}
      \frac{ \sin({\pi t m}+\tfrac{\pi \ell}{r})^2}{ \sin({\pi t}+\tfrac{\pi (\ell+rp)}{rm})^2 } 
      =
      \sum_{\ell=0}^{r-1}
      \sum_{p=1}^{m}
      \sum_{k \in \Z}
      \frac{1}{m}
      \frac{ \sin({\pi t m}+\tfrac{\pi \ell}{r})^2}{({\pi t}+\tfrac{\pi (\ell + rp + rmk)}{rm})^2} \\
      &=
      m
      \sum_{\ell=0}^{r-1}
      \frac{ \sin({\pi t m}+\tfrac{\pi \ell}{r})^2}{\sin(\pi tm + \tfrac{\pi \ell}{r})^2}
      =rm.
    \end{aligned}
  \]
  In the penultimate display, we have extracted a factor of $m$ and again applied \eqref{eq:sineresidue}.
\end{proof}

We can now give a proof of Theorem \ref{thm:interpolation}.
\begin{proof}

  Define for any $m \in \N$ with $m > 1$
  \[
    R(z) = \frac{kmF_{km}(z) - kF_{k}(z)}{k(m-1)(2mk)}.
  \]
  Then we can write
  \[
    R(z) = \frac{1}{2mk} \sum_{s=-km}^{km} \lambda_s z^s,
  \]
  where $\lambda_s = 1$ for $-k \leq s \leq k.$
  In particular, for any $0 \leq r < km$
  \[
    \sum_{ \omega:\omega^{2mk}=1}
    \omega^r R(z \bar{\omega})
    =
    \frac{1}{2mk} \sum_{s=-km}^{km}
    \lambda_s z^s
    \cdot
    \biggl[
      \sum_{ \omega:\omega^{2mk}=1}
      \omega^r \bar{\omega}^s
    \biggr]
    = \lambda_r z^r,
  \]
  and so it follows that for any polynomial $Q(z)$ of degree $k$ and $z$ on the unit circle,
  \begin{equation}
      |Q(z)|^2
      =
      \sum_{ \omega:\omega^{2mk} = 1}
      |Q(\omega)|^2 R(z \bar{\omega}) 
      \leq
      \sum_{ \omega:\omega^{2mk} = 1}
      |Q(\omega)|^2 \frac{F_{km}(z\bar{\omega})}{2k(m-1)}
      \eqqcolon X(z),
    \label{eq:Qinterpolation}
  \end{equation}
  where the first equality follows as $|Q(z)|^2$ for $|z|=1$ can be represented as a Laurent polynomial with Fourier support contained in $[-k,k]$
  and the inequality follows from the positivity of the Fej\'er kernel.

  Using Lemma \ref{lem:Fejersums}, we have
  \[
    \sum_{ \omega:\omega^{2mk} = 1}
    \frac{F_{km}(z \bar \omega)}{2k(m-1)}
    =\frac{m}{m-1},
  \]
  and hence for all $|z|=1,$
  \[
    X(z) \leq \frac{m}{m-1}\cdot \max_{\omega : \omega^{2mk} = 1}|Q(\omega)|^2.
  \]
  If we further partition the roots of unity into $\mathcal{N}$ and $\mathcal{F}$ as in the statement of the theorem, we can bound $X(z)$ using \eqref{eq:Fejer} by
  \[
    X(z) \leq
    \frac{m}{m-1}\cdot \max_{\omega \in \mathcal{N}}|Q(\omega)|^2
    +
    \max_{\omega \in \mathcal{F}}|Q(\omega)|^2
    \cdot
    \sum_{\omega \in \mathcal{F}} \frac{2}{km(m-1)|1-z\bar{\omega}|^2}.
  \]
  If $z$ satisfies that $|z-1| \leq \frac{b}{k},$ then $|\omega - z| \geq \frac{b}{k}$ and there is therefore an absolute constant $C > 0$ so that
  \begin{equation}\label{eq:FB}
    \sum_{\omega \in \mathcal{F}} \frac{2}{km(m-1)|1-z\bar{\omega}|^2}
    \leq \frac{C}{b(m-1)},
  \end{equation}
  which completes the proof of the upper bound.

  For the lower bound, starting from \eqref{eq:Qinterpolation},
  \[
    \begin{aligned}
      |Q(z)|^2
      &=
      \sum_{ \omega:\omega^{2mk} = 1}
      |Q(\omega)|^2 R(z \bar{\omega}) \\
      &\geq
      \sum_{ \omega:\omega^{2mk} = 1}
      |Q(\omega)|^2 \frac{F_{km}(z\bar{\omega})}{2k(m-1)}
      -
      \max_{\omega:\omega^{2mk}=1} |Q(\omega)|^2
      \cdot
      \sum_{ \omega:\omega^{2mk} = 1}
      \frac{F_{k}(z\bar{\omega})}{2mk(m-1)}.
    \end{aligned}
  \]
  Using Lemma \ref{lem:Fejersums}, we conclude
  \[
      |Q(z)|^2
      \geq
      X(z)
      -
      \frac{1}{m-1}
      \cdot
      \max_{\omega:\omega^{2mk}=1} |Q(\omega)|^2.
  \]
  Furthermore
  \[
      X(z)
      \geq
      \min_{\omega \in \mathcal{N}} |Q(\omega)|^2
      \cdot
      \sum_{\omega \in \mathcal{N}}
      \frac{F_{km}(z\bar{\omega})}{2k(m-1)}
      \geq
      \min_{\omega \in \mathcal{N}} |Q(\omega)|^2
      \biggl(\frac{m}{m-1} - \frac{C}{b(m-1)}\biggr),
  \]
  using \eqref{eq:FB}.
\end{proof}

\section{Convergence of the derivative martingale}
\label{sec:mgle}

We recall from \eqref{eq:upsilon}
\begin{equation*}
  \varphi_{k+1}(\theta)=\varphi_{k}(\theta) + 2\Re \{ \sigma \left( \log(1-\gamma_ke^{i\Psi_k(\theta)}) \right) \}
  ,\quad \varphi_0(\theta) = 0,
\end{equation*}
for $\sigma \in \left\{ 1, i \right\}.$

We also recall (from \eqref{eq:betagamma}) that $\gamma_j$ are independent, rotationally invariant in law, and have $|\gamma_j|^2$ distributed as $\Beta(1, \beta_j)$ where $\beta_j^2 = \tfrac{\beta}{2}(j+1)$.  Hence we have an explicit expression for the moment generating functions of $\varphi,$ given by (see \cite[Proposition 2.5]{CMN} or \cite[Lemma 2.3]{BHNY}).
\begin{lemma}\label{lem:mgf}
  For any $s,t \in \C$ with $\Re s \geq -1$
  \[
    \Exp[ e^{s \Re( \log(1-\gamma_j))+ t \Im( \log(1-\gamma_j))}]
    = \frac{\Gamma(1+\beta_j^2)\Gamma(1+s+\beta_j^2)}{\Gamma(1+\beta_j^2 + (s+it)/2)\Gamma(1+\beta_j^2 + (s-it)/2)}.
  \]
  For $s \geq 0$ and $t \in \R,$
  \[
    \Exp[ e^{s \Re( \log(1-\gamma_j))+ t \Im( \log(1-\gamma_j))}]
    \leq
    \exp\left(
    \frac{s^2+t^2}{2}\frac{1}{1+\beta(j+1)}
    \right).
  \]
\end{lemma}

We define for any $j\in \N$ and any $|\sigma|=1,$ $H^{(\sigma)}_j(s) = H_j(s) \coloneqq \log \Exp[ e^{2s\Re(\sigma\log(1-\gamma_{j-1}))}].$ Define
\[
  \mathscr{M}_j(\theta,s) \coloneqq e^{ s \varphi_j(\theta) - \sum_{k=1}^j H_k(s)},
\]
which is a martingale.  Set $s_\beta \coloneqq \sqrt{\tfrac{\beta}{2}}$ and define
\[
  \widehat{\mathscr{D}}_j(\theta) \coloneqq -\partial_s  \mathscr{M}_j(\theta,s) \vert_{s=s_\beta}
  =e^{ s_\beta \varphi_j(\theta) - \sum_{k=1}^j H_k(s_\beta)}
  \biggl(\sum_{k=1}^j H_k'(s_\beta) - \varphi_j(\theta)\biggr),
\]
which is also a martingale.  Define
\begin{equation}
  \widehat{\mathscr{B}}_j
  \coloneqq
  \frac{1}{2\pi}\int_0^{2\pi} \widehat{\mathscr{D}}_j(\theta) d\theta.
  \label{eq:Bmart}
\end{equation}
We need an elementary computation of the asymptotic behavior of the sums of $H_k.$
\begin{lemma}\label{lem:gmc_hksum}
  The following limits exist and are finite
  \[
    \lim_{j\to \infty}
    \sum_{k=1}^j H_k(s_\beta)
    -\log j
    = \mathfrak{g}_\beta,
    \quad
    \text{and}
    \quad
    \lim_{j\to \infty}
    \sum_{k=1}^j H_k'(s_\beta)
    - \sqrt{\frac{8}{\beta}}\log j
    = \mathfrak{h}_\beta.
  \]
\end{lemma}
\begin{proof}
  We recall the ratio asymptotic for the $\Gamma$ function (see \cite[5.11]{DLMF})
  \[
    \frac{\Gamma(\beta_k^2 + x)}
    {\Gamma(\beta_k^2+y)}
    =\beta_k^{2(x-y)}\left( 1+ \frac{\tfrac12(x-y)(x+y-1)}{\beta_k^2} + O(\beta_k^{-4}) \right).
  \]
  Then for any $s,t \in \C,$
  \[
    \Exp[ e^{s \Re( \log(1-\gamma_k))+ t \Im( \log(1-\gamma_k))}]
    =
    1
    + \frac{s^2+t^2}{4\beta_k^2}
    + O(\beta_k^{-4}),
  \]
  and the asymptotic can be differentiated on both sides with respect to $s$ and $t$ as well.
  Thus
  \[
    H_k(s) = \frac{s^2}{\beta_k^2}
    + O(\beta_k^{-4}).
  \]
  For $H_k(s_\beta)$ we therefore have
  \[
    H_k(s_\beta)
    = \frac{1}{k+1} + O(k^{-2}),
  \]
  which leads directly to the claimed asymptotic.
  For the derivative, we have that
  \[
    H_k'(s_\beta) = \frac{2s_\beta}{\beta_k^2} + O(\beta_{k}^{-4})
    = \sqrt{\frac{8}{\beta}} \frac{1}{k+1}+ O(\beta_{k}^{-4}).
  \]
\end{proof}
We observe that the field $\{\sqrt{\frac{8}{\beta}}\log j -  \varphi_j(\theta)\}$ is rarely very negative, and is in fact almost surely positive for all $j$ sufficiently large (but random).
\begin{lemma}\label{lem:gmc_maxbnd}
  \[
    \inf\biggl\{ \sqrt{\frac{8}{\beta}}(\log j - \tfrac{1}{8}\log\log j) -  \varphi_j(\theta) : j \in 2^{\N}, \theta \in [0,2\pi]\biggr\} > -\infty \quad \As
  \]
\end{lemma}
\begin{proof}
  We use Proposition \cite[Propositions 3.1, 4.5]{CMN}, due to which
  \begin{equation}\label{eq:upsilongmc}
    \sup\biggl\{
      \sup_{\theta \in [0,2\pi]} \varphi_j(\theta)
      -
      \sup_{0 \leq k \leq 2j} G_j( \tfrac{\pi k}{j})
      : j \in \N
    \biggr\}
    < \infty \As
  \end{equation}
  Using \cite[(4.6) and the display following with $C=t$ ]{CMN} for any $j$ and any $t \geq 1$ (we take $t = (\tfrac12+\epsilon)\log\log j$),
  \[
    \Pr
    [\sup_{0 \leq k \leq 2j} G_j( \tfrac{\pi k}{j})
      > \sqrt{\tfrac{8}{\beta}}( \log j - \tfrac34 \log\log j + t)
    ]
    \leq C_\beta (1+t)^3e^{-2t}
  \]
  for some constant $C_\beta>0.$  By Borel--Cantelli applied to the sequence $j \in 2^\N,$
  we have for any $\epsilon > 0$ and for all such $j\in 2^\N$ sufficiently large
  \[
    \sup_{0 \leq k \leq 2j} G_j( \tfrac{\pi k}{j})
    \leq \sqrt{\tfrac{8}{\beta}}( \log j - (\tfrac14-\epsilon) \log\log j ).
  \]
\end{proof}

With the control provided by the proof of Lemma \ref{lem:gmc_maxbnd} and Lemma \ref{lem:goodstuff1}, we may work on an event $\mathcal{E}_\kappa$
\[
  \mathcal{E}_\kappa
  \coloneqq
  \biggl\{
    \sup\{
      \max\{
        |G_j(\theta) - \varphi_j(\theta)|,
        \varphi_j(\theta) - \sqrt{\tfrac{8}{\beta}}(\log j-\tfrac 18 \log\log j)
      \} : j \in 2^{\N}, \theta \in [0,2\pi]
    \} \leq  \sqrt{\tfrac{8}{\beta}}\kappa
  \biggr\}.
\]
On the event $\mathcal{E}_\kappa$ we can use Girsanov and the ballot theorem for the Gaussian random walk $j \mapsto G_j(\theta)$ to conclude for any $\delta >0$ any $\theta \in [0,2\pi],$
\begin{equation}\label{eq:properballot}
  \Pr[ \mathcal{E}_\kappa, \log j + \kappa - \sqrt{\tfrac{\beta}{8}} \varphi_j \in
    [{k-1},k]
  ] \leq
  \begin{cases}
    \frac{C_{\beta,\kappa,\delta}k}{(\log j)^{3/2}}{\exp\left( -\frac{(\log j - k)^2}{\log j} \right)},
    & k \leq (1-\delta)\log j, \\
    \frac{C_{\beta,\kappa,\delta}}{(\log j)^{1/2}}{\exp\left( -\frac{(\log j - k)^2}{\log j} \right)},
    & k > (1-\delta)\log j.
  \end{cases}
\end{equation}

We begin with the observation that the improperly normalized mass tends almost surely to $0$ at the critical $s=s_\beta.$
\begin{lemma}\label{lem:Mj0}
  For
  any $\sigma \in \left\{ 1, i \right\}$
  and any $\beta > 0,$
  \[
    Z_j \coloneqq \frac{1}{2\pi}\int_0^{2\pi} \mathscr{M}_j(\theta,s_\beta) d\theta
    \Asto[j] 0.
  \]
  Furthermore
  \(
  \{Z_j \sqrt{\log j} : j \in \N\}
  \)
  is tight, and for any $\epsilon > 0,$ there is a compact $K\subset (0,\infty)$
  so that if
  \[
  \chi(\theta) = \one[{ (\sqrt{\tfrac{8}{\beta}} \log j -  \varphi_j(\theta))/\sqrt{\log j} \not\in K }]
  \]
  then for any $j \in \N$
  \begin{equation*}
    \Pr
    \biggl(\int_0^{2\pi}
    \mathscr{M}_j(\theta,s_\beta) \bigl|\sqrt{\tfrac{8}{\beta}} \log j -  \varphi_j(\theta) \bigr| \chi(\theta) d\theta > \epsilon
    \biggr)
    <\epsilon.
  \end{equation*}
\end{lemma}
\begin{proof}
  The process $Z_j$ is a positive martingale and so it converges almost surely.
  After establishing  the claimed tightness, it follows that $Z_j\Prto[j]0.$
 Hence along some subsequence, it converges almost surely to $0.$ The almost sure convergence of $Z_j$ then completes the proof of the first point.

  The remaining statements now follow by taking expectations of $Z_j$ on the event $\mathcal{E}_\kappa.$  In particular using \eqref{eq:properballot}
  we have (with $\log j - \sqrt{\tfrac{8}{\beta}} \varphi_j = {u}{\sqrt{\log j}}$) and $t=(1-\delta)\sqrt{\log j}$
  \[
    \begin{aligned}
      \Exp \biggl(
      \one[\mathcal{E}_\kappa]
      \mathscr{M}_j(\theta, s_\beta)
      \biggr)
      &\leq
      C_{\beta,\kappa}
      \int\limits_{0}^t
      \frac{u}{\sqrt{\log j}}
      \exp(-2u\sqrt{\log j} + \log j)
      \exp\left( -\frac{(\log j - u\sqrt{\log j})^2}{\log j} \right)
      du\\
      &+
      C_{\beta,\kappa}
      \int\limits_{t}^{\infty}
      \exp(-2u\sqrt{\log j} + \log j)
      \exp\left( -\frac{(\log j - u\sqrt{\log j})^2}{\log j} \right)
      du
    \end{aligned}.
  \]
  This simplifies to
  \[
    \Exp \biggl(
    \one[\mathcal{E}_\kappa]
    \mathscr{M}_j(\theta, s_\beta)
    \biggr)
    \leq
    C_{\beta,\kappa}
    \int\limits_{0}^t
    \frac{u}{\sqrt{\log j}}
    \exp(-u^2)
    du
    +O( (\log j)^{(1-\delta)^2}).
  \]
  In particular, we conclude that
  \[
    \biggl\{
      \one\{\mathcal{E}_\kappa\} \sqrt{\log j}\int_0^{2\pi} \mathscr{M}_j(\theta,s_\beta)\,d\theta : j \in \N
    \biggr\}
  \]
  is tight, and as $\cup_{\kappa \in \N} \mathcal{E}_\kappa$ has probability $1$, the tightness without the indicator holds.

  Essentially the same computation shows the claimed estimates for the final display of the lemma.  With $K = [\eta/2, 2\eta^{-1}],$ for all $j$ sufficiently large
  \[
    \Exp \biggl(
    \one[\mathcal{E}_\kappa]
    \mathscr{M}_j(\theta, s_\beta)
    |\log j - \sqrt{\tfrac{8}{\beta}}\varphi_j(\theta)|
    \chi(\theta)
    \biggr)
    \leq
    C_{\beta,\kappa}
    \biggl(
    \int\limits_{0}^\eta
    +
    \int\limits_{\eta^{-1}}^\infty
    \biggr)
    u^2
    \exp(-u^2)
    du.
  \]
  This may be made as small as desired by picking $\eta$ small.
\end{proof}
\noindent We will use this convergence to compare $\widehat{\mathscr{B}}_j$ and ${\mathscr{B}_j},$ and  $\widehat{\mathscr{D}}_j$ and ${\mathscr{D}_j}$.
 Before doing this comparison, we will show the convergence of $\widehat{\mathscr{B}}_j$ and of $\int \widehat{\mathscr{D}}_j (\theta) f(\theta) d\theta$ for positive bounded test functions $f$.  Similar ideas appear in \cite{DRSV}.
\begin{lemma}\label{lem:Bhat}
  There is an almost surely nonnegative finite random variable $\widehat{\mathscr{B}}_\infty$ and nonnegative measure $\widehat{\mathscr{D}}_\infty$ so that
for any bounded positive deterministic test function $f$,
  \[
    \widehat{\mathscr{B}}_{2^\ell}
    \Asto[\ell]
    \widehat{\mathscr{B}}_\infty,\qquad   \int_0^{2\pi} \widehat{\mathscr{D}}_{2^\ell}(\theta) f(\theta) d\theta
    \Asto[\ell]
    \int_0^{2\pi} \widehat{\mathscr{D}}_\infty(\theta) f(\theta) d\theta.
  \]
\end{lemma}
\begin{proof}
  Define the positive and negative parts
  \[
    \mathscr{D}^{\pm}_j(\theta) \coloneqq e^{ s_\beta \varphi_j(\theta) - \sum_{k=1}^j H_k(s_\beta)}
    \biggl(\sum_{k=1}^j H_k'(s_\beta) - \varphi_j(\theta)\biggr)_{\pm},
  \]
  and define
  \[
    \widehat{\mathscr{B}}_{2^{\ell}}^{\pm}
    \coloneqq
    \frac{1}{2\pi}
    \int_0^{2\pi}
    \mathscr{D}^{\pm}_{2^\ell}(\theta)
    d\theta.
  \]
 Define
  \[
    Y_{2^{\ell}} \coloneqq
    \Exp[
      \widehat{\mathscr{B}}_{2^{\ell+1}}^{-}
      ~\vert~ \filt_{2^\ell}
    ].
  \]
  We will show $\sum_{\ell=1}^\infty Y_{2^\ell} < \infty$ almost surely.  Having done so, the lemma will follow, as we now explain.  Observe
  \begin{equation}
\label{eq-comp}
    \Exp[
      \widehat{\mathscr{B}}_{{2^{\ell+1}}}^{+}
      ~\vert~ \filt_{2^{\ell}}
    ]
    =
    \Exp[
      \widehat{\mathscr{B}}_{{2^{\ell+1}}}
      ~\vert~ \filt_{2^{\ell}}
    ]
    +Y_{2^{\ell}}
    =
    \widehat{\mathscr{B}}_{{2^{\ell}}}
    +Y_{2^{\ell}}
    \leq
    \widehat{\mathscr{B}}_{{2^{\ell}}}^{+}
    +Y_{2^{\ell}},
  \end{equation}
  it follows that both $T_{2^\ell}^{\pm} \coloneqq \widehat{\mathscr{B}}_{2^{\ell}}^{\pm}-\sum_{k=1}^{{\ell-1}} Y_{2^k}$ are supermartingales with respect to the filtration $(\filt_{2^\ell} : \ell \in \N)$.  Letting $\tau$ be the stopping time that $\ell \mapsto \sum_{k=1}^{\ell} Y_{2^k}$ exceeds $R,$ then $T_{2^{j\wedge \tau}}^{\pm}$ is a supermartingale bounded below by $-R$ (by predictability) which converges almost surely.  As this holds for any $R \in \N$ it follows that $T_{2^\ell}^{\pm}$ converges almost surely.  As the sum of $Y_{2^k}$ converges almost surely as well, $\widehat{\mathscr{B}}_{2^{\ell}}^{\pm}$ converges almost surely.  Hence so does their difference.
  \begin{remark}
    From Lemma \ref{lem:gmc_maxbnd},
    it follows that in fact $\widehat{\mathscr{B}}_{j}^{-}$ is eventually $0$ for all $j$ sufficiently large, almost surely.
  \end{remark}
The statement concerning $\int_0^{2\pi}\widehat{\mathscr{D}_j}(\theta)f(\theta) d\theta$ follows similarly: instead of \eqref{eq-comp}, use that
  \begin{align}
\label{eq-comp1}
    \Exp[
      \int_0^{2\pi} \widehat{\mathscr{D}}_{{2^{\ell+1}}}^{+}(\theta) f(\theta) d\theta
      ~\vert~ \filt_{2^{\ell}}
    ]
   & \leq 
    \Exp[
      \int_0^{2\pi} \widehat{\mathscr{D}}_{{2^{\ell+1}}} f(\theta) d\theta
      ~\vert~ \filt_{2^{\ell}}
    ]
    +2 \pi \|f\|_\infty Y_{2^{\ell}}\\
 &   =
   \int_0^{2\pi}  \widehat{\mathscr{D}}_{{2^{\ell}}}(\theta) f(\theta) d\theta
   +2 \pi \|f\|_\infty  Y_{2^{\ell}}
    \leq
    \int_0^{2\pi} \widehat{\mathscr{D}}_{{2^{\ell}}}^{+}(\theta) f(\theta) d\theta
    +2\pi \|f\|_\infty Y_{2^{\ell}},\nonumber
\end{align}
and the rest follows as in the treatment of $\widehat{\mathscr{B}}_{2^\ell}$.

  So, it remains to show $\sum_{j=1}^\infty Y_{2^j} < \infty$ almost surely.
  There is a $C_\beta$ so that for all $j$ sufficiently large
  \[
    Y_{2^{j+1}}
    \leq
    C_\beta \int_0^{2\pi}
    \mathscr{M}_{2^j}(\theta,s_\beta)
    \Exp\left[
      e^{
        \sqrt{\beta/2}\left( \varphi_{2^{j+1}}(\theta)-\varphi_{2^{j}}(\theta)\right)
        -\log 2
      }
      \bigl(
      \varphi_{2^{j+1}} - \sqrt{\tfrac{8}{\beta}}\log 2^{j+1}
      \bigr)_+
      ~\vert~ \filt_{2^{j}}
    \right].
  \]
  The increment $\sqrt{\beta/2}\bigl(\varphi_{2^{j+1}}(\theta)-\varphi_{2^{j}}(\theta)\bigr)$ is uniformly subgaussian over all $j$ (from Lemma \ref{lem:mgf}).  Hence on the event $\mathcal{E}_\kappa,$ the restriction that
  \[
    \varphi_{2^{j+1}} - \varphi_{2^{j}}
    >\sqrt{\tfrac{8}{\beta}}\log 2^{j+1} - \varphi_{2^{j}}
    >\sqrt{\tfrac{8}{\beta}}(\kappa + \tfrac{1}{8}\log\log 2^j)
  \]
  implies that there is a constant $c_{\beta,\kappa} > 0$ so that for all $j \in \N$
  \[
    \Exp\left[
      e^{
        \sqrt{\beta/2}\left( \varphi_{2^{j+1}}(\theta)-\varphi_{2^{j}}(\theta)\right)
        -\log 2
      }
      \bigl(
      \varphi_{2^{j+1}} - \sqrt{\tfrac{8}{\beta}}\log 2^{j+1}
      \bigr)_+
      ~\vert~ \filt_{2^{j}}
    \right]
    \leq
    e^{-c_{\beta,\kappa} (\log j)^2}.
  \]
  Hence we arrive at the conclusion that almost surely
  \[
    \sum_{j=1}^\infty \bigl(\one \{ \mathcal{E}_\kappa \} Y_{2^j}\bigr)
    < \sum_{j=1}^\infty \biggl(e^{-c_{\beta,\kappa} (\log j)^2}\int_0^{2\pi}
    \mathscr{M}_{2^j}(\theta,s_\beta)d\theta\biggr)
    \quad
    \As
  \]
  From Lemma \ref{lem:Mj0} this is finite almost surely, which completes the proof as the $(\mathcal{E}_\kappa : \kappa \in \N)$ exhaust the probability space.
\end{proof}

We conclude with the proof of Theorem \ref{thm:Bj}.
\begin{proof}[Proof of Theorem \ref{thm:Bj}]
  We begin by noting that
  $\widehat{\mathscr{D}}_j$
  and $\mathscr{D}_j$ are exactly related by the identity
  \[
    \sqrt{\tfrac{4}{\beta}}
    \mathscr{D}_j(\theta)
    e^{\log j - \sum_{k=1}^j H_k(s_\beta)}
    -\widehat{\mathscr{D}_j}(\theta)
    =
    \mathscr{M}_j(\theta,s_\beta)
    \biggl(\sum_{k=1}^j H_k'(s_\beta) - \sqrt{\tfrac{8}{\beta}} \log j\biggr)
  \]
  Hence integrating against a positive bounded test function $f$, from Lemma \ref{lem:gmc_hksum}, Lemma \ref{lem:Mj0}, and Lemma \ref{lem:Bhat},
  \begin{equation} \label{eq:masslimitrelationship}
    \int_0^{2\pi} f(\theta)\mathscr{D}_{2^\ell}(\theta) d\theta
    \Asto[\ell]
    \sqrt{\tfrac{\beta}{4}}
    e^{\mathfrak{g}_\beta}
    \int_0^{2\pi} f(\theta)\widehat{\mathscr{D}}_{\infty}(\theta) d\theta
    \quad
    \text{and}\quad
    \mathscr{B}_{2^\ell}
    \Asto[\ell]
    \sqrt{\tfrac{\beta}{4}}
    e^{\mathfrak{g}_\beta}
    \widehat{\mathscr{B}}_\infty.
  \end{equation}

It remains to prove the non-atomicity of $\widehat{\mathscr{D}}_\infty$.
Mimicking Lemma \ref{lem:Mj0},  with $\epsilon\in (0,1/3)$, introduce the interval $K_k=[k^{-\epsilon}, k^\epsilon]$ and, for $j_0$ fixed,  the function
  \[
  \widehat{ \chi}_{j_0,j}(\theta) =\prod_{k=j_0}^j \one[{ (\sqrt{\tfrac{8}{\beta}}  k\log 2-  \varphi_{2^k}(\theta))/\sqrt{ k} \in K_k }].
  \]
The same argument as in Lemma \ref{lem:Mj0}  shows that 
\[\int  \mathscr{M}_{2^j}(\theta,s_\beta)
\biggl|\biggl(
\varphi_{2^j}(\theta)
   - \sqrt{\tfrac{8}{\beta}}  j\log 2\biggr)\biggr| (1-\widehat{\chi}_{j_0,j}(\theta)) d\theta\leq C \sum_{k=j_0}^j \frac{1}{k^{3\epsilon}}\leq \frac{C'}{j_0^{3\epsilon-1}}.\]
Taking $j_0$ large, it thus suffices to prove the non-atomicity of the limit of the positive  measures $\mathscr{M}_{2^j}(\theta,s_\beta)
   \biggl|\biggl(\varphi_{2^j}(\theta)
   - \sqrt{\tfrac{8}{\beta}}  j\log 2\biggr)\biggr| \widehat{\chi}_{j_0,j}(\theta) d\theta$ for large fixed $j_0$. Toward this end, 
divide $[0,2\pi]$ to intervals $\Delta_i$ of length $\delta$. Using Lemma \ref{lem:goodstuff}, we can replace $\varphi_{2^k}(\theta)-\varphi_{2^{j_0}}(\theta)$ by  $Z_{2^k}^{2^{j_0}}(\theta)$.

Let $A_{i,j}=\int_{\Delta_i} \mathscr{M}_{2^j}(\theta,s_\beta)
   \biggl|\biggl(\varphi_{2^j}(\theta)
   - \sqrt{\tfrac{8}{\beta}}  j\log 2\biggr)\biggr| \widehat{\chi}_{j_0,j}(\theta) d\theta$. Using Lemma \ref{Prmomen1Lower}
we have that 
\begin{equation}
\label{eq-firstmommeas}
\Exp\big( A_{i,j}\mid \filt_{2^{j_0}})=C_{i,j_0} \Delta
\end{equation}
where $\max_i C_{i,j_0}<\infty$ is a random variable independent of $\Delta$ or $j$. Using Lemmas \ref{lemma:kDebut} and \ref{lemma:ktoutDebut} as in Proposition
\ref{prop:2ray},
we  obtain that 
\begin{equation}
\label{eq-secmommeas}
\Exp\big( A_{i,j}^2\mid \filt_{2^{j_0}})=C_{i,j_0}' \Delta o(\Delta),
\end{equation}
where $\max_i C_{i,j_0}'<\infty$ is a random variable independent of $\Delta$ or $j$. Hence, for any $\delta>0$,
\( \Pr (A_{i,j}>\delta\mid \filt_{2^{j_0}})\leq \frac{\max_i C_{i,j_0}' \Delta o(\Delta)}{\delta^2},\)
and therefore
\[ \Pr (\exists i: A_{i,j}>\delta\mid \filt_{2^{j_0}})\leq \frac{\max_i C_{i,j_0}' o(\Delta)}{\delta^2}.\]
Since this is true uniformly in $j$, taking the limit as first $j\to\infty$ and then $\Delta\to 0$ gives the claim of nonatomiticity of $\widehat{\mathscr{D}}_\infty$.
\end{proof}

\section{Decoration process technical estimates}
\label{sec:decoration2nd}

In this section, we collect other estimates about the decoration.  In all, we need a relatively sharp upper estimate on a single ray of the decoration being high (that is sharp up to constants $c(k_4)$).  We also provide a weak upper estimate on two rays being high which is sharp up to a power of $(\log k_1)$.

We also need a few technical estimates of a different nature.
Due to the technical nature of the two-ray estimates we produce up to level $n_1^+$,
we show a continuity estimate for the markov kernel $(x,df)\mapsto\mathfrak{s}(x,e^x df)$ in the parameter $x$.
Finally, we show a mixing estimate of the decoration measure comparing $\mathfrak{s}(x,e^{i\alpha+x} df)$ in the case of $\sigma=1$ to the same with uniformly random $\alpha$.

\subsection{Moment estimates}
The following are needed to produce estimates for $\mathfrak{s}$.  These concern the diffusions $\mathfrak{U}^{o,j}$ and the events $\mathscr{P}_j$,
see \eqref{eq:decorationbarrier2}.  As the estimates
have no dependence on $j$, we suppress the $j$.
\begin{lemma}\label{lem:SDE1st2nd}
  For some $0 < c(k_7,x) < c(k_7,\infty)$ which is a continuous function of $k_7$,
  \begin{equation}
    \label{eq-0407aa}
    \Pr\biggl( \bigl\{ \mathfrak{U}^o_{T_+}(\theta)-\sqrt{\tfrac{4}{\beta}}h \in
  [-k_7,x] \bigr\} \cap \mathscr{P}(\theta,h) ~\big\vert~\filt_{n_1^+} \biggr)
    =
    (1+o_{k_4})\frac{c(k_7,x) h \sqrt{k_4}}{(T_+-T_-)^{3/2}}
    e^{-(T_+-T_-)} e^{-\sqrt{2}h -h^2/2(T_+-T_-)}
    ,
  \end{equation}
  for all 
  \begin{equation}
    \label{eq-defh}
    h \in [(\log k_1^+)^{1/14}, (\log k_1^+)^{13/14}].
  \end{equation}
  Further, the upper bound in \eqref{eq-0407aa} holds for all $h \geq 0.$
  Next, if $\theta_1,\theta_2$ are such that $|\theta_1|,|\theta_2| \leq 3n^{8\delta}$ and $h_1,h_2$ are as in \eqref{eq-defh} then, for some $c(k_7)>0$, with
  $|\theta_1-\theta_2|e^{T_*}=k_1$ and $h=(h_1\vee h_2)$,
  \begin{align}
    &\Pr\biggl(
    \bigcap_{i\in \{1,2\}}
    \bigl\{ \mathfrak{U}^o_{T_+}(\theta_i) 
    -\sqrt{\tfrac{4}{\beta}}h_i 
    \in [-k_7,\infty) \bigr\} \cap \mathscr{P}(\theta_i,h_i)
    ~\big\vert~\filt_{n_1^+} \biggr)\label{eq-0407ba} \\
    &\leq\!
    \begin{cases}
      c(k_7)h\frac{(T_+-T_*)\wedge (T_*-T_-)}{(T_*-T_-+1)^{3/2}}\exp\biggl(\! -S(T_-,T_+,h)\!-\!
      \Big(T_+-T_*+(T_+-T_*)^{1/14}\Big)\!
      -\!{h^2/2(T_+-T_-)}
      \biggr), &\\
      \qquad\qquad\qquad\qquad\qquad\qquad \qquad\qquad\qquad\qquad\qquad\qquad
      \text{if } e^{k_4}\leq |\theta_1-\theta_2| \leq k_1e^{-T_-},  &\\
      c(k_7)\exp \biggl(
      -S(T_-,T_+,h_1)-S(T_-,T_+,h_2)-(h_1^2+h_2^2)/2(T_+-T_-)\biggr), \\
      \qquad\qquad\qquad\qquad\qquad\qquad \qquad\qquad\qquad\qquad\qquad\qquad
      \text{if } |\theta_1-\theta_2| \geq k_1e^{-T_-}, &
    \end{cases}
    \nonumber
  \end{align}
  where
  \[
    S(T_-,T_+,x)=
    (T_+-T_-)+\sqrt{2}x.
  \]
\end{lemma}
We emphasize that in \eqref{eq-0407aa}, we implicitly used that $k_{i+1}\ll k_i$ in the $o$ notation. We also recall that $T_-<0$, so that the bottom inequality in \eqref{eq-0407ba} represents separation of angles much larger than $k_1$.
\begin{proof}
We write the proof for $\sigma=1$, the case $\sigma=i$ being similar but simpler.  To alleviate notation, we introduce a  useful change of variables.
Recall that
\begin{equation}\label{eq:LUbis}
  \begin{aligned}
    d \mathfrak{L}_t(\theta)
    &=
    {i\theta e^t}{k_1^{-1}} dt
    +\sqrt{\tfrac{4}{\beta}} e^{i \Im \mathfrak{L}_t(\theta)} d \mathfrak{W}_t.
      \end{aligned}
\end{equation}
Let $(W_t^1,W_t^2)=(\Re \mathfrak{W}_t,\Im \mathfrak{W}_t)$ and set
\begin{equation}
  \label{eq-defV}
  d\left(\begin{array}{l} V_t^1\\
    V_t^2
  \end{array}\right)=
  \left(\begin{array}{cc}
    \cos(\Im L_t(0)) & -\sin(\Im L_t(0))\\
    \sin(\Im L_t(0)) & \cos (\Im L_t(0))
  \end{array}
  \right)d\left(\begin{array}{l}
    W_t^1\\
    W_t^2
  \end{array}
    \right).
\end{equation}
Note that the pair of processes $V_t^1,V_t^2$ are
independent standard Brownian motions.
Let $(R_t(\theta),I_t(\theta))$ $=$ $(\Re \mathfrak{L}_t(\theta),
\Im \mathfrak{L}_t(\theta))$, and set $(R_t,I_t)=(R_t(0),I_t(0))$.
We then have, with $\hat I_t(\theta)=I_t(\theta)-I_t$, that
\begin{equation}
  \label{eq-0207a}
  dR_t= \sqrt{\tfrac{4}{\beta}}dV_t^1, \quad dI_t=\sqrt{\tfrac{4}{\beta}}
  dV_t^2,
\end{equation}
and
\begin{eqnarray}
  \label{eq-0207b}
  dR_t(\theta)&=&\sqrt{\tfrac{4}{\beta}}(\cos(\hat I_t(\theta)) dV_t^1+
  \sin(\hat I_t(\theta)) dV_t^2),\\
  d\hat I_t(\theta)&=&\theta e^t k_1^{-1} dt+\sqrt{\tfrac{4}{\beta}}(\sin(\hat I_t(\theta)) dV_t^1-(1-
  \cos(\hat I_t(\theta)) dV_t^2).
  \nonumber
\end{eqnarray}

We begin with the proof of \eqref{eq-0407aa}.
Note that for any $\theta$, we have $t\mapsto \mathfrak{U}^o_t(\theta)=R_t(\theta)$ has the same law as that of $t\mapsto
R_t$, and further $R_t$ is a $\sqrt{4/\beta}$ multiple of a standard Brownian motion. Thus, by standard barrier estimates and the Markov property, see e.g. \cite[Proof of Lemma 5.3]{BRZ}, we  obtain, with $h'$ as in
\eqref{eq-defh}, that
\begin{eqnarray*}
  &&\Pr\biggl( 
    \bigl\{ 
    \mathfrak{U}^o_{T_+}(\theta) -\sqrt{\tfrac{4}{\beta}}h  \in [-k_7,x]
    \bigr\} 
    \cap \mathscr{P}(\theta,h) ~\big\vert~\filt_{n_1^+} 
    \biggr)
  = \frac{(1+o_{k_4})}{\sqrt{2\pi}}
  \frac{h}{(T_+-T_--k_4)^{3/2}}\\
  &&\qquad \times
  \int_{-x}^{k_7} dy \int_{z\in J_{k_4}} dz e^{-(\sqrt{2}(T_+-T_--k_4)+h-z)^2/2(T_+-T_--k_4)}z\tfrac{1}{\sqrt{2\pi k_4}} e^{-( z-y+\sqrt{2}k_4)^2/2k_4}, \
\end{eqnarray*}
where $J_{k_4}=[-k_4^{13/14}, -k_4^{1/14}]$.
We then obtain (recalling that $k_i\gg k_{i+1}$)
\begin{equation}
  \label{eq-0607a}
  \begin{aligned}
  &\Pr\biggl( 
    \bigl\{ 
      \mathfrak{U}^o_{T_+}(\theta) -\sqrt{\tfrac{4}{\beta}}h \in [-k_7,x] 
      \bigr\} 
      \cap \mathscr{P}(\theta,h)
       ~\big\vert~\filt_{n_1^+} 
       \biggr) \\
  &= (1+o_{k_4}) \frac{\sqrt{k_4}}{2{\pi}}
  \frac{h}{(T_+-T_-)^{3/2}}
  e^{-(T_+-T_-)} e^{-\sqrt{2}h -h^2/2(T_+-T_-)}
  \int_{-x}^{k_7}
  e^{\sqrt{2}y}\,dy.
\end{aligned}
\end{equation}

We turn to the two ray estimates; note that we are shooting for a rough bound only, essentially accurate at the exponential scale. We may and will take $\theta_1=0$ and write $\theta=\theta_2$.
Set $ L_t= e^{\lambda (R_t+R_t(\theta))}$ and $\theta_t= \theta e^{t} k_1^{-1}$. We emphasize that the estimates we will obtain will be uniform with respect to the initial condition
$\hat I_{T_-}(\theta)$.
Fix $\tilde  I_t=\hat I_t-\theta_t$. We have from \eqref{eq-0207b} that
$$ dL_t = \sqrt{\tfrac{4\lambda^2 }{\beta}}L_t
\big( (1+ \cos(\theta_t+\tilde I_t)) dV_t^1+\sin(\theta_t+\tilde I_t) dV_t^2\big)+\tfrac{4\lambda^2}{\beta}(1+\cos(\theta_t+\tilde I_t))L_t dt$$
and therefore, setting
$\ell_t=\E L_t$, $\ell_t^c= \E(\cos(\tilde I_t) L_t)$, $\ell_t^s=
\E(\sin(\tilde I_t) L_t)$, we obtain that
\begin{equation}
  \label{eq-0307a}
  \frac{d\ell_t}{dt}=\tfrac{4\lambda^2}{\beta}\big(\ell_t+\cos(\theta_t) \ell_t^c-
  \sin(\theta_t) \ell_t^s\big).
\end{equation}
On the other hand, for some explicit martingale $M_t$,
\begin{eqnarray*}
  d(\cos(\tilde I_t) L_t)&=& dM_t
  +
\tfrac{4\lambda^2}{\beta}(1+\cos(\theta_t+\tilde I_t))\cos(\tilde I_t)L_t dt-
\tfrac{4}{\beta}\cos(\tilde I_t) (1-\cos(\theta_t+\tilde I_t))L_t dt\\
&&- \tfrac{4\lambda}{\beta} \sin(2(\theta_t+\tilde I_t))\sin(\tilde I_t)L_t dt,
\end{eqnarray*}
with a similar expression for $d(\sin(\tilde I_t) L_t)$.
Therefore, there exists a constant $\alpha=\alpha(\lambda,\beta)$
so that
\begin{equation}
  \label{eq-0307b}
  \left|\tfrac{d\ell_t^c}{dt}\right|\leq \alpha \ell_t,\quad
  \left|\tfrac{d\ell_t^s}{dt}\right|\leq \alpha \ell_t.
\end{equation}
 From now on, we take $4\lambda^2/\beta=a$ ($a\sim2$ is the exponent we will need for Chebychev's inequality).
  Define  $\hat \ell_t =\ell_t e^{-a (t-T_-)}$, and similarly
  $\hat \ell_t^c,\hat\ell_t^s$.
  Then,
  $$  \frac{d \hat \ell_t}{dt}=a  e^{-a (t-T_-)} \big(\cos(\theta_t)\ell_t^c-\sin(\theta_t) \ell_t^s\big)=
  a   \big(\cos(\theta_t)\hat \ell_t^c-\sin(\theta_t) \hat \ell_t^s\big) ,$$
  and therefore, with $t\geq  T_{-}$,
  $$ \hat \ell_{t}-\hat \ell_{T_-}=a
  \int_{T_-}^t \big(\cos(\theta_u)\hat \ell_u^c-\sin(\theta_u) \hat
  \ell_u^s\big)du.$$
Using the change of variables
  $e^{u}=s$, we have that $\theta_u=\theta s/k_1$ and
  thus we obtain that
  \begin{equation}
    \label{eq-0307c}\hat \ell_t-\hat \ell_{T_-}=
 a \int_{e^{T_-}}^{e^{t}} s^{-1}\big(\cos(\theta s/k_1) \hat\ell^c_{\log s}-\sin(\theta s/k_1) \hat \ell^s_{\log s}\big) ds.
\end{equation}
Consider first the case $\theta e^{T_-}/k_1> c$ for some fixed constant
$c>1$, that is $T_*\leq T_--\log c$.
Set
$T_{\ell}=T_-+\ell$, $\ell\geq 0$. We consider first $s\in [e^{T_\ell},e^{T_{\ell+1}}] \eqqcolon I_\ell$.
  Let
  \[F_c(s)=\int_{e^{T_\ell}}^s u^{-1}\cos(\theta u/k_1)du,
  \quad
   F_s(s)=\int_{e^{T_\ell}}^s u^{-1}\sin(\theta u/k_1)du.\]
   Then,
   for $s\in I_\ell$, we have (using integration by parts)  that
   $$ |F_c(s)|, |F_s(s)|\leq 3 \tfrac{k_1 e^{-T_\ell} }{\theta}=:b_\ell.$$
   Performing in \eqref{eq-0307c} integration by parts and using \eqref{eq-0307b} and that $|\hat \ell_t^c|,|\hat \ell_t^s|\leq \hat \ell_t$, we obtain
   (using that $4ab_\ell\ll 1$ for $\lambda$ bounded by say $4$ if $c$ is chosen large enough)
   that
   for $e^t\in I_\ell$,
   $$ \hat \ell_t \leq \hat \ell_{T_\ell}+4ab_\ell \hat \ell_t+4\alpha ab_\ell
\int_{e^{T_\ell}}^{e^t} s^{-1}\hat \ell_{\log s} ds
   =\hat \ell_{T_\ell}+4ab_\ell \hat \ell_t+ 4\alpha ab_\ell
\int_{T_\ell}^t \hat \ell_u du .$$
   An application of Gronwall's inequality then gives that
   for such $t$,
   $$\hat \ell_{t} \leq  \tfrac{1}{1-4ab_\ell} \hat \ell_{T_\ell}e^{4ab_\ell\alpha(t-T_\ell)/(1-4ab_\ell)}.$$
Unravelling the definitions, we have obtained that for $e^t\in I_\ell$,
\begin{equation}
\label{eq-0407a}
\E\big(e^{\lambda(R_t+R_t(\theta))}\big)=\ell_t \leq
\left(\prod_{j=1}^\ell \tfrac{1}{1-4ab_j}\right)e^{a(t-T_-)+4a\sum_{j=1}^\ell a b_j
\alpha(T_{j+1}-T_j)/(1-4ab_j)}
\leq C e^{a(t-T_-)},
\end{equation}
uniformly in the initial conditions of $\hat I_{T_-}(\theta)$,
where $C$ is a universal constant independent of $\ell$ and we used that
$\sum_{j=1}^\ell a b_j
\alpha(T_{j+1}-T_j)/(1-4ab_j)$ is uniformly bounded.
In particular, for $\theta_1,\theta_2$ with $|\theta_1-\theta_2|\geq c
k_1 e^{-T_-}$,
using Chebycheff's inequality with $\lambda\sim 2$
we obtain that
\begin{eqnarray}
  \label{eq-finallargegap}
  &&\Pr (R_{{T_+}}(\theta_i)\geq \sqrt{8/\beta} ((h_i-k_7)/\sqrt{2}+(T_+-T_-)), i=1,2)\\
  &&\leq C \nonumber
  e^{-
  \sum_{i=1}^2 \big(\sqrt{2} (T_+-T_-)+(h_1-k_7)\big)^2/2(T_+-T_-) },
\end{eqnarray}
which yields the bottom inequality in \eqref{eq-0407ba}.

We next turn to the case {$T_+-k_4\geq T_*>T_- -\log c$}. In that case, we apply the barrier estimate for time $(T_*-T_- -\log c)_+$ and the same computation
as above for larger times. That is, assuming without loss of generality that $h=h_1$, we write
 $T_*'=(T_*\vee T_-)$ and, with $\mathscr{P}(\theta_i,h_i)^{T_*'}$ the part of the barrier event $\mathscr{P}(\theta_i,h_i)$ up to time $T_*'$,
\begin{eqnarray}
\label{eq-240522}
  &&\Pr (R_{{T_+}}(\theta_i)\geq \sqrt{8/\beta} ((h_i-k_7)/\sqrt{2}+(T_+-T_-))\cap \mathscr{P}(\theta_i,h_i), i=1,2)
\\
&&\leq \int  \Pr(\sqrt{\beta/8}R_{{T_*'}}(\theta_1)\in dz\cap\mathscr{P}(\theta_1,h_1)^{T_*'} )\times\nonumber\\
&&\qquad
\Pr\Big(R_{T_+}(\theta_1)-R_{T_*'}(\theta_1)\geq  \sqrt{8/\beta} ((h_1-k_7)/\sqrt{2}+(T_+-T_-)-z),\nonumber\\
&&\qquad \qquad \qquad
R_{T_+}(\theta_2)-R_{T_*'}(\theta_2)\geq \sqrt{8/\beta} (-k_7/\sqrt{2}+(T_+-T_*)+T_*^{1/14}\Big).
\nonumber
\end{eqnarray}
Note that the integration over the variable $z$ in the right side of \eqref{eq-240522} is restricted in particular to the range $z\in [(T_+-T_*')^{1/14},(T_+-T_*')^{13/14}]$.
We control the second probability in the right hand side of \eqref{eq-240522} as in the case $T_*\leq T_--\log c$, while the first probability is controlled by a standard
Gaussian bound. This yields the top inequality in \eqref{eq-0407ba}. Further details are omitted.
\end{proof}

\subsection{Regularity of the kernel}
\begin{lemma}
  Let $F$ be the subset of $\mathcal{C}([-2\pi k_1,0],\mathbb{C})$ for which $\max_\theta |f(\theta)| > e^{-k_7}$ (which corresponds to those decorations in $\Gamma_{k_7}^+$).
  Let $\mathfrak{r}(x,\alpha,f)$ be the Radon--Nikodym derivative of the measure $\mathfrak{s}(\alpha+x, e^{\sqrt{4/\beta}(\alpha+x)}df)$ with respect to $\mathfrak{s}(x,e^{\sqrt{4/\beta}x}df)$ on $F$.  Then, uniformly in $|x|\leq (\log k_1)^{17/18}$ and $|\alpha| \leq 1$,
  \[
    \mathfrak{r}(x,\alpha,f)=e^{\sqrt{2}\alpha}(1+o_{k_1}).
  \]
  \label{lem:kregular}
\end{lemma}
   The same holds for $\mathfrak{p}$ trivially by integrating over the random phase:
   \begin{corollary}\label{cor:pregularity}
   Using the same $F$, and for the same set of $|x| \leq (\log k_1)^{17/18}$,
   the Radon-Nikodym derivative $\mathfrak{t}(x,\alpha,f)$ of $\mathfrak{p}(\alpha+x, e^{\sqrt{4/\beta}(\alpha+x)}df)$ with respect to $\mathfrak{p}(x,e^{\sqrt{4/\beta}x}df)$ satisfies
   \[
     \mathfrak{t}(x,\alpha,f) = e^{\sqrt{2}\alpha}(1+o_{k_1}).
   \]
 \end{corollary}
 \begin{proof}[Proof of Lemma \ref{lem:kregular}]
We recall the definition of $\mathfrak{U}_t^{o,j}(h)$
  from below \eqref{eq:decoration}.
  In particular, we have a decomposition
  \begin{equation}\label{eq:Djo_1old}
    \begin{aligned}
    &\mathfrak{U}_{T_+}^{o}(\theta)
    =\mathfrak{U}_{T_\dagger}^o(0)
    +(\mathfrak{U}_{T_+}^o(\theta)-\mathfrak{U}_{T_\dagger}^o(0)), \\
    &\mathfrak{L}_{T_+}^{o}(\theta)
    =\mathfrak{L}_{T_\dagger}^o(0)
    +(\mathfrak{L}_{T_+}^o(\theta)-\mathfrak{L}_{T_\dagger}^o(0)), \\
    \end{aligned}
  \end{equation}
  where the entire process $\{(\mathfrak{L}_{T_+}^o(\theta)-\mathfrak{L}_{T_\dagger}^o(0)) : \theta\}$ is independent of $\mathfrak{L}_{T_\dagger}^o(0)$ (and hence also $\mathfrak{U}_{T_\dagger}^o(0)$).
  Then $D_j^o$ can be decomposed as
  \begin{equation}\label{eq:Djo_1}
    D_j^o(x)
    \eqqcolon
    e^{\mathfrak{U}_{T_\dagger}^o(0)}
    \times
    \widehat{D}(x+\sqrt{\tfrac{\beta}{4}}\mathfrak{U}_{T_\dagger}^o(0)),
  \end{equation}
  which makes $\widehat{D}(x)$ independent of $\mathfrak{U}_{T_\dagger}^o(0)$.  Let $\widehat{\mathfrak{s}}(x,\cdot)$ be the law of $\widehat{D}(x)$.
  Note to have $D_j^o(x) \in e^{\sqrt{4/\beta}x}F$, we must have
  \[
    \sqrt{\tfrac8\beta}\mathcal{A}_{T_\dagger}^-
    \leq
    \sqrt{\tfrac4\beta}x
    +
    \mathfrak{U}_{T_\dagger}^{o}(0)
    \leq
    \sqrt{\tfrac8\beta}\mathcal{A}_{T_\dagger}^+,
  \]
  as if not, then by definition $D_j^o \equiv 0$.  Define the interval 
  \[
    J = 
    \sqrt{2}(T_\dagger - T_-)
    +
    [
      \sqrt{2}\mathcal{A}_{T_\dagger}^-,
      \sqrt{2}\mathcal{A}_{T_\dagger}^+
    ],
  \] 
  which has width which is at most $O( (\log k_1^+)^{13/14})$ and allows us to write the previous display as $x+\sqrt{\tfrac{\beta}{4}}\mathfrak{U}_{T_\dagger}^{o}(0) + \sqrt{2}(T_\dagger - T_-) \in J$.
  Thus we have a representation for $f \in F$, with $T=T_\dagger-T_-$, and where integrate over the law of $y = x+\sqrt{\tfrac{\beta}{4}}\mathfrak{U}_{T_\dagger}^{o}(0) + \sqrt{2}T$,
  \[
    \begin{aligned}
    \mathfrak{s}(x,e^{\sqrt{{4}/{\beta}}x}df)
    &=
    \int_{J}
    e^{-(y-x-\sqrt{2}T)^2/(2T)}\widehat{\mathfrak{s}}(y-\sqrt{2}T,e^{\sqrt{{4}/{\beta}}(y-\sqrt{2}T)} df)
    \,\frac{dy}{\sqrt{2\pi T}}.
    \end{aligned}
  \]
  
  Hence this probability is differentiable in $x$ and in fact
  \[
    \frac{d}{dx} \mathfrak{s}(e^{\sqrt{{4}/{\beta}}x}df)
    =
    \int_{J}
    (\sqrt{2}-(x-y)/T)
    e^{-(y-x-\sqrt{2}T)^2/(2T)}\widehat{\mathfrak{s}}(y-\sqrt{2}T,e^{\sqrt{{4}/{\beta}}(y-\sqrt{2}T)} df)
    \,\frac{dy}{\sqrt{2\pi T}}.
  \]
  On the assumptions on $x,y$, the factor $(\sqrt{2}-(x-y)/T)$ is bounded uniformly over the $x$ and $y$ considered by $\sqrt{2} + (\log k_1)^{-\delta}$ for some $\delta>0,$ and hence we conclude
  \[
    \frac{d}{dx} \mathfrak{s}(x,e^{\sqrt{{4}/{\beta}}x}df)
    =
    (\sqrt{2} + o_{k_1})\mathfrak{s}(x,e^{\sqrt{{4}/{\beta}}x}df).
  \]
  Hence from Gronwall's inequality, we conclude the claimed estimate.
\end{proof}

\subsection{Mixing of phase}

The final technical estimate that must be done for the decoration, which only concerns the case $\sigma=1$, is the mixing of phase.  Specifically, we must estimate the statistical difference between
\[
  \mathfrak{s}(x,e^{i\alpha+x} df)
  \quad\text{and}\quad
  \Exp\mathfrak{s}(x,e^{i 2\pi U + x} df),
\]
where $U \lawequals \Unif([0,1])$.  This is the measure that appears in $\overline{\mathfrak{m}}$ in \eqref{eq:2ndintensityprime}.
Once more we let $F$ be the subset of $\mathcal{C}([-2\pi k_1,0],\mathbb{C})$ for which $\max_\theta |f(\theta)| > e^{-k_7}$.

We use a representation that is similar to \eqref{eq:Djo_1}, namely that
\begin{equation}\label{eq:Djo_2}
    D_j^o(x) \eqqcolon
    e^{i W}
    \times
    \widecheck{D}(x),
  \end{equation}
  where $W$ is a centered Gaussian variable, independent of $\widecheck{D}$ of variance $\tfrac{4}{\beta}(T_\dagger-T_-)$, corresponding to the imaginary part of the increment of $\mathfrak{L}_{T_\dagger} - \mathfrak{L}_{T_-}$.
  Now by virtue of the variance being large, we can make a total variation comparison of $e^{i(\alpha + W)}$ to $e^{iU}$.
\begin{lemma}
  With $\operatorname{dtv}$ the total variance distance, for a real Gaussian
  $Z$
  of variance $V$ and a uniform variable $U$ on $[0,1]$, there is an absolute constant $C>0$ so that
  \[
    \sup_{\alpha \in \R} \operatorname{dtv}( e^{i(\alpha + Z)}, e^{i2\pi U}) \leq
    Ce^{-2\pi^2 V}.
  \]
  \label{lem:dtv}
\end{lemma}
\begin{proof}
  By periodicity, we may take $\alpha \in [0,2\pi]$.  We compute the Fourier coefficients of the law of $e^{i(\alpha + Z)}$.  Note that the $k$-th coefficient is given by
  \[
    \Exp e^{i2\pi k(\alpha + Z)}
    =e^{i2\pi k \alpha} e^{ -2\pi^2 k^2 V}.
  \]
  As the total variation distance is given by integrating the
  $\text{L}^1$--distance of the densities, we have
  \[
    \sup_{\alpha \in \R} \operatorname{dtv}( e^{i(\alpha + Z)}, e^{i2\pi U}) \leq
    {2\pi}
    \sum_{k \neq 0} \bigl|\Exp e^{i2\pi k(\alpha + Z)}\bigr|
    \leq
    Ce^{-2\pi^2 V}.
  \]
\end{proof}
This leads to:
\begin{corollary}
\label{cor:dtv}
Over any subset $F \subseteq \mathcal{C}( (-2\pi k_1,0), \C)$ for which $e^{i\alpha}F = F$ for any $\alpha \in \R$,
\[
  \sup_{\alpha \in \R}
  d_{\operatorname{BL}}
  (\mathfrak{s}(x,e^{i\alpha+x} df),
  \Exp\mathfrak{s}(x,e^{i 2\pi U + x} df)
  )
  \leq
  Ce^{-(8/\beta)\pi^2 (T_\dagger - T)}
  \mathfrak{s}(x,e^{x}F).
\]
\end{corollary}
\begin{proof}
  Rescaling both sides of the equation by $\mathfrak{s}(e^{x}F),$ we then have a bound for the bounded variation distance of two probability measures,
  \[
    \mathfrak{s}(x,e^{i\alpha+x}df)/\mathfrak{s}(x,e^{x}F)
    \quad\text{and}\quad
    \Exp\mathfrak{s}(x,e^{i 2\pi U + x} df)/\mathfrak{s}(x,e^{x}F).
  \]
  From \eqref{eq:Djo_2}, we have that the phase $e^{iW}$ is independent of $\widecheck{D}$, and hence we have a coupling of these two measures that holds with probability $1-\operatorname{dtv}( e^{i(\alpha + Z)}, e^{i2\pi U})$.  By the definition of bounded--Lipschitz metric, the claimed bound follows.
\end{proof}

\appendix

{\bf  Appendices}
\vspace{-0.5cm}
\section{Point processes}
\label{sec:pointprocesses}

We record some elementary properties about point processes, and in particular approximation of point processes by Poisson processes.
We will let $\Gamma$ be a complete separable metric space, with a metric $\partial_0$ which is bounded by $1.$
We will use the metrics $\partial_1$ and $\partial_2$ from the introduction.
We also define two metrics over finite point measures on $\Gamma$ and their laws: for any finite point measures
 $\xi_1 = \sum_{i=1}^m \delta_{y_i}$ and $\xi_2 = \sum_{i=1}^n \delta_{z_i}$  
on  $\Gamma$,
\[
  d_{1}(\xi_1, \xi_2) \coloneqq
  \begin{cases}
    0, & \text{if } m = n = 0, \\
    \min_{\pi} \frac{1}{n}\sum_{i=1}^n \partial_0(y_i, z_{\pi(i)}), & \text{if } m = n > 0, \\
    1, & \text{if } m \neq n, \\
  \end{cases}
\]
with the minimum being 
 over all permutations of $\{1,2,\dots,n\}.$
Finally, for two laws on point processes $Q_1$ and $Q_2$ we let
\[
  d_2( Q_1, Q_2) \coloneqq \inf_{(\xi_1,\xi_2)} \Exp( d_1( \xi_1, \xi_2)).
\]
We note that it is possible to bound
\(
\partial_1(\xi_1,\xi_2) \leq d_{1}(\xi_1, \xi_2)\max\{|\xi_1|,|\xi_2|\}
\)
where $|\xi_j|$ is the cardinality of the point set for $j=1,2$,
and hence
\begin{equation}\label{eq:d2del2}
  \partial_2( Q_1, Q_2)
  \leq
  \inf_{(\xi_1,\xi_2)}
  \biggl\{
  \sqrt{\Exp( d_1^2( \xi_1, \xi_2))}
  \sqrt{\vspace{3em}\Exp( |\xi_1|^2 + |\xi_2|^2)}
\biggr\}.
\end{equation}

We will formulate a corollary of \cite{ChenXia} for Poisson process approximation with local dependency structure.
We let $\left\{ \Xi_i : i \in \mathcal{I} \right\}$ be a set of Bernoulli point processes, so that $\Xi_i(\Gamma) \in \{0,1\}.$   We will suppose that 
elements of  this set have local dependence, in that for each $i \in \mathcal{I}$ there are sets $A_i \subseteq B_i \subseteq \mathcal{I}$ with $i \in A_i$ so that
\begin{equation}\label{eq:PPndp}
  \forall~i \in \mathcal{I}
  \quad
  \{ \Xi_j : j \in A_i \}
  \quad
  \text{is independent of}\quad
  \{ \Xi_j : j \in \mathcal{I}\setminus B_i\}.
\end{equation}
For all $i \in \mathcal{I}$  we let ${\lambda}_i$ be the intensity measure of $\Xi_i$, and we set $\Xi := \sum_{i \in \mathcal{I}} \Xi_i$,  ${\lambda} \coloneqq \sum_i \mathbf{\lambda}_i$ and $\Lambda \coloneqq {\lambda}(\Gamma).$
\begin{theorem}\label{thm:PP}
  Suppose $\left\{ \Xi_i : i \in \mathcal{I} \right\}$ satisfy \eqref{eq:PPndp}, and define for any $i \in \mathcal{I},$
  \[
    T_i \coloneqq \sum_{j \in A_i \setminus \{i\}} \Xi_i(\Gamma),
    \quad
    P_i \coloneqq \Xi_i(\Gamma),
    \quad
    \text{and}
    \quad
    L_i \coloneqq  \sum_{j \in \mathcal{I}\setminus B_i}\Exp \bigl(P_j\bigr).
  \]
  Then there is a numerical constant $C>0$ so that for a Poisson process $\Pi$ with intensity $\lambda,$
  \[
    \partial_2( \Xi, \Pi) \leq
    C
    (\Var(\Xi(\Gamma)) + 3\Lambda)^{3/2}
    \sum_{i \in \mathcal{I}}
    \biggl( \frac{ \Exp(P_i)\Exp(T_i) + \Exp( T_i P_i) + (\Exp( P_i) )^2}{ L_i^2 }\biggr).
  \]
\end{theorem}
\begin{proof}
  We start by using \cite[(2.6)]{ChenXia} to establish that there is a numerical constant $C>0$ so that
  \begin{equation}\label{eq:CX1}
    d_2( \Xi, \Pi) \leq
    C\sum_{i \in \mathcal{I}}
    \biggl( \frac{1}{ L_i } + \frac{\Var(\Xi(\Gamma))}{L_i^2} \biggr)
    \cdot
    \bigl( \Exp(P_i)\Exp(T_i) + \Exp( T_i P_i) + (\Exp( P_i) )^2\bigr).
  \end{equation}
  We note that in the notation of that equation $|V_i| = T_i,$ and their  
  {\boldmath{$\lambda$}}
  are our $\lambda.$  Finally, we bound their $\kappa_i$ by
  \(
  \Var( |\Xi(\Gamma)|)/L_i,
  \)
  and we bound $\sqrt{\kappa_i}$ by $1+\kappa_i.$

  We also note that \cite{ChenXia} requires that $\Gamma$ is a locally compact metric space.  As $\lambda$ is a finite Borel measure on $\Gamma,$ it is tight.  Hence there are compact sets $K_i$ so that with $\tilde{\Gamma} =\cup_{i=1}^\infty K_i,$  $\lambda( \Gamma \setminus \tilde{\Gamma} ) = 0.$ The space $\tilde{\Gamma}$ is locally compact, and both point processes $\Xi$ and $\Pi$ put $0$ mass on $\Gamma \setminus \tilde{\Gamma}$ with probability $1.$  Hence we may apply the Theorem of \cite{ChenXia} to $\Xi$ and $\Pi$ considered as point processes on $\tilde{\Gamma}$ with the same metrics $\partial_0,d_1$ and $d_2$.  Furthermore, coupling the point processes as random measures on $\tilde{\Gamma}$ gives a coupling of the point processes as random measures on $\Gamma,$ and hence \eqref{eq:CX1} follows.

  From \eqref{eq:d2del2}, and using that $d_1$ is bounded by $1,$
  \[
    \partial_2
    ( \Xi, \Pi)
    \leq
    d_2
    ( \Xi, \Pi)
    \sqrt{\Var(\Xi(\Gamma)) + 3\Lambda}.
  \]
  As $L_i \leq \Lambda,$ the claim follows.
\end{proof}

We also formulate a change-of-intensity lemma that controls the distance between two Poisson processes of similar intensity.  Define for any two finite Borel measures $\pi$ and $\lambda$ on $\Gamma,$
\begin{equation}\label{eq:BL}
  d_{\operatorname{BL}}(\pi,\lambda) = \sup_{ f } \biggl|\int_{\Gamma} f d(\pi -\lambda)\biggr|,
\end{equation}
with the supremum over all $f$ with both $|f(x)-f(y)| \leq \partial_0(x,y)$ for all $x,y \in \Gamma$ and $|f(x)| \leq 1$ for all $x \in \Gamma.$  Then:
\begin{theorem}\label{thm:PPcom}
  Let $\pi$ and $\lambda$ be two finite measures on $\Gamma.$
  Let $\alpha = \min\{ \pi(\Gamma), \lambda(\Gamma) \}.$  Then
  \[
    d_2(\Pi(\pi), \Pi(\lambda)) \leq \tfrac{1-e^{-\alpha}}{\alpha} d_{\operatorname{BL}}(\pi,\lambda)
    \quad\text{and}\quad
    \partial_2(\Pi(\pi), \Pi(\lambda)) \leq \sqrt{2}\max\{ \pi(\Gamma), \lambda(\Gamma) \}\tfrac{1-e^{-\alpha}}{\alpha} d_{\operatorname{BL}}(\pi,\lambda).
  \]
\end{theorem}
\begin{proof}
  The first is nearly a corollary of \cite[Theorem 1.2]{BrownXia}.
  To derive it, apply Lemma 2.6 directly to the second line of the displayed equation at the top of \cite[p.259]{BrownXia}.
  The second follows from the first by using \eqref{eq:d2del2}.
\end{proof}
\begin{corollary}\label{cor:PPcom}
  For any compact $\Gamma$, and for any $\epsilon > 0$, there is a finite list $f_1,f_2, \dots, f_n$ of non-negative functions $f_j$ with $|f_j(x)-f_j(y)| \leq \partial_0(x,y)$ and $|f_j(x)| \leq 1$ for all $x,y \in \Gamma$ and a $\delta$ so that if
  \[
    \max_{j=1,\dots,n} \biggl|\int_\Gamma f_jd(\pi-\lambda)\biggr| \leq \delta
  \]
  then
  \(
    \partial_2(\Pi(\pi), \Pi(\lambda)) \leq \epsilon.
  \)
\end{corollary}
\begin{proof}
  By Arzel\`a--Ascoli, we can find a list $\{g_j\}$ for $j=1,\dots,n$ so that for all $1$--Lipschitz, 1--bounded $f$,
  \[
    \min_j \sup_{x \in \Gamma} |f(x)-g_j(x)|  \leq \epsilon/2.
  \]
  Let $\{f_j\}$ be all the positive and negative parts of the functions in this list.
  We may also assume the constant-1 function is in the list.  Now suppose that
  \[
    \max_{j=1,\dots,n} \biggl|\int_\Gamma f_jd(\pi-\lambda)\biggr| \leq \delta.
  \]
  Then $|\pi(\Gamma)-\lambda(\Gamma)|\leq \delta$, and so by Theorem \ref{thm:PPcom}
  \[
    \partial_2(\Pi(\pi), \Pi(\lambda)) \leq \sqrt{2}\max\{ \pi(\Gamma), \lambda(\Gamma) \}\tfrac{1-e^{-\alpha}}{\alpha} d_{\operatorname{BL}}(\pi,\lambda)
    \leq \sqrt{2}\tfrac{\max\{ \pi(\Gamma), \lambda(\Gamma) \}}{\max\{ \pi(\Gamma), \lambda(\Gamma) \}-{\delta}}(\epsilon/2 + \delta).
  \]
  If $\max\{ \pi(\Gamma), \lambda(\Gamma) \}=0,$ there was nothing to prove in the first place.  Otherwise by taking $\delta$ sufficiently small, we conclude the claim.
\end{proof}

\def\Nc{{\mathcal N}}

\def\e{\mathtt e}
\def\k{\mathtt k}
\def\half{\frac{1}{2}}
\def\quart{\frac{1}{4}}
\def\I{\mathfrak I}
\def\E{{\mathbb E}}
\def\R{{\mathbb R}}
\def\C{{\mathbb C}}
\def\D{{\mathbb D}}
\def\G{{\mathbb G}}
\def\Lc{{\mathcal L}}
\def\Gc{{\mathcal G}}

\def\P{{\mathbb P}}

\section{Auxilliary one and two-rays estimates}
\label{section:lower_bound}

We collect in this appendix the key two-rays estimates used in the bulk 
of the evolution. The argument were developed in details in \cite{CMN}, 
and in fact even our tex file is based on theirs. The 
hurried reader may skip the content of
the appendix, keeping only the statements of Propositions
\ref{Prmomen1Lower}, \ref{Prmomen1Lowerbis} and Lemmas 
 \ref{lemma:kDebut}, \ref{lemma:ktoutDebut}, \ref{lemma:kMil},
 \ref{lemma:kDebutbis} and \ref{lemma:ktoutDebutbis}
in mind.

Throughout, we adopt exactly, for ease of reference, 
the notation of \cite{CMN}.

\paragraph{Further notation:}
For any $n\geq p$, $\theta\in [0,2\pi)$, we define
\begin{align*}
Z_n^{(p)}(\theta) \coloneqq Z_n(\theta) - Z_p(\theta) =  \sum_{j=p}^{n-1} \frac{ \Nc_j^\C }{\sqrt{j+1}} e^{i((j+1)\theta + A_j(\theta))} .
\end{align*}
We stress that $Z_n^{(p)}$ depends implicitly of $\beta$ because of the Pr\"ufer phase. Similarly,
we define
\begin{align*}
A_n^{(p)}(\theta) \coloneqq \sum_{j=p}^{n-1} a_j(\theta) = A_n(\theta) - A_p(\theta),\ .
\end{align*}
where
$$a_j(\theta) \coloneqq A_{j+1}(\theta) - A_j(\theta).$$
 More generally, for any family of quantities depending on an index $k$, we will denote the difference
 of the quantities indexed by $k$ and $p$ by the same notation, with $k$ as an index and $p$ as
 an superscript.

In the following, it will be convenient to study the field at times which are powers of $2$. 

\subsection{The Chhaibi-Madaule-Najnudel coupling}
\label{section:new_coupling}

As in \cite{CMN},
we introduce a new process in order to gain more independence. Recall that 
$N=\log_2 n$.

\paragraph{A more independent field:}
For each fixed $\theta$,  $(Z_{2^k}(\theta))_{k\geq 0}$ is a complex Gaussian random walk. Moreover we
could compute the correlations of $Z_{2^k}(\theta)$ and $Z_{2^k}(\theta')$ and observe that
they behave logarithmically with respect to the distance between $\theta$ and $\theta'$ modulo $2  \pi$.
However, $Z$ is not globally Gaussian, so we cannot directly apply known results on the maximum of Gaussian fields,
but we will still provide its approximative branching structure. To achieve this aim we will  gain
some independence by making small changes on $Z$.

Let us fix some integer $r$, which will be  assumed to be larger than some suitable universal constant.
For $l \geq r$, we denote 
$\Delta=\Delta^{(l)} \coloneqq e^{\sqrt{\log r}}
+ 100\lfloor  \log^2 l \rfloor $. Observe that for any $N \geq r$, we can rewrite formula \cite[(3.3)]{CMN} as:
\begin{align}
Z_{2^{N}}^{(2^r)}(\theta) \coloneqq & \sum_{l=r}^{N -1 } \sum_{p=0}^{2^{\Delta}-1}  \left(\sum_{j=0}^{2^{l-\Delta}-1}
\Nc^\C_{2^{l}+p 2^{l-\Delta} + j}
\frac{e^{i\psi_{ 2^{l}+p 2^{l-\Delta}+j}(\theta )}}
     {\sqrt{ 2^{l}+p 2^{l-\Delta} + j +1 }}
\right).
\end{align}
Note that $\Delta^{(l)}$ and $l - \Delta^{(l)} $
are strictly positive if $l \geq r$ and $r$ is large enough.
Now, let $Z^{(2^r,\Delta)}$ be the process defined by
\begin{align}
Z_{2^{N}}^{(2^r,\Delta)}(\theta) \coloneqq & \sum_{l=r}^{N -1 } \sum_{p=0}^{2^{\Delta}-1}  \left(\sum_{j=0}^{2^{l-\Delta}-1}    \Nc^\C_{2^{l}+p 2^{l-\Delta} + j} e^{i j\theta } \right)
\frac{e^{i\psi_{ 2^{l}+p 2^{l-\Delta}}(\theta )}}
     {\sqrt{ 2^{l}+p 2^{l-\Delta}}}
.
\end{align}
Observe that $Z_{2^{N}}^{(2^r)}(\theta)$ and $Z_{2^{N}}^{(2^r,\Delta)}(\theta)$ only differ
by the change in the square root of the denominator, and by the replacement of
some increments of the  Pr\"ufer phases by their mean. The following is a slight variation of \cite[Proposition 5.2]{CMN}; the additional statement
\eqref{eq-addstat} actually follows from the proof 
in \cite{CMN}.
\begin{proposition}
For $r$ large enough,
	\label{proposition:new_coupling}
	\begin{align*}
\sum_{l\geq r} \P\left(	\sup_{\theta \in [0,2\pi)} | Z_{2^{l+1}}^{(2^{l},\Delta)}(\theta) - Z_{2^{l+1}}^{(2^{l})}(\theta)|\geq 2 l^{-2} \right) <+\infty
	\end{align*}
In particular, as $\sum_{l\geq r } 2 l^{-2}<+\infty$, we have that almost surely,
\begin{align}
\sum_{l \geq r} \sup_{\theta \in [0,2\pi)}
|  Z_{2^{l+1}}^{(2^{l},\Delta)}(\theta) - Z_{2^{l+1}}^{(2^{l})}(\theta) | < \infty \ 
\end{align}
and
\begin{align}
  \label{eq-addstat}
  \limsup_{r\to\infty}
  \sum_{l \geq r} \sup_{\theta \in [0,2\pi)}
|  Z_{2^{l+1}}^{(2^{l},\Delta)}(\theta) - Z_{2^{l+1}}^{(2^{l})}(\theta) | =
0.
\end{align}
\end{proposition}
In the proof of Proposition \ref{proposition:new_coupling}, 
\cite{CMN} introduce the variables,
for any $l\geq 0$, $\theta\in [0,2\pi)$,
\begin{align}
\label{tighti2}
\blacksquare_l(\theta) & \coloneqq  Z_{2^{l+1}}^{(2^{l},\Delta)}(\theta) - Z_{2^{l+1}}^{(2^{l})}(\theta)
\\
\nonumber &  = \sum_{p=0}^{2^{\Delta}-1}
               \frac{ e^{i\psi_{ 2^{l}+p 2^{l-\Delta} }(\theta )} }
                    { \sqrt{{ 2^{l}+p 2^{l-\Delta}}} }
               \left(\sum_{j=0}^{2^{l-\Delta}-1}
               \Nc^\C_{2^{l}+p 2^{l-\Delta} + j} e^{i j \theta }   \square_j^{ (2^{l}+p 2^{l-\Delta}) }(\theta)
               \right) \ ,
\end{align}
with
\begin{align}
\label{eq:ineq_square}
       \left|\square_j^{(2^{l}+p 2^{l-\Delta} )}(\theta)\right|
\leq &  \left| A_{j+ 2^{l}+p 2^{l-\Delta} }^{(2^{l}+p 2^{l-\Delta})}(\theta) \right| + 2^{-\Delta}.
\end{align}
The main steps in the proof consist of the estimates
\begin{align}
\label{eq:one_point_square}
\P\left( | \blacksquare_l(\theta) | \geq l^{-2} \right)
\ll_{\beta} e^{- 2^{ \quart \Delta} }.
\end{align}
 and, for $\lambda \in \R$:
\begin{align}
\label{eq:laplace_blacksquare}
\sup\left(
\E\left( e^{ \lambda \Re  \blacksquare_l(\theta)} \mathds{1}_{G_l(\theta)}  \Big| \mathcal{G}_{2^l}\right),
\E\left( e^{ \lambda \Im  \blacksquare_l(\theta)} \mathds{1}_{G_l(\theta)}  \Big| \mathcal{G}_{2^l}\right)
\right)
\leq &
e^{\lambda^2 \left( 2^{-\quart \Delta} + 2^{-\Delta} \right)^2 } \ .
\end{align}
(see \cite[(5.6), (5.8)]{CMN}.)
\subsection{One and two ray estimates}
\label{subsection:second_moment}

In the sequel, for all fields denoted by $Z$ with some indices and superscripts, we write $R$ with the same indices
and superscripts for the real part of $\sigma$ times the initial field (recall that $\sigma \in \{1,i,-i\}$).
\paragraph{An envelope for the paths of $R_{2^{N}}^{(2^r,\Delta)}(\theta) $:} For $j \geq r$, let
\begin{align}
\label{eq:def_tau_r}   \tau^{(r)}_j& \coloneqq  \sum_{l=r}^{j-1} \sum_{p=0}^{2^{\Delta^{(l)} }-1}
\frac{2^{l-\Delta^{(l)}}}{2^l +p2^{l-\Delta^{(l)}}}=  \sum_{l=r}^{j-1} \sum_{p=0}^{2^{\Delta^{(l)} }-1}
\frac{1}{2^{\Delta^{(l)}} + p} =: (j-r) \log 2 + \lambda_{j}^{(r)},
\end{align}
where $(\lambda_j^{(r)})_{j\geq r}$ is a nonnegative and
increasing sequence, tending, when $j \rightarrow \infty$, to a limit $\lambda_{\infty}^{(r)}$ such that
$$\lambda_\infty^{(r)} < \sum_{l\geq r} 2^{-\Delta^{(l)}} \leq  
2^{-e^{\sqrt{\log r}}} 
\sum_{l=1}^{\infty} 
2^{- 100 \lfloor \log^2 l \rfloor }  
 \ll  
 2^{-e^{\sqrt{\log r}}} 
\to 0  $$
when $r$ goes to $\infty$.
 Let $\alpha_+ \coloneqq 1 - \frac{1}{10}$ and $\alpha_-\coloneqq \frac{1}{10}$,
and for $N \geq 2r$, $k \in [|r, N|]$,  define
\begin{align}
  \label{eq-ukplus}
u_k^{(N)} \coloneqq \left\{ \begin{array}{ll} - k^{\alpha_-},\qquad &\text{if  } k\leq \lfloor N/2\rfloor,
\\
 - (N-k)^{\alpha_-} -   \frac{3}{4}  \log N,\qquad &\text{if  } \lfloor N/2\rfloor <k \leq N,
\end{array} \right.
\end{align}
and
\begin{align}
  \label{eq-ukminus}
l_k^{(N)} \coloneqq \left\{ \begin{array}{ll} - k^{\alpha_+},\qquad &\text{if  } k\leq \lfloor N/2\rfloor,
\\
 - (N-k)^{\alpha_+} -    \frac{3}{4} \log  N,\qquad &\text{if  } \lfloor N/2\rfloor <k \leq N.
\end{array} \right.
\end{align}
We then define an envelope by its lower bound and its upper bound at each $k\in [|r,N|]$:
\begin{align*}
U_k^{(N)} \coloneqq     \tau^{(r)}_k+  u_k^{(N)} \quad \text{and    }\quad L_k^{(N)} \coloneqq       \tau^{(r)}_k+ l_k^{(N)}  .
\end{align*}

Fix $x<0$ (depending on $k_2$), $z<0$ (depending on $k_1^+$), and $\upsilon\in (0,1)$. Set $N_+=\log n_1^+$. The basic event  we need to consider is, for $\theta\in [0,2\pi)_{2^N}$,
\begin{align*}
& \I_N(\theta)=\I_N(\theta,x,z) \coloneqq \\
&{\{ \forall k\in [|r,N_1^+|],\,  L_k^{(N)}  \leq x+R_{2^k}^{(2^r,\Delta)}(\theta) \leq U_k^{(N)}, 
x+R_{2^{N_+}}^{(2^r,\Delta)} (\theta)\in \tau^{(r)}_{N_+}-\frac34 \log N +[z,z+\upsilon)\}}.\end{align*}
We  will consider $x\in [- r^{1/20},- r^{19/20}]$ and $z\in [-(k_1^+)^{1/20},-(k_1^+)^{19/20}]$.
The  walk  $(R_{2^k}^{(2^r,\Delta)}(\theta))_{r \leq k \leq N}$ is a Gaussian random walk whose distribution is the same as $\sqrt{\frac{1}{2}} (W_{\tau_j^{(r)}})_{r \leq k \leq N}$.
In the case where an  event  $\I_N(\theta)$ occurs for some $\theta \in [0, 2\pi)_{2^N}$, it means that  
$x+R_{2^k}^{(2^r,\Delta)}(\theta)$ is around $\tau_k^{(r)}$ for $r \leq k \leq N$, i.e. the Brownian motion $W$ is roughly growing linearly with rate $\sqrt{2}$. For this reason, in the sequel of this part of the paper, 
we will often compare the probability of an event $Ev$ 
concerning the random walk $(R_{2^k}^{(2^r,\Delta)}(\theta))_{r \leq k \leq N}$ to the probability of a similar event $GEv$, where 
a linear function $t \mapsto t \sqrt{2}$ has been subtracted from the possible trajectories of the underlying Brownian motion $W$ for which the event $Ev$ is satisfied. If $Ev$ depends only on the trajectory of $W$ up to a certain time $T$, we get, by using the Girsanov transfomation, an equality of the form
$$\P [Ev((W_t)_{0 \leq t \leq T})] =
  \P [ GEv((W_t - t \sqrt{2})_{0 \leq t \leq T})]
= \E[ e^{-\sqrt{2}W_T - T} \mathds{1}_{GEv((W_t)_{0 \leq t \leq T})}],$$
and then the inequality
$$\P [Ev] \leq e^{-T - \sqrt{2} \mu}
\P [GEv]$$
where $\mu$ denotes the smallest possible value of $W_T$ for which the event $GEv$ can occur.

The following proposition gives a lower bound for the first moment of $\I_N$.
\begin{proposition}[First moment of $\I_N$]
\label{Prmomen1Lower} For any $\upsilon\in (0,1)$, $r\in \N$ large enough
and  $N$ large enough depending on $r$:
\begin{align}
 \label{momen1Lower}
\P(\I_N(\theta)) &\geq  
 e^{2(x-z)}2^{r-N_+}e^{-2   \upsilon -2\lambda_\infty^{(r)} } 
 N^{\frac{3}{2} }  \P(Event_{r,N}),\\
\P(\I_N(\theta)) &\leq  
 e^{2(x-z)}2^{r-N_+}e^{2   \upsilon +2\lambda_\infty^{(r)} } 
 N^{\frac{3}{2} }  \P(Event_{r,N}),
\end{align}
where, with $W$ being a standard Brownian motion,
\begin{align*}
& Event_{r,N} = Event_{r,N}(x,z) \coloneqq \\
&\left\{  \forall j\in [|r,N_+|],\, l_j^{(N)}  \leq x+ \sqrt{\frac{1}{2}} W_{\tau^{(r)}_j}
\leq u_j^{(N)} , x+ \sqrt{\frac{1}{2}} W_{\tau^{(r)}_{N_+}}\in \tau^{(r)}_{N_+}-\frac34 \log N+[z,z+\upsilon)\right\}.\end{align*}
Further, for $x,z$ as above, the asymptotics of $\P(Event_{r,N})$ are given
in  Lemma \ref{lem-onerayLB} below.
 \end{proposition}
\begin{proof}
 Since  $(R_{2^j}^{(2^r, \Delta)}(\theta))_{r \leq j \leq N_+}$ is a Gaussian random walk whose distribution does not depend on $\theta$, we have
\begin{align*}
&\P\left( \I_N (\theta)\right) =\\
& \P\left( \forall j\in [|r,N_+|],\,  L_j^{(N)}  \leq  x+R_{2^j}^{(2^r,\Delta)}(0) \leq U_j^{(N)},
x+R_{2^{N_+}}^{(2^r,\Delta)} (0)\in \tau^{(r)}_{N_+}-\frac34 \log N+[z,z+\upsilon) \right).
\end{align*}
More precisely, we know that $(R_{2^j}^{(2^r,\Delta)}(0))_{j\geq r}$ is distributed like $  \sqrt{\frac{1}{2}}   (W_{    \tau^{(r)}_j})_{j\geq r}  $. By Girsanov's transform, with density $ e^{\sqrt{2}W_{\tau^{(r)}_{N_+} }-   \tau^{(r)}_{N_+}}   $, we have
\begin{align*}
&\P(\I_N(\theta))\\
&=  \E\left( e^{ - \sqrt{2} \left( W_{\tau^{(r)}_{N_+}} + \sqrt{2} \tau^{(r)}_{N_+} \right) + \tau_{N_+}^{(r)} } \mathds{1}_{\{ \forall j\in 
[|r, N_+]|],\, l_j^{(N)}
\leq x+ \sqrt{\frac{1}{2}}  W_{\tau^{(r)}_j} \leq u_j^{(N)}, x+\sqrt{\frac{1}{2}}  W_{\tau^{(r)}_{N_+}}\in-\frac34 \log N + [z,z+\upsilon)}  \}      \right)
\\
&=  e^{ -\tau^{(r)}_{N_+}} \E\left( e^{- \sqrt{2}  W_{\tau^{(r)}_{N_+}} }   \mathds{1}_{\{ \forall j\in [|r,N_+|],\, l_j^{(N)}  \leq  x+\sqrt{\frac{1}{2}}  W_{\tau^{(r)}_j} \leq u_j^{(N)} ,  x+  \sqrt{\frac{1}{2}}W_{\tau^{(r)}_{N_+}}\in-\frac34 \log N+  [z,z+\upsilon)  \}}   \right)
\\
&\geq 2^{r-N_+} e^{2(x-z)- 2\upsilon -\lambda_\infty^{(r)} } N^{\frac{3}{2} }  \P\left(  Event_{r,N} \right),
\end{align*}
with a similar upper bound replacing $- 2\upsilon -\lambda_\infty^{(r)} $ by
$ 2\upsilon +\lambda_\infty^{(r)}$. 
Using  Lemma \ref{lem-onerayLB} completes the proof of
the right inequality in \eqref{momen1Lower}.
\end{proof}
We will also need a similar estimate for shorter times.
Let $t<N/4$ and set
\begin{align*}
  & \I_{t,f}(\theta)=\I_{N,t}(\theta,x,z):=\\
&{\{ \forall k\in [|r,t|],\,  L_k^{(N)}  \leq x+R_{2^k}^{(2^r,\Delta)}(\theta) \leq U_k^{(N)}, 
x+R_{2^{t}}^{(2^r,\Delta)} (\theta)\in \tau^{(r)}_{t}+[z,z+\upsilon)\}}.\end{align*}
Note that because of our choice of $t$, 
only the first part of the barrier is employed in $\I_{t,f}$,
and therefore $\I_{t,f}(\theta)$  does not depend on $N$. We have the following
analogue of Proposition \ref{Prmomen1Lower}.
\begin{proposition}[First moment of $\I_{N,t}$]
\label{Prmomen1Lowerbis} For any $\upsilon\in (0,1)$, $C>0$, $r\in \N$ large enough,
$x$ as above, $z\in [-\sqrt{t}/C, -C\sqrt{t}]$
and  $t$ large enough depending on $r$:
\begin{align}
 \label{momen1Lowerbis}
 \P(\I_{t,f}(\theta)) &\geq  
 e^{2(x-z)}2^{r-N_+}e^{-2   \upsilon -2\lambda_\infty^{(r)} } 
\P(Event_{r,t,f}) \gg_C \frac{|xz|}{t^{3/2}}e^{2(x-z)}2^{r-t} ,
\end{align}
where, with $W$ being a standard Brownian motion,
\begin{align*}
& Event_{r,t,f} = Event_{r,N,t}(x,z):= \\
&\left\{  \forall j\in [|r,t|],\, l_j^{(N)}  \leq x+ \sqrt{\frac{1}{2}} W_{\tau^{(r)}_j}
\leq u_j^{(N)} , x+ \sqrt{\frac{1}{2}} W_{\tau^{(r)}_{t}}\in \tau^{(r)}_{t}+[z,z+\upsilon)\right\}.\end{align*}
 \end{proposition}
\noindent
The proof is similar to that of Proposition \ref{Prmomen1Lower}, and is omitted.

\subsection{Two rays estimate}
We will bound here
\begin{align}
\label{eq:def_P_N}
\P_N(\theta,\theta')=\P_N(\theta, \theta',x,z,x',z'):= \P(\I_N(\theta,x,z)\cap \I_N(\theta',x',z')).
\end{align}
(For the sake of shortness of notation, when no confusion occurs 
we write $\P_N(\theta,\theta')$, omitting the $x,z,x',z'$ from the notation.)
This study, which is technical, is based on the fact that the random walks
$(R_{2^j}^{(2^r,\Delta)}(\theta))_{r \leq j \leq N}$ and $(R_{2^j}^{(2^r,\Delta)}(\theta'))_{r \leq j \leq N}$ are Gaussian random walks whose increments are approximately independent after some branching time which is roughly minus
the logarithm in base $2$ of the distance modulo $2 \pi$ between $\theta$ and $\theta'$.

The general idea is as follows. For given $\theta, \theta'$, we will consider the integer $\k$ such that $2^{-\k} \leq \frac{||\theta- \theta'||}{2\pi} < 2^{-(\k-1)}$, $|| \cdot ||$ denoting the distance on the set $\R / 2 \pi \Z$. One can understand $\k$ as the time of (approximate) branching between the field at $\theta $ and at $\theta' $. We will show that after some time $\k_+ = \k_+(\k) $ ``slightly larger'' than $\k$, we are able to bring out independence between the increments of  $(R_{2^k}^{(2^r,\Delta)}(\theta))_{k\geq r} $ and $ (R_{2^k}^{(2^r,\Delta)}(\theta'))_{k\geq r}$. By analogy with the Gaussian field, we will see this time as a time of decorrelation. It is defined as follows:
 \begin{itemize}
 \item For $\k  \leq  N/2$, the time of decorrelation is $\k_+ := \k + 3 \Delta^{(\k)}$. Recall that:
   $$ \Delta^{(\k)}:= e^{\sqrt{\log r}}
   + 100\lfloor  \log^2 \k \rfloor \ .$$
 In particular, for $\k \leq r/2$ and $r$ large enough, the  walks  $(R_{2^k}^{(2^r,\Delta)}(\theta))_{k\geq r} $ and $ (R_{2^k}^{(2^r,\Delta)}(\theta'))_{k\geq r}$ will have ``almost independent increments'' from the starting time $r$, since $r \leq \k_+$.
 \item For $N/2 < \k \leq N _+$, we will require a faster decorrelation. We take $\k_+
 := \k + 3 \kappa^{(\k)}$, where
 $$\kappa^{(\k)} :=  \lfloor r/100 \rfloor + 100 \lfloor \log (N-\k) \rfloor^2.$$
 However, the price to pay is that we will have to modify our field $(R_{2^j}^{(2^r,\Delta)}(\theta))_{r \leq j \leq N}$ in the spirit of subsection \ref{section:new_coupling}.
 \end{itemize}
 
 Our two rays estimates are divided to three lemmas,
 numbered \ref{lemma:kDebut}, \ref{lemma:ktoutDebut} and \ref{lemma:kMil}.
 The statement of each of these lemmas gives a suitable majorization of $\P_N(\theta, \theta')$, for a given range of values of $\k$.
\begin{lemma}[Time of branching $\k \leq N/2$]
  \label{lemma:kDebut}
  For any $\upsilon \in  (0,1)$, $r$ large enough, $N$ large enough depending on $r$,   $\k \leq \frac{N}{2}$ and $2^{-\k}\leq \frac{||\theta-\theta'||}{2\pi} < 2^{-(\k-1)}$,  we have
  \begin{align}
  \label{eqkDebut}
  \P_N(\theta, \theta')&\ll \left\{
  \begin{array}{ll}
   |zz' xx'|  2^{2r-2N_+ +2(x-z+x'-z')}                                    ,\qquad &\text{when } \k_+ \leq r,\\
  |zz'x| e^{2(x-z-z')}2^{r-N_+} 2^{\k_+-N_+} e^{- \k_+^{\alpha_-}} ,\qquad &\text{when } \k_+ \geq r.
  \end{array}
  \right.
  \end{align}
\end{lemma}

The following lemma studies pairs whose time of branching happen before $r/2$.  It refines the estimate obtained in the previous lemma.
\begin{lemma}[Time of branching $\k \leq r/2$]
  \label{lemma:ktoutDebut}
  For any $\upsilon \in (0,1)$, $r\in \N$ large enough, $ \k \leq \frac{r}{2}$, $2^{-\k} \leq \frac{||\theta-\theta'||}{2\pi} \leq 2^{-(\k-1)}$ and  $N$ large enough depending on $r$, we have
  \begin{align*}
  \P_N(\theta, \theta')&\leq   (1+\eta_{r,\upsilon  })    2^{2r -2N_+}e^{2(x-z+x'-z')} N^{3}   \P(Event_{r,N}(x,z) ) \P(Event_{r,N}(x',z')), 
  \end{align*}
  where
  $$\underset{\upsilon \rightarrow 0}{\limsup} \, \underset{r \rightarrow \infty} {\limsup} \, \eta_{r,\upsilon} = 0.$$
\end{lemma}

\begin{lemma}[Time of branching $N/2 < \k \leq N_+$]
  \label{lemma:kMil}
  For $\upsilon \in (0,1)$, $r\in \N$ large enough, $N$ large enough depending on $r$,  $N/2 < \k \leq N_+$ and $2^{-\k}\leq \frac{||\theta-  \theta'||}{2\pi}  < 2^{-(\k-1)}$,  we have $\k_+ < N$ and
  \begin{align}
  \P_N(\theta,\theta') \ll  |xz| e^{2(x-z-z')} 2^{r-N_+} 2^{\k_+-N_+} e^{- \half (N-\k_+ )^{\alpha_-}} ,
  \end{align}
\end{lemma}
\begin{rmk}
The proof of the Lemma \ref{lemma:kMil} is the unique place where we use $(l_j^{(N)})_{r \leq j\leq N}$, the lower part of the envelope.
\end{rmk}
\subsubsection{Dyadic case}
We start by assuming $\frac{||\theta-\theta'||}{2\pi} = 2^{-\k}$ and prove Lemmas \ref{lemma:kDebut} and \ref{lemma:kMil} in this dyadic case. This part is mainly for pedagogical purposes while laying the ground for the general case. It illustrates perfectly the machinery of the proof in a simpler setting.

It will be convenient to
denote,  for any $l\in [|r,N_+|]$, $p\in [|0,2^{\Delta^{(l)}}-1|]$,
\begin{align}
 \label{defI_n}
 I_{l,p}(\theta) := \frac{ \sum_{j=0}^{2^{l-\Delta^{(l)}}-1} \Nc^\C_{2^{l}+p 2^{l-\Delta^{(l)}} + j} e^{i \theta j} }
                         { \sqrt{{ 2^{l}+p 2^{l-\Delta^{(l)}}  }}  }
                 \eqlaw \Nc^{\C}\left(0, (2^{\Delta^{(l)}}+ p)^{-1} \right).
\end{align}
Recall that $R_{2^N}^{(2^r,\Delta)}(\theta)$ and $R_{2^N}^{(2^r, \Delta)}(\theta')$ can be written as
$$
R_{2^N}^{(2^r,\Delta)}(\theta)= \Re \left( \sigma \sum_{l=r}^{N -1 } \sum_{p=0}^{2^{\Delta^{(l)}}-1}  I_{l,p}(\theta) e^{ i\psi_{2^{l}+p2^{l-\Delta^{(l)}}} (\theta) } \right),
$$
$$
R_{2^N}^{(2^r,\Delta)}(\theta')= \Re \left( \sigma \sum_{l=r}^{N -1 } \sum_{p=0}^{2^{\Delta^{(l)}}-1}  I_{l,p}(\theta') e^{ i\psi_{2^{l}+p2^{l-\Delta^{(l)}}} (\theta') } \right).
$$
{\it The crucial observation} is that for any $  l \in [|\k_+,N_+|] $ 
 (we easily check that $\k_+ <N_+$ if $r$ is large enough, $N \geq r$ and $\k \leq N/2$) and any $p\leq 2^{\Delta^{(l)}}-1$, the random variables $ I_{l,p}(\theta)$ and $ I_{l,p}(\theta')$  are independent and identically distributed. Indeed, they form a complex Gaussian vector, and they are uncorrelated, since for $l \geq \k_+$, one has $l - \Delta^{(l)} \geq \k$ if $l \geq r $ and $r$ is large enough,
and then
$$ \sum_{j=0}^{2^{l - \Delta^{(l)}} } e^{i( \theta - \theta') j}
 = \sum_{j=0}^{2^{l - \Delta^{(l)}}} e^{\pm 2  i j \pi / 2^{\k}} = 0.$$

We deduce that the increments of $R_{2^{j}}^{(2^r,\Delta)}(\theta) $ and  $R_{2^{j}}^{(2^r,\Delta)}(\theta') $ after the time $\k_+$ are independent and identically distributed. Recalling the definition of $\tau^{(r)}$ in (\ref{eq:def_tau_r}) and \eqref{defI_n}, it follows that
\begin{align*}
(R_{2^{j}}^{(2^r,\Delta)}(\theta))_{j\geq r} \eqlaw (R_{2^{j}}^{(2^r,\Delta)}(\theta'))_{j\geq r} \eqlaw \sqrt{\frac{1}{2}} ( W_{\tau_j^{(r)}})_{j\geq r}.
\end{align*}
For any $k \in [|r, N_+|]$, $x, z<0$, we introduce the events
\begin{align}
\nonumber
Ev(k,x,z)&:= \left\{   \forall j\in [|k,N_+|],\,  l_j^{(N)}  \leq\sqrt{\frac{1}{2}}W_{\tau_j^{(k)}}-  \tau^{(k)}_j  +x
 \leq u_j^{(N)}, \right.\nonumber\\
&\qquad  \left. \sqrt{\frac{1}{2}}W_{\tau_{N_+}^{(k)}}-  \tau^{(k)}_{N_+} +x\in -\frac34 \log N +[z,z+\upsilon)    \right\},\nonumber
\\
\label{eq:def_Ev_GEv}
GEv(k,x,z)&:=\left\{   \forall j\in [|k,N_+|],\,  l_j^{(N)}  \leq\sqrt{\frac{1}{2}}W_{\tau_j^{(k)}}  +x
 \leq u_j^{(N)}, \right.\nonumber\\
&\qquad  \left. \sqrt{\frac{1}{2}}W_{\tau_{N_+}^{(k)}}+x\in -\frac34 \log N +[z,z+\upsilon)    \right\},
\end{align}
Note that $GEv(r,x,z)=Event_{r,N}(x,z)$ from Proposition \ref{Prmomen1Lower}, we only keep the separate notation for
compatibility with \cite{CMN}. Furthermore note that $GEv(k,x,z)$ is equal to the event obtained from $Ev(k,x,z)$ after  Girsanov transform with density $ \exp\left( \sqrt{2}W_{\tau^{(k)}_{N_+} }- \tau^{(k)}_{N_+} \right) $. Performing the transform yields:
\begin{align}
\label{eq:girsanov1}
\P\left(Ev(k,x,z)\right) & = e^{-\tau_{N_+}^{(k)}} \E\left( e^{-\sqrt{2} W_{\tau^{(k)}_{N_+} }  } \mathds{1}_{ GEv(k,x,z) } \right)
                         \leq 2^{k-N_+} e^{2(x-z+\upsilon)} N^{\frac{3}{2}} \P\left( GEv(k,x,z) \right),
\end{align}
where in the last inequality we used the definition \eqref{eq:def_tau_r} of $\tau_{N_+}^{(k)}$ and the fact that $e^{-\sqrt{2} W_{\tau^{(k)}_{N_+} }  }\leq N^{\frac{3}{2}} e^{2(x-z+\upsilon)} $ on $GEv(k,x,z)$.

In order to allow for more flexibility and for later use, let us record the following analogous events. For $k \in [|r,N_+|]$, $a\geq 0$, $E=(E_j)_{j\geq k}$ a sequence of reals such that $(\tau_{j}^{(k)} - E_j)_{j\geq k}$ is positive and nondecreasing, $x,z\in \R$, define
\begin{align}
\nonumber
Ev(k,a,E,x.z)  &:= \left\{   \forall j\in [|k, N_+ |] ,\, l_j^{(N)}-a \leq \sqrt{\frac{1}{2}}W_{\tau_j^{(k)}-E_j}+x- \tau^{(k)}_{j}  \leq u_j^{(N)}+a, \right.\\
&\quad \left.  \sqrt{\frac{1}{2}}W_{\tau_{N_+}^{(k)}-E_{N_+}}+x-\tau^{(k)}_{N_+}\in -\frac34 \log N +[z,z+\upsilon)   \right\},\\
\label{eq:def_Ev_GEv_ext}
GEv(k,a,E,x,z) &:= \left\{   \forall j\in [|k, N_+ |] ,\, l_j^{(N)}-a \leq \sqrt{\frac{1}{2}}W_{\tau_j^{(k)}-E_j} +x\leq u_j^{(N)}+a ,\right.\nonumber\\
&\quad \left.   \sqrt{\frac{1}{2}}W_{\tau_{N_+}^{(k)}-E_{N_+}}+x\in -\frac34 \log N +[z,z+\upsilon)  \right\}.
\end{align}
Again, the event $GEv(k,a,E,x,z) $ is, up to an error due to the time shift $E$, "quasi equal" to what we obtain when we apply the Girsanov' transform with density
$$ \exp\left( \sqrt{2}W_{\tau_{N_+}^{(k)}-E_{N_+} } - (\tau_{N_+}^{(k)}- E_{N_+}) \right) $$ to the event $ Ev(k,a,E,x,z)$. This time, the inequality takes the form:
\begin{align}
\label{eq:girsanov2}
\P\left(Ev(k,a,E,x,z)\right) & \leq 2^{k-N_+} e^{-E_{N_+}+2(x-z+a+\upsilon)} N^{\frac{3}{2}} \P\left( GEv(k,a+\sup_{k \leq j \leq N} |E_j|,E,x,z) \right).
\end{align}
Indeed, by the Girsanov transform and then using the barrier at time $N_+$:
\begin{align*}
  & \P\left(Ev(k,a,E,x,z)\right) \\
&=  e^{-\tau_{N_+}^{(k)}+E_{N_+}} \E\left( e^{-\sqrt{2} W_{\tau^{(k)}_{N_+} - E_{N_+}} }  
\mathds{1}_{ \left\{   \forall j\in [|k, N_+ |] ,\, l_j^{(N)}-a \leq \sqrt{\frac{1}{2}} W_{\tau_j^{(k)}-E_j}+x- E_j  \leq u_j^{(N)}+a \right\}}\right.\\
&\qquad\qquad\qquad \qquad \qquad \left.\mathds{1}_{ \left\{\sqrt{\frac{1}{2}}W_{\tau_{N_+}^{(k)}-E_{N_+}}+x\in -\frac34 \log N +[z,z+\upsilon) \right\}  }\right)\\
\leq & 2^{k-N_+} e^{E_{N_+}+2(x-z-E_{N_+}+a+\upsilon)} N^{\frac{3}{2}} \P\left( \forall j\in [|k, N_+ |] , l_j^{(N)}\!-\!a \leq \sqrt{\frac{1}{2}}W_{\tau_j^{(k)}-E_j}+x- E_j  \leq u_j^{(N)}\!+\!a ,\right.\\
&\qquad \qquad\qquad\qquad \qquad \qquad  \left.  \left\{\sqrt{\frac{1}{2}}W_{\tau_{N_+}^{(k)}-E_{N_+}}+x\in -\frac34 \log N +[z,z+\upsilon) \right\} 
\right) \\
\leq &  2^{k-N_+} e^{-E_{N_+}+2(x-z+a+\upsilon)} N^{\frac{3}{2}} \P\left( GEv(k,a+\sup_{k \leq j \leq N_+} |E_j|,E,x,z) \right).
\end{align*}

\begin{proof}[Proof of Lemma \ref{lemma:kDebut} in dyadic case]{ \quad }

{ \bf When $\k_+ \leq r$:}
The increments of $R_{2^{j}}^{(2^r,\Delta)}(\theta) $ and  $R_{2^{j}}^{(2^r,\Delta)}(\theta') $ after the time $\k_+$ are independent and identically distributed, thus we have
\begin{align*}
\P_N(\theta,\theta) = \quad & \P\left(Ev(r,x,z)\right) \P\left(Ev(r,x',z')\right)\\
                     \stackrel{Eq. \eqref{eq:girsanov1}}{\ll_v} & 
 2^{2r-2N_+} e^{2(x-z+x'-z')} \left( N^{\frac{3}{2}} \P\left( GEv(r,x,z) \right) \right) \left( N^{\frac{3}{2}} \P\left( GEv(r,x',z') \right) \right).
\end{align*}
Finally by applying \textcolor{red}{\cite[(A.15)]{CMN}}
 (with $E_{j-r}=\lambda_j^{(r)}/\log 2$, which implies $||E|| \leq 1$ for $r$ large enough), we obtain \eqref{eqkDebut}.

{ \bf When $\k_+ \geq r$:}
The increments of $R_{2^{j}}^{(2^r,\Delta)}(\theta) $ and  $R_{2^{j}}^{(2^r,\Delta)}(\theta') $ after time $\k_+$ are independent and identically distributed. Moreover,
all these increments are independent of
those of $R_{2^{j}}^{(2^r,\Delta)}(\theta) $
for $j$ between $r$ and $\k_+$ (we see this fact  by first conditioning with respect to the $\sigma$-algebra $\mathcal{G}_{2^{\k_+}}$). We then have:
\begin{align*}
&\P_N(\theta, \theta'x,z,x',z') \leq   \P\left(  Ev(r,x,z) \right) \sup_{ x'+l_{\k_+}^{(N)}  \leq w\leq x'+u_{\k_+}^{(N)} }\P( Ev(\k_+,w,z' ) )\\
\stackrel{Eq. \eqref{eq:girsanov1}}{\ll}
                                  & 2^{r-N_+}e^{2(x-z)} \left( N^{\frac{3}{2}} \P\left( GEv(r,x,z) \right) \right)\\
&\qquad\qquad\qquad
                              2^{\k_+-N_+} \sup_{ l_{\k_+}^{(N)}  \leq w+x'\leq u_{\k_+}^{(N)} } \left( N^{\frac{3}{2}} e^{2(w+x'-z)} \P\left( GEv(\k_+,w+x',z') \right) \right) \\
                      \ll   \quad & |xzz'|2^{r-N_+}
                                    2^{\k_+-N_+} e^{2(x-z-z')}\sup_{ l_{\k_+}^{(N)}  \leq w'\leq u_{\k_+}^{(N)} } \left( N^{\frac{3}{2}} e^{2(w')} \P\left( GEv(\k_+,w',z') \right) \right).
\end{align*}
If $\k_+ \leq N/4$, according to  \cite[(A.15]{CMN},
for any $w'\in \R$, we have
\begin{align*}
\P\left( GEv(\k_+,w',z' ) \right) \ll  |w'| |z'|(N-\k_+)^{-\frac{3}{2}} \ .
\end{align*}
Recalling that $ w'\leq u_{\k_+}^{(N)}=  -(\k_+)^{\alpha_-} $, we have
\begin{align}
\label{tocliam}
\sup_{ l_{\k_+}^{(N)}  \leq w'\leq u_{\k_+}^{(N)} } \left( N^{\frac{3}{2}} e^{2(w')} \P\left( GEv(\k_+,w',z') \right) \right)
\ll e^{- \k_+^{\alpha_-}}.
\end{align}
On the other hand, for $\k_+ > N/4$, we can crudely bound $\P\left( GEv(\k_+,w',z; ) \right)$ by $1$, and using the fact that $w' \leq - \k_+^{\alpha_-}$ if $\k_+<N/2$ and
$w'\leq -(N-\k_+)^{\alpha_-}\leq -0.9 \k_+^{\alpha_-}$ if $\k\leq N/2\leq \k_+$,
we obtain that
\begin{align}
\label{tocliam1}
\sup_{ l_{\k_+}^{(N)}  \leq w'\leq u_{\k_+}^{(N)} } \left( N^{\frac{3}{2}} e^{2w'} \P\left( GEv(\k_+,w',z) \right) \right)
\leq e^{2 w'} N^{\frac{3}{2}}
\ll e^{- 1.8 \k_+^{\alpha_-}} N^{3/2}
\ll e^{- \k_+^{\alpha_-}},
\end{align}
since $\k_+\geq N/5$ for $N$ large. (Recall that we take $N\to\infty$ before we take $k_1\to\infty$.)

In all cases, by combining Eq. \eqref{tocliam} and \eqref{tocliam1}, we deduce:
\begin{align*}
\P_N(\theta,\theta') \ll |xzz'|   2^{r-N_+} 2^{\k_+-N_+}e^{2(x-z-z')} e^{- \k_+^{\alpha_-}} .
\end{align*}
It concludes the proof of Lemma \ref{lemma:kDebut} when $\frac{||\theta- \theta'||}{2\pi}$ is a negative power of $2$.
\end{proof}

\begin{proof}[Proof of Lemma \ref{lemma:kMil} in the dyadic case]
 Now we shall study $\P_N(\theta,\theta')$, when the branching between the field in $\theta$ and the field in $\theta'$  appears after the time $N/2$. This time we shall prove that when one restricts to the paths which are in the envelope, the increments of the path of the field at $\theta$ and at $\theta'$ are approximately independent after the time of decorrelation $k_+$.  We recall that for this range
 $\k_+ = \k+ 3\kappa^{(\k)}$ where $\kappa^{(\k)}:= \lfloor \frac{r }{100}\rfloor + 100 \lfloor \log (N-\k)\rfloor^2$, if
$\k \leq N-1$.

We need to exhibit the independence between the increments of the processes $(R_{2^j}^{(2^{\k_+},\Delta)}(\theta))_{j\geq \k_+}$ and $(R_{2^j}^{(2^{\k_+},\Delta)}(\theta'))_{j\geq \k_+}$. The {\it crucial observation} we used in case of $\k \leq N/2$ does not work anymore for such a short decorrelation time. We first need to modify our field using similar arguments to those used for the proof of Proposition \ref{proposition:new_coupling}.

In the following we shall use the quantity
\begin{align*}
	J_{l,p}^{(\k)}(\theta) := \frac{ \sum_{j=0}^{2^{l-\kappa^{(\k)}}-1}    \Nc^\C_{2^{l}+p 2^{l-\kappa^{(\k)}} + j} e^{i \theta j}     }{\sqrt{{ 2^{l}+p 2^{l-\kappa^{(\k)}}  }}  } \eqlaw \Nc^{\C}\left((0, 2^{\kappa^{(\k)}} +p)^{-1} \right)
\end{align*}

Let $\k \in [N/2,N_+]$ and $2^{-\k}\leq \frac{||\theta- \theta'||}{2\pi} < 2^{-(\k-1)}$. We have, for $r$ large enough, since $N - \k \geq r/2$ is large,
$$\k_+ \leq  \k + (r/33) + 300 \log^2 (N-\k) \leq \k + (N- \k)/16 < N. $$
Recall that for $ \tilde{\theta}\in \{\theta,\theta'  \}$, $r\leq k\leq N$, $ R_{2^k}^{(2^{r},\Delta )}( \tilde{\theta})   =  \Re(\sigma Z_{2^k}^{(2^{r},\Delta )}( \tilde{\theta}))   $.  Since for $l \geq \k_+$, $\kappa^{(\k)} \leq \Delta^{(l)}$, we can write for all $\k_+ \leq k \leq N$, $ \tilde{\theta}\in \{  \theta,\theta'\}$:
\begin{align*}
	Z_{2^k}^{(2^{\k_+},\Delta )}( \tilde{\theta}) &=    \sum_{l=\k_+}^{k-1 } \sum_{p=0}^{2^{\Delta^{(l)}}-1}  I_{l,p}( \tilde{\theta}) e^{ i\psi_{ 2^{l}+p 2^{l-\Delta^{(l)}}} ( \tilde{\theta} )  }\\
	&=    \sum_{l=\k_+}^{k -1 } \sum_{p=0}^{2^{\Delta^{(l)}}-1}  \frac{ \sum_{j=0}^{2^{l-\Delta^{(l)}}-1}    \Nc^\C_{2^{l}+p 2^{l-\Delta^{(l)}} + j} e^{i   \tilde{\theta} j}     }{\sqrt{{ 2^{l}+p 2^{l-\Delta^{(l)}}  }}  }e^{ i\psi_{ 2^{l}+p 2^{l-\Delta^{(l)}}}(   \tilde{\theta}) }
	\\
	=  &  \sum_{l=\k_+}^{k-1 } \sum_{p=0}^{2^{\kappa^{(\k)}}-1}     \left(\sum_{j=0}^{2^{l-\kappa^{(\k)}}-1}        \frac{   \Nc^\C_{2^{l}+p 2^{l-\kappa^{(\k)}} + j} e^{i  \tilde{\theta} \left\{  j-      \mathfrak{j}       \right\}} }{\sqrt{  2^{l}+p 2^{l-\kappa^{(\k)}} + \mathfrak{j}  }  } e^{ i  \psi_{ 2^{l}+p 2^{l-\kappa^{(\k)} } +   \mathfrak{j}   }(  \tilde{\theta})}  \right), \text{(with $\mathfrak{j}=   \lfloor \frac{j}{2^{l-\Delta^{(l)}}}\rfloor    2^{l-\Delta^{(l)}}  $)}
	\\
	&=  Z_{2^k}^{(2^{\k_+},\kappa^{(\k)})}(  \tilde{\theta})  + \mathfrak{E}_{2^k}( \tilde{\theta})
\end{align*}
where
\begin{align*}
	Z_{2^k}^{(2^{\k_+},\kappa^{(\k)})}(   \tilde{\theta}) &:= \sum_{l=\k_+}^{k-1} \sum_{p=0}^{2^{\kappa^{(\k)}}-1}  J_{l,p}^{(\k)}(  \tilde{\theta}) e^{ i  \psi_{ 2^{l}+p 2^{l-\kappa^{(\k)}}  }(  \tilde{\theta})}  ,
	\\
	\mathfrak{E}_{2^k}(  \tilde{\theta})&:=   \sum_{l=\k_+}^{k-1} \sum_{p=0}^{2^{\kappa^{(\k)}}-1} \frac{  e^{ i  \psi_{ 2^{l}+p 2^{l-\kappa^{(\k)}}  }( \tilde{\theta} )}  }{\sqrt{{ 2^{l}+p 2^{l-\kappa^{(\k)}}  }}  } \left(\sum_{j=0}^{2^{l-\kappa^{(\k)}}-1}    \Nc^\C_{2^{l}+p 2^{l-\kappa^{(\k)}} + j} e^{i  \tilde{\theta} j}   \lozenge_j^{ (2^{l}+p 2^{l-\kappa^{(\k)}})  }(  \tilde{\theta}) \right),
\end{align*}
with
\begin{align}
	\label{eq:ineq_squarebis}
	\left|\lozenge_j^{(2^{l}+p 2^{l-\kappa^{(\k)}} )}(  \tilde{\theta})\right|
	\leq &  \left|  A_{ 2^{l}+p 2^{l-\kappa^{(\k)}} +  \mathfrak{j} }^{(2^{l}+p 2^{l-\kappa^{(\k)}})    }(  \tilde{\theta})   ) \right| + 2^{-\kappa^{(\k)}}.
\end{align}
Indeed:
\begin{align*}
	\lozenge_j^{  (2^{l}+p 2^{l-\kappa^{(\k)}} )  }(  \tilde{\theta})
	= & -1+ \sqrt{\frac{2^l+p2^{l-\kappa^{(\k)}}}{2^l+p2^{l-\kappa^{(\k)}} +  \mathfrak{j}   }  }e^{i A_{ 2^{l}+p 2^{l-\kappa^{(\k)}} +\mathfrak{j}  }^{(2^{l}+p 2^{l-\kappa^{(\k)}})}(  \tilde{\theta})  }
	\\
	= & - 1+ \left( 1 + \Theta \right)
	e^{i A_{2^{l}+p 2^{l-\kappa^{(\k)}} +  \mathfrak{j}}^{(2^{l}+p 2^{l-\kappa^{(\k)}})    }(  \tilde{\theta})  },
\end{align*}
with $|\Theta| \leq 2^{-\kappa^{(\k)}}$
which implies inequality \eqref{eq:ineq_squarebis}. In the following, we shall denote:
\begin{align*}
	\blacklozenge^{(\k)}_l(  \tilde{\theta}):=\sum_{p=0}^{2^{\kappa^{(\k)}}-1} \frac{  e^{ i  \psi_{ 2^{l}+p 2^{l-\kappa^{(\k)}}  }(  \tilde{\theta} )}  }{\sqrt{{ 2^{l}+p 2^{l-\kappa^{(\k)}}  }}  } \left(\sum_{j=0}^{2^{l-\kappa^{(\k)}}-1}    \Nc^\C_{2^{l}+p 2^{l-\kappa^{(\k)}} + j} e^{i   \tilde{\theta} j}   \lozenge_j^{ (2^{l}+p 2^{l-\kappa^{(\k)}})  }(   \tilde{\theta}) \right).
\end{align*}

Notice that, on the contrary of the proof of Proposition \ref{proposition:new_coupling}, where $\Delta^{(l)}$ varies with $l$, here we fix $\kappa^{(\k)}$ as soon as we know that  $2^{-\k} \leq \frac{||\theta- \theta'||}{2\pi}  < 2^{-(\k-1)} $.  By using the same arguments used to prove  \eqref{eq:one_point_square} and (\ref{eq:laplace_blacksquare}), one can show similarly that for any $l\geq \k_+$, $ \tilde{\theta}\in \{\theta, \theta'\}$,
\begin{align}
\label{nonind}
	& \P\left( |  \blacklozenge^{(\k)}_l( \tilde{\theta}) |\geq 2^{-\frac{\kappa^{(\k)}}{8}},  \cap_{l \geq \k_+} \widetilde{G}_{l}( \tilde{\theta})  \Big|  \Gc_{2^{\k_+}}  \right)  \ll e^{-2^{\frac{\kappa^{(\k)}}{8}   }},
	\\ & \qquad  \P\left(( \cap_{l\geq \k_+} \widetilde{G}_{l}(  \tilde{\theta} ))^c  \Big|    \Gc_{2^{\k_+}} \right)  \ll_{\beta}      \exp\left( - \frac{\beta}{33} 2^{\half \kappa^{(\k)}}    \right).
\end{align}
with
\begin{align*}
	\widetilde{G}_{l}(  \tilde{\theta} ) := & \ \bigcap_{p=0}^{2^{\kappa^{(\k)}}-1}
	\left\{ \sup_{0\leq j\leq 2^{l-{\kappa^{(\k)}}}-1} \left| A_{j+ 2^{l}+p 2^{l- \kappa^{(\k)}} }^{(2^{l}+p 2^{l-  \kappa^{(\k)}})}(  \tilde{\theta} ) \right|
	\leq 2^{-\quart \kappa^{(\k)} }
	\right\} \ .
\end{align*}
Moreover it is plain to observe that for any $r$ large enough, $N$ large enough depending on $r$, and $\k\in [|N/2, N-r/2|]$, and under the complement of the two events just above,
\begin{align}
	\sum_{l=\k_+}^{N-1}  |  \blacklozenge^{(\k)}_l(  \tilde{\theta}) |  \leq   \sum_{l=\k_+}^{N-1} 2^{- \frac{\kappa^{(\k)}}{8}} \leq (N-\k_+) 2^{- \frac{1}{8} \lfloor \frac{r}{100} \rfloor -  \frac{25}{2}\lfloor\log (N-\k)\rfloor^2  } \leq 1.
\end{align}
So for any $\tilde{\theta}\in \{\theta,\theta'\}$,  we can replace $(R_{2^k}^{(2^{\k_+},\Delta )}(  \tilde{\theta}))_{\k_+ \leq k \leq N} $ by $(R_{2^k}^{(2^{\k_+},\kappa^{(\k)})}(  \tilde{\theta}))_{ \k_+ \leq k \leq N}  $ with an error at most $1$.
 Thus we have
\begin{align*}
 &  \P_N(\theta,\theta') \\
&\leq   \P\left( \forall j\in [|r,N_+|],\,  L_j^{(N)} -1 \leq  x+ R_{2^{ j\wedge \k_+}}^{(2^{r},\Delta )}(\theta) +  R_{2^j}^{(2^{\k_+},\kappa^{(\k)})}(\theta)\mathds{1}_{{\{ j >k_0  \}}} \leq U_j^{(N)} +1,\, \right.
 \\
 & \left. \qquad\qquad\qquad \forall j\in [|\k_+,N_+|],\,  L_j^{(N)} -1 \leq  x'+ R_{2^{ \k_+}}^{(2^{r},\Delta )}(\theta') +  R_{2^j}^{(2^{\k_+},\kappa^{(\k)})}(\theta') \leq U_j^{(N)} +1,\right.\\
& \qquad\qquad\qquad\left. x+ R_{2^{ \k_+}}^{(2^{r},\Delta )}(\theta) +  R_{2^j}^{(2^{\k_+},\kappa^{(\k)})}(\theta)\in -\frac34 \log N +[z,z+v],\right.\\
&\qquad\qquad\qquad\left.
 x'+ R_{2^{ \k_+}}^{(2^{r},\Delta )}(\theta') +  R_{2^j}^{(2^{\k_+},\kappa^{(\k)})}(\theta')\in -\frac34 \log N +[z,z+v]  \right)\\
 + & \sum_{\tilde{\theta}\in \{\theta,\theta'\}} \E\Bigg(  \mathds{1}_{\{  \forall j\in [|r,\k_+|],\,  L_j^{(N)}  \leq  \tilde x
 +R_{2^j}^{(2^r,\Delta)}( \tilde{\theta}) \leq U_j^{(N)}\}}  \Big( \sum_{l=\k_+}^{N_+} \P\left( |  \blacklozenge^{(\k)}_l(  \tilde{\theta}  ) |\geq 2^{-\frac{\kappa^{(\k)}}{8}},  \cap_{l\geq \k_+} \tilde{G}_{l}( \tilde{\theta} ) \Big| \Gc_{2^{\k_+}}  \right)\\
 & \qquad \qquad\qquad \qquad\qquad\qquad \qquad\qquad\qquad \qquad\qquad\qquad+   \P\left( ( \cap_{l\geq \k_+} \tilde{G}_{l}( \tilde{\theta} ))^c \Big|  \Gc_{2^{\k_+}}  \right) \Big) \Bigg),\\
&=: \P_N^{(\Delta,\kappa)}(\theta,\theta')+ Q_N(\theta,\theta'),
\end{align*}
where $\tilde x=x'$ if $\tilde\theta=\theta'$ and $\tilde x=x$ otherwise.

We first deal with the sum in $\tilde{\theta}\in \{  \theta,\theta'\}$, that is with $Q_N(\theta,\theta')$. By using \eqref{nonind}, then
the Girsanov transfom with density $e^{\sqrt{2 }W_{\tau^{(r)}_{\k_+} }-    \tau_{\k_+}^{(r)} }$, and  \cite[Corollary A.6]{CMN}
(when $\k_+\geq \frac{2N}{3}$), the sum is
\begin{align*}
 \ll 2^{r-\k_+} N^{\frac{3}{2}} e^{ 2 (N-\k_+)^{\alpha_+}}  \left( (N-\k_+)e^{-2^{\frac{\kappa^{(\k)}}{8}}}   +    \exp\left( - \frac{\beta}{33} 2^{\half \kappa^{(\k)}} \right)  \right),\;\text{when }\, \k_+ \leq \frac{2N}{3},
\\
 \ll 2^{r-\k_+} \frac{N^{\frac{3}{2}}( N-\k_+)^{2\alpha_+}   }{(\k_+-r)^{\frac{3}{2}}} e^{ 2 (N-\k_+)^{\alpha_+}}  \left( (N-\k_+)e^{-2^{\frac{\kappa^{(\k)}}{8}}}   +    \exp\left( - \frac{\beta}{33} 2^{\half \kappa^{(\k)}} \right)  \right),\; \text{when }\, \k_+ \geq \frac{2N}{3}
\end{align*}
which are both dominated by $|xz| e^{-(x-z-z')}2^{-2N_++ \k_+ +r}      e^{- \half (N-\k_+ )^{\alpha_-}} $. It remains to bound the expression 
$\P_N^{(\Delta,\kappa)}(\theta,\theta')$.
Let $\tau_j^{(r,\k_+)}:= \tau_j^{(r)}$ if $j\leq \k_+$ and $\tau_j^{(r,\k_+)}:= \tau_{\k_+}^{(r)}+ \sum_{l=\k_+}^{j-1} \sum_{p=0}^{2^{\kappa^{(\k)}}-1} (2^{\kappa^{(\k)}} +p)^{-1}$ if $j\geq \k_+$ and $E^{(\k_+)}_j:= \tau_j^{(r)}-\tau_j^{(r,\k_+)}$, for any $r\leq j\leq N_+$ . It is plain to check  that $\sum_{j=r}^{N_+}|E^{(\k_+)}_{j+1} -E^{(\k_+)}_{j} | \ll 2^{- \frac{r}{100}}$ and
\begin{align*}
(R_{2^{ j\wedge \k_+}}^{(2^{r},\Delta )}(\theta') +  R_{2^j}^{(2^{\k_+},\kappa^{(\k)})}(\theta')\mathds{1}_{{\{ j >k_0  \}}})_{r\leq j\leq N} \eqlaw \sqrt{\frac{1}{2}} ( W_{\tau_j^{(r,\k_+)}})_{j\geq r} \ .
\end{align*}
Now it suffices to reproduce the proof of Lemma \ref{lemma:kDebut}. In this first part, we assume $\frac{||\theta - \theta'||}{2 \pi} = 2^{-\k}$. In this case, we check the independence of $J_{l,p}^{(\k)}( \theta)$ and $J_{l,p}^{(\k)}( \theta')$ for $l \geq \k_+$, since $\k_+ - \kappa^{(\k)} \geq \k$. We then show, by doing the suitable conditionings, that the increments of $(R_{2^k}^{(2^{\k_+},\kappa^{(\k)})}(\theta))_{\k_+ \leq k \leq N_+}$, $(R_{2^k}^{(2^{\k_+},\kappa^{(\k)})}(\theta'))_{\k_+ \leq k \leq N_+}$,
$(R_{2^k}^{(2^{r},\Delta)}(\theta'))_{r \leq k \leq \k_+}$ are independent. Thus we have
\begin{align}
\nonumber
  &  \P_N^{(\Delta,\kappa)}(\theta,\theta')\leq  \P\left(  Ev(r,1,E^{(\k_+)},x,z) \right) \max_{    l_{\k_+}^{(N)} -1 \leq w\leq u_{\k_+}^{(N)}+1   }\P( Ev(\k_+,1,E^{(\k_+)},w,z')   ).
\end{align}

By the same arguments as in the proof of Lemma \ref{lemma:kDebut} in the dyadic case we have
\begin{align*}
	\max_{    l_{\k_+}^{(N)}-1 \leq w\leq u_{\k_+}^{(N)} +1  }\P( Ev(\k_+,1,E^{(\k_+)},w,z')   )\ll_{\upsilon}
&	|z'|2^{-(N_+-\k_+)}  
\max_{    l_{\k_+}^{(N)}-1 \leq w\leq u_{\k_+}^{(N)} +1  }|e^{2(w-z')}\\
&\leq |z'|2^{-(N_+-\k_+)}  e^{-2z'-(N-\k_+)^\alpha_-},
\end{align*}
where we used that in the stated range, $|w'|\geq (N-\k_+)^\alpha_-/2$,
	and
	$$
	\P\left(  Ev(r,1,E^{(\k_+)},x,z)\right) \ll |xz|2^{-N_++r}.$$
 Finally one gets
\begin{align*}
	\P_N(\theta,\theta') \ll|xzz'| 2^{r-N_+}  2^{\k_+-N_+} e^{-  (N-\k_+)^{\alpha_-}}e^{2(x-z-z')}   \ .
\end{align*}
This concludes the proof of Lemma \ref{lemma:kMil}, when $\frac{||\theta-\theta'||}{2\pi}$ is a negative power of $2$.
\end{proof}

\subsubsection{General case}

\begin{proof}[Proof of Lemma \ref{lemma:kDebut} in general case]
Fix $\upsilon >0$, $r$ large enough, $N$ large enough depending on $r$, and $\k$ such that  $ \k \leq \frac{N}{2} $. Unlike the previous dyadic case, for $ \k_+ \leq l \leq N-1$ and $p\leq 2^{\Delta^{(l)}}-1$, the random variables $ I_{l,p}(\theta)$ and $ I_{l,p}(\theta')$ are not rigorously independent. However observe that for any $\k_+ \leq l \leq N-1$, the absolute value of their correlations decreases exponentially with $l$. Indeed,
if $C_{l,p}(\theta,\theta') := \E \left( I_{l,p}(\theta)\overline{I_{l,p}(\theta')} \right)$, then
\begin{align*}
\left| C_{l,p}(\theta,\theta')\right| &=\frac{1}{{2^l+p2^{l-\Delta^{(l)}}}}\left|  \sum_{j=0}^{2^{l-\Delta^{(l)}}-1}  e^{i(\theta-\theta') j} \right|  \leq \frac{4}{2^l ||\theta-\theta'|| } \ll 2^{\k-l}.
\end{align*}
Since $ \E \left( I_{l,p}(\theta)I_{l,p}(\theta') \right) = 0$, and $(I_{l,p}(\theta), I_{l,p} (\theta'))$ is a centered complex Gaussian vector, one checks, by computing  covariances, that it is possible to write
 $$I_{l,p} (\theta) = \frac{C_{l,p}(\theta,\theta')}{C_{l,p} (\theta', \theta')}
 I_{l,p} (\theta')
 + I^{ind}_{l,p} (\theta, \theta'),$$
where the two terms of the sums are independent, with an expectation of the square equal to zero. Note that
 $$C_{l,p} := C_{l,p}(\theta', \theta')
 = \frac{2^{l-\Delta^{(l)}}}{2^l + p2^{l-\Delta^{(l)}}} = (2^{\Delta^{(l)}} + p)^{-1}$$
does not depend on $\theta'$. Moreover, we have by Pythagoras' theorem:
 $$\E [|I_{l,p}^{ind}(\theta, \theta')|^2] = \E [|I_{l,p}(\theta)|^2]
 - \left|\frac{C_{l,p}(\theta, \theta')}{C_{l,p}} \right|^2 \E [|I_{l,p}(\theta')|^2]
 = C_{l,p} - \frac{|C_{l,p}(\theta, \theta')|^2}{C_{l,p}}$$
 Using this decomposition of $I_{l,p}(\theta)$ and the measurability of the different quantities with respect to the $\sigma$-algebras of the form $\mathcal{G}_j$, we deduce that one can write:
\begin{align}
\label{decomposi}
 (R_{2^l}^{(2^{\k_+},\Delta)}(\theta))_{l\geq \k_+}= (R_{2^l}^{(2^{\k_+},ind)}+ E_{2^l}^{(2^{\k_+})} )_{l\geq\k_+} \ .
\end{align}
Here $(R_{2^l}^{(2^{\k_+},ind)})_{l\geq \k_+}$ is a Gaussian process, independent of $(R_{2^l}^{(2^{\k_+},\Delta)}(\theta'))_{l\geq \k_+}$, and distributed as $  ( \sqrt{\frac{1}{2}}W_{ \tau^{(\k_+)}_l- \mathtt{C}_l}  )_{l\geq \k_+}$ with
$$\mathtt{C}_l:= \sum_{t=\k_+}^{l-1}  \sum_{p=0}^{2^{\Delta^{(t)}} -1 } \frac{ |C_{t,p}(\theta, \theta')|^2}{C_{t,p}} \ .$$
Notice that $\mathtt{C}=\mathtt{C}^{(\k_+)}$ implicitly depends of $\k_+$. $(E_{2^l}^{(2^r)} )_{l\geq\k_+}$ on the other hand is defined by
\begin{align}
E_{2^{l+1}}^{(2^{l})}  = \Re \left(\sigma \sum_{p=0}^{2^{\Delta^{(l)}} -1 }e^{ i\psi_{ 2^{l}+p 2^{l-\Delta^{(l)}}} (\theta )   }
\frac{C_{l,p} (\theta, \theta')}{C_{l,p}}
 I_{l,p}(\theta') \right).  \label{DefE}
\end{align}

Furthermore notice that
\paragraph{Fact 1:} For any $l\geq \k_+$,
\begin{align*}
|E_{2^{l+1}}^{(2^{l})}| \leq |E|_{2^{l+1}}^{(2^{l})}:= \sum_{p=0}^{2^{\Delta^{(l)}} -1 } |C_{l,p}(\theta, \theta')| (2^{\Delta^{(l)}}+p ) |I_{l,p}(\theta') |
\end{align*}
is measurable with respect to the sigma field $\sigma\left(  \Nc_t^\C,\, t\in  [|2^{l},2^{l+1}-1|]\right)$.

\paragraph{Fact 2:} The process $(R_{2^l}^{(2^{\k_+},ind)})_{l\geq\k_+} $ is independent of the couple $( R_{2^l}^{(2^{\k_+},\Delta)}(\theta') , |E|_{2^l}^{(2^r)})_{l\geq \k_+}$

\paragraph{Fact 3:}  $\sup_{l\geq \k_+}|\mathtt{C}_l| \ll 2^{-\Delta^{(\k)}}$ if $r$ is large enough. Indeed, in this case,
since $\k_+ \geq 3 \Delta^{(\k)} \geq 3e^{\sqrt{\log r}}$
is also large, we have
\begin{align*}
\sup_{l\geq \k_+}|\mathtt{C}_l|
& \ll \sum_{l = \k_+}^{\infty}
\sum_{p=0}^{2^{\Delta^{(l)}} - 1}  (2^{\k - l})^2
(2^{\Delta^{(l)}}  + p)
\ll  \sum_{l = \k_+}^{\infty} 2^{2\k - 2l + 2 \Delta^{(l)}}
\ll \sum_{l = \k_+}^{\infty} 2^{2\k - 2l + 2e^{\sqrt{\log r}}+
200 \log^2 l}
\\ &  \ll \sum_{\k_+ \leq l \leq 100 \k}
2^{2\k - 2 l + 2e^{\sqrt{\log r}}
+ 300 \log^2 \k}
+ \sum_{l \geq \max (100\k, \k_+)} 2^{2\k - 2 l + 0.8 \k_+ + 200 \log^2 l}
\\ & \ll \sum_{\k_+ \leq l \leq 100 \k}
2^{2 \k - 2 l + 3 \Delta^{(\k)}}
+ \sum_{l\geq \max (100\k, \k_+)} 2^{2 \k + 0.8 \k_+ - 1.99 l}
\\ & \ll 2^{2 \k - 2 \k_+ +  3 \Delta^{(\k)}}
+ 2^{2 \k + 0.8 \k_+  - 1.99[(0.05)(100 \k) +
0.95 (\k_+)]}
\ll  2^{- 3 \Delta^{(\k)}} + 2^{- \k_+}
\ll 2^{- 3 \Delta^{(\k)}}.
\end{align*}
It means (see \cite[Lemma A.5]{CMN})
that the process $(R_{2^l}^{(2^{\k_+},ind)})_{l\geq \k_+}$ is very "similar" to the process
$ \sqrt{\frac{1}{2}} (W_{ \tau_{l}^{(\k_+)}})_{l\geq \k_+ }$.
 Moreover, if for $\k_+ \leq r \leq l$,
  $\mathtt{C}_l^{(r)} := \mathtt{C}_l -
  \mathtt{C}_r$, then we have for $r$ large enough:
  $$\sup_{l \geq r} |\mathtt{C}_l^{(r)}|
  \leq \sup_{l \geq \k_+}|\mathtt{C}_l|
  \leq 2^{-  \Delta^{(\k)}}
  \leq 2^{-e^{\sqrt{\log r}} } \leq 1/2.$$

\paragraph{Fact 4:} $|E|$ is small. For any $m \geq 0$, $l\geq \k_+$, we introduce the event
$$\mathtt{E}_m^{(l)}:=\{ |E|_{2^{l+1}}^{(2^l)}  \geq 2^{-\frac{\Delta^{(l)}}{4}}m\} \ .$$
For some universal constants $c, c' > 0$, and $m \geq 1/2$, the probability of $\mathtt{E}_m^{(l)} $ is smaller than
\begin{align}
 \nonumber  \P\left( |E|_{2^{l+1}}^{(2^l)}
 \geq 2^{-\frac{\Delta^{(l)}}{4}}m \right) \leq 2^{\Delta^{(l)}} \P\left( c 2^{\k-l} |I_{l,0}(\theta')|\geq m  2^{-\frac{9}{4}\Delta^{(l)}} \right) &\ll 2^{\Delta^{(l)}} e^{- c'm^2 2^{2(l-\k - \frac{7}{4}\Delta^{(l)}) }    }
\\
\label{fact4}&\ll  e^{- c'm^2 2^{2(l-\k - 2\Delta^{(l)}) }    }.
\end{align}
For $r$ (and then $\k_+$ and $l$) large enough and $l \leq 100 \k$, $2\Delta^{(l)} \leq 3 \Delta^{(\k)}
+(\log (c')/\log 4)$, and
then
$$ \P\left( |E|_{2^{l+1}}^{(2^l)} \geq 2^{-\frac{\Delta^{(l)}}{4}}m \right)  \ll e^{-m^2 2^{2(l-\k_+)} }.$$
If $r$ is large enough and $l \geq \sup(100 \k, \k_+)$ (and therefore,  also large), we use that
\begin{eqnarray*}
  2(l-\k-2\Delta^{(l)})&=&l-\k_++(l-2\k+\k_+-4\Delta^{(l)})
=l-\k_++(l-k+3\Delta^{(\k)}-4\Delta^{(l)})\\
&\geq&
l-\k+(l-\k-400(\log l)^2)\geq l-\k+(0.99 l-400 (\log l)^2)\\
&\geq& l-\k-\log c',
\end{eqnarray*}
and then, using \eqref{fact4},
$$\P\left( |E|_{2^{l+1}}^{(2^l)}  \geq
2^{-\frac{\Delta^{(l)}}{4}}m \right)
\ll e^{-c' m^2 2^{2(l - \k - 2\Delta^{(l)})}}
\leq  e^{-m^2(l-\k_+)}.
$$
Hence in any case, for $r$ large and $l \geq \k_+$,
$$\P\left(\mathtt{E}_m^{(l)}\right)=\P\left( |E|_{2^{l+1}}^{(2^l)}  \geq
2^{-\frac{\Delta^{(l)}}{4}}m \right)
\ll e^{-m^2 2^{l -\k_+}}.$$

{\bf When $\k_+\leq r$:} Using the decomposition (\ref{decomposi}) and  Fact 2, and noticing that 
$\sum_{l=r}^{+\infty} m 2^{-\frac{\Delta^{(l)}}{4}} \leq  m2^{-e^{\sqrt{\log r}}}$,
for $r$ large enough, we can affirm that
\begin{align}
\label{REM}
\nonumber &\P_N(\theta,\theta')\leq \P\left(  Ev(r,0,0,x,z)  \right) \P\left( Ev(r, 2^{-e^{\sqrt{\log r}}},
\mathtt{C}^{(r)},x',z')\right)
\\
&\qquad +\sum_{m\geq 1} \P\left(  Ev(r,0,0,x,z) ,\, \cup_{j\in [|r,N_+|]}\mathtt{E}_m^{(j)}  \right) \P\left( Ev(r,(m+1)2^{-e^{\sqrt{\log r}}},
\mathtt{C}^{(r)},x',z')\right),
\end{align}
where the Brownian motion involved in the event $Ev(r,0,0,x,z)$ is suitably coupled with the complex Gaussian random walk whose increments are of the form
$ I_{l,p}(\theta') e^{i \psi_{2^l + p2^{l-\Delta^{(l)}}} (\theta')}$
for $r \leq l \leq N-1$ and $0 \leq p \leq 2^{\Delta^{(l)}} - 1$.

By using Eq. \eqref{eq:girsanov2} and then the fact that 
$\sup_{r \leq j \leq N}\mathtt{C}^{(r)}_j \leq 2^{1-e^{\sqrt{\log r}}}$
if $r$ is large enough, we have for any $m \geq 0$:
\begin{align}
  \nonumber &\P\left( Ev(r,(m+1)2^{-e^{\sqrt{\log r}}},
\mathtt{C}^{(r)},x',z')\right)
\\
\nonumber &\leq e^{-2(x'-z')}2^{r -N_+} e^{2\upsilon + 2(m+1)2^{-e^{\sqrt{\log r}}}
- \mathtt{C}^{(r)}_N}N^{\frac{3}{2}}
                \P\left( GEv(r,(m+1)2^{-e^{\sqrt{\log r}}}
		+ \sup_{r \leq j \leq N}\mathtt{C}^{(r)}_j,\mathtt{C}^{(r)},x',z') \right)
\\
\label{subiel}&\leq e^{-2(x'-z')} 2^{r -N_+} e^{2\upsilon + 2(m+1) 2^{-e^{\sqrt{\log r}}}
- \mathtt{C}^{(r)}_N}N^{\frac{3}{2}} \P\left( GEv(r,2(m+2)2^{-e^{\sqrt{\log r}}},
\mathtt{C}^{(r)},x',z') \right) \ .
\end{align}
Now, we invoke \cite[Corollary A.6]{CMN}
as, by the Fact 3, with the notation of the corollary, $||E||_1 \leq 1$ if $r$ is large enough. Thus, we deduce, for $N$ large enough depending on $r$, that
\begin{align}
\label{kpluspetit0}
\P\left( Ev(r,(m+1)2^{-e^{\sqrt{\log r}}},
\mathtt{C}^{(r)},x',z')\right)
\ll_{\upsilon}  |x'z'|e^{2(x'-z')}2^{r -N} e^{ 2m 2^{-\frac{r}{400}} } (1 + (m+1)
2^{- e^{\sqrt{\log r}}})^3 \ .
\end{align}

Similarly, to compute $\P\left(Ev(r,0,0,x,z) ,\,  \cup_{j\in [|r,N_+|]}\mathtt{E}_m^{(j)}\right)$ we will apply the Girsanov transform with the density $e^{\sqrt{2 }W_{\tau_{N_+}^{(r)}}- \tau_{N_+}^{(r)} }$. It requires to study what is the effect of this density on the event $ \cup_{j\in [|r,N_+|]}\mathtt{E}_m^{(j)}$.
The increments of the complex random walk which
were $I_{j,p}(\theta') e^{i \psi_{2^j + p 2^{j - \Delta^{(j)}}}(\theta')}$ before the Girsanov transform, increase
by $\sigma^{-1} C_{j,p}$ afterwards. Hence, between the two situations, before and after the Girsanov
transform, $|E|_{2^{j+1}}^{(2^{j})}$, defined as the sum, for $0 \leq p \leq 2^{\Delta^{(j)}} - 1$, of
the absolute value of the increments of the random walk multiplied by $|C_{j,p} (\theta, \theta')|/C_{j,p}$, vary, for $r$ large enough, at most by $  2^{2\Delta^{(j)} +\k -j}$, since
$$\sum_{p = 0}^{2^{\Delta^{(j)}}-1}
|C_{j,p} (\theta, \theta')|
\ll 2^{\Delta^{(j)}+ \k - j}.$$
Now, for $j\geq r\geq \k_+$ and $r$ large enough, we  have:
$$  2^{2\Delta^{(j)} +\k -j}
< \frac{1}{2} 2^{-\frac{1}{4}\Delta^{(j)}}.$$
Indeed, for $j \leq 100 \k$,
$$ 2^{2\Delta^{(j)} +\k -j}
\leq 2^{2\Delta^{(j)} +\k -\k_+}
\leq 2^{2\Delta^{(j)} - 3\Delta^{(\k)}}
\leq 2^{-0.9\Delta^{(j)} },
$$
and for $j \geq \max(\k_+, 100 \k)$ (necessarily large, since $r$ is large),
\begin{align*} 2^{2\Delta^{(j)} +\k -j}
& \leq 2^{2\Delta^{(j)} -0.99j}
\leq 
2^{e^{\sqrt{\log r}}+ 200 \log^2 j  -0.99j}
\leq 2^{e^{\sqrt{\log j}}+ 200 \log^2 j  -0.99j}
\\ & \leq 2^{-0.96j} \leq 2^{-e^{\sqrt{\log r}}
- 0.94j}
\leq 2^{-\Delta^{(j)}}.
\end{align*}

Hence, if for $m \geq 1$, before (respectively after) the Girsanov transform, $ \cup_{j\in [|r,N_+|]}\mathtt{E}_m^{(j)} $ occurs, then  $  \cup_{j\in [|r,N_+|]}\mathtt{E}_{\half m}^{(j)}  $ still occurs after (respectively before) the transform. Finally we get, for any $m \geq 1$,
\begin{align*}
&\P\left(  Ev(r,0,0,x,z) ,\, \cup_{j\in [|r,N_+|]}\mathtt{E}_m^{(j)}  \right)\\
  &\leq  e^{2(x-z)}2^{r -N}  \E\left( e^{-\sqrt{2}W_{\tau_N^{(r)}}  }  \mathds{1}_{\{   \forall j\in [|r, N_+ |] ,\, l_j^{(N)}\leq x+\sqrt{\frac{1}{2}}  W_{\tau_j^{(r)}}  \leq u_j^{(N)} \}} \mathds{1}_{ \cup_{j\in [|r,N_+|]}\mathtt{E}_{\half m}^{(j)}   }   \right.\\
&\qquad\qquad\qquad\qquad \qquad  \left.
\mathds{1}_{x+\sqrt{\frac{1}{2}}  W_{\tau_{N_+}^{(r)}} \in -\frac34 \log N+[z,z+v)}\right)
\\
&\ll_{\upsilon} e^{2(x-z)}2^{r-N} N^{\frac{3}{2}} \P\left( GEv(r,0,0,x,z)\cap\left\{ \cup_{j\in [|r,N_+|]}\mathtt{E}_{\half m}^{(j)} \right\} \right)
\end{align*}
As $\mathtt{E}_{\frac{m}{2}}^{(l)}$ is measurable with respect to $\sigma( \Nc_t^\C,\, t\in  [|2^{l},2^{l+1}-1|])$, by applying  
\cite[Corollary A.6]{CMN}
 and using the {\bf Fact 4} we get, for $r$ large enough and $N$ large enough depending on $r$:
\begin{align*}
\P\left( GEv(r,0,0,x,z)\cap \left\{ \cup_{j\in [|r,N_+|]}\mathtt{E}_{\half m}^{(j)} \right\}  \right)   & \ll_{\upsilon} |xz| e^{2(x-z)}N^{-\frac{3}{2}} \sum_{j\geq r} \sqrt{\P(\mathtt{E}_{\frac{1}{2}m}^{(j)})  }\\
& \ll |xz| e^{2(x-z)} N^{-\frac{3}{2}}
\sum_{j \geq r} e^{-(m^2/8) 2^{j - \k_+}}\\&
 \ll |xz| e^{2(x-z)} N^{-\frac{3}{2}} e^{-2^{r-\k_+}m^2/8},
\end{align*}
since $r \geq \k_+$.

By combining this inequality with  \eqref{kpluspetit0} and \eqref{REM}, we get:
\begin{align}
\nonumber & \P_N(\theta, \theta',x,z,x',z')\ll_{\upsilon} |xzx'z'|e^{2(x-z+x'-z')}2^{2 (r-N_+)}  \\
& \nonumber  \qquad+   \sum_{m\geq 1}
|xzx'z'|e^{2(x-z+x'-z')}2^{r-N_+} 2^{2m 2^{-e^{\sqrt{\log r}}
}} (1+ (m+1) 2^{-e^{\sqrt{\log r}}})^3
2^{r-N} e^{-(m^2/8) 2^{r-\k_+}}
\\ & \leq |xzx'z'|e^{2(x-z+x'-z')}2^{2 (r-N_+)}   \left[ 1+ \left(\sum_{m \geq 1}
2^{2m} (1+m)^3 e^{-m^2/8} \right) e^{-(1/8)(2^{r-\k_+} - 1)} \right]\nonumber\\
&
\ll|xzx'z'|e^{2(x-z+x'-z')}2^{2 (r-N_+)}  \label{Gillet},
\end{align}
which concludes the case $\k_+\leq r$.

{\bf When $\k_+\geq r$:} Using the decomposition (\ref{decomposi}) and  Fact 2 and noticing that $\sum_{l=\k_+}^{+\infty} m 2^{-\frac{\Delta^{(l)}}{4}}\leq m2^{-e^{\sqrt{\log r}}/8}$
for $r$ large enough,  we can affirm that
\begin{align}
  \label{REM2}  \P_N(\theta,\theta',x,z,x',z') &\leq \sum_{m\geq 0} \P\left(  Ev(r,0,0,x,z) ,\, \cup_{j\in [|\k_+,N_+|]}\mathtt{E}_m^{(j)}  \right)\\
& \sup_{    l_{\k_+}^{(N)}  \leq 
x'+w\leq u_{\k_+}^{(N)} } \P\left( Ev(\k_+,(m+1) 2^{-e^{\sqrt{\log r}}/8}
,\mathtt{C},x'+w,z')\right).
\end{align}
 By a similar computation as what we have done in the case $\k_+ \leq r$, we get:
\begin{align}
\label{oto}
\P\left(  Ev(r,0,0,x,z) ,\, \cup_{j\in [|\k_+,N_+|]}\mathtt{E}_m^{(j)}  \right)
\ll_{\upsilon} |xz|e^{2(x-z)} 2^{r-N_+-(m^2/8)}.
\end{align}
On the other hand, by using Eq. \eqref{eq:girsanov2}, for any $w'\in [ l_{\k_+}^{(N)} , u_{\k_+}^{(N)} ]$,  we obtain:
\begin{align}
  \nonumber &\P\left( Ev(\k_+, (m+1) 2^{-e^{\sqrt{\log r}}/8}
  ,\mathtt{C},w',z')\right)
\\
\nonumber &\ll_{\upsilon}
e^{2(w'-z')}2^{\k_+  -N_+}  e^{   2(m+1) 2^{-\frac{r}{400}}} N^{\frac{3}{2}} \P\left( GEv(\k_+, (m+1) 2^{-e^{\sqrt{\log r}}/8}
+ \sup_{\k_+ \leq j \leq N} |\mathtt{C}_j|,\mathtt{C},w',z') \right)
\\
\label{subielbi}&  \leq
e^{2(w'-z')}2^{\k_+  -N_+}  e^{   2(m+1) 2^{-\frac{r}{400}}} N^{\frac{3}{2}}  \P\left( GEv(\k_+,   2(m+1) 2^{-e^{\sqrt{\log r}}/8}
,\mathtt{C},w',z') \right),
\end{align}
where we used that $\sup_{j\geq \k_+} |\mathtt{C}_j|\leq 2^{-\Delta^{(\k)}} \leq \frac{1}{2}(m+1) 2^{-e^{\sqrt{\log r}}/8}$.

For $\k_+ \leq N/4$, we can use \cite[Corollary A.6]{CMN}
 to deduce, for $w'\leq u_{\k^+}^{(N)}=- (\k_+)^{\alpha_-}$,
\begin{align}
  &\P\left( Ev(\k_+,   (m+1) 2^{-e^{\sqrt{\log r}}/8}
  ,\mathtt{C},w',z')\right) \nonumber\\&
  \ll e^{2(w'-z')} 2^{\k_+ -N_+} e^{   2(m+1) 2^{-e^{\sqrt{\log r}}/8}
  }  (1+ 2(m+1) 2^{-e^{\sqrt{\log r}}/8}
  )^3 |w'z'|  
\nonumber \\ & \ll_{\upsilon} |w'z'| e^{2(w'-z')}2^{\k_+ -N_+} e^{3(m +1)- (\k_+)^{\alpha_-}}\label{kpluspetit0bi} \ .
\end{align}

For $\k_+ \geq N/4$, we bound the probability of the $GEv$ event by $1$ and use the fact that $w'\leq  - 0.9(\k_+ )^{\alpha_-}$
(the factor $0.9$ coming from the case $\k \leq N/2 \leq \k_+$). We then get
\begin{align*}
  P\left( Ev(\k_+,   (m+1) 2^{-e^{\sqrt{\log r}}/8}
  ,\mathtt{C},w',z')\right) & \ll_{\upsilon}
  |w'z'|e^{2(w'-z')}2^{\k_+  -N_+}  e^{   2(m+1) 2^{-e^{\sqrt{\log r}}/8}
  }N^{\frac{3}{2}}
\\ & \ll_{\upsilon} |z'|e^{-2z'}2^{\k_+  -N_+}  e^{2m+2} e^{-(\k_+ )^{\alpha_-}} (N^{3/2} e^{-0.8(\k_+ )^{\alpha_-}})
\\ & \ll|z'|e^{-2z'}  2^{\k_+  -N_+}  e^{2m+2} e^{-(\k_+ )^{\alpha_-}}.
\end{align*}
 This again implies \eqref{kpluspetit0bi}.
Finally, by combining  this equation with (\ref{oto}) and (\ref{REM2}), we get
\begin{align*}
\P_{N}(\theta, \theta')& \ll_{\upsilon}  |xzz'|e^{2(x-z-z')} 2^{r-N_+} 2^{\k_+ -N_+}  e^{- (\k_+)^{\alpha_-}}   \sum_{m\geq 0}  e^{3m+3-(m^2/8)}\\
&
\ll  |xzz'|e^{2(x-z-z')} 2^{r-N_+} 2^{\k_+ -N_+}  e^{- (\k_+)^{\alpha_-}}
\end{align*}
which concludes the proof of Lemma \ref{lemma:kDebut}.
\end{proof}

\begin{proof}[Proof of Lemma \ref{lemma:ktoutDebut}]
We can use \eqref{subiel} in order to get (for $r$ large enough and $N$ large enough depending on $r$):
\begin{align*}
&\P(Ev(r,0,0,x,z))
\P(Ev(r, 2^{-e^{\sqrt{\log r}}/8}
, \mathtt{C}^{(r)}, x',z'))
\\ & \leq e^{2(x-z+x'-z')}e^{4 \upsilon + 2^{1-e^{\sqrt{\log r}}/8}}
2^{2(r-N_+)} N^{3}
\P(GEv(r,0,0,x,z))
\P(GEv(r,2^{1-e^{\sqrt{\log r}}/8}
,\mathtt{C}^{(r)},x',z')),
\end{align*}
and then, by the majorization of the second term of \eqref{REM} which is involved in \eqref{Gillet}:
\begin{eqnarray*}
&&  \P_N(\theta, \theta',x,z,x',z')\\
&\leq &
e^{2(x-z+x'-z')} e^{4 \upsilon + 2^{1-e^{\sqrt{\log r}}/8}}
2^{2(r-N_+)} N^{3}
\P(GEv(r,0,0,x,z))
\P(GEv(r,2^{1-e^{\sqrt{\log r}}/8}
,\mathtt{C}^{(r)},x',z'))\\
&&+ \mathcal{O}_{\upsilon}
\left( 2^{2(r-N_+)} e^{-\frac{2^{r-\k_+} - 1}{8}} \right).
\end{eqnarray*}
Hence, we have:
\begin{align*} \P_N(\theta, \theta')
  & \leq e^{2(x-z+x'-z')} e^{4 \upsilon + 2^{1 -e^{\sqrt{\log r}}/8}}
2^{2(r-N_+)}\\
&\qquad \times  N^{3}
\P(GEv(r,0,0,x,z))
\P(GEv(r,2^{1-e^{\sqrt{\log r}}/8}
,\mathtt{C}^{(r)},x',z'))
\\ & + \mathcal{O}_{\upsilon}
\left( 2^{2(r-N_+)}
 e^{-\frac{2^{r - (r/2)- 3e^{\sqrt{\log r}}/8 - 100 \log^2 (r/2)} - 1}{8}} \right).
 \end{align*}
By applying \cite[Lemmas 2.1,2.3]{BRZ} with the same argument as 
in Lemma \ref{lem-onerayLB},
using the fact that 
$||E||_1$  goes to zero when $r$ goes to infinity, we get
\begin{align*}
\P_N(\theta, \theta')
&\leq e^{2(x-z+x'-z')}
e^{4 \upsilon} 2^{2(r -N_+)} N^{3}   \P(Event_{r,N}(x,z)) \P(Event_{r,N}(x',z'))(1+ \eta_r)\\
&\qquad+  \mathcal{O}_{\upsilon}
\left( 2^{2(r-N)} e^{-\frac{2^{r/3} - 1}{8}} \right),
\end{align*}
 where $\eta_r$ goes to zero when $r$ goes to infinity. Using \eqref{momen1Lower}, we obtain  Lemma \ref{lemma:kDebut}.
 \end{proof}

\begin{proof}[Proof of Lemma \ref{lemma:kMil} in the general case]
The general case needs to uses exactly the same arguments as 
in the general case of the proof of  Lemma \ref{lemma:kDebut}.  This time, for $ \k_+ \leq l \leq N-1$ and $p\leq 2^{\kappa^{(\k)}}-1$, the random variables $ J_{l,p}^{(\k)}(\theta)$ and $ J_{l,p}^{(\k)}(\theta')$ are not rigorously independent. However, we observe that for $\k_+ \leq l \leq N_+$, the absolute value of their correlations,  decreases exponentially with $l$. Indeed, if $C_{l,p}^{(\k)}(\theta,\theta') := \E \left( J_{l,p}^{(\k)}(\theta)\overline{J_{l,p}^{(\k)}(\theta')} \right)$, then

\begin{align*}
\left| C_{l,p}^{(\k)}(\theta,\theta')\right| &=\frac{1}{{2^l+p2^{l-\kappa^{(\k)}}}}\left|  \sum_{j=0}^{2^{l-\kappa^{(\k)}}-1}  e^{i(\theta-\theta') j} \right|  \leq \frac{4}{2^l ||\theta-\theta'|| } \ll 2^{\k-l} \ .
\end{align*}

Since $ \E \left( J_{l,p}^{(\k)}(\theta)J_{l,p}^{(\k)}(\theta') \right) = 0$, and the vector
$(J_{l,p}^{(\k)}(\theta), J_{l,p}^{(\k)} (\theta'))$ is centered complex Gaussian, one checks, by computing  covariances, that it is possible to write
$$J_{l,p}^{(\k)} (\theta) = \frac{C_{l,p}^{(\k)}(\theta,\theta')}{C_{l,p}^{(\k)} (\theta', \theta')}
J_{l,p}^{(\k)} (\theta')
+ J^{(\k),ind}_{l,p} (\theta, \theta'),$$
where the two terms of the sums are independent, with an expectation of the square
equal to zero. Note that
$$C^{(\k)}_{l,p} := C^{(\k)}_{l,p}(\theta', \theta')
= \frac{2^{l-\kappa^{(\k)}}}{2^l + p2^{l-\kappa^{(\k)}}} = (2^{\kappa^{(\k)}} + p)^{-1},$$
does not depend on
$\theta'$. Moreover, we have by Pythagoras' theorem:
$$\E [|J_{l,p}^{(\k),ind}(\theta, \theta')|^2] = \E [|J_{l,p}^{(\k)}(\theta)|^2]
- \left|\frac{C_{l,p}^{(\k)}(\theta, \theta')}{C_{l,p}^{(\k)}} \right|^2 \E [|J_{l,p}^{(\k)}(\theta')|^2]
= C_{l,p}^{(\k)} - \frac{|C_{l,p}^{(\k)}(\theta, \theta')|^2}{C_{l,p}^{(\k)}}.$$
Using this decomposition of $J_{l,p}^{(\k)}(\theta)$ and the measurability
of the different quantities with respect to
the $\sigma$-algebras of the form $\mathcal{G}_j$, we deduce that one can write:
\begin{align}
\label{decomposiBIS}
(R_{2^l}^{(2^{\k_+},\kappa^{(\k)})}(\theta))_{l\geq \k_+}= (R_{2^l}^{(2^{\k_+},ind)}+ E_{2^l}^{(2^{\k_+})} )_{l\geq\k_+},
\end{align}
with $(R_{2^l}^{(2^{\k_+},ind)})_{l\geq \k_+}$ is a Gaussian process, independent of
$(R_{2^l}^{(2^{\k_+},\kappa^{(\k)})}(\theta'))_{l\geq \k_+}$, and
distributed as $  ( \sqrt{\frac{1}{2}}W_{ \tau^{(\k_+,\k_+)}_l- \mathtt{C}^{(\k)}_l}  )_{l\geq \k_+}$ with $\tau^{(\k_+,\k_+)}_j=  \sum_{l=\k_+}^{j-1} \sum_{p=0}^{2^{\kappa^{(\k)}}-1} (2^{\kappa^{(\k)}} +p)^{-1}$ and
$$\mathtt{C}^{(\k)}_l:= \sum_{t=\k_+}^{l-1}  \sum_{p=0}^{2^{\kappa^{(\k)}} -1 } \frac{ |C_{t,p}^{(\k)}(\theta, \theta')|^2}{C_{t,p}^{(\k)}} $$ and $(E_{2^l}^{(2^{\k_+})} )_{l\geq\k_+}$ defined by
\begin{align}
E_{2^{l+1}}^{(2^{l})}  = \Re \Bigg(\sigma \sum_{p=0}^{2^{\kappa^{(\k)}} -1 }e^{ i\psi_{ 2^{l}+p 2^{l-\kappa^{(\k)}}} (\theta )   }
\frac{C_{l,p}^{(\k)}(\theta, \theta')}{C_{l,p}^{(\k)}}
J_{l,p}^{(\k)}(\theta') \Bigg).  \label{DefEbis}
\end{align}
Note that $\mathtt{C}^{(\k)}_l$ and  $E_{2^{l+1}}^{(2^{l})}$ here represent quantities which are different from those denoted in the same way in the proof of Lemma \ref{lemma:kDebut}.
Furthermore notice that
\paragraph{Fact 1:} For any $l\geq \k_+$,
\begin{align*}
|E_{2^{l+1}}^{(2^{l})}| \leq |E|_{2^{l+1}}^{(2^{l})}:= \sum_{p=0}^{2^{\kappa^{(\k)}} -1 } |C_{l,p}^{(\k)}(\theta, \theta')| (2^{\kappa^{(\k)}}+p ) |J_{l,p}^{(\k)}(\theta') |
\end{align*}
is measurable with respect to the sigma field $\sigma\left(  \Nc_t^\C,\, t\in  [|2^{l},2^{l+1}-1|]\right)$.

\paragraph{Fact 2:} The process $(R_{2^l}^{(2^{\k_+},ind)})_{l\geq\k_+} $ is independent of the couple $( R_{2^l}^{(2^{\k_+},\kappa^{(\k)})}(\theta') , |E|_{2^l}^{(2^{\k_+})})_{l\geq \k_+}$.

\paragraph{Fact 3:}  $\sup_{l\geq \k_+}|\mathtt{C}_l^{(\k)}| \ll 2^{-\kappa^{(\k)}}$ if $r$ is large enough. Indeed we have
\begin{align*}
\sup_{l\geq \k_+}|\mathtt{C}_l^{(\k)}|
& \ll \sum_{l = \k_+}^{\infty}
\sum_{p=0}^{2^{\kappa^{(\k)}} - 1}  (2^{\k - l})^2
(2^{\kappa^{(\k)}}  + p)
\ll  \sum_{l = \k_+}^{\infty} 2^{2\k - 2l + 2 \kappa^{(\k)}}
\ll \sum_{l = \k + 3\kappa^{(\k)}}^{\infty} 2^{2 (\k +3\kappa^{(\k)} - l) -4\kappa^{(\k)}}   	\ll 2^{- 2\kappa^{(\k)}}.
\end{align*}
It means that the process $(R_{2^l}^{(2^{\k_+},ind)})_{l\geq \k_+}$ is very "similar" to $ \sqrt{\frac{1}{2}} (W_{ \tau_{l}^{(\k_+,\k_+)}})_{l\geq \k_+ }$.

\paragraph{Fact 4:} $|E|$ is small. For any $m \geq 0$, $l\geq \k_+$, we introduce the event
$$\mathtt{E}_m^{(l)}:=\{ |E|_{2^{l+1}}^{(2^l)}  \geq 2^{-\frac{\kappa^{(\k)}}{4}}m\} \ . $$
For some universal constants $c, c' > 0$, and $m  \geq 1/2$, the probability of $\mathtt{E}_m^{(l)} $ is smaller than
\begin{align*}
    \P\left( |E|_{2^{l+1}}^{2^l} \geq 2^{-\frac{\kappa^{(\k)}}{4}}m \right)
\leq 2^{\kappa^{(\k)}} \P\left( c 2^{\k-l} |J_{l,0}^{(\k)}(\theta')|\geq m  2^{-\frac{9}{4}\kappa^{(\k)}} \right)
& \ll 2^{\kappa^{(\k)}} e^{- c'm^2 2^{2(l-\k - \frac{7}{4}\kappa^{(\k)}) } }\\
& \ll e^{- c'm^2 2^{2(l-\k - 2\kappa^{(\k)}) } }.
\end{align*}
Then for $r$ large and $l \geq \k_+$,
$$\P\left( |E|_{2^{l+1}}^{2^l}  \geq 2^{-\frac{\kappa^{(\k)}}{4}}m \right) \ll e^{-m^2 2^{l -\k_+}} \ .$$

Using the decomposition (\ref{decomposiBIS}) and Fact 2 and noticing that $\sum_{l=\k_+}^{N} m 2^{-\frac{\kappa^{(\k)}}{4}}\leq m2^{-\frac{r}{400}}$ for $r$ large enough,  we can affirm that
\begin{align}
     & \P_N^{(\Delta,\kappa)}(\theta,\theta') \label{REM2BIS} \\
\nonumber
\leq & \sum_{m\geq 0} \P\left(  Ev(r,1,E^{(\k_+)},x,z) ,\, \cup_{j\in [|\k_+,N-1|]}\mathtt{E}_m^{(j)}  \right) \\
\nonumber
& \quad \quad \sup_{    l_{\k_+}^{(N)}-1  \leq x'+w'\leq u_{\k_+}^{(N)}+1  } \P\left( Ev(\k_+,1 +(m+1) 2^{-\frac{r}{400}},\mathtt{C}^{(\k)}+E^{(\k_+)},x'+w',z')\right).
\end{align}
Here, by abuse of notation, we refer to the same event $Ev(k, a, E, z)$ as in Eq. \eqref{eq:def_Ev_GEv_ext} but for the new time clock $\tau_.^{(\k_+, \k_+)}$. By the same arguments used to prove \eqref{Gillet} we have:
\begin{align}
\label{otoBIS}
\P\left(  Ev(r,1,E^{(\k_+)},x,z) ,\, \cup_{j\in [|\k_+,N-1|]}\mathtt{E}_m^{(j)}  \right)
\ll_{\upsilon}  e^{2(x-z)}|xz|2^{r-N_+-(m^2/8)}.
\end{align}
On the other hand, by using the inequality \eqref{eq:girsanov2} deduced from the Girsanov transform (which still holds for the time clock $\tau^{(\k_+,\k_+)}_j =  \sum_{l=\k_+}^{j-1} \sum_{p=0}^{2^{\kappa^{(\k)}}-1} (2^{\kappa^{(\k)}} +p)^{-1}$), we obtain for any $x'+w'\in [ l_{\k_+}^{(N)}-1, u_{\k_+}^{(N)}+1 ]$:
\begin{align}
\nonumber &\P\left( Ev(\k_+,1+ (m+1) 2^{-\frac{r}{400}},\mathtt{C}^{(\k)}+E^{(\k_+)},x'+w',z')\right)
\\
\label{subielbiBIS}&  \ll_{\upsilon}
2^{\k_+  -N_+}  e^{  1+ 2(m+1) 2^{-\frac{r}{400}}} N^{\frac{3}{2}} e^{2(x'+w'-z')  }\nonumber \\
&\qquad \times\P\left( GEv(\k_+,  1+ 2(m+1) 2^{-\frac{r}{400}},\mathtt{C}^{(\k)}+E^{(\k_+)},x'+w',z')) \right),
\end{align}
where we used that $\sup_{j\geq \k_+} |\mathtt{C}^{(\k)}_j|\leq 2^{-\kappa^{(\k)}} \leq (m+1) 2^{-\frac{r}{400}} $.  We bound the probability
of the $GEv$ event by $1$ and use the fact that
$x'+w' \leq 1+ u_{k_0}^{(N)}\leq 1 -(N_+-\k_+)^{\alpha_-}-\frac{3}{4}\log N$.
We then get
\begin{align*}
&P\left( Ev(\k_+,  1+ (m+1) 2^{-\frac{r}{400}} ,\mathtt{C}^{(\k)}+E^{(\k_+)},x'+w',z')\right) \\
&\ll
2^{\k_+ -N_+}  e^{   2(m+1) 2^{-\frac{r}{400}}} N^{\frac{3}{2}} e^{2(x'+w'-z')} 
\\ & \ll_{\upsilon} 2^{-2z'}2^{\k_+  -N_+}  e^{2m+2}   N^{3/2} e^{- 2(N_+ - \k_+)^{\alpha_-} -\frac{3}{2}\log N}
\\ & \ll 2^{-2z'}2^{\k_+  -N_+}  e^{2m+2} e^{-(N_+ - \k_+)^{\alpha_-}},
\end{align*}
Finally, by combining  this equation with (\ref{otoBIS}) and (\ref{REM2BIS}), we get
\begin{align*}
  \P_{N}^{(\Delta,\kappa)}(\theta, \theta')& \ll_{\upsilon} 
  |xz| 2^{2(x-z-z')}
  2^{r-N_+} 2^{\k_+ -N}  e^{- (N_+-\k_+)^{\alpha_-}}   
  \sum_{m\geq 0}  e^{2m+2-(m^2/8)}\\
  &
  \ll |xz|2^{2(x-z-z')} 2^{r-N_+} 2^{\k_+ -N}  e^{- (N_+-\k_+)^{\alpha_-}} \ ,
\end{align*}
which concludes the proof of Lemma \ref{lemma:kMil}.
\end{proof}

\subsection{Short time barriers}
We will also need the analogue of Lemmas \ref{lemma:kDebut} and \ref{lemma:ktoutDebut}
with $\I_N$ replaced by $\I_{t,f}$. Thus, let
\begin{align}
\label{eq:def_P_Nbis}
\P_{t,f}(\theta,\theta')=\P_{t,f}(\theta, \theta',x,z,x',z'):= \P(\I_{t,f}(\theta,x,z)\cap \I_{t,f}(\theta',x',z')).
\end{align}
We have the following.
\begin{lemma}[Time of branching $\k \leq N/2$]
  \label{lemma:kDebutbis}
  For any $\upsilon \in  (0,1)$, $r$ large enough, $C>1$,
  $t$ large enough depending on $r$,   $z,z'\in [-\sqrt{t}/C,C\sqrt{t}]$, 
  $\k \leq t$ and $2^{-\k}\leq \frac{||\theta-\theta'||}{2\pi} < 2^{-(\k-1)}$,  we have
  \begin{align}
  \label{eqkDebutbis}
  \P_{t,f}(\theta, \theta')&\ll \left\{
  \begin{array}{ll}
    \frac{|zz' xx'|}{t^3}  2^{2r-2t +2(x-z+x'-z')}                                    ,\qquad &\text{when } \k_+ \leq r,\\
    \frac{|zz'x|}{t^{3/2}} e^{2(x-z-z')}2^{r-t} 2^{\k_+-t} e^{- \k_+^{\alpha_-}} ,\qquad &\text{when } \k_+ \geq r.
  \end{array}
  \right.
  \end{align}
\end{lemma}

\begin{lemma}[Time of branching $\k \leq r/2$]
  \label{lemma:ktoutDebutbis}
  For any $\upsilon \in (0,1)$, $r\in \N$ large enough, $r$ large enough, $C>1$,
  $t$ large enough depending on $r$,   $z,z'\in [-\sqrt{t}/C,C\sqrt{t}]$, 
  $ \k \leq \frac{r}{2}$, $2^{-\k} \leq \frac{||\theta-\theta'||}{2\pi} \leq 2^{-(\k-1)}$ and  $t$ large enough depending on $r$, we have
  \begin{align*}
    \P_{t,f}(\theta, \theta')&\leq   (1+\eta_{r,\upsilon  })    2^{2r -2t}
    e^{2(x-z+x'-z')}    \P(Event_{r,N}(x,z) ) \P(Event_{r,N}(x',z')), 
  \end{align*}
  where
  $$\underset{\upsilon \rightarrow 0}{\limsup} \, \underset{r \rightarrow \infty} {\limsup} \, \eta_{r,\upsilon} = 0.$$
\end{lemma}

\section{Classical estimates on Gaussian walks}
\label{section:appendix}

The following estimates are classical and extend those in
\cite{CMN} by keeping track of the dependence in starting and ending points. 
In the following, the process $\left( W_s ; s \in \R_+ \right)$ is a standard Brownian motion. We use the notation from \ref{subsection:second_moment}.

The following is classical. Since the proof is short, we bring it.
\begin{lemma}
  \label{lem-onerayLB}
  Notation
as in  Section 
\ref{subsection:second_moment}, with $x,z$ in the appropriate range. 
Then,
\begin{equation}
  \label{eq-precise}
   (1-\eta_{N,r,k_1,v}) \sqrt{\frac{2}{\pi}}|xz|v\leq 
  N^{\frac{3}{2} }  \P(Event_{r,N}) \leq (1+\eta_{N,r,k_1,v}) 
  \sqrt{\frac{2}{\pi}}|xz|v,
\end{equation}
where
\[ 
  \lim_{v\to 0}
\lim_{r\to\infty}\lim_{k_1\to\infty}\lim_{N\to\infty} \eta_{N,r,k_1}=0.\]
\end{lemma}
\begin{proof}
  We show the lower bound, the upper bound is similar.
Recall that
\begin{align*}
& Event_{r,N} = Event_{r,N}(x,z)\coloneqq \\
&\left\{  \forall j\in [|r,N_+|],\, l_j^{(N)}  \leq x+ \sqrt{\frac{1}{2}} W_{\tau^{(r)}_j}
\leq u_j^{(N)} , x+ \sqrt{\frac{1}{2}} W_{\tau^{(r)}_{N_+}}\in \tau^{(r)}_{N_+}-\frac34 \log N+[z,z+\upsilon)\right\}.\end{align*}
For $r,k_1$ large, we have that $|\tau^{(r)}_j-(j-r)\log 2|\leq 1$.
Set, comparing with \eqref{eq-ukplus} and \eqref{eq-ukminus}, for $k\in [r,N_++1]$,
\begin{align}
  \label{eq-ukplus'}
u_k^{+,(N)}\coloneqq-2+\left\{ \begin{array}{ll} - k^{2\alpha_-},\qquad &\text{if  } k\leq \lfloor N/2\rfloor,
\\
 - (N-k)^{2\alpha_-} -   \frac{3}{4}  \log N,\qquad &\text{if  } \lfloor N/2\rfloor <k \leq N,
\end{array} \right.
\end{align}
and
\begin{align}
  \label{eq-ukminus'}
l_k^{+,(N)}\coloneqq2+\left\{ \begin{array}{ll} - k^{\alpha_+-\alpha_-},\qquad &\text{if  } k\leq \lfloor N/2\rfloor,
\\
 - (N-k)^{\alpha_+-\alpha_-} -    \frac{3}{4} \log  N,\qquad &\text{if  } \lfloor N/2\rfloor <k \leq N.
\end{array} \right.
\end{align}

The crucial fact is that for $N$ large we have that for all $k\in[r,N_+]$,
\begin{equation}
  \label{eq-comp+}
  \inf_{\theta\in [(k-1)\vee r, k+1]} u_k^{(N)} \geq u_k^{+,(N)},
  \quad
  \sup_{\theta\in [(k-1)\vee r, k+1]} l_k^{(N)} \leq l_k^{+,(N)}.
\end{equation}
In particular, we have that
\begin{align*}
 & Event_{r,N}'\coloneqq\\
&\left\{  \forall t\in [r,\tau^{(r)}_{N_+}],\, l_t^{+,(N)}  \leq 
x+ \sqrt{\frac{1}{2}} W_{t}
\leq u_t^{+,(N)} , x+ \sqrt{\frac{1}{2}} W_{\tau^{(r)}_{N_+}}\in \tau^{(r)}_{N_+}-\frac34 \log N+[z,z+\upsilon)\right\}\\
&\subset  Event_{r,N}.\end{align*} 
The conclusion of the lemma follows from a variant of
\cite[Lemma 2.1]{BRZ}, and our assumptions on $x,z$.
\end{proof}  

\newcommand{\etalchar}[1]{$^{#1}$}


\end{document}